\documentclass[letterpaper,11pt]{article}

\RequirePackage{amsthm,amsmath,amsfonts,amssymb}
\RequirePackage[authoryear]{natbib} %
\usepackage[CJKbookmarks=true,
            bookmarksnumbered=true,
            bookmarksopen=true,
            colorlinks=true,
            citecolor=blue,
            linkcolor=blue,
            anchorcolor=red,
            urlcolor=blue
            ]{hyperref}
\RequirePackage{graphicx}%
 
\usepackage{amssymb,amsfonts,amsmath,amstext,amsthm,mathrsfs}
\usepackage{graphics,latexsym,epsfig,psfrag,wrapfig,comment,paralist}
\usepackage{xspace}
\usepackage{subcaption,bm,multirow}
\usepackage{booktabs}
\usepackage{mathdots}
\usepackage{bbold}

\usepackage{algorithm, algorithmicx, algpseudocode}
\usepackage{prettyref}
\usepackage{listings}
\usepackage{amsmath}
\usepackage{tikz}
\usepackage{pgfplots}
\usetikzlibrary{shapes}
\usetikzlibrary{positioning, arrows, matrix, calc, patterns, fit}
\usepackage{caption}
\usepackage{array}
\usepackage{mdwmath}
\usepackage{multirow}
\usepackage{mdwtab}
\usepackage{eqparbox}
\usepackage{multicol}
\usepackage{amsfonts}
\usepackage{tikz}
\usepackage{multirow,bigstrut,threeparttable}
\usepackage{amsthm}
\usepackage{array}
\usepackage{bbm}
\usepackage{epstopdf}
\usepackage{mdwmath}
\usepackage{mdwtab}
\usepackage{eqparbox}
\usepackage{tikz}
\usepackage{latexsym}
\usepackage{amssymb}
\usepackage{bm}
\usepackage{graphicx}
\usepackage{mathrsfs}
\usepackage{epsfig}
\usepackage{psfrag}
\newcommand{\Gmean}{\GG_{\mathsf{mean}}}

\usepackage{pgfplots}
\pgfmathdeclarefunction{gauss}{3}{%
  \pgfmathparse{#3/(#2*sqrt(2*pi))*exp(-((x-#1)^2)/(2*#2^2))}%
}
\pgfmathdeclarefunction{mingauss}{4}{%
  \pgfmathparse{#4/(#3*sqrt(2*pi))*exp(-max((x-#1)^2, (x-#2)^2)/(2*#3^2))}%
}
\usetikzlibrary{shapes}
\usetikzlibrary{positioning, arrows, matrix, calc, patterns, fit}

\tikzset{
  invisible/.style={opacity=0},
  visible on/.style={alt={#1{}{invisible}}},
  alt/.code args={<#1>#2#3}{%
    \alt<#1>{\pgfkeysalso{#2}}{\pgfkeysalso{#3}}
  },
}

\newcommand{\tTV}{\widetilde{\mathsf{TV}}}
\newcommand{\tD}{\widetilde{D}}
\newcommand{\tW}{\widetilde{W}}
\newcommand{\iid}{i.i.d.\xspace}

\newtheorem{definition}{\textbf{Definition}}[section]
\newtheorem{corollary}{\textbf{Corollary}}[section]
\newtheorem{lemma}{\textbf{Lemma}}[section]
\newtheorem{theorem}{\textbf{Theorem}}[section]

\newtheorem*{insight*}{\textbf{Observation}}

\newtheorem{proposition}{\textbf{Proposition}}[section]
\newtheorem*{proposition*}{\textbf{Proposition}}

\newtheorem*{lemmai*}{\textbf{Lemma (informal)}}
\newtheorem{remark}{\textbf{Remark}}[section]

\newtheorem{example}{\textbf{Example}}[section]

\newcommand{\sH}{\mathcal{H}}

\newcommand{\bR}{\text{\boldmath{$R$}}}

\newcommand{\bbP}{\mathbb{P}}

\newcommand{\bbE}{\mathbb{E}}

\newcommand{\cM}{\mathcal{M}}

\newcommand{\cG}{\mathcal{G}}

\def\[#1\]{\begin{align}#1\end{align}}
\def\(#1\){\begin{align*}#1\end{align*}}

\def\argmax{\operatornamewithlimits{arg\,max}}
\def\argmin{\operatornamewithlimits{arg\,min}}

\newcommand{\bprf}{\begin{proof}}
\newcommand{\eprf}{\end{proof}}
\newcommand{\blem}{\begin{lemma}}
\newcommand{\elem}{\end{lemma}}
\newcommand{\eps}{\epsilon}

\newcommand{\eqdef}{\triangleq}
\newcommand{\bP}{\mathbb{P}}
\newcommand{\bE}{\mathbb{E}}
\newcommand{\sF}{\mathcal{F}}

\newcommand{\sU}{\mathcal{U}}

\newtheorem{assumption}[theorem]{Assumption}
\theoremstyle{definition}

\newcommand{\TV}{\mathsf{TV}}

\newcommand{\modu}{\mathfrak{m}}

\usepackage[normalem]{ulem}

\newcommand{\GG}{\mathcal{G}}
\newcommand{\MM}{\mathcal{M}}

\renewcommand{\cite}[1]{\citep{#1}}

\usepackage{color}

\title{Generalized Resilience and Robust Statistics}
\author{Banghua Zhu, Jiantao Jiao, Jacob Steinhardt\thanks{Banghua Zhu is with the Department of Electrical Engineering and Computer Sciences, University of California, Berkeley. Jiantao Jiao is with the Department of Electrical Engineering and Computer Sciences and the Department of Statistics, University of California, Berkeley. Jacob Steinhardt is with the Department of Statistics and the Department of Electrical Engineering and Computer Sciences, University of California, Berkeley. Email: \{banghua, jiantao,jsteinhardt\}@berkeley.edu.}}
\date{\today}

\begin{document}

\maketitle

\begin{abstract}

Robust statistics traditionally focuses on outliers, or perturbations in total variation distance. 
However, a dataset could be maliciously corrupted in many other ways, such as systematic measurement errors and missing covariates. 
We consider corruption in either $\TV$ or Wasserstein distance,
and show that robust estimation is possible whenever the true population distribution satisfies a property called \emph{generalized resilience}, which  holds under moment or hypercontractive conditions. %
For $\TV$ corruption model, our finite-sample analysis improves over previous results for mean estimation with bounded $k$-th moment, linear regression, and joint mean and covariance estimation. For $W_1$ corruption, we provide the first finite-sample guarantees for second moment estimation and linear regression.   %

Technically, our robust estimators are a generalization of minimum distance (MD) functionals, 
which project the corrupted distribution onto a given set of well-behaved distributions.  
The error of these MD functionals is bounded by a certain modulus of continuity, and we provide 
a systematic method for upper bounding this modulus for the class of generalized resilient distributions, 
which usually gives sharp population-level results and good finite-sample guarantees.
\end{abstract}

\tableofcontents
\newpage

\section{Introduction}\label{introduction}

We study the problem of robust estimation from high-dimensional corrupted data. Corruptions can occur in many forms, such as 
\emph{process error} that affects the outputs, \emph{measurement error} that affects the covariates, or some fraction of arbitrary 
\emph{outliers}. We will provide a framework for analyzing these and other types of corruptions, 
study minimal assumptions needed to enable robust estimation at the population level, 
corruptions, and construct estimators with provably good performance in finite samples.

We model corruptions in terms of a perturbation distance $D(p,q)$. 
Specifically, we posit a true population 
distribution $p^*$  that lies in some family of distribution $\mathcal{G}$, but observe 
samples $X_1, \ldots, X_n$ from a corrupted distribution $p$ such that  $D(p^*, p) \leq \epsilon$. %
Our goal is to output an estimate 
$\hat{\theta}(X_1, \ldots, X_n)$ such that some cost $L(p^*, \hat{\theta})$ is small. Note in particular 
that our goal is to estimate parameters of the original, uncorrupted distribution $p^*$.
As a result, even as $n \to \infty$ we typically incur some non-vanishing error that depends on $\epsilon$. We also consider a more powerful \emph{adaptive} model where  $X_1, \ldots, X_n$ are first sampled from true distribution, and then perturbed by adversary, which is formally defined in Section~\ref{subsec.corruption_models}. %

Throughout the paper, we focus on the case of corruption distance $D= \TV$ or $W_1$, though many of our results extend to 
Wasserstein distance over an arbitrary metric space. 
High dimensional robust statistics for $D = \TV$ has a long history.  
The majority of the classical statistics papers  focus  on  the minimum distance functional when the true distribution is   Gaussian or elliptical~\cite{huber1973robust,donoho1988automatic,adrover2002projection,huber2011robust,chen2002influence,  gao2017robust, gao2019generative}, while recent computationally efficient algorithms are instead based on assumed tail bounds via e.g.~moments or sub-Gaussianity~\cite{diakonikolas2017being, steinhardt2017resilience, steinhardt2017certified, diakonikolas2018learning, diakonikolas2018sever,  liu2018high,  chen2018robust, bateni2019minimax, lecue2019robust}.   
In this paper, we propose a different assumption called \emph{generalized resilience},  which is more general than either the widely 
used Gaussian or tail bound assumptions. Generalized resilience enables the systematic design of statistically efficient algorithms, 
which can also be efficiently computed in some cases. Furthermore, it gives near-optimal statistical rates even in the 
special case of Gaussian or tail-bounded distributions.

Corruptions under $\TV$ only allow an $\epsilon$-fraction of outliers or deletions. 
In many applications we might instead believe that all of the data have been slightly corrupted. 
We can model this by letting $D = W_1$ be the (standard) Wasserstein distance between $p$ and $q$, defined as the minimum cost in 
$\ell_2$-norm needed to move the points in $p$ to the points in $q$. For $W_1$ corruptions, mean estimation is trivial since the 
adversary can shift the mean by at most $\epsilon$. However, estimation of higher moments, as well as least squares estimation, are 
non-trivial and we focus on these in the $W_1$ case. %

In this paper, we connect ideas from both the classical and modern approaches to handling $\TV$ perturbations, and extend these ideas 
to Wasserstein perturbations. We summarize our main contributions as follows:
\begin{itemize}
    \item Motivated by  the minimum-distance (MD) functional \cite{donoho1988automatic} and the recent progress in efficient algorithms, we construct explicit non-parametric assumptions, \emph{generalized resilience}, under which MD functionals automatically give tight worst-case population error for both $\TV$ and $W_1$ perturbations, matching or improving previous bounds obtained under much stronger assumptions.  
   \item  We design  statistically efficient finite-sample algorithms based on MD functionals 
          for both $\TV$ corruption and $W_1$ corruption. We propose two different approaches, \emph{weakening the distance} and \emph{expanding the set},  that guarantee statistical efficiency and pave the way for designing computationally efficient algorithms.
  \end{itemize}
For $\TV$ corruption, our results improve the best existing bounds for tasks including mean estimation, linear regression and covariance estimation. For $W_1$ corruption, we are the first to provide any good robustness guarantee under natural assumptions.

\subsection{Main results for $\TV$ corruption (Section~\ref{sec.robust_TV})}

Throughout the paper, we design algorithms based on the \emph{minimum distance (MD) functional} estimator~\cite{donoho1988automatic}, which projects the corrupted empirical distribution $\hat p_n$ onto some set of distributions $\mathcal{M}$ under a discrepancy measure $\tilde D$:
\begin{equation}\label{eq.mdfintro}
\hat{\theta}(\hat p_n) = \theta^*(q), \text{ where } q = \argmin_{q \in \cM} \tilde D(q, \hat p_n), \theta^*(q) = \argmin_{\theta} L(q,\theta)
\end{equation}
Here $\cM$ and $\tilde D$ are design parameters to be specified (think of them as relaxations of $\cG$ and $D$).  In other words, this estimator projects the observed distribution $p$ onto the distribution set $\cM$ to get $q$, then outputs the optimal parameters for $q$.

\subsubsection{Design of set: generalized resilience (Section~\ref{sec.population_TV})}\label{subsec.intro_set_TV}
We begin with the main results for $\TV$ corruption. In the infinite sample case, if the true distribution $p^*$ lies in some family $\GG$ and we observe the population corrupted distribution $p$,  the performance of the MD functional estimator $q = \argmin_{q \in \GG} \TV(q, p)$ is upper bounded by the modulus of continuity (Lemma~\ref{lemma.population_limit},~\cite{donoho1988automatic}), defined as $\modu(\GG, 2\epsilon) = \sup_{p_1, p_2 \in \GG : \TV(p_1, p_2) \leq 2\epsilon} L(p_1, \theta^*(p_2))$. While the adversary can choose distributions outside of $\GG$, the modulus $\modu$ only involves pairs of distributions 
that lie within $\GG$, making it amenable to analysis. In mean estimation, when the set $\GG$ is taken as the set of resilient distributions~\cite{steinhardt2018resilience}, defined as 
\begin{align}\label{eqn.oldresilience}
\mathcal{G}_{\mathsf{mean}}(\rho, \epsilon) = \{p \mid \|\bE_{r}[X] - \bE_{p}[X]\| \leq \rho, \forall r\leq \frac{p}{1-\epsilon}\},
\end{align}
the modulus of continuity can be proved to be upper bounded by $2\rho$.  The notation $r \leq \frac{p}{1-\epsilon}$ indicates $r$ can be obtained from $p$ by conditioning on an event of probability $1-\epsilon$; thus \eqref{eqn.oldresilience} specifies the set of 
distributions whose mean is stable under deleting an $\epsilon$ fraction of points. 

The reason for the bounded modulus is a \emph{mid-point property} of $\TV$ 
distance: if $\TV(p_1, p_2) \leq \epsilon$ then there is a midpoint $r$ that can be obtained from either of the $p_i$ 
by conditioning on an event of probability $1-\epsilon$. Thus $\bE_r[X]$ is close to both $\bE_{p_1}[X]$ and $\bE_{p_2}[X]$ by 
resilience, and so $\bE_{p_1}[X]$ and $\bE_{p_2}[X]$ are close by the triangle inequality. 
The argument appears implicitly in \citet{steinhardt2018resilience} and~\citet{diakonikolas2017being}. Here we make it explicit  in Lemma~\ref{lem.G_TV_mean_modulus}.

Next, suppose that  the loss $L$ is arbitrary. We generalize 
\citeauthor{steinhardt2018resilience}'s definition of resilience to yield a family with bounded modulus for any given loss $L$.
For loss $L$ we will need two conditions: the first condition asks that the optimal parameters for $p$ do well on any $r\leq \frac{p}{1-\epsilon}$, while 
the second asks that if a parameter does well on $r$ then it also does well on $p$.
These are stated formally below:
\begin{align}
L(r, \theta^*(p)) &\leq \rho_1 \text{ whenever } r \leq \frac{p}{1-\epsilon}, \tag{$\downarrow$} \label{eq.down_intro} \\
L(p, \theta) &\leq \rho_2 \text{ whenever } L(r, \theta) \leq \rho_1 \text{ and } r \leq \frac{p}{1-\epsilon}. \tag{$\uparrow$} \label{eq.up_intro}
\end{align}

We show in Section~\ref{sec.robust_TV} that the family of distributions satisfying both \eqref{eq.down_intro} and \eqref{eq.up_intro} has modulus bounded by $\rho_2$ via the mid-point lemma (Theorem~\ref{thm.G_fundamental_limit}), and that it specializes to the 
family $\Gmean(\rho, \epsilon)$ for mean estimation by taking $\rho_1 = \rho$, $\rho_2 = 2\rho$.

While these conditions  may appear abstract, they  subsume bounded Orlicz-norm  type assumptions and provide a tight modulus 
for many concrete examples. 
More importantly, they pave the way for the universal finite-sample algorithms in the next section.

\subsubsection{Finite sample results for $\TV$ corruption}
A core difficulty in finite samples is that the $\TV$ distance between a continuous distribution $p$ and its empirical distribution $\hat p_n$ does not converge 
to zero  at all. Thus in the MD functional we cannot set $\tilde D = \TV$  when $\mathcal{M}$ only contains continuous distributions. We present two approaches to overcome the issue; the first approach, \emph{weaken the distance}, replaces $\TV$ with a weaker distance 
$\tTV$ that converges at a parametric rate, while the second approach, \emph{expand the set}, projects under $\TV$ to some set $\MM$ that is large enough. %

\paragraph{First approach: weaken the distance (Section~\ref{sec.finite_sample_algorithm_weaken})}
Intuitively, the issue is that the $\TV$ distance is too fine---it reports a large distance even 
between a population distribution $p$ and the finite-sample distribution $\hat{p}_n$. 
A solution to this is to relax the distance. We will replace $\TV$ with a smaller \emph{generalized Kolmogorov-Smirnov distance}: %
\begin{equation}
\tTV_{\sH}(p, q) \eqdef \sup_{f \in \sH, t \in \bR} |\bP_p(f(X) \geq t) - \bP_q(f(X) \geq t)|.
\end{equation}

The distance $\tTV$ is smaller than $\TV$ because it takes a supremum over only the events defined by threshold 
functions in $\sH$, while $\TV$ takes the same supremum over all measurable events. We then apply the minimum distance functional under the new distance $\tTV$.  The distance, when specified to the  case of Gaussian mean estimation, is analyzed in~\cite{donoho1988pathologies} with $\mathcal{H} = \{ v^{\top}x \mid v\in\bR^d\}$.

To show that projection under $\tTV$ works, we need to check two properties:
\begin{itemize}
\item The distance $\tTV_{\sH}(p, \hat p_n)$ between $p$ and its empirical distribution is small.
\item The modulus $\modu(\GG, \epsilon)$ is still bounded when replacing $\TV$ with $\tTV_{\sH}$.
\end{itemize}

As shown in Fig.~\ref{fig:intro_p_star}, if $\tTV_{\sH}(p, \hat p_n)$ is small, the triangle inequality implies that $\tTV_{\sH}(p^*, \hat p_n)$ is also small. 
Thus as long as the modulus  $\modu(\GG, \epsilon)$ is bounded under $\tTV$, the worst-case error for this projection algorithm can be bounded.  %

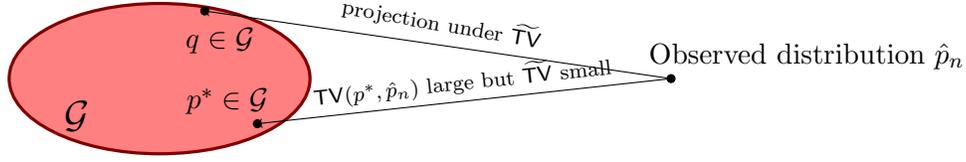
\begin{figure}[!htbp]
\begin{tikzpicture}
  \begin{scope}[shift={(0,0)}]
    \draw [opacity=1, very thick, fill=red!50, draw= red!50!black](2.4, 4.5) {ellipse (2cm and 1cm)};
     \draw  [](3.3, 4.2) node {$p^*\in\GG$};
     \draw  [](1.3, 4) node {\Large $\GG$};
       \draw [radius=1.5pt,fill](3.7, 3.9) circle;
      \draw  [](3.2, 5) node {$q\in\mathcal{G}$};
      \draw [radius=1.5pt,fill](3.0,5.4) circle; 
      \draw (9.2, 4.5) edge[->] node[above, rotate=-10]{{\scriptsize projection under $\tTV$}} (3.0,5.4);
     \draw  [](11,4.8) node {Observed distribution $\hat p_n$};
      \draw [radius=1.5pt,fill] (9.2, 4.5) circle;
    \draw (9.2, 4.5) edge[->] node[above, rotate=6,yshift=-1.5]{{\scriptsize $\TV(p^*, \hat p_n)$ large but $\tTV$ small }} (3.7, 3.9);
    \end{scope}
\end{tikzpicture}
\caption{Illustration of the analysis for \emph{weaken the distance}.}
\label{fig:intro_p_star}
\end{figure}
The first property is an instance of the VC inequality \citep{vapnik2015uniform,dudley1978central}, which implies that $\tTV_{\sH}(p, \hat p_n) = O(\sqrt{d/n})$ with high probability when the family of sets $\{\{x \mid f(x)\geq t\} \mid f\in \mathcal{H},t\in \mathbf{R}\}$ has VC-dimension $d$. 
With the first property, we can generally replace the perturbation level $\epsilon$ in the population worst-case risk with $\epsilon + \sqrt{(\mathsf{VC}(\mathcal{H})+\log(1/\delta))/n}$ for the finite-sample case. %

We establish the second property via a careful design of $\mathcal{H}$ for different tasks. The key underlying tool is 
a \emph{mean-cross} lemma %
showing that any pair of 
one-dimensional resilient distributions that are close in $\tTV$ have $\epsilon$-deletions whose means cross each other:

\begin{lemma}[Mean cross, Lemma~\ref{lem.mean_cross_tv}]\label{lem.intro_midpoint_KS}Suppose two distributions $p,q$ on the real line satisfy
$
\sup_{t\in \bR} |\bP_p(X\geq t) - \bP_q(Y\geq t)| \leq \epsilon.
$
Then one can find some $r_p \leq \frac{p}{1-\epsilon}$ and $r_q \leq \frac{q}{1-\epsilon}$
such that $r_p$ is stochastically dominated by $r_q$,
which implies that
$\mathbb{E}_{r_p}[X] \leq \mathbb{E}_{r_q}[Y]$.
\end{lemma}

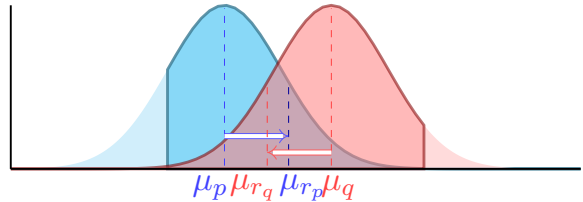
\begin{figure}[!htp]
    \centering
   \begin{tikzpicture}[
  implies/.style={double,double equal sign distance,-implies},
  scale=0.9,  every node/.style={scale=1.3}
]
\begin{axis}[
  no markers, domain=0:8, samples=50,
  axis lines*=left,
  every axis y label/.style={at=(current axis.above origin),anchor=south},
  every axis x label/.style={at=(current axis.right of origin),anchor=west},
  height=4cm, width=10cm,
  xtick=\empty, ytick=\empty,
  enlargelimits=false, clip=false, axis on top,
  axis line style = thick,
  grid = major
  ]
  \addplot [opacity=0.8,very thick,fill=cyan!20, draw=none, domain=0:2.2] {gauss(3.0,0.8,1)} \closedcycle;
  \addplot+[opacity=0.8,very thick,fill=cyan!50, draw=cyan!50!black, domain=2.2:8] {gauss(3.0,0.8,1)} \closedcycle;
  \addplot+[opacity=0.7,very thick,fill=red!20, draw=none, domain=5.8:8] {gauss(4.5,0.8,1)} \closedcycle;
  \addplot+[opacity=0.55,very thick,fill=red!50, draw=red!50!black, domain=0:5.8] {gauss(4.5,0.8,1)} \closedcycle;
\draw [color=blue!75, yshift=0.0cm, dashed](axis cs:3.0,0.01) -- (axis cs:3.0,0.49);
\draw [color=blue!75, yshift=-0.27cm](axis cs:2.8,0) node {$\mu_{p}$};
\draw [color=blue!75, yshift=-0.27cm](axis cs:4.1,0) node {$\mu_{r_{p}}$};
\draw [color=blue!60!black, yshift=0.0cm, dashed](axis cs:3.9,0.01) -- (axis cs:3.9,0.26);
\draw [color=blue!66](axis cs:3.0,0.1) edge[implies] (axis cs:3.9,0.1);
\draw [color=red!75, yshift=0.0cm, dashed](axis cs:3.6,0.01) -- (axis cs:3.6,0.26);
\draw [color=red!75, yshift=-0.27cm](axis cs:3.4,0) node {$\mu_{r_{q}}$};
\draw [color=red!75, yshift=0.0cm, dashed](axis cs:4.5,0.01) -- (axis cs:4.5,0.49);
\draw [color=red!75, yshift=-0.27cm](axis cs:4.6,0) node {$\mu_{q}$};
\draw [color=red!66](axis cs:4.5,0.05) edge[implies] (axis cs:3.6,0.05);
\end{axis}
\end{tikzpicture}
\caption{Illustration of mean cross lemma. For any distributions $p_1, p_2$ that are close under $\tTV$, we can truncate the $\epsilon$-tails of each distribution to make their means cross.}
\label{fig.mean_cross_tv}
\end{figure}

As a concrete application of the mean cross lemma, we illustrate how to bound the modulus for mean estimation: $\mathfrak{m} = \sup_{p,q: \tTV(p, q)\leq \epsilon, p, q\in\GG_{\mathsf{mean}}(\rho, \epsilon)} \| \bE_q[X] - \bE_{p}[X]\|$. We take $\sH$ in $\tTV_\sH$ to consist of functions of the form $f(x) = v^{\top}x$.  For any pair $p, q$ with $\tTV_\sH(p, q)\leq \epsilon$, one can find a projection vector $v^*$ with $\|v^*\|_2 \leq 1$ such that $ \| \bE_q[X] - \bE_{p}[X]\| = v^{*\top} (\bE_q[X] - \bE_{p}[X])$. Then from $\sup_{t\in \bR} |\bP_p( v^{*\top}X\geq t) - \bP_q( v^{*\top}Y\geq t)| \leq \epsilon$  and Lemma~\ref{lem.intro_midpoint_KS}, there exists some distribution $r_{p}, r_{q}$ such that $r_{p} \leq \frac{p}{1-\epsilon}$, $r_{q} \leq \frac{q}{1-\epsilon}$, and their mean is crossed, i.e. $v^{*\top}(\bE_{r_{q}}[X] - \bE_{r_{p}}[X]) \leq 0$, as is shown in Figure~\ref{fig.mean_cross_tv}. From $p, q\in\GG_{\mathsf{mean}}$ we know that $v^{*\top}\bE_q[X]$ is within $\rho$ distance to the mean of  deleted distribution $v^{*\top}\bE_{r_q}[X]$, $v^{*\top}\bE_p[X]$ is within $\rho$ distance to the mean of  deleted distribution $v^{*\top}\bE_{r_p}[X]$. Thus the distance between $v^{*\top}\bE_p[X]$ and $v^{*\top}\bE_q[X]$ is at most $2\rho$ from Figure~\ref{fig.mean_cross_tv}. This gives $\|\bE_p[X] - \bE_q[X]\|\leq 2\rho$. %

The mean-cross lemma provides a new and systematic approach for bounding the modulus of generalized resilience sets, while traditional approaches~\cite{chen2018robust, gao2019generative} are specific to Gaussian and elliptical assumptions. 

For other tasks beyond mean estimation, it suffices to design an appropriate $\sH$ in   $\tTV_\sH$ and apply the mean-cross lemma.
We provide a general way of constructing $\sH$  based on the Fenchel-Moreau dual representation of the cost function (Section~\ref{subsec.discussion_tTV}) and prove its effectiveness in Appendix~\ref{appendix.general_sHandtTV}. Intuitively, we want to represent the cost function as the supremum over the expectation of a set of 1-d functions of the data. Then we take $\sH$ as the set of 1-d functions.

As a concrete example, for linear regression the cost function considered is the excess loss $L(p, \theta) = \bE_p[(Y-X^{\top}\theta)^2 - (Y-X^{\top}\theta^*(p))^2]$. Assume $\TV(p^*, p) \leq \epsilon$ and let $\hat{p}_n$ denote the corrupted empirical distribution given $n$ samples. We take  $\sH$ to consist of functions of the form $f(x) = (v_1^{\top}x)^2 - (v_2^{\top}x)^2$, which gives the 
following result:

\begin{theorem}[A special case of Theorem~\ref{thm.linearregressiontvtildeproof}]\label{thm.intro_linreg_finite}
For $(X, Y)\sim p^*$,  let $Z = Y-X^{\top}\theta^*(p^*)$ denote the residual error. Suppose that
    \begin{align}
        \bE_{p^*}[(v^{\top}X)^4] &\leq \kappa^4 \bE_{p*}[(v^{\top}X)^2]^2 \text{ for all } v \in \bR^d, \text{ and} \\
          \bE_{p^*}[Z^4] &\leq \sigma^4.
    \end{align}
 Then the MD functional estimator $\hat{\theta}(\hat p_n)$ under $\tTV_\mathcal{H}$ projection  satisfies 
$L(p^*, \hat{\theta}(\hat{p}_n)) = O(\sigma^2\kappa^2(\epsilon + \sqrt{(d+\log(1/\delta))/n}))$ with probability at least 
$1-\delta$.
\end{theorem}

When $n \gg d/\epsilon^2$, our performance matches the infinite-data limit $O(\sigma^2\kappa^2\epsilon)$ up to constants. In contrast, \citet{klivans2018efficient} achieve a weaker error of $O(\sqrt{\epsilon})$ under the same assumptions on $p^*$ when $n \gg \mathsf{poly}(d^4, 1/\epsilon)$.  As a follow-up of our work, \citet{bakshi2020robust} achieves $O(\epsilon)$ error with sample size $O((d\log(d))^2)$, which is quadratic in dimension.  %

For the task of joint mean and covariance  estimation, the target cost functions considered are    $\|\Sigma_{p^*}^{-1/2}(\mu_{p^*} -\hat \mu) \|_2$ and $
    \| I_d - \Sigma_{p^*}^{-1/2}\hat\Sigma\Sigma_{p^*}^{-1/2}\|_2 $, where $\hat \mu$ and $\hat \Sigma$ are estimation of the mean and covariance, separately.  We  show that the same $\sH$ for mean estimation also works for joint mean and covariance estimation, which 
yields the following: %
\begin{theorem}[A special case of Theorem~\ref{thm.tTV_joint_multiplicative}]\label{thm.intro_tTV_joint_multiplicative}
Denote $\tilde \epsilon = 2\epsilon + 2C^{\mathsf{vc}}\sqrt{\frac{d+1+\log(1/\delta)}{n}}$. Assume $\tilde \epsilon < 1/2$, and that $p^*$ satisfies
\begin{align*} 
    \bE_{p^*}\left[ {({v^{\top}(X-\mu_p)})^4}\right] \leq {\kappa^4\bE_{p^*}[(v^\top (X-\mu_p))^2]^2} \text{ for all } v \in \bR^d.
\end{align*}
Let $\hat \mu(\hat p_n)$, $\hat \Sigma(\hat p_n)$ be  the MD functional estimator  under $\tTV_\mathcal{H}$ projection. Then there exist some $C>0$ such that when $\tilde \epsilon \leq C$, with probability at least $1-\delta$, 
\begin{align*}
     \|\Sigma_{p^*}^{-1/2}(\mu_{p^*} -\hat \mu(\hat p_n)) \|_2 & \lesssim \kappa\tilde \epsilon^{3/4}, \\
    \| I_d - \Sigma_{p^*}^{-1/2}\hat \Sigma(\hat p_n)\Sigma_{p^*}^{-1/2}\|_2 & \lesssim \kappa^2 \tilde \epsilon^{1/2}.
\end{align*}
\end{theorem}
This matches the dependence on $\epsilon$ in~\citet{kothari2017outlier} while improving the sample complexity  from $O((d\log(d))^k)$ to $O(d)$.  Theorem~\ref{thm.tTV_joint_multiplicative} in the main text provides a more general population-level result under the assumption of bounded Orlicz norm. %

\paragraph{Second approach: expand the set (Section~\ref{sec.expand})}\label{sec.intro_expand}
We show in this section that if the target set is large enough to cover the true ``uncorrupted'' empirical distribution, 
then projecting under $\TV$ distance has good finite-sample guarantees even though the population and empirical 
distribution are far apart in $\TV$. This analysis is closest in spirit to existing analyses of computationally-efficient 
estimators, and versions of it appear implicitly in the computation-oriented literature
\citep{lai2016agnostic,diakonikolas2019robust,diakonikolas2018robustly,steinhardt2018resilience,steinhardt2018robust}. 

For this analysis, we consider an adaptive corruption model (Definition~\ref{def.adaptivecont}), 
where we first sample from $p^*$ to get the empirical true distribution $\hat p_n^*$, and the adversary inspects the samples and 
modifies an $\epsilon'$ fraction of them. When $\epsilon' \asymp\epsilon+1/n$, the adaptive corruption is stronger than the 
level-$\epsilon$ oblivious corruption discussed above (where we corrupt the population distribution before sampling).

We describe the intuition of the analysis  here.  %
In all the tasks we considered, it suffices to   
find an estimator that does well on $\hat p_n^*$ due to standard generalization bounds ($L(\hat p_n^*, \theta)$ small implies $L(p^*, \theta)$ small when $n$ is large enough). Because of this, it suffices for the projection algorithm to find a $q$ that is close to 
$\hat p_n^*$, even if it is far from $p^*$ itself. More specifically, 
we rely on the following three conditions for the analysis to go through:
\begin{itemize}
\item $\MM$ is large enough (hence the name ``expand the set''): $\hat p_n^* \in \MM$ with high probability. %
\item The modulus is still bounded: $\modu(\MM, \epsilon)$ is small.
\item The empirical loss $L(\hat p_n^*, \theta)$ is a good approximation to the population loss $L(p^*, \theta)$.
\end{itemize}

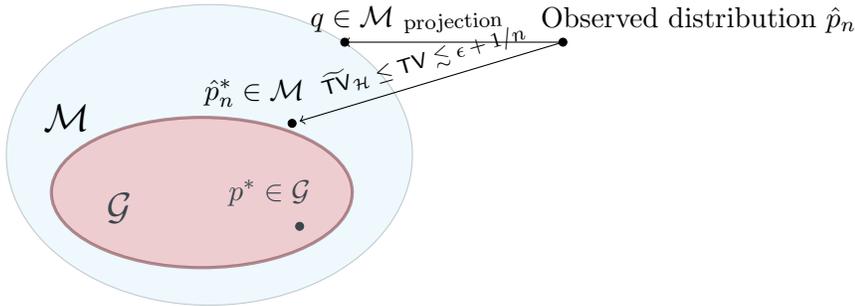
\begin{figure}[!htbp]
\begin{tikzpicture}
  \begin{scope}[shift={(0,0)}]
    {\draw [opacity=0.6, very thick, fill=red!40, draw= red!40!black](2.4, 4.5) {ellipse (2cm and 1cm)};
     \draw  [](3.3, 4.5) node {$p^*\in\GG$};
     \draw  [](1.3, 4.3) node {\Large $\GG$};
     \draw [radius=1.5pt,fill](3.7, 4.05) circle;
    }
      \draw [opacity=0.3, fill= cyan!20, draw= cyan!20!black](2.5, 5.0) {ellipse (2.7cm and 2 cm)};
    
     \draw  [](0.6, 5.5) node {\Large $\cM$};
       
     \draw  [](3.1, 5.82) node {$\hat p_n^*\in\cM$};
      \draw [radius=1.5pt,fill](3.60,5.42) circle;
    \draw  [](4.4, 6.8) node {$q\in\cM$};
      \draw [radius=1.5pt,fill](4.3,6.5) circle;
       \draw (7.2, 6.5) edge[->] node[above, rotate=16, xshift=-0.5]{{\scriptsize $\tTV_\sH\leq  \TV\lesssim \epsilon + 1/n$}} (3.7,5.45); 
    \draw (7.2, 6.5) edge[->] node[above, rotate=2]{{\scriptsize projection}} (4.3,6.5) ;
     \draw  [](9,6.8) node {Observed distribution $\hat p_n$};
      \draw [radius=1.5pt,fill] (7.2, 6.5) circle;
    \end{scope}
\end{tikzpicture}
\caption{Illustration of the analysis for \emph{expand the set}.}
\label{fig.intro_hatp}
\end{figure}

We actually need slightly weaker properties than the three above, as we discuss in Section~\ref{sec.expand}; for instance 
we only need to bound a certain generalized modulus that is smaller than $\modu(\MM, \epsilon)$.  %
In addition, the same argument can be applied for $\tTV_\mathcal{H}$ projection as well as $\TV$.

This analysis strategy requires a good choice of $\mathcal{M}$ such that its modulus is bounded and $\hat p_n^* \in\mathcal{M}$ with high probability.  As a concrete example, suppose the true distribution has bounded $k$-th moment. Although $\hat p_n^*$ does not have small $k$-th moment with less than $d^{k/2}$ samples, 
our key insight is that we can bound a certain \emph{linearized} $k$-th moment with only $\Theta(d)$ samples, which is {sufficient} to ensure {resilience} (see Lemma~\ref{lem.empirical_psi_resilient} for a rigorous statement and extension to any Orlicz norm). 
For instance, if $p^*$ has $4$-th moments bounded by $\sigma^4$ then we will bound $\sup_{\|v\|_2 \leq 1} \bE_{\hat p_n^*}[\psi(|v^{\top}X|)]$, 
where $\psi(x)$ is the smallest convex function on $[0,\infty)$ that coincides with $x^4$ for small $x$ (for instance, for $x\leq 4\sigma$ in the case that $n=d$). %
With the linearized moment technique, we are able to  design an MD functional  for
bounded $k$-th moment  isotropic distributions  as below:  %

\begin{theorem}[Theorem~\ref{thm.projection_kth_moment_identity_covariance}]\label{thm.intro_kth}
Suppose that $p^*$ has mean $\mu$, identity covariance, and bounded $k$-th moment: $\sup_{\|v\|_2 \leq 1} \bE_{p^*}[|v^{\top}(X-\mu)|^k] \leq \sigma^k$. 
Then, given $n$ samples, the MD functional which  projects the corrupted distribution $\hat p_n$ to the set of distributions with bounded covariances under either $\TV$ or $\tTV$ distance leads to an estimator with $\ell_2$ error 
$O(\sigma \cdot (\epsilon^{1-1/k} + \sqrt{d\log d / n}))$ with high probability.
\end{theorem}
In the past literature, the best general analysis yields a suboptimal complexity of $d^{1.5}$ \citep{steinhardt2018resilience}, while sum-of-squares based approach requires %
at least $d^{k/2}$ samples~\cite{kothari2017outlier}.
Our analysis achieves the optimal infinite sample error $\epsilon^{1-1/k}$ and near-linear sample complexity.  When combined with the statistical results in Theorem~\ref{thm.intro_kth}, the filtering algorithm in~\cite{diakonikolas2016robust, diakonikolas2017being, diakonikolas2019recent, zhu2020robust} achieves efficient computation. The follow-up work in~\citet{diakonikolas2020outlier} improved our results with a sub-gaussian rate and the same dependence on dimension.

\subsection{Main results for $W_1$ corruption (Section~\ref{sec.robust_W1})}

\subsubsection{Generalized resilience set for $W_1$ (Section~\ref{sec.population_w1})}\label{sec.W1_population_intro}

We show how to extend the idea of resilience to the Wasserstein distance $W_1$ (many of the ideas below also hold for 
Wasserstein distances over an arbitrary metric). As in the $\TV$ case, we focus on bounding the modulus 
of continuity, which remains an upper bound on the error of the MD functional. 

For $\TV$ distance, bounding the modulus crucially relied on the midpoint property that 
any $p_1$, $p_2$ have a midpoint $r$ obtained via \emph{deletions} of $p_1$ or $p_2$. 
In other words, we used the fact that any $\TV$ perturbation can be decomposed into a ``friendly'' 
operation (deletion) and its opposite (addition). We think of deletion as friendlier than addition, 
as the latter can move the mean arbitrarily far by adding probability mass at infinity.

To extend this to other Wasserstein distances, we need to identify a similar way of decomposing a 
Wasserstein perturbation into a friendly perturbation and its inverse. To do this, we 
re-interpret deletion in the following useful way:
\emph{Deletion is equivalent to movement towards the mean under $\TV$}. 
More precisely:
\begin{quote}
$\hat{\mu}$ is a possible mean of an $\epsilon$-deletion of $p$ if and only if some 
$r$ with mean $\hat{\mu}$ can be obtained from $p$ by moving points \emph{towards} $\hat{\mu}$ 
with $\TV$ distance at most $\epsilon$.
\end{quote}
This is more easily seen in the following diagram:
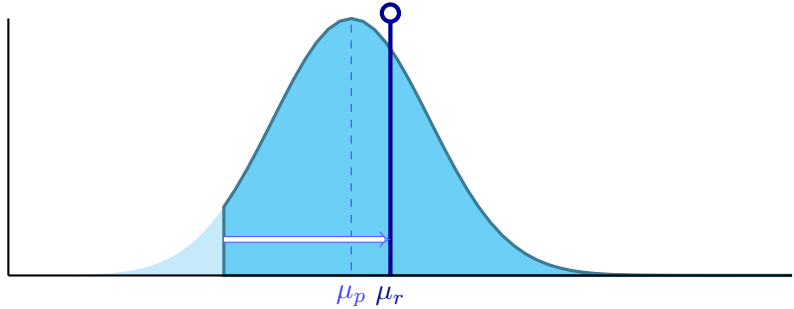
\begin{center}
\begin{tikzpicture}[
  implies/.style={double,double equal sign distance,-implies},
]
\begin{axis}[
  no markers, domain=0:8, samples=50,
  axis lines*=left,
  every axis y label/.style={at=(current axis.above origin),anchor=south},
  every axis x label/.style={at=(current axis.right of origin),anchor=west},
  height=5cm, width=12cm,
  xtick=\empty, ytick=\empty,
  enlargelimits=false, clip=false, axis on top,
  axis line style = thick,
  grid = major
  ]
  \addplot [very thick,fill=cyan!20, draw=none, domain=0:2.2] {gauss(3.5,0.8,1)} \closedcycle;
  \addplot [very thick,fill=cyan!50, draw=cyan!50!black, domain=2.2:8] {gauss(3.5,0.8,1)} \closedcycle;
\draw [color=blue!75, yshift=0.0cm, dashed](axis cs:3.5,0.01) -- (axis cs:3.5,0.49);
\draw [color=blue!75, yshift=-0.27cm](axis cs:3.5,0) node {$\mu_p$};
\draw [color=blue!60!black, yshift=0.0cm, dashed](axis cs:3.9,0.01) -- (axis cs:3.9,0.40);
\draw [color=blue!60!black, yshift=-0.27cm](axis cs:3.9,0) node {$\mu_r$};
\draw [color=blue!60!black, yshift=0.0cm, ultra thick, -o](axis cs:3.9,0.00) -- (axis cs:3.9,0.53);
\draw [color=blue!66](axis cs:2.2,0.07) edge[implies] (axis cs:3.9,0.07);
\end{axis}
\end{tikzpicture}
\end{center}
Here we can equivalently either delete the left tail of $p$ or 
shift all of its mass to $\mu_r$; both yield a modified distribution 
with the same mean $\mu_r$. This motivates the following informal definition for friendliness in $W_1$:
\begin{quote}
    A distribution $r$ is an $\epsilon$-friendly perturbation of $p$ for a function $f(x)$ and distance $W_1$ if and only if one can transport $X\sim p$ to $Y\sim r$ with cost no more than $\epsilon$, while ensuring $f(Y)$ is between $f(X)$ and $\bE_{r}[f(Y)]$ almost surely.
\end{quote}

The friendliness is defined only in terms of one-dimensional functions $f : \mathcal{X} \to \bR$; we will see how to handle 
higher-dimensional objects later. 
Intuitively, a friendly perturbation is a distribution $r$ for which there exists a coupling that ``squeezes'' $p$ to $\mu_r$. 
The key property of friendliness is the \emph{midpoint} lemma, which states that
every pair of nearby distributions has a friendly midpoint, analogously to the $\TV$ case:
\begin{lemma}[formal version Lemma~\ref{lem.wc_friend_pert}]\label{lemma.midpointintro}
    Given any $p,q$ with $W_{1}(p,q)\leq \epsilon$ and any  one-dimensional function $f$, there exists an $r$ that is an 
$\epsilon$-friendly perturbation of \emph{both} $p$ and $q$ for the function $f$.
\end{lemma}
We generalize resilience to Wasserstein distances by saying that a distribution is resilient if $\bE_r[f(X)]$ is close to 
$\bE_p[f(X)]$ for every $\epsilon$-friendly %
perturbation $r$ and every function $f$ lying within some appropriate family $\sF$.
For instance, for second moment estimation we would consider functions $f_v(x) = \langle x, v \rangle^2$ with $\|v\|_2 = 1$.
For more general losses $L(p, \theta)$, we show how to obtain an appropriate family $\sF$ via the Fenchel-Moreau 
representation~\cite{borwein2010convex} of $L$, as long as $L$ is convex in $p$ for fixed $\theta$. Convexity in $p$ is a mild condition that often holds, 
e.g.~any loss of the form $L(p, \theta) = \bE_{X \sim p}[\ell(\theta; X)]$ is linear (and hence convex) in $p$.
This construction  yields robust estimators in concrete cases, 
such as the $W_1$ linear regression
example. We discuss it in detail in the next section.

\subsubsection{Finite sample results for $W_1$ corruption (Section~\ref{sec.finite_sample_w1})}

We produce finite sample algorithms under $W_1$ perturbation by weakening $W_1$ to a distance $\tW_1$, similar to the 
$\TV \to \tTV$ weakening discussed above. The Kantorovich-Rubinstein duality yields the representation $W_1(p,q) = \sup\{ \bE_{p}[f(X)] - \bE_{q}[f(X)] \mid f \text{ is $1$-Lipschitz}\}$, so to weaken the $W_1$ distance we pick out a subset of $1$-Lipschitz functions; 
specifically, we will take the set of all $1$-Lipschitz linear or rectified linear functions.

We show that the weakened distance $\tW_1$  gives a mean cross lemma similar to that under $\TV$ perturbation: if $\tW_1(p, q) \leq \epsilon$ then for any convex function $f$ there are $7\epsilon$-friendly perturbations %
$r_p$, $r_q$ of $p$ and $q$ such that that $\bE_{r_q}[f(X)] \leq \bE_{r_p}[f(X)]$ (Lemma~\ref{lem:tw1_cross_mean}).
The main idea is to use the \emph{integral representation} of convex functions~\citep{lukevs2009integral} in one dimension, which states that any one-dimensional convex function can be decomposed 
into a weighted integral of linear or rectified linear $1$-Lipschitz functions. %

We present the first finite sample result for $W_1$ perturbations based on the two analysis techniques above. %
Assume $W_1(p^*, p )\leq \epsilon$ and let $\hat{p}_n$ denote the corrupted empirical distribution given $n$ samples. For linear regression where the cost function considered is    the predictive loss $L(p, \theta) = \bE_p[(Y-X^{\top}\theta)^2]$,  we obtain the following:

\begin{theorem}[A special case of Theorem~\ref{thm.linreg_tw_1_proj}]\label{thm.w1_intro}

For $(X, Y)\sim p^*$,  let $Z = Y-X^{\top}\theta^*(p^*)$ denote the residual error, where $\theta^*(p^*) \triangleq \argmin_{\theta \in \Theta} L(p^*,\theta), \Theta = \{\theta \mid  \|\theta\|_2 \leq R\}$. Let $X' = [X, Z] \in \bR^{d+1}$ be the concatenation of $X$ and $Z$. Assume that $p^*$ satisfies:
\begin{align}
    \bE_{p^*}[Z^2] \leq \sigma_2^2,  \bE_{p^*}\bigg[{|v^{\top}X'|^{3}}\bigg] \leq {\sigma_1^{3}} 
  \text{ for all } v\in\bR^{d+1}, \|v\|_2=1.
\end{align} 
Let $\hat\theta(\hat p_n)$  be the MD functional estimator under $\tW_1$ projection and $\bar R = \max(R, 1)$. Then with high probability, $L(p^*, \hat\theta(\hat p_n)) \leq \sigma_2^2 +  O(\sigma_1^{3/2}\bar R^{5/2}\sqrt{\tilde \epsilon} +  \bar R^4\tilde \epsilon^2)$. Here $\tilde \epsilon = O\left(\epsilon + \sigma_1 \bar R\sqrt{ d/n}+ \sigma_1 \bar R/n^{1/4}\right)  $.
\end{theorem}
In $W_1$ perturbation, $\epsilon$ has the same units as $X$, hence $\sigma_1 \bar R$ matches the unit of $\epsilon$.%
The sample complexity needed to achieve the same performance as the population is 
$O(\frac{d}{\epsilon^2} + \frac{1}{\epsilon^{4}})$. %

We also generalize the expanding the set idea and  provide the first computationally efficient result for second moment estimation via $W_1$ projection: %
\begin{theorem}[A special case of Theorem~\ref{thm.efficient_sec_w1}]
Assume  the adversary is able to corrupt the true empirical distribution $\hat p_n^*$ to some distribution $\hat p$ such that $W_1(\hat p_n^*, \hat p)\leq \epsilon$, and $p^*$ satisfies
    $\sup_{v\in\bR^d, \|v\|_2=1}\bE_{p^*}[|v^\top X|^4] \leq \sigma^4$
for some $\sigma>0$.  When $n\gtrsim (d\log(d/\delta))^{2}$, we can design an MD functional estimator $\hat M(\hat p)$ under $W_1$ or $\tW_1$ projection that  satisfies
\begin{align*}
    \| \hat M(\hat p) - \bE_{p^*}[XX^{\top}]\|_2 \leq \min(\sigma^2, C_1 \sigma^{4/3}  \epsilon^{2/3})+C_2\sigma^2 \max(\sqrt{\frac{ d\log(d/\delta)}{n\delta^{1/2}}},  \frac{ d\log(d/\delta)}{n\delta^{1/2}})
\end{align*}
with probability at least $1-\delta$, 
where $C_1, C_2$ are some universal constants.
\end{theorem}
The MD functional can be computed via sum-of-squares programming if we strengthen the bounded moment assumption to sum-of-squares certifiable. This gives the first computationally efficient algorithm under Wasserstein perturbation.

We summarize the analysis techniques and frameworks for both \emph{weaken the distance} and \emph{expand the set} algorithms in Appendix~\ref{sec.general_analysis}. The two analyses can be further extended to corruption under a general discrepancy metric and loss function.  Furthermore,  the spirit of two analyses can be applied to distributional robust optimization framework. We  discuss their connections in Appendix~\ref{sec.connection_dro}. %

\section{Preliminaries}\label{sec.preliminaries}

In this section, we provide definitions for frequently used terms throughout the paper, 
introduce the two corruption models considered in this paper, and discuss the  information theoretic limits for robust estimation. 

\subsection{Notations and Definitions}
\label{subsec.definitions}

We first collect important notations and definitions  throughout this paper. The Wasserstein-1 distance is defined as  $W_{1}(p, q) =  
    \inf_{\pi_{p, q} \in \Pi(p, q) } \int \|x-y\| d \pi_{p, q} (x, y)$, where  $\Pi(p, q)$ denotes the set of all couplings between $p$ and $q$. A function $\psi: [0, +\infty) \mapsto [0, +\infty)$ is called an \emph{Orlicz function} if $\psi$ is convex, non-decreasing, and satisfies $\psi(0) = 0$, $\psi(x) \to \infty$ as $x\to \infty$. For a given Orlicz function $\psi$, the Orlicz norm of a random variable $X$ is defined as
$  \| X \|_{\psi} \triangleq \inf\left \{ t>0 : \bE_p\left [\psi\left(|X|/t\right )\right ] \leq 1\right \}.$
For univariate random variables $X$ and $Y$, we say that
$Y$ \emph{stochastically dominates} $X$ (in first order) 
if $\bP(X \leq t)\geq  \bP(Y \leq t)$ for all $t \in \bR$ \citep[Definition B.19.a, B.19.b]{marshall1979inequalities}.
We denote this as $X \leq_{sd}Y$ or $P_X \leq_{sd} P_Y$. It implies that $\bE[X] \leq \bE[Y]$. A \emph{pseudometric} is a function $d:X\times X \rightarrow\mathbf{R}_+$ satisfying the 
following three properties:
$d(x,x)=0$, $d(x,y) = d(y,x)$ (symmetry), and $d(x,z)\leq d(x,y)+ d(y,z)$ (triangle inequality).
Unlike a metric, one may have $d(x,y)=0$ for $x\neq y$.
Similarly, a function $\| \cdot\|:X \rightarrow\mathbf{R}_+$ is a \emph{pseudonorm} if $\|0\| = 0$, $\|cx\| = |c| \|x\|$ (homogeneity), and $\|x+y\|\leq \|x\|+\|y\|$ (triangle inequality). 
We define the generalized inverse of a non-decreasing function $\psi$ as 
$\psi^{-1}(y) = \inf\{x \mid \psi(x) > y\}$%
The rest of the notations are collected in the beginning of Appendix~\ref{sec.appendixgeneralemmas}.

\subsection{Corruption Models}
\label{subsec.corruption_models}

We consider two types of corruption model under perturbation discrepancy $D$: 
corruptions either to the population distribution (with samples drawn from the corrupted distribution), 
or to the empirical distribution on $n$ samples (with samples originally drawn from the true distribution $p^*$ and then perturbed).
The latter defines a more powerful adversary that is allowed to make decisions after seeing the random draw 
generating the data. It recovers the \emph{finite sample robustness} model of~\citep{donoho1988automatic} and the \emph{full adversary} setup in~\citet{diakonikolas2019robust} when we take $D = \TV$. We accordingly call the first model the \emph{oblivious} model and the second the 
\emph{adaptive} model:

\begin{definition}[Oblivious Corruption]\label{def.obliviouscorruption}
Denote the true distribution as $p^*$. The oblivious adversary with level $\epsilon$ under $D$ is allowed to do the following. It first perturbs $p^*$ to $p$ such that $D(p^*, p)\leq \epsilon$, then takes $n$ $i.i.d.$ samples from $p$ to get $(X_1, X_2, \ldots, X_n)$. We denote the empirical distribution of the observations $(X_1, X_2, \ldots, X_n) $ as $\hat p_n$.
\end{definition} 

\begin{definition}[Adaptive Corruption]\label{def.adaptivecont}
Denote the true distribution as $p^*$. The adaptive adversary with level $\epsilon$ under $D$ is allowed to do the following. We take $n$ i.i.d. samples from $p^*$ to get the empirical distribution $\hat p_n^*$. The adversary inspects $\hat p_n^*$ and produces another empirical distribution supported on $n$ points, denoted as $\hat p_n$. The conditional distribution $\hat p_n \mid \hat{p}_n^*$ is allowed if there exists some coupling $\pi_{X,Y}$ satisfying $\pi_X = p^*, D(\pi_X, \pi_Y) \leq \epsilon$ such that $D(\hat p_n^*, \hat p_n)$ is stochastically dominated by $D\left(\frac{1}{n}\sum_{i = 1}^n \delta_{X_i}, \frac{1}{n}\sum_{i=1}^n \delta_{Y_i} \right )$, where $\{(X_i,Y_i)\}_{i = 1}^n$ denotes $n$ i.i.d. samples from $\pi_{X,Y}$. 
\end{definition} 

From the definition of adaptive corruption, we see that oblivious corruption is a special case when $\pi_Y = p$. Thus the adaptive corruption with level $\epsilon$ is always stronger than oblivious corruption with level $\epsilon$. We characterize the behavior of $D(\hat p_n^*, \hat p_n)$  under adaptive corruption model in Appendix~\ref{proof.lem_coupling}.

\subsection{Population (Infinite-Sample) Limit}
\label{sec.population_limit}

It will be helpful to consider the infinite-sample regime where $\hat{p}_n^* = p^*$ and $\hat{p}_n = p$; in this case 
both corruption models defined above become identical. We will be interested in understanding the \emph{population limit}, 
or the best (worst-case) error achievable by any estimator $\theta(p)$:
\begin{definition}[Population limit with $(\mathcal{G},D,L, \epsilon)$] \label{def.populationlimit}
Given distribution family $\mathcal{G}$, discrepancy $D$, loss $L$, and perturbation level $\epsilon$, 
the population limit for the robust inference problem is
\begin{align}\label{eqn.fundamental_limit}
   \inf_{\theta(p)} \sup_{(p^*,p): D(p^*,p)\leq \epsilon, p^*\in \mathcal{G}} L(p^*, \theta(p))
\end{align}
\end{definition}

This population limit is defined in~\citet[Prop. 4.2]{donoho1988automatic}. It draws a precise connection between robust statistics and robust optimization, which provides one potential answer to the questions raised in~\citet[page xvi and 339]{ben2009robust}.

\subsubsection*{Modulus of continuity bounds population limit}
The following bound on the population limit will guide our design of $\GG$ in the sequel: 

\begin{lemma}\cite{donoho1988automatic} \label{lemma.population_limit}
Suppose $D$ is a pseudometric. Then the population limit in~(\ref{eqn.fundamental_limit}) is 
at most the maximum loss between any pair of 
distributions in $\GG$ of distance at most $2\epsilon$:
\begin{align}
\label{eqn.modulus}
\modu(\GG, 2\epsilon, D, L) \triangleq \sup_{p_1, p_2 \in \GG:  D(p_1, p_2) \leq 2 \epsilon} L(p_1, \theta^*(p_2)).
\end{align}
The upper bound is achieved by the projection (minimum distance functional) algorithm $q = \argmin\{D(q,p) \mid q \in \GG\}$, where we output $\theta = \theta^*(q) = \argmin_{\theta \in \Theta} L(q, \theta)$.
\end{lemma}

The quantity $\modu$ 
is called the \emph{gauge function} or \emph{modulus of continuity}~\cite{donoho1988automatic, chen2018robust}. Later we omit $D, L$ in the parameters if they are obvious from context. 
The projection procedure in Lemma~\ref{lemma.population_limit} is oblivious
to the cost $L$ as well as the perturbation level $\epsilon$, which are desirable features in practice. When we know the perturbation level $\epsilon$, then the upper bound also holds if output $\theta^*(q)$ for \emph{any} $q \in \GG$ such that $D(p, q) \leq \epsilon$.
We show in Lemma~\ref{lemma.population_limit_opt} that the projection algorithm is optimal up to constants if $L(p, \theta^*(q))$ is a pseudometric over $(p, q)$, and the modulus of continuity is a nearly tight upper bound if we further assume the space induced by $D$ is a complete convex metric space~\citep[Theorem 1.97]{penot2012calculus} such as $\TV$ or $W_1$.

\section{Robust Inference under $\TV$ corruption}\label{sec.robust_TV}

In this section, we focus on the robust inference problem under $\TV$ corruption. We begin with introducing the population assumption as generalized resilience, and show the generality of the assumption via finite sample results for mean estimation, linear regression and joint mean and covariance estimation under transformed space. Our results match the population limit in the infinite sample case in terms of dependence on corruption level $\epsilon$, and give near optimal finite-sample rate. Furthermore, they enable the design of computationally efficient algorithms.

\subsection{Population assumption: generalized resilience}\label{sec.population_TV}

The \emph{distribution family $\GG^\TV$} encodes our assumptions about the true distribution $p^*$. 
Restricting $p^*$ to some non-trivial family $\GG^\TV$ is generally necessary in the robust setting. Take for instance 
robust mean estimation under $\TV$ corruption; then two distributions 
could be nearby in $\TV$ but have arbitrarily different means due to a small amount of mass at $\infty$. We would like to have some assumption $\GG^\TV$ with the following property:
\begin{enumerate}
   \item \emph{Not too big:} the modulus of continuity of this set can be controlled (and thus population limit by Lemma~\ref{lemma.population_limit}); 
   \item \emph{Not too small:} $\GG^\TV$ is a superset of  most of the widely used assumptions underlying the recent provably computationally efficient algorithms; %
   \item \emph{Near-optimal finite sample algorithm:} we can design a general finite-sample algorithm that guarantees robust inference for the $\GG^\TV$ we designed. It is near-optimal for known cases, and can inspire the design of computationally efficient algorithms. 
\end{enumerate}

As is discussed in Section~\ref{subsec.intro_set_TV}, in the task of mean estimation, the resilience set defined in~\eqref{eqn.oldresilience} has bounded modulus and subsumes tail bound type assumptions. %
We provide rigorous lemma for bounding the modulus as below. 

\begin{lemma}\label{lem.G_TV_mean_modulus}
The modulus of continuity $\modu$ in (\ref{eqn.modulus}) for $\GG^\TV_{\mathsf{mean}}(\rho, \eta)$ in~\eqref{eqn.oldresilience}
satisfies the bound
$    \modu(\GG^\TV_{\mathsf{mean}}(\rho, \eta), 2\epsilon) \leq 2\rho$
for any $ 2\epsilon \leq \eta < 1$. 
\end{lemma}

By Lemma~\ref{lemma.population_limit}, the population limit for 
$\GG^\TV_{\mathsf{mean}}(\rho, \eta)$ is at most $2\rho$ when $2\epsilon \leq \eta < 1$. 
\begin{proof}[Proof of Lemma~\ref{lem.G_TV_mean_modulus}]

Denote $\mu_{p_1} = \bE_{p_1}[X], \mu_{p_2} = \bE_{p_2}[X]$. Recall that
\begin{align}
    \modu(\GG^\TV_{\mathsf{mean}}(\rho, \eta), 2\epsilon) = \sup_{(p_1, p_2):  \TV(p_1, p_2) \leq 2 \epsilon,  p_1, p_2 \in \GG^\TV_{\mathsf{mean}}(\rho, \eta) } \|\mu_{p_1}- \mu_{p_2} \|. 
\end{align}
From $\TV(p_1, p_2)\leq 2\epsilon \leq \eta$,
we know that the distribution $r = \frac{\min(p_1, p_2)}{1-\TV(p_1,p_2)}$ satisfies 
$r\leq \frac{p_1}{1-\eta}$ and $r\leq \frac{p_2}{1-\eta}$ (see Lemma~\ref{lem.deletionrproperties} for a formal proof). 
It then follows from $p_1, p_2 \in \GG_{\mathsf{mean}}^{\TV}(\rho, \eta)$ and the triangle inequality that
$\|\mu_{p_1} - \mu_{p_2} \| \leq \|\mu_{p_1} - \mu_{r} \| + \|\mu_{p_2} - \mu_{r} \| \leq 2\rho$.
\end{proof}

A close inspection of the proof above shows that the loss can be generalized from $\| \mu_{p} - \mu_q\|$ to all pseudonorms $W_\sF(p,q) = \sup_{f\in\sF}\bE_p[f(X)] - \bE_q[f(X)]$. We discuss it in Appendix~\ref{sec.resilienceforwfdefinition}. 
We defer the discussion on the relationship with tail bound type assumptions to Section~\ref{sec.finite_sample_algorithm_weaken}. %

Now we extend the definition of {resilience} for mean estimation to arbitrary cost 
functions $L(p, \theta)$ that may not satisfy the triangle inequality. 
The general definition below imposes two conditions: (1) the optimal parameter for the distribution $\theta^*(p)$ should do 
well on all deleted distributions $r \leq \frac{p}{1-\eta}$, and (2) any parameter that does 
well on some deleted distribution $r \leq \frac{p}{1-\eta}$ also does well on $p$. We measure performance on 
$r$ with a \emph{bridge function} $B(r, \theta)$, which is often the same as the loss 
$L$ but need not be.

\begin{definition}[$\GG^\TV(\rho_1, \rho_2, \eta)$]\label{def.G_TV}
Given an arbitrary loss function $L(p, \theta)$, we 
define 
$\mathcal{G}(\rho_1, \rho_2, \eta) = \mathcal{G}_{\downarrow}(\rho_1, \eta) \cap \mathcal{G}_{\uparrow}(\rho_1, \rho_2, \eta)$, where:
\begin{align}
    \GG^\TV_{\downarrow}(\rho_1, \eta) & \triangleq \{ p \mid   \sup_{r\leq\frac{p}{1-\eta}}B(r, \theta^*(p))\leq \rho_1\},\label{eqn.G_TV_down} \\
    \GG^\TV_{\uparrow}(\rho_1, \rho_2, \eta) & \triangleq \{ p \mid  \text{ for all } \theta \in \Theta,  r \leq \frac{p}{1-\eta}, \left( B(r, \theta)\leq \rho_1 \Rightarrow  L(p, \theta)\leq \rho_2\right) \},
\end{align}
The function $B(p, \theta)$ is an arbitrary cost function that serves the purpose of bridging. Here  $\theta^*(p) \in \argmin_{\theta  \in \Theta} B(p,\theta)$.
\end{definition}
If we do not specify $\Theta$, then it is clear from the context. 
The added flexibility in choosing $B$ is important for finite-sample algorithms, which we will show in later sections. 
If we take $B(p, \theta) = L(p, \theta) = \|\bE_p[X] - \bE_\theta[X]\|, \rho_2 =2\rho_1$, then our design exactly reduces to the resilient set for mean estimation\footnote{To see the reduction, note that resilience is equivalent to $\GG^\TV_{\downarrow}$ in Equation (\ref{eqn.G_TV_down}). Thus we only need to show that $\GG^\TV_{\uparrow}$ is a subset of $\GG^\TV_{\downarrow}$. By our choice of 
$B, L$ and $\rho_2$, the implication condition in $\GG^\TV_{\uparrow}$ follows from the triangle inequality.}.

We show that $\GG^\TV$ is \emph{not too big} by bounding its modulus of 
continuity, and that it is \emph{not too small} by showing it subsumes various assumptions imposed in the literature in the next section together with the finite-sample 
results. 
\begin{theorem}\label{thm.G_fundamental_limit}
For $\GG^\TV(\rho_1, \rho_2, \eta)$  in Definition~\ref{def.G_TV}, if $2\epsilon \leq \eta < 1$, we have 
\begin{align}
   \modu(\GG^\TV(\rho_1, \rho_2, \eta), 2\epsilon)  \leq \rho_2.
\end{align}
\end{theorem}
\begin{proof}
We  rely on the midpoint distribution $r$ to bridge the modulus. 
Consider any  $p_1, p_2$ satisfying $\TV(p_1, p_2)\leq 2\epsilon \leq \eta$.  
From Lemma~\ref{lem.deletionrproperties}, there is a midpoint $r$ such that $r \leq \frac{p_1}{1-\eta}$ and $r \leq \frac{p_2}{1-\eta}$. From the fact that $p_1\in \GG^\TV(\rho_1, \rho_2, \eta) \subset \mathcal{G}_{\downarrow}(\rho_1, \eta)$, we have $B(r, \theta^*(p_1))\leq \rho_1$. From this and the fact that $p_2\in \GG^\TV(\rho_1, \rho_2, \eta) \subset \mathcal{G}_{\uparrow}(\rho_1, \rho_2, \eta)$, we then have
$    L(p_2, \theta^*(p_1))\leq \rho_2$. Since $p_1$ and $p_2$ are arbitrary, this bounds the modulus of continuity by $\rho_2$. %

\end{proof}

\subsection{Finite Sample Algorithms for $\TV$: weaken the distance}\label{sec.finite_sample_algorithm_weaken}
In Lemma~\ref{lemma.population_limit} we saw that the minimum distance functional defined as $q  = \argmin_{q\in \GG^\TV} \TV(q,p) $, $\hat \theta   = \theta^*(q) = \argmin_{\theta \in \Theta} L(q, \theta)$
yield good bounds whenever the modulus of continuity $\modu(\GG^\TV, \epsilon)$ is small. However, in the finite-sample case, the $\TV$ distance reports a large distance even 
between a population distribution $p$ and the finite-sample distribution $\hat{p}_n$.  Thus simply replacing $p$ with $\hat p_n$ in the MD functional would fail.
We provide solutions to this is via relaxing the distance and expanding the set in projection.
The general projection algorithm is presented below. 

\begin{algorithm}[H]
\begin{algorithmic}
\State \textbf{Input} 
observed distribution $\hat{p}_n$, discrepancy $\tD$, destination set $\cM$, optional parameter $\tilde \epsilon$
\If {$\tilde \epsilon$ is given}  
\\ {\qquad
find any $q \in \mathcal{M}$ such that $\tD(q, \hat{p}_n)\leq \tilde \epsilon$. }
\Else 
{\quad find $q = \argmin_{q\in \mathcal{M}} \tD(q, \hat{p}_n).$
}
\EndIf
\State \textbf{Output}  $q$.
\end{algorithmic}

 \caption{Projection algorithm $\Pi(\hat{p}_n; \tD, \mathcal{M})$ or $\Pi(\hat{p}_n; \tD, \mathcal{M}, \tilde \epsilon)$}\label{alg.general_projection}
\end{algorithm}
After obtaining $q$ from Algorithm~\ref{alg.general_projection}, 
we output $\hat \theta= \theta^*(q) \triangleq \argmin_{\theta \in \Theta} L(q, \theta)$ as the estimated parameter.
This algorithm is a generalization of the minimum distance functional~\citep{donoho1988automatic}, where we allow the flexibility of choosing  $\tD$   for the projection function, and 
 $\MM$ for the destination set, and we allow the algorithm to output any $q\in \mathcal{M}$ satisfying $\tD(q, \hat{p_n})\leq \tilde \epsilon$ if the parameter $\tilde \epsilon$ is given.

In this section we focus on the approach of weakening the discrepancy $\TV$  and keeping the destination set intact ($\cM = \GG^\TV$)\footnote{All the argument can be extended to the case of projecting onto $\mathcal{M} = \GG^\TV_\downarrow \supset \GG^\TV$ when $\GG^\TV$ is the generalized resilience set. Here for simplicity we keep $\mathcal{M} = \GG^\TV$.}.
We define a family of pseudonorms that is weaker than $\TV$, called  \emph{generalized Kolmogorov--Smirnov distance} $\tTV_\mathcal{H}$,  to avoid the issue:
\begin{align} \label{eqn.ksdistancedefinition}
    \tTV_\mathcal{H}(p, q) \triangleq \sup_{f\in\mathcal{H},  t\in \bR}|\bP_p[f(X)\geq t] - \bP_q[f(X) \geq t]|.
\end{align}
When $X$ is a one-dimensional random variable and $\mathcal{H}$ is the singleton $\{\mathrm{Id}: x \mapsto x\}$, 
we recover the Kolmogorov-Smirnov distance~\cite{massey1951kolmogorov}.
When $\mathcal{H}$ contains all functions, $\tTV_\mathcal{H} = \TV$; otherwise $\tTV_\mathcal{H}$ is weaker than $\TV$ distance: $\tTV_\mathcal{H} \leq \TV$. For any $\mathcal{H}$, $\tTV_\mathcal{H}$ is a pseudometric.

Assume $p^* \in \GG^\TV$ and $\TV(p^*, p) \leq \epsilon$ (oblivious corruption model).
The following result generalizes~\citet{donoho1988automatic} and guarantees good performance of the projection algorithm 
as long as we can bound $\tTV_\mathcal{H}(p, \hat p_n)$ along with the modulus of continuity under $\tTV_{\sH}$.
\begin{proposition}\label{prop.tTV_general}
The projection algorithm $q=\Pi(\hat{p}_n; \tTV_{\mathcal{H}}, \GG^\TV)$  satisfies
\begin{align}\label{eqn.tTV_modulus_general}
    L(p^*, \theta^*(q)) \leq  
    \modu(\GG^\TV, \tilde\epsilon, \tTV_\mathcal{H}, L) = \sup_{p_1, p_2\in \GG^\TV: \tTV_\mathcal{H}(p_1, p_2) \leq \tilde \epsilon} L(p_2, \theta^*(p_1)), 
\end{align} where $\tilde \epsilon = 2\epsilon + 2\tTV_\mathcal{H}(p, \hat p_n)$. Same bound can be achieved via $q=\Pi(\hat{p}_n; \tTV_{\mathcal{H}}, \GG^\TV, \tilde \epsilon).$
\end{proposition}

As its proof below shows, the conclusions in Proposition~\ref{prop.tTV_general} remain unchanged if we change the corruption model to 
allow $\tTV_{\sH}$ perturbations instead of $\TV$.
In other words, Proposition~\ref{prop.tTV_general} analyzes the original algorithm under larger corruptions, and we will obtain good finite sample error bounds throughout this section for this more difficult problem.

\begin{proof}
Since $\tTV_\mathcal{H}$ satisfies triangle inequality, we have 
\begin{align}
    \tTV_\mathcal{H}(q, p^*) & \leq  \tTV_\mathcal{H}(q, \hat p_n) + \tTV_\mathcal{H}(\hat p_n, p^*) \leq 2\tTV_\mathcal{H}(\hat p_n, p^*),
\end{align} 
where the second step is because $\tTV_{\sH}(q, \hat{p}_n)$ is minimizes $\tTV_{\sH}(\cdot, \hat{p}_n)$ over $\GG^\TV$. 
Applying the triangle inequality again and $\tTV_\sH(p^*,p)\leq \epsilon$, we obtain
\begin{align}
2\tTV_{\mathcal{H}}(\hat p_n, p^*) \leq 2(\tTV_\mathcal{H}(\hat p_n, p) + \tTV_\mathcal{H}(p, p^*))  \leq 2\epsilon + 2\tTV_\mathcal{H}(p, \hat p_n).
\end{align}
Thus from $p^*, q \in \GG^\TV$ and Lemma~\ref{lemma.population_limit},
we know that the error is bounded by $ \modu(\GG^\TV, \tilde\epsilon, \tTV_\mathcal{H}, L)$.
\end{proof}

We can bound $\tTV_\mathcal{H}(p, \hat p_n)$ with the complexity of $\mathcal{H}$ using the following lemma. We defer its proof to Appendix~\ref{Appendix.proof_vcinequality}.
\newcommand{\vc}{\mathsf{vc}}
\begin{lemma}\label{lemma.vcinequality}
Let $\hat{p}_n$ be the empirical distribution of $n$ \iid samples from $p$ and let $\vc(\sH)$ be the VC dimension of the collection of sets $\{\{x \mid f(x)\geq t\} \mid f\in \mathcal{H},t\in \mathbf{R}\}$. Then, each of the following holds with probability at least $1-\delta$:
\begin{align}\label{eqn.vc_inequality}
    \tTV_\mathcal{H}(p,\hat{p}_n) &\leq C^{\mathsf{vc}} \cdot \sqrt{\frac{\mathsf{vc}(\mathcal{H})+\log(1/\delta)}{n} } \text{ for some universal constant $C^{\vc}$}, \\
\label{eqn.dkwtvh}
        \tTV_\mathcal{H}(p,\hat{p}_n) &\leq \sqrt{\frac{\ln(2|\mathcal{H}|/\delta)}{2n}}, \text{ where $|\sH|$ denotes the cardinality of $\sH$.}
    \end{align}
\end{lemma}

Since Lemma~\ref{lemma.vcinequality} bounds $\tTV_{\sH}(p, \hat{p}_n)$, it remains to bound the modulus $\modu(\GG^\TV, \epsilon, \tTV_{\sH})$ for some $\sH$ with small VC dimension. We provide a general analysis via the mean cross lemma.  We show that for two 1-dimensional distributions that are close under $\tTV_\mathcal{H}$, we can delete a small fraction of probability mass to make their means cross:

\begin{lemma}[Mean cross]
\label{lem.tvt_cross_mean}\label{lem.mean_cross_tv}
Suppose two distributions $p,q$ on the real line satisfy
\begin{align} 
\sup_{t\in \bR} |\bP_p(X\geq t) - \bP_q(Y\geq t)| \leq \epsilon.
\end{align}
Then one can find some $r_p \leq \frac{p}{1-\epsilon}$ and $r_q \leq \frac{q}{1-\epsilon}$
such that $r_p$ is stochastically dominated by $r_q$,
which implies that
$\mathbb{E}_{r_p}[X] \leq \mathbb{E}_{r_q}[Y]$.
\end{lemma}
\begin{proof}
The idea of the proof is illustrated in Figure~\ref{fig.mean_cross_tv}. 
Suppose $X \sim p, Y\sim q$. Starting from $p, q$, we delete $\epsilon$ probability mass corresponding to the largest points of $X$ in $p$ to get $r_p$, and delete $\epsilon$ probability mass corresponding to the smallest points $Y$ in $q$ to get $r_q$.  Equation (\ref{eqn.tTV_XY}) implies that $\bP_{r_p}(X\geq t) \leq \bP_{r_q}(X\geq t)$ holds for all $t\in\mathbb{R}$. Hence, $r_q$ stochastically dominates $r_p$ and $\bE_{r_p}[X] \leq \bE_{r_q}[Y]$.
\end{proof}

With this one-dimensional mean cross lemma, the key idea to show modulus of continuity in high dimension is to identify the optimal projected direction and apply the  lemma in that direction. We show in the next sections via concrete results how the lemma can be applied to derive finite-sample rates for different cases.
 
\subsubsection{$\tTV_{\sH}$ projection for mean estimation}
Recall that in mean estimation, we have $ \GG^\TV_{\mathsf{mean}}(\rho, \eta) \eqdef \big\{p  \mid \|\bE_r[X] - \bE_p[X] \| \leq \rho \text{ for all } r \leq \frac{p}{1-\eta}\big\}$. 
We choose $\mathcal{H}$ as
$\{v^\top X \mid v\in \bR^d\}.
$
This particular $\tTV_\mathcal{H}$ is also used in~\citet{donoho1982breakdown,donoho1988automatic}.
Intuitively, the reason for choosing this $\sH$ is that linear projections of our data contain all information needed to recover the mean, so perhaps it is enough for distributions to be close only under these projections. We have the following results for mean estimation.

\begin{theorem}\label{thm.tTV_mean} 
Denote $\tilde \epsilon = 2\epsilon + 2C^{\mathsf{vc}}\sqrt{\frac{d+1+\log(1/\delta)}{n}}$, where $C^{\mathsf{vc}}$ is from Lemma~\ref{lemma.vcinequality}. 
For any Orlicz function $\psi$, assume that $p^*\in\GG(\psi)$, where  
\begin{align} \label{eqn.boundedpsimeanexample}
   \GG(\psi) = \left\{p \mid \sup_{v\in\bR^d, \|v \|_*=1} \bE_{p}\left[\psi\left(\frac{|\langle v, X - \bE_{p}[X]\rangle|}{\sigma}\right)\right] \leq 1\right\}.
\end{align}
Here $\| \cdot\|_*$ is the dual norm of $\|\cdot\|$. 
 For $\mathcal{H} = \{v^\top X \mid v\in \bR^d\}$, let $q $ denote the output of the projection algorithm $ \Pi(\hat{p}_n; \tTV_{\mathcal{H}},  \GG(\psi))$. Then  for any $\tilde\epsilon \in [0, 1)$, with probability at least $1-\delta$,
\begin{align}\label{eqn.finalboundmeanestimation}
    \|\bE_{p^*}[X] - \bE_q[X]\|\leq \frac{2\sigma\tilde\epsilon \psi^{-1}(1/\tilde\epsilon)}{1-\tilde\epsilon}.
\end{align} 
Furthermore, the population limit (Definition~\ref{def.populationlimit}) for $\GG(\psi)$ is  $\Theta(\sigma\epsilon\psi^{-1}(\frac{1}{2\epsilon}))$ for $\epsilon < 0.499$. 
\end{theorem}
The proof is deferred to  Appendix~\ref{proof.mean_psi}, where we first show that $\GG(\psi)$ is a subset of the resilience set $\GG^\TV_{\mathsf{mean}}(\rho(\tilde \epsilon), \tilde \epsilon)$, and then bound the performance of projection algorithm on the resilience set via Proposition~\ref{prop.tTV_general}. Thus the conclusion also holds if we consider the larger  resilience set as assumption and target projection set.

\textbf{Interpretation and Comparison.}
If $p^*$ is sub-Gaussian, then the population limit is $\rho(\epsilon) = \Theta(\epsilon \sqrt{\log(1/\epsilon)})$ for $\epsilon<0.499$ and projecting onto the set of sub-Gaussian distributions achieves the matching rate 
when $n \gtrsim \frac{d + \log(1/\delta)}{\epsilon^2}$.
If $p^*$ has bounded covariance, then the population limit is $\rho(\epsilon) = \Theta(\sqrt{\epsilon})$ and projection achieves the matching rate at the same sample complexity $n \gtrsim \frac{d + \log(1/\delta)}{\epsilon^2}$. This also holds true for sub-exponential and bounded $k$-th moment distributions.
Although the dependence of $n$ on $\epsilon$ might be sub-optimal, we improve the sample complexity by another analysis method in Section~\ref{sec.expand}, and defer the detailed discussion to Section~\ref{sec.mean_comparison_discussion}.

\subsubsection{$\tTV_{\sH}$ projection for linear regression}
For the linear regression problem, we take the bridge function and cost function as
$B(p, \theta) = L(p, \theta) = \bE_p[(Y-X^{\top}\theta)^2 - (Y-X^{\top}\theta^*(p))^2]$ in Definition~\ref{def.G_TV}. Here $\theta^*(p) = \argmin_\theta \bE_p[(Y-X^\top \theta)^2]$. 
We design the corresponding $\mathcal{H}$ as
\begin{align}\label{eqn.H_linreg}
\mathcal{H}& = \{(Y - X^\top\theta_1)^2 - (Y - X^\top\theta_2)^2 \mid  \theta_1, \theta_2 \in \bR^d\}.
\end{align}
The following theorem characterizes the performance of $\tTV_{\mathcal{H}}$ projection:
\begin{theorem} \label{thm.linearregressiontvtildeproof}
Assume $p^*\in\GG(\psi)$, where
    \begin{align*} 
  \GG(\psi) = \Bigg\{p \mid \bE_{p}\bigg[\psi \bigg(\frac{(v^{\top}X)^2}{\sigma_1^2 \bE_{p}[(v^{\top}X)^2]}\bigg)\bigg] &\leq 1 \text{ for all } v \in \bR^d, \text{ and }  
          \bE_{p}\left[\psi \left(\frac{(Y - X^\top \theta^*(p))^2}{\sigma_2^2}\right)\right] \leq 1 \Bigg\}.
    \end{align*}

Denote $\tilde \epsilon =  2\epsilon + 2 C^{\mathsf{vc}}\sqrt{\frac{10d +\log(1/\delta)}{n} }$.
For $\mathcal{H}$ designed in (\ref{eqn.H_linreg}), let $q$ denote the output of the projection algorithm $\Pi(\hat{p}_n; \tTV_{\mathcal{H}}, \GG(\psi))$. Then for any $ \tilde \epsilon$ satisfying $\sigma_1^2\tilde \epsilon\psi^{-1}(\frac{1}{\tilde \epsilon}) < \frac{1}{2}$, with probability at least $1-\delta$, 
\begin{align*}
    \bE_{p^*}[(Y-X^\top\theta^*(q))^2 - (Y-X^\top\theta^*(p^*))^2]\leq  \left(\frac{2\sigma_1\sigma_2\tilde\epsilon \psi^{-1}(1/\tilde\epsilon)}{1-\tilde\epsilon}\right)^2.
\end{align*}   

The  population limit (Definition~\ref{def.populationlimit}) for $\GG(\psi)$  is $ \Theta( ({\sigma_1\sigma_2\epsilon \psi^{-1}(1/\epsilon)})^2)$ when the perturbation level $\epsilon$ satisfies  $\epsilon < 1/2$ and $ 2\sigma_1^2\epsilon\psi^{-1}(\frac{1}{2\epsilon}) < 1/2$.
\end{theorem}

 We defer the proof to Appendix~\ref{appendix.proof_linreg_tvt}. 
We first show that any $\GG(\psi)$ is a subset of the generalized resilience set. Then we show that for any $p^*$ in the generalized resilience set, the worst-case error of projection algorithm is bounded via Proposition~\ref{prop.tTV_general}.

\textbf{Interpretation and Comparison.} 
The first condition in $\GG(\psi)$ bounds the tails of $X$ in every direction \emph{relative to} the 
second moment (also known as hyper-contractivity or anti-concentration), while  the second bounds the tails of the error $Z$.  The hyper-contractivity condition is satisfied for general Gaussian distribution $X\sim \mathcal{N}(\mu, \Sigma)$, which is shown in e.g. ~\citep[Corollary 5.21]{boucheron2013concentration}.
One might wonder whether a simpler condition such as sub-Gaussianity of $X$ and $Z$ 
would also guarantee a finite population limit.  We prove 
in Appendix~\ref{proof.delete_dimension_linreg} that even if $Z \equiv 0$, sub-Gaussianity of $X$ is \emph{not} sufficient. 
The hyper-contractivity condition prevents the deletion of dimension, and only can be replaced by usual bounded Orlicz norm assumption if we assume $\|\theta\|$ is upper bounded.

Taking $\psi(x) = |x|^{k/2}$, we recover the  $k$-hyper-contractivity condition of $X$ and bounded $k$-th moment condition of $Y-X^\top \theta^*(p)$ as is studied in \citet{klivans2018efficient}. 
Our projection algorithm guarantees the excess predictive loss to be  $O(\epsilon^{2-4/k})$ given $O(d/\epsilon^2)$ samples, while \citet{klivans2018efficient} gives $O(\epsilon^{1-2/k})$ assuming $O(\mathsf{poly}(d^k,1/\epsilon))$ samples. This also matches the dependence on $\epsilon$ with a better dependence on dimension compared to the follow-up work in~\citet{bakshi2020robust}. When $X$ and $Z$
are independent, we can improve this further to 
$\epsilon^{2-2/k}$, as we show in the proof. 

When $X$ and $Z$ are both sub-Gaussian, our dependence on $\epsilon$ is the same as the Gaussian example in~\citet{gao2017robust} up to a log factor while the sample complexity matches~\citet{gao2017robust} exactly. 
\citet{diakonikolas2019efficient} guarantee parameter error $\|\hat \theta - \theta^*\|_2 \lesssim (\epsilon\log(1/\epsilon))^2$ given $O(d/\epsilon^2)$ samples when $X$ is isotropic Gaussian and $Z$ has bounded second moment, which is implied by our analysis by taking the Orlicz function $\psi$  as exponential function and use it as the generalized resilience set.

\subsubsection{$\tTV$ projection for joint mean and covariance estimation}

For joint estimation of the mean and covariance of a distribution $p$,   we use 
the recovery metric as in~\citet{kothari2017outlier}:
\begin{align}
   L(p, (\mu, \Sigma)) & = \max\big(\|\Sigma_p^{-1/2}(\mu_p-\mu)\|_2^2/\eta, \|I_d - \Sigma^{-1/2}_p\Sigma\Sigma_p^{-1/2}\|_2 \big), \label{eqn.joint_L}
\end{align}
where $\mu_p$ and $\Sigma_p$ are the mean and covariance of $p$. 
Making $L$ small requires that the estimated covariance $\Sigma$ is close to $\Sigma_p$, and also that 
$\mu$ is close to $\mu_p$ under the norm induced by $\Sigma_p$.
We apply the $\tTV$ projection algorithm for joint mean and covariance estimation and  present the theorem as below.  
 \begin{theorem}\label{thm.tTV_joint_multiplicative}Assume that $p^*\in\GG(\psi)$, where 
\begin{align}\label{eqn.joint_sufficient}
   \GG(\psi) = \left\{  \sup_{v\in\bR^d, \|v\|_2 = 1}\bE_p\left[\psi\left(\frac{({v^{\top}(X-\mu_p)})^2}{\kappa^2\bE_{p}[(v^\top (X-\mu_p))^2]}\right)\right] \leq 1\right\},
\end{align}

Denote $\tilde \epsilon = 2\epsilon + 2C^{\mathsf{vc}}\sqrt{\frac{d+1+\log(1/\delta)}{n}}$.  For $\mathcal{H} = \{v^\top x \mid v\in\bR^d, \|v\|_2=1\}$, let $q$ denote the output of the projection algorithm $\Pi(\hat{p}_n; \tTV_{\mathcal{H}}, \GG(\psi))$. Then there exist some $C$ such that when $\tilde \epsilon \leq C$, with probability at least $1-\delta$, 
\begin{align}
     \|\Sigma_{p^*}^{-1/2}(\mu_{p^*} -\mu_q) \|_2 & \lesssim \kappa\tilde\epsilon\sqrt{\psi^{-1}({1}/{\tilde\epsilon})}, 
    \| I_d - \Sigma_{p^*}^{-1/2}\Sigma_q\Sigma_{p^*}^{-1/2}\|_2 \lesssim \kappa^2\epsilon{\psi^{-1}({1}/{\epsilon})}.
\end{align}

\end{theorem}

The proof is deferred to Appendix~\ref{proof.G_TV_joint}. 
We first show that any $\GG(\psi)$ is a subset of the generalized resilience set. Then we show that for any $p^*$ in the generalized resilience set, the worst-case error of projection algorithm is bounded via Proposition~\ref{prop.tTV_general}. In the analysis we provide a new way of bounding the modulus by deleting the two distributions in the modulus formulation differently in the transformed space, and utilizing the integral representation of mean and covariance to connect with the $\tTV$ distance.  %

\textbf{Interpretation and Comparison.}
First take $\psi(x) = x^k$;  when $\Sigma_p^{-1/2}X$ has $2k$-th central moment bounded by $\kappa^{2k}$, 
we can estimate the mean and covariance with respective errors $O(\kappa\epsilon^{1-1/(2k)})$ and 
$O(\kappa^2\epsilon^{1-1/k})$.
This matches results in~\citet{kothari2017outlier} while improving the sample complexity's dependence on dimension from $O((d\log(d))^k)$ to $O(d)$.
Next take $\psi(x) = e^x -1$; when $\Sigma_p^{-1/2}X$ is sub-Gaussian with parameter $\kappa$, we can estimate the mean and
covariance with respective errors $O(\kappa \epsilon\sqrt{\log(1/\epsilon)})$ and $O(\kappa^2 \epsilon\log(1/\epsilon))$.
For another cost under 2-norm (Theorem~\ref{thm.tTV_joint_correct}), our results generalize  \citet{gao2019generative} from Gaussians to a non-parametric set with a necessary sacrifice of a log factor. 

\subsubsection{Discussion and remarks on the results}\label{subsec.discussion_tTV}

\paragraph*{Not only the minimizer works in MD functional}

We show in Proposition~\ref{prop.any_q_suffices_oblivious}  that instead of looking for the exact projection $ q = \Pi(\hat{p}_n; \tTV_{\mathcal{H}}, \GG)$, all the above results also hold if we simply find some distribution $q=\Pi(\hat{p}_n; \tTV_{\mathcal{H}}, \GG, \tilde \epsilon/2)$ in Algorithm~\ref{alg.general_projection}, i.e. it suffices to find some $q \in \GG$ that is within $\epsilon + f(n, \delta, \mathcal{H})$ $ \tTV_\mathcal{H}$-distance of $\hat p_n$, where $f(n, \delta, \mathcal{H})$ is the $1- \delta$ quantile of $\TV(p, \hat p_n)$, which usually takes the form $C^{\mathsf{vc}}\cdot \sqrt{\frac{\mathsf{vc}(\mathcal{H})+\log(1/\delta)}{n} }$ (Lemma~\ref{lemma.vcinequality}).

\paragraph*{Design of $\mathcal{H}$ in $\tTV_{\mathcal{H}}$ for general cases}

In the above examples, we give specific choice of $\mathcal{H}$ for different $\GG^\TV$, including $\GG^\TV$ for mean estimation, linear regression and joint mean and covariance estimation. We remark that all the choice of $\mathcal{H}$ can be unified by the dual representation of bridge function $B$ in the generalized resilience definition.
Assume $B(p, \theta)$ is convex in $p$ for all $\theta$ in $\GG^\TV(\rho_1, \rho_2, \eta)$, consider the Fenchel-Moreau dual representation~\cite{borwein2010convex} of $B$:
\begin{align*}
    B(p, \theta) = \sup_{f\in\sF_\theta} \bE_p[f(X)] - B^*(f, \theta).
\end{align*}
Here $\sF_\theta = \{f \mid B^*(f,\theta) <\infty\}$. Then we take 
$\mathcal{H}=  \bigcup_{\theta\in \Theta} \sF_\theta.
$
All the three examples above are taking the $\mathcal{H}$ according to this rule. We justify in Appendix~\ref{appendix.general_sHandtTV} that with this design of $\sH$, the modulus  for generalized resilience set is bounded. 

\paragraph*{Performance guarantee for adaptive corruption model}

So far we only focus on the performance guarantee for oblivious corruption model. In fact, all the above results also hold for adaptive corruption model. We only need to substitute the term $\tilde \epsilon = 2\epsilon + 2 C^{\mathsf{vc}} \sqrt{\frac{\mathsf{vc}(\mathcal{H})+\log(1/\delta)}{n}}$ in the results with $\tilde \epsilon = 2(\sqrt{\epsilon}+\sqrt{\frac{\log(1/\delta)}{n}})^2   + 2 C^{\mathsf{vc}} \sqrt{\frac{\mathsf{vc}(\mathcal{H})+\log(1/\delta)}{n}}$. The results can be seen from a combination of Theorem~\ref{thm.admissible-1-adaptive} for analysis technique and Lemma~\ref{lem.coupling_TV} for the upper bound on $\TV(\hat p_n, \hat p_n^*)$.  The difference in analysis is sketched in Theorem~\ref{thm.admissible-1} and Theorem~\ref{thm.admissible-1-adaptive}.

Besides all the examples above, we also show in Appendix~\ref{appendix.classification} that with different choices of $B$ and $L$, we are able to derive different $\GG^\TV$ and corresponding sufficient conditions for robust classification.

\subsection{Finite Sample Algorithms - Expand Destination Set $\mathcal{M}$}\label{sec.expand}

Besides weakening the distance $  \TV$, an alternative way to rescue the projection algorithm 
is to ensure that the destination set $\cM$ is {large enough}. 
In this section, we assume the corruption model is  adaptive corruption (Definition~\ref{def.adaptivecont}) of level $\epsilon$, which is stronger than the oblivious corruption (Definition~\ref{def.obliviouscorruption}) of the same level. We denote $\hat p_n$ as the corrupted empirical distribution and $\hat{p}_n^*$ as the empirical distribution sampled from the clean population distribution $p^*$.  %

We will focus on mean estimation because it is well-studied, but our analysis 
strategy applies more generally.  
All the results in this section apply to both $\TV$ and $\tTV$ projection. For $\TV$ projection, the existing filtering and convex programming approaches can be viewed as obtaining one solution of our projection; for $\tTV_\mathcal{H}$ projection, our results provide new bounds for the algorithm in Section~\ref{sec.finite_sample_algorithm_weaken}.  The key difference of this approach compared to distance weakening is that we no longer search for the population distribution $p^*$ in projection; we instead search for the empirical distribution $\hat{p}_n^*$ in projection. The intuitive reasoning is summarized in Section~\ref{sec.intro_expand}. We formalize and generalize it in the following proposition:

\begin{proposition}
\label{prop.expand_sufficient}
For a set $\GG' \subset \MM$, define the generalized modulus of continuity as 
\begin{equation}
\modu(\GG', \MM, \epsilon) \eqdef \min_{p \in \GG', q \in \MM : \TV(p, q) \leq \epsilon} L(p, \theta^*(q)).
\end{equation}
Assume $\TV(\hat p', \hat p_n^*)\leq \epsilon_1$  with probability at least $1-\delta$ and $\hat p' \in \GG'$ with probability at least $1-\delta$.
Then the minimum distance functional projecting under $\TV$ onto $\MM$ has empirical error 
$L(\hat p', \hat{\theta})$ at most $\modu(\GG', \MM, \tilde \epsilon)$ with probability at least 
$1 - 3\delta$, where $\tilde \epsilon =2(\sqrt{\epsilon}+\sqrt{\frac{\log(1/\delta)}{2n}})^2+2\epsilon_1$.
\end{proposition}

\begin{proof}
From Lemma~\ref{lem.coupling_TV}, we know that with probability at least $1-\delta$, $\TV(\hat p_n^*, \hat p_n)\leq (\sqrt{\epsilon}+\sqrt{\frac{\log(1/\delta)}{2n}})^2$. Thus by triangle inequality, $\TV(\hat p_n, \hat p')\leq \epsilon_1 + (\sqrt{\epsilon}+\sqrt{\frac{\log(1/\delta)}{2n}})^2 = \tilde \epsilon/2$ with probability at least $1-2\delta$.
If $\hat p'$ lies in $\GG'$, then since $\GG' \subset \MM$ we know that $\hat{p}_n$ has distance 
at most $\tilde \epsilon/2$ from $\MM$, and so the projected distribution $q$ satisfies $\TV(q, \hat{p}_n) \leq \tilde \epsilon/2$ 
and hence $\TV(q, \hat p') \leq \tilde \epsilon$. It follows from the definition that 
$L(\hat p', \hat{\theta}) = L(\hat p', \theta^*(q)) \leq \modu(\GG', \MM, \tilde \epsilon)$.
\end{proof}

To employ Proposition~\ref{prop.expand_sufficient}, we construct $\MM$ and $\GG'$ such that 
the generalized modulus is small, then exhibit some $\hat{p}' \in \GG'$ that is close to 
$\hat{p}_n^*$. Often $\hat{p}'$ will just be $\hat{p}_n^*$, but sometimes it is a perturbed 
version of $\hat{p}_n^*$ that deletes heavy-tailed ``bad'' events.

The key difficulty in analysis is that the empirical distribution $\hat p_n^*$ may not inherit the good properties of $p^*$, e.g.~the empirical distribution of an isotropic Gaussian distribution does not have constant $k$-th moment unless the sample size $n \gtrsim d^{k/2}$ for constant $k$. Additionally, when $k = 2$, even if the original distribution has bounded Euclidean norm $\sqrt{d}$, its empirical distribution with $d$ samples is not resilient with the right scale of $\rho = O(\sqrt{\eta})$.
We discusse this in Appendix~\ref{sec.negativeresultresilient}.
Since we cannot hope to establish properties like bounded $k$th moments for $\hat p_n^*$, we instead rely on two main techniques 
to control some other properties of $\hat p_n^*$ or $\hat p'$:
moment linearization, and perturbing $\hat p_n^*$ to $\hat p'$.

\textbf{Linearized moment.} %
Although $\hat p_n^*$ does not have small $k$-th moment with less than $d^{k/2}$ samples, 
our key insight is that we can bound a certain \emph{linearized} $k$-th moment with only $\Theta(d)$ samples, which is {sufficient} to ensure {resilience} (see Lemma~\ref{lem.empirical_psi_resilient} for a rigorous statement and extension to any Orlicz norm). 
For instance, if $p^*$ has $4$th moments bounded by $\sigma^4$ then we will bound $\sup_{\|v\|_2 \leq 1} \bE_{\hat p_n^*}[\psi(|v^{\top}X|)]$, 
where $\psi(x)$ is the smallest convex function on $[0,\infty)$ that coincides with $x^4$ when $0\leq x\leq 4\sigma$ and $n=d$.

\textbf{Perturb $\hat p_n^*$ to $\hat p'$.} If $p^*$ has covariance operator norm bounded by $\sigma$, then in general with even $d\log d$ samples one cannot guarantee that $\hat p_n^*$ has covariance with operator norm $O(\sigma)$. However, it was realized in~\citet{steinhardt2017resilience} that one may construct another distribution $\hat p'$ with $\TV(\hat p_n^*, \hat p')\lesssim \epsilon$ such that $\hat p'$ has covariance bounded by $O(\sigma)$ given $d\log d$ samples.
Thus instead of checking the empirical distribution $\hat p_n^* \in\GG'$, we can construct some $\hat p' \in \GG'$ such that $\TV(\hat p_n^*, \hat p' ) \lesssim \epsilon$. 
This allows us to take $\cM$ to be the set of bounded covariance distributions instead of all resilient distributions; 
the advantage of this is that there are 
computationally efficient algorithms that approximately solve 
the projection in some cases~\cite{diakonikolas2017being, diakonikolas2019efficient}. 

This motivates us to consider mean estimation for bounded $k$-th moment distribution with identity covariance. The identity covariance assumption allows us to take $\cM$ to be the set of bounded covariance distributions, which we believe admits an efficient projection 
algorithm analogous to that of \citet{diakonikolas2017being} for isotropic sub-Gaussians.
    The statistical analysis requires the simultaneous application of the {linearized moment} and {perturbation} techniques described above, as well as a generalized modulus result~(Lemma~\ref{lem.empirical_identity_cov_modulus}). We obtain:

\begin{theorem}[Bounded $k$-th moment and identity covariance]\label{thm.projection_kth_moment_identity_covariance}
Denote
$\tilde \epsilon = 4(\sqrt{\epsilon + \frac{\log(1/\delta)}{n}} + \sqrt{\frac{\log(1/\delta)}{2n}})^2$.
Take $\GG^\TV$ to be the set of isotropic distributions with bounded $k$-th moment, $k\geq 2$, and $\MM$ to be the set of bounded covariance distributions:
\begin{align}
    \mathcal{G} &= \{p \mid \mathbb{E}_p[(X-\mu_p)(X-\mu_p)^{\top}] = I_d,  \sup_{v\in \bR^d, \|v\|_2 = 1} \bE_p[ |v^{\top}(X - \mu_p)|^k] \leq \sigma^k \},\\
    \mathcal{M} &= \{p \mid  \|\bE_p[(X-\mu_p)(X-\mu_p)^\top] \|_2 \leq 1+ f(n, d, \epsilon, \delta, \sigma)\}.
\end{align}
If $p^*\in\GG^\TV$ and $\tilde \epsilon < 1/2$, then there exists some $f(n, d, \epsilon, \delta, \sigma)$ such that the projection $q = \Pi(\hat p_n; \TV / \tTV_\mathcal{H}, \cM)$ of $\hat{p}_n$ onto $\MM$ satisfies
\begin{align}
     \|\mu_{p^*}-\mu_q\|_2 & =   O\left(k\sigma\cdot \left(\frac{ \epsilon^{1-1/k}}{\delta^{1/k}} + \frac{1}{\delta}\sqrt{\frac{d\log(d)}{n}}\right) \right)
\end{align}
with probability at least $1-8\delta$. Moreover,
this bound holds for any $q \in \MM$ within $\TV$ (or $\tTV$) distance $\tilde \eps/2$ of $\hat{p}_n$.
\end{theorem}

The proof is deferred to Appendix~\ref{proof.bounded_kth_identity}. When $\delta$ is constant and $n$ goes to infinity, this bound recovers $O(\epsilon^{1-1/k})$, which is the population limit for bounded $k$-th moment distributions. It  guarantees sample complexity of $O(d\log(d)/\epsilon^{2-2/k})$.

\textbf{Interpretation and Comparison.}
For  $k$-th moment bounded distributions, 
the prior work~\citet{steinhardt2018robust} shows that   projection onto resilient set under $\TV$ distance  works  with $d^{3/2}$ samples. \citet{prasad2019unified} used the approach of reducing high-dimensional mean estimation to one dimensional via covering to achieve error $\epsilon^{1-1/k} + \sqrt{\frac{d+\log(1/\delta)}{n}}$.  When combined with the statistical results in Theorem~\ref{thm.projection_kth_moment_identity_covariance}, the filtering algorithm in~\cite{diakonikolas2016robust, diakonikolas2017being, diakonikolas2019recent, zhu2020robust} achieves efficient computation. The follow-up work in~\citet{diakonikolas2020outlier} improved our results with a sub-gaussian rate and the same dependence on dimension. 

With the same analysis framework, we improve more rates of the mean estimation $\tTV$ projection algorithm in Appendices. We briefly summarize the rate we achieve in Table~\ref{tab.projection_summary}. 
The general framework of this analysis approach  for different $D$, $L$, and $\GG^\TV$ is summarized in Theorem~\ref{thm.admissible-2}.

\begin{table}[H]
\hskip-0.9cm
\begin{tabular}{|c|c|c|c|}\hline 
 $\GG$ & $\cM$ & $D$ & $\|\mu_{p^*}-\mu_q\|_2 $ \\
 \hline
Resilience set $\GG_{\mathsf{mean}}$ & $\GG_{\mathsf{mean}}$ & $\tTV$ & $\rho(2\epsilon + 2C^{\mathsf{vc}}\sqrt{\frac{d+ 1+\log(1/\delta)}{n}})$ [Theorem~\ref{thm.tTV_mean}]  \\
  sub-Gaussian & $\GG_{\mathsf{mean}}$ & $\TV$ or $\tTV$ & $\epsilon\sqrt{\log(1/\epsilon)} + \sqrt{\frac{d+\log(1/\delta)}{n}} $ [Theorem~\ref{thm.general_subgaussian_tv_projection}]\\
 bdd $k$-th moment ($k\geq 2$) &  $\GG_{\mathsf{mean}} $ &  $\TV$ or $\tTV$ & $\frac{\epsilon^{1-1/k}}{\delta^{1/k}} +\frac{1}{\delta}\sqrt{\frac{d}{n}}$ 
 [Theorem~\ref{thm.general_kth_tv_projection}]\\
 bdd cov  &  bdd cov &  $\TV$ or $\tTV$ & $\sqrt{ \epsilon} +\sqrt{\frac{d\log(d/\delta)}{n}}$ [Theorem~\ref{thm.projection_bounded_covariance}] \\
 bdd $k$-th moment ($k>2$) + $I_d$   &  bdd cov  & $\TV$ or $\tTV$ & $\frac{ \epsilon^{1-1/k}}{\delta^{1/k}} + \frac{1}{\delta}\sqrt{\frac{d\log(d)}{n}}$  [Theorem~\ref{thm.projection_kth_moment_identity_covariance}] \\
 \hline
\end{tabular}
\caption{Summary of results for generalized projection algorithm
assuming $p^*\in\GG$. Here `bdd' is short for `bounded'. }
\label{tab.projection_summary}
\end{table}

\section{Robust inference under $W_1$  corruption}\label{sec.robust_W1}

In this section, 
we present a general recipe for constructing estimator that are robust to Wasserstein-1 perturbations. 

\subsection{Population assumption: generalized resilience}\label{sec.population_w1}

In $\TV$ perturbation, we bound the modulus of continuity via the exisence of a midpoint:  we used the fact that any $\TV$ perturbation can be decomposed into a ``friendly'' 
operation (deletion) and its opposite (addition). We think of deletion as friendlier than addition, 
as the latter can move the mean arbitrarily far by adding probability mass at infinity. As is discussed in Section~\ref{sec.W1_population_intro},  we can extend  this to other Wasserstein distances via decomposing a 
Wasserstein perturbation into a friendly perturbation and its inverse, where the friendly perturbation shall not make huge effect on the target loss. We provide one definition of friendly perturbation as follows:
\begin{definition}[Friendly perturbation]\label{def.friendly_perturbation}
For a distribution $p$ over $\mathcal{X}$, fix a function $f : \mathcal{X}\to \bR$. A distribution $r$ 
is an {$\eta$-friendly perturbation} of $p$ for $f$ under $W_{1}$, denoted as $r \in \mathbb{F}(p, \eta, f) $, if there is a coupling $\pi_{X,Y}$ between $X\sim p$ 
and $Y \sim r$ such that:
\begin{itemize}
\item The cost $(\bE_{\pi}[\|X-Y\|])$ is at most $\eta$.
\item All points move towards the mean of $r$: $f(Y)$ is between $f(X)$ and $\bE_{r}[f(Y)]$ almost surely.
\end{itemize}
\end{definition}
The friendliness is defined only in terms of one-dimensional functions $f : \mathcal{X} \to \bR$; we will see how to handle 
higher-dimensional objects later. 
Intuitively, a friendly perturbation is a distribution $r$ for which there exists a coupling that `squeezes' $p$ to $\mu_r$. 

We provide the midpoint lemma for friendly perturbation in $W_{1}$ perturbation. We   show that  given any $p,q$ with $W_{1}(p,q)\leq \epsilon$ and any $f$, there exists an $r$ that is an 
$\epsilon$-friendly perturbation of \emph{both} $p$ and $q$ for the function $f$. 
To show the existence of a midpoint, we rely on the intuition that any coupling between two one-dimensional distributions can be separated into two stages: in one stage all the mass only moves towards some point, in the other stage all the mass moves away from that point. This is illustrated in Figure~\ref{fig.midpoint_wc}.
\begin{figure}[!htp] 
    \centering
   \begin{tikzpicture}[
  implies/.style={double,double equal sign distance,-implies},
  scale=0.9,  every node/.style={scale=1.5}
]
\begin{axis}[
  no markers, domain=0:8, samples=50,
  axis lines*=left,
  every axis y label/.style={at=(current axis.above origin),anchor=south},
  every axis x label/.style={at=(current axis.right of origin),anchor=west},
  height=4cm, width=10cm,
  xtick=\empty, ytick=\empty,
  enlargelimits=false, clip=false, axis on top,
  axis line style = thick,
  grid = major
  ]
  \addplot [opacity=0.8,very thick,fill=cyan!20, draw=none, domain=0:2.2] {gauss(3.5,0.8,1)} \closedcycle;
  \addplot+[opacity=0.8,very thick,fill=cyan!50, draw=cyan!50!black, domain=2.2:8] {gauss(3.5,0.8,1)} \closedcycle;
  \addplot+[opacity=0.7,very thick,fill=red!20, draw=none, domain=5.8:8] {gauss(4.5,0.8,1)} \closedcycle;
  \addplot+[opacity=0.55,very thick,fill=red!50, draw=red!50!black, domain=0:5.8] {gauss(4.5,0.8,1)} \closedcycle;
\draw [color=blue!75, yshift=0.0cm, dashed](axis cs:3.5,0.01) -- (axis cs:3.5,0.49);
\draw [color=blue!75, yshift=-0.27cm](axis cs:3.3,0) node {$\mu_{p_1}$};
\draw [color=blue!60!black, yshift=0.0cm, dashed](axis cs:3.9,0.01) -- (axis cs:3.9,0.44);
\draw [,color=blue!60!black, yshift=-0.27cm](axis cs:3.9,0) node {$\mu_r$};
\draw [,color=blue!60!black, yshift=0.0cm, ultra thick, -o](axis cs:3.9,0.00) -- (axis cs:3.9,0.53);
\draw [color=blue!40!red!80!black, yshift=-0.27cm](axis cs:3.9,0) node {$\mu_r$};
\draw [color=blue!40!red!80!black, yshift=0.0cm, ultra thick, -o](axis cs:3.9,0.00) -- (axis cs:3.9,0.53);
\draw [color=blue!66](axis cs:2.2,0.07) edge[implies] (axis cs:3.9,0.07);
\draw [color=red!75, yshift=0.0cm, dashed](axis cs:4.5,0.01) -- (axis cs:4.5,0.49);
\draw [color=red!75, yshift=-0.27cm](axis cs:4.5,0) node {$\mu_{p_2}$};
\draw [color=red!66](axis cs:5.8,0.07) edge[implies] (axis cs:3.9,0.07);
\end{axis}
\end{tikzpicture}
\caption{Illustration of midpoint lemma. For any distributions $p_1, p_2$ that are close under $W_{1}$, the coupling between 
$p_1$ and $p_2$ can be split into couplings $\pi_{p_1, r}$, $\pi_{p_2, r}$ such that $p_1, p_2$ only move towards $\mu_r$. We do this by ``stopping'' the movement from $p_1$ to $p_2$ at $\mu_r$.}
\label{fig.midpoint_wc}
\end{figure}
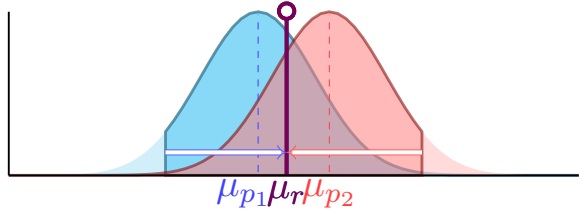
\begin{lemma}[Midpoint lemma for $W_{1}$ perturbation]\label{lem.wc_friend_pert}
Assume that $f$ is continuous under the topology induced by the metric $\|\cdot \|$. Then for
 any $p_1$ and $p_2$ such that $W_{1}(p_1,p_2)<\eta$ and any $f$, there exists a distribution $r$ such that
\begin{align}
 r \in \mathbb{F}(p_1, \eta, f) \cap   \mathbb{F}(p_2, \eta, f).
\end{align}
In other words, $r$ is an $\eta$-friendly perturbation of both $p_1$ and $p_2$ for $f$ under $W_{1}$.  
\end{lemma}
See Appendix~\ref{proof.coupling_unimodal} for a formal proof. 
With this lemma in hand, 
we generalize resilience to Wasserstein distance by saying that a distribution is resilient if $\bE_r[f(X)]$ is close to 
$\bE_p[f(X)]$ for every $\eta$-friendly perturbation $r$ and every function $f$ lying within some appropriate family $\sF$.
For instance, for second moment estimation we would consider functions $f_v(x) = \langle x, v \rangle^2$ with $\|v\|_2 = 1$.
We discuss this in more detail below.

\subsubsection{Warm-up: Second Moment Estimation under $W_1$ Perturbation}

Consider estimation of the second moment ($L(p, M) = \|M - \bE_p[XX^{\top}]\|_2$ where $M\in \bR^{d\times d}$). We do not consider mean estimation since it is trivial under $W_1$ perturbation (outputting the mean of 
$p$ incurs error $\epsilon$, which is optimal).

Recall that the resilience set for mean estimation  $\GG^{\TV}_{\mathsf{mean}}$ is defined by asking that friendly perturbations 
(in that case deletions) did not move the mean by too much. 
With our definition of friendly perturbation for $W_{1}$ in hand, we similarly define $\GG^{W_1}_{\mathsf{sec}}$ for second moment estimation as the following:
\begin{align}\label{eqn.G_w1_sec}
    \GG^{W_1}_{\mathsf{sec}}(\rho, \eta) = \{p \mid \sup_{ \|v\|_2 = 1, r\in \mathbb{F}(p, \eta, W_1, |v^{\top}X|^2)} |\bE_{p}[|v^{\top}X|^2] - \bE_{r}[|v^{\top}X|^2]| \leq \rho  \},
\end{align}
This asks that in all unit directions $v$, friendly perturbations under $|v^{\top} x|^2$ cannot move the second moment by more than $\rho$. 
As before we will show that the set $\GG^{W_1}_{\mathsf{sec}}$ is not too big (has bounded modulus) and not too small (contains natural nonparametric distribution families).

\paragraph*{Not too big}
As in the $\TV$ perturbation case, we show that $\mathcal{G}^{W_1}$ has controllable population limit by upper bounding its modulus of continuity.
\begin{theorem}
The modulus of continuity $\modu$ in (\ref{eqn.modulus}) for $\GG^{W_{1}}_{\mathsf{sec}}(\rho, \eta)$ is bounded above as 
$\modu(\GG^{W_{1}}_{\mathsf{sec}}(\rho, \eta), 2\epsilon) \leq 2\rho$ for any $ 2\epsilon \leq \eta$.
\end{theorem}
\begin{proof}
Denote $M_p=\bE_p[XX^\top]$. 
The modulus is defined as
\begin{align}
\sup_{p_1, p_2 \in \GG^{W_1}_{\mathsf{sec}}(\rho, \eta), W_1(p_1, p_2) \leq 2 \epsilon} \|M_{p_1} - M_{p_2}\|_2.
\end{align}
By Lemma~\ref{lem.wc_friend_pert}, for any unit vector $v$ we can find $r$ such that $W_1(p_1, r)\leq 2\epsilon$,  $W_1(p_2, r)\leq 2\epsilon$, and $r \in  \mathbb{F}(p_1, 2\epsilon, W_1, |v^\top X|^2) \bigcap \mathbb{F}(p_2, 2\epsilon, W_1, |v^\top X|^2)$ is a friendly perturbation for both  $p_1$ and $p_2$. Now take some $v^*$ with $\|v^*\|_2 = 1$ such that 
\begin{align}
    |(v^*)^\top (M_{p_1} - M_{p_2})v^*| = \|M_{p_1} - M_{p_2}\|_2. 
\end{align}
By symmetry of $p_1, p_2$ we may assume that the term inside the absolute value is positive.
From $p_1, p_2\in \GG^{W_1}_{\mathsf{sec}}(\rho, \eta)$, we know that for any $2 \epsilon \leq \eta$, 
\begin{align}
\bE_{p_1}[(v^{*\top}X)^2] - \bE_{r}[(v^{*\top}X)^2] \leq \rho,\\
\bE_{r}[(v^{*\top}X)^2] - \bE_{p_2}[(v^{*\top}X)^2] \leq \rho.
\end{align}
Combining the two equations together gives us 
\begin{align}
    \bE_{p_1}[(v^{*\top}X)^2] - \bE_{p_2}[(v^{*\top}X)^2]\leq 2\rho.
\end{align}
This shows that $\|M_{p_1} - M_{p_2}\|_2\leq 2\rho$.
\end{proof}

The set $\GG^{W_1}_{\mathsf{sec}}$ is a superset of Orlicz-norm  bounded distributions, which is shown in Section~\ref{sec.finite_sample_w1}.
We also extend the design of $\GG^{W_1}_{\mathsf{sec}}$ to arbitrary  pseudonorm cost in Appendix~\ref{sec.G_Wc_WF}. 

The function $f$ in the friendly perturbation requirement ($r\in\mathbb{F}(p, \eta, W_1, f)$) can be chosen differently without changing the conclusions above. For example, we may replace the constraint  $r\in\mathbb{F}(p, \eta, W_1, |v^\top X|^2)$ with $r\in\mathbb{F}(p, \eta, W_1, |v^\top X|)$, and the sufficient condition and modulus of continuity remains the same. This is discussed in Appendix~\ref{proof.psi_w1_resilience} and is used in Section~\ref{sec.finite_sample_w1} for finite sample algorithm design. 

\subsubsection{General Design of $\GG^{W_1}$}

Recall that for the general $\GG^\TV$ (Definition~\ref{def.G_TV}), we said a distribution is resilient if 
(1) the parameter $\theta^*(p)$ does
well on all deletion $r \leq \frac{p}{1-\eta}$, and (2) any parameter that does 
well on some deletion $r \leq \frac{p}{1-\eta}$ also does well on $p$. 
We used a bridge function $B$ to measure performance on $r$.

Inspired by this argument, 
we extend the definition of resilience for second moment estimation $\GG^{W_1}_{\mathsf{sec}}$    to other losses. 
Since friendly perturbation for $W_{1}$ is only defined for a one-dimensional random variable $f(X)$, we apply the Fenchel-Moreau 
representation~\cite{borwein2010convex} of $B$ to decompose $B$ to the expectation of one dimensional functions $f\in\sF$, as long as $B$ is lower semi-continuous and convex in $p$ for fixed $\theta$:
\begin{align}\label{eqn.dual_B}
    B(p, \theta) = \sup_{f\in\sF_\theta} \bE_p[f(X)] - B^*(f, \theta).
\end{align}
Here $B^*(f, \theta)$ is the convex conjugate of $B$, and $\sF_\theta = \{f\mid B^*(f,\theta) < \infty\}$. Convexity in $p$ is a mild condition that often holds, 
e.g.~any loss of the form $B(p, \theta) = \bE_{X \sim p}[\ell(\theta; X)]$ is linear (and hence convex) in $p$.

We thus define the resilient set in the same way as $\TV$: we say a distribution $p$ is resilient if (1) $\theta^*(p)$ does well on all friendly perturbation perturbation $r$ and every function $f\in\sF_{\theta^*(p)}$, and (2) for any parameter $\theta$, if all $f\in\sF_{\theta}$, $\theta$ have a friendly perturbation $r\in\mathbb{F}(p, \eta, W_{1}, f)$ where $\theta$ does well, then $\theta$ also does well on $p$ under $L$. We formally define the set $\GG^{W_1}$ below.

\begin{definition}[$\GG^{W_1}$]\label{def.G^Wc}
We define \begin{align}\label{eqn.def_w1_G}
& \GG^{W_1}(\rho_1, \rho_2, \eta) = \GG^{W_1}_{\downarrow}(\rho_1, \eta) \cap \GG^{W_1}_{\uparrow}(\rho_1, \rho_2, \eta), \text{ where}
\end{align}
\begin{align*}
     \GG_{\downarrow}^{W_1}(\rho_1, \eta) = \{p \mid  & \sup_{f\in\sF_{\theta^*(p)},  r\in \mathbb{F}(p, \eta, W_{1}, f)} \bE_r[f(X)] - B^*(f, \theta^*(p)) \leq \rho_1  \}, \\
     \GG_{\uparrow}^{W_1}(\rho_1, \rho_2, \eta) =  \bigg\{p \mid  & \text{for all } \theta\in\Theta, \bigg( \Big(  \sup_{f\in \sF_\theta} \inf_{r\in \mathbb{F}(p, \eta, W_{1}, f)} \bE_r[f(X)] - B^*(f, \theta) \leq \rho_1 \Big) \Rightarrow  L(p, \theta) \leq \rho_2  \bigg) \bigg\}.
\end{align*}
\end{definition}

The construction of $\GG^{W_1}$ generalizes the idea  in $\GG^{\TV}$ (Definition~\ref{def.G_TV}) and $\GG^{W_1}_\mathsf{sec}$ (Equation~\eqref{eqn.G_w1_sec}). %
We control the population limit of $\GG^{W_1}$ by bounding its modulus of continuity. 
\begin{theorem}\label{thm.G_Wc_fundamental_limit}

The modulus of continuity $\modu$ in (\ref{eqn.modulus}) for $\GG^{W_1}(\rho_1, \rho_2, \eta)$ is bounded above by 
$\modu(\GG^{W_1}(\rho_1, \rho_2, \eta), 2\epsilon) \leq \rho_2$ for any $ 2\epsilon \leq \eta$.
\end{theorem}
The proof is deferred to Appendix~\ref{proof.thm_G_WcG_Wc_fundamental_limit} and follows the same lines as 
Theorem~\ref{thm.G_fundamental_limit}, using 
Lemma~\ref{lem.wc_friend_pert} to produce the required midpoint distribution $r$. By taking $B$ and $ L$ as the cost of second moment estimation, we can recover the definition of $\GG_{\mathsf{sec}}^{W_1}$. The concrete examples of $\GG$ is deferred to  the next section

\subsection{Finite Sample Algorithm for $W_1$}\label{sec.finite_sample_w1}

In this section, we design finite sample algorithms for robust estimation under $W_1$ perturbations.  Throughout the section, we assume that $\hat p_n$ follows oblivious corruption model (Definition~\ref{def.obliviouscorruption}) of level $\epsilon$ under $W_1$ perturbation. It can also be extended to the adaptive corruption model using Theorem~\ref{thm.admissible-1-adaptive}. 

Similar to the issue of $\TV(p, \hat p_n)$ discussed in Section~\ref{sec.finite_sample_algorithm_weaken}, in general $W_{1}(p, \hat p_n)$ converges to zero slowly even when $p$ is well behaved; indeed,  
$\mathbb{E}[W_1(p, \hat p_n)] \gtrsim_p n^{-1/d}$ for 
any measure $p$ that is absolutely continuous with respect to the Lebesgue measure on $\bR^d$ \citep{dudley1969speed}.
As Section~\ref{sec.finite_sample_algorithm_weaken}, we therefore weaken $W_1$ to a distance $\tW_1$ to achieve better finite-sample performance.

\subsubsection{$\tW_1$ projection for Second Moment Estimation}
We first introduce our design of $\tW_1$ for second moment estimation. 
In below analysis we assume an oblivious corruption model (Definition~\ref{def.obliviouscorruption}), i.e. $\TV(p^*, p)\leq \epsilon$ and we observe an empirical distribution of $p$. But our results also apply directly to adaptive corruption model (Definition~\ref{def.adaptivecont}) as $\TV$ case, which is shown  from a combination of Theorem~\ref{thm.admissible-1-adaptive} for analysis technique and Lemma~\ref{lem.coupling} for the upper bound on $W_1(\hat p_n, \hat p_n^*)$.

Recall the dual representation of $W_1$ as $ W_1(p, q) = \sup_{u\in \mathcal{U}}  \bE_p[u(X)] - \bE_q[u(X)] $, 
where $\mathcal{U} = \{u: \bR^d \mapsto \bR \mid |u(x) - u(y)| \leq \|x-y \|\}$ is the set of all
$1$-Lipschitz functions.
To design the weakened distance $\tW_1(p, q)$, we take the supremum over a smaller set. Let $\sU'$ be the set of all 
$1$-Lipschitz linear or rectified linear functions:
\begin{align}
\mathcal{U}' = \{\max(0, v^{\top}x-a) \mid v\in \bR^d, \|v\|\leq 1, a\in\bR \} \cup \{v^{\top}x \mid v \in \bR^d, \|v\|\leq 1\}.
\end{align} 
Since all functions in $\mathcal{U}'$ are Lipschitz-1, $\mathcal{U}' \subset \mathcal{U}$, and we define $\tW_1$ as follows:
\begin{align}
\label{eqn.tw_def}
    \tW_1(p, q) & = \sup_{u \in \mathcal{U}'} \left|\bE_p[u(X)] - \bE_q[u(X)] \right|.
\end{align}
Similarly to Proposition~\ref{prop.tTV_general}, the key to analyzing $\tW_1(p, q)$ is to control (i) $\tW_1(p, \hat p_n)$ and (ii) the modulus of continuity for $\GG^{W_1}$ under $\tW_1$. 

We first use the following lemma to bound the statistical error $\tW_1(p, \hat p_n)$:

\begin{lemma}\label{lem.tW_1_empirical_to_population}
Assume $W_1(p^*, p)\leq \epsilon$ and for some Orlicz function $\psi$, $p^*$ satisfies 
\begin{align}
   \sup_{v\in\bR^d, \|v\|_2=1} \bE_{p^*}[\psi(|v^\top(X-\mu_{p^*})|/\kappa)] \leq 1, \\
    \|\bE_{p^*}[(X-\mu_{p^*})(X-\mu_{p^*})^\top]\|_2\leq \sigma^2.
\end{align}
Then for $\tW_1$ defined in (\ref{eqn.tw_def}), we have \begin{align}
    \bE_{p}[\tW_1(p, \hat p_n)]\leq 2\epsilon + 8\sigma \sqrt{\frac{d}{n}}+\frac{3\kappa \psi^{-1}(\sqrt{n})}{\sqrt{n}}.
\end{align}
\end{lemma}

The proof is deferred to Appendix~\ref{proof.tW_1_empirical_to_population}. Compared to the lower bound  $\bE[W_1(p, \hat p_n)]\gtrsim n^{-1/d}$, the upper bound for $\tW_1$ reduces the  sample complexity from exponential in $d$ to polynomial in $d$.

To bound the modulus of continuity, we establish a 
mean-cross lemma showing that $p$ and $q$ have friendly perturbations $r_p$ and $r_q$ with 
$\bE_{r_q}[f(X)] \leq \bE_{r_p}[f(X)]$ (cf.~Lemma~\ref{lem.tvt_cross_mean}).
In Lemma~\ref{lem.tvt_cross_mean} for $\tTV$, we established the mean crossing property 
by showing that $r_p$ stochastically dominated $r_q$; for $\tW_1$ we will instead 
show that $r_p$ dominates $r_q$ in the convex order, which establishes 
mean crossing for convex functions $f$:

\begin{figure}[!htp]
    \centering
   \begin{tikzpicture}[
  implies/.style={double,double equal sign distance,-implies},
  scale=0.9,
]
\begin{axis}[
  no markers, domain=0:8, samples=50,
  axis lines*=left,
  every axis y label/.style={at=(current axis.above origin),anchor=south},
  every axis x label/.style={at=(current axis.right of origin),anchor=west},
  height=4cm, width=10cm,
  xtick=\empty, ytick=\empty,
  enlargelimits=false, clip=false, axis on top,
  axis line style = thick,
  grid = major
  ]
  \addplot+[opacity=0.8,very thick,fill=cyan!50, draw=cyan!50!black, domain=0:8] {gauss(3.6,0.6,1)} \closedcycle;
  \addplot+[opacity=0.7,very thick,fill=red!30, draw=none, domain=4.3:8]
  {gauss(3.6,0.9,0.8)} \closedcycle;
  \addplot+[opacity=0.7,very thick,fill=red!30, draw=none, domain=0:2.9]
  {gauss(3.6,0.9,0.8)} \closedcycle;
  \addplot+[opacity=0.55,very thick,fill=red!50, draw=red!50!black, domain=2.9:4.3] {gauss(3.6,0.9,0.8)} \closedcycle;
\draw [color=red!66](axis cs:5.1,0.05) edge[implies] (axis cs:4.3,0.05);
\draw [color=red!66](axis cs:2.1,0.05) edge[implies] (axis cs:2.9,0.05);
\end{axis}
\end{tikzpicture}
\caption{The mean cross lemma for $\tW_1$ in the special case that $p_1$ and $p_2$ have the same mean. If $\tW_1(p_1,p_2)$ is small, we can push the $\epsilon$-tails of $p_2$  to two points on the boundary of the new distribution to create a friendly perturbation 
$r_{p_2}$ that guarantees convex ordering.}
\label{fig.mean_cross_w1}
\end{figure}
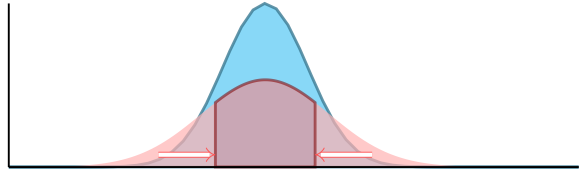

\begin{lemma}
\label{lem:tw1_cross_mean}
Consider two distributions $p,q$ on the real line such that $\tW_1(p, q) \leq \epsilon$.
Then, for $g(x) = x$ or $g(x) = |x|$, one can find $r_p\in \mathbb{F}(p, 7\epsilon, W_1, g(x))$ and $r_q \in \mathbb{F}(q, 7\epsilon, W_1, g(x))$ 
such that $r_p$ is less than $r_q$ in the convex order with respect to the random variable $g(X)$:
$r_p \leq_{\mathrm{cx}} r_q$. This 
is equivalent to saying that for all convex functions $f$ such that the expectations exist, 
\begin{equation}
\mathbb{E}_{r_p}[f(g(X))] \leq \mathbb{E}_{r_q}[f(g(X))].
\end{equation}
\end{lemma}
The proof is deferred to Appendix~\ref{proof.tw1_cross_mean}. The idea is illustrated in Figure~\ref{fig.mean_cross_w1} in the special case that $p$ and $q$ have equal means. Intuitively, one can guarantee convex ordering by squeezing one of the distributions towards the mean. 
To apply the one-dimensional mean cross lemma, we consider all one-dimensional projections and identify the optimal direction.

The reason we only consider $g(x)$  in the friendly perturbation constraint is that the landscape of general functions can be complicated and it is harder to characterize the moving towards mean operation. Thus for instance for 
second moment estimation, we consider friendly perturbations $r\in\mathbb{F}(p, \eta, W_1, |v^\top x|)$ instead of $r\in\mathbb{F}(p, \eta, W_1, |v^\top x|^2)$. %

With the mean-cross lemma and the generalized resilience set, we can prove the following result on second moment estimation:

\begin{theorem}[Second moment estimation, $\tW_1$ projection]\label{thm.sec_tw_1_proj}
Assume $p^*\in\GG(k)$, where
    $\GG(k) = \{p \mid \sup_{v\in\bR^d, \|v\|_2=1}\bE_{p}[|v^\top X|^k] \leq \sigma^k\}$
for some $\sigma>0$, $k>2$.
Denote $ \tilde \epsilon = \frac{C_1}{\delta} \left(\epsilon +  \sigma \sqrt{\frac{d}{n}}+\frac{\sigma }{\sqrt{n^{1-1/k}}}\right)$, where $C_1$ is some universal constant. 
Then the projection algorithm $q = \Pi(\hat p_n; \tW_1, \GG(k))$ or $q = \Pi(\hat p_n; \tW_1, \GG(k), \tilde \epsilon/2)$    satisfies
\begin{align}
    \|\bE_{q}[XX^{\top}] - \bE_{p^*}[XX^{\top}]\|_2 \leq C_2 \min(\sigma^2, \sigma^{1+1/(k-1)}\tilde \epsilon^{1-1/(k-1)})
\end{align}
with probability at least $1-\delta$, 
where $C_2$ is some universal constant. 
The population limit when the perturbation level is $\epsilon$ is  $\Theta(\min(\sigma^2, \sigma^{1+1/(k-1)}\tilde \epsilon^{1-1/(k-1)}))$.

\end{theorem}

The proof is deferred to Appendix~\ref{proof.psi_w1_resilience}, where we first show that $\GG(k)$ is a subset of the resilience set $\GG^{W_1}_{\mathsf{sec}}$, and then show that the projection algorithm has bounded modulus for the larger resilience set via the mean ross lemma.

\textbf{Interpretation and Comparison.} When the third moment of $X$ is bounded by ${\sigma}^{3}$, the projection algorithm guarantees error
$O(\sigma^{3/2}\sqrt{\epsilon})$ for $\epsilon$ sufficiently small and $n\gtrsim \max(d/\epsilon^2, 1/\epsilon^4)$, while the same condition in $\TV$ perturbation would give error of $O(\sigma \epsilon^{2/3})$ for $n\gtrsim d/\epsilon^2$ (Theorem~\ref{thm.tTV_joint_correct}).

Different from $\TV$ case, the error  is the minimum of two terms: $\sigma^2$ and $\sigma\eta\psi^{-1}(\frac{\sigma}{\eta})$. The second term is similar to $\TV$ and dominates the first term when $\eta$ is smaller than $\sigma$. Although $\eta$ in $W_1$ perturbation can be arbitrarily large, the error cannot exceed $\sigma^2$ since the assumption already implies the second moment to be bounded by $\sigma^2$.
The breakdown point for $W_1$ perturbation is in general infinity in contrast to $1/2$ in $\TV$ case.

\subsubsection{$\tW_1$ projection for Linear Regression}

We further study linear regression under $W_1$ perturbation below and provide sufficient conditions for distributions to be inside the $W_1$-resilient set. Here we measure the cost in $W_1$ as $c((x_0,y_0),(x_1,y_1)) = \sqrt{ \|x_0 - x_1\|_2^2 + \|y_0 - y_1\|_2^2 }$, i.e. the adversary is allowed to perturb both $X$ and $Y$ simultaneously. 
We obtain the following result for robust linear regression under $W_1$ perturbation via taking $B(p, \theta) = L(p, \theta)  = \bE_p[(Y - X^\top \theta)^2] $ in Definition~\ref{def.G^Wc}.

\begin{theorem}[$\tW_1$ projection for linear regression]
\label{thm.linreg_tw_1_proj}
Denote by $X' = [X, Z]$ the $d+1$ dimensional vector that concatenates $X$ with the noise $Z = Y-X^{\top}\theta^*(p)$.
Assume $p^*\in\GG(k)$ for some $k>2$,
where \begin{align*} 
   \GG(k) = \left\{ p \mid \sup_{v\in\bR^{d+1}, \|v\|_2=1}  \bE_{p}\bigg[{|v^{\top}X'|^{k}}\bigg] \leq \sigma_1^{k}, 
   \bE_{p}[Z^2] \leq \sigma_2^2, \|\theta^*(p)\|_2\leq R\right\}.
\end{align*} 
Denote $\bar R = \max(R, 1)$,   $ \tilde \epsilon = \frac{C_1}{\delta} \left(\epsilon + \sigma_1  \bar R\sqrt{ d/n}+ \sigma_1 \bar R/\sqrt{n^{1-1/k}}\right)$,  where $C_1$ is some universal constant. Let $q = \Pi(\hat p_n; \tW_1, \GG(k))$ or $q = \Pi(\hat p_n; \tW_1, \GG(k),  \tilde \epsilon/2)$. Then with probability at least $1-\delta$, there exists some universal constant $C_2$,
\begin{align}
    \bE_{p^*}[(Y-X^\top \theta^*(q))^2 ] \leq \sigma_2^2 + C_2 \bar R^2 (\sigma_1^{1+1/(k-1)} (\bar R \tilde \epsilon)^{1-1/(k-1)} + (\bar R \tilde \epsilon)^2).
\end{align}  

\end{theorem}

The proof is deferred to   Appendix~\ref{proof.linreg_tw_1_proj}, where we first show that $\GG(k)$ is a subset of the generalized resilience set under the choice of $B$ and $L$ for linear regression, and then show that the worst-case error of the projection algorithm is upper bounded via the mean cross lemma.

\textbf{Interpretation and Comparison.} Taking $k=3$, we see that if the vector $X'=[X, Z]$ has its third moment bounded by $\sigma_1^3$ and $Z$ has its second moment bounded by $\sigma_2^2$, our algorithm guarantees error $\sigma_2^2 +O(\bar R^2(\bar R^2\epsilon^2 + \sigma_1^{3/2}\sqrt{\bar R\epsilon})$  for $\epsilon$ sufficiently small and $n\gtrsim \max(d/\epsilon^2, 1/\epsilon^4)$.
Although higher moment bound naturally implies bounded second moment of $Z$, we introduce $\sigma_2$ since it can be much smaller than $\sigma_1$ in practice, and our final error would approach $\sigma_2^2$ which is the population limit for non-robust problem. 
We also present another set of sufficient conditions in Appendix~\ref{appendix.proof_W1_linreg_sufficient} that micmic the hypercontractivity condition in TV  linear regression case.

We  also show in Appendix~\ref{appendix.necessity_hyper_w1} that the boundedness assumption on $\|\theta\|_2$ is necessary: even if $X$ is a Gaussian distribution and $Z \equiv 0$, the population limit would be infinite since $\|\theta\|_2$ can go to infinity. 
The boundedness assumption on $\theta$ can  be dropped when we consider total regression~\cite{golub1980analysis} instead of linear regression, i.e. $L(p, \theta) = \bE_{p}[(\theta^\top X'')^2]$, where $X'' = [X, Y], \theta\in\Theta=\{\theta\mid \|\theta\|_2 = 1\}$.   
This is equivalent to  estimating the second moment along some direction $\theta$. With the same proof as the second moment estimation, we know that for the set $\GG = \{ p \mid \bE_{p}[(v^\top X'')^{k+1}] \leq \sigma^{k+1}\}$, we have
the population limit upper as $ O( \min(\sigma^2, \sigma^{1+1/k} \epsilon^{1-1/k} ))$.  The proof in Appendix~\ref{appendix.necessity_hyper_w1} also shows that this set $\GG$ for total regression has infinite population limit when we consider linear regression cost.

\subsection{Expanding the set for second moment estimation}

Similar to Section~\ref{sec.expand}, we implement the expanding set idea to provide both statistically efficient and computationally efficient algorithm via $W_1$ projection:

\begin{theorem}[Second moment estimation, $W_1$ or $\tW_1$ projection]\label{thm.efficient_sec_w1}
Assume that $p^*$ has bounded $k$-th moment for $k>2$, i.e. 
    $\sup_{v\in\bR^d, \|v\|_2=1}\bE_{p^*}[|v^\top X|^k] \leq \sigma^k$
for some $\sigma>0$, and that the adversary is able to corrupt the true empirical distribution $\hat p_n^*$ to some distribution $\hat p$ such that $W_1(\hat p_n^*, \hat p)\leq \epsilon$. 
Take the projection set as $ \GG = \{p \mid \sup_{v\in\bR^d, \|v\|_2=1}\bE_{p}[|v^\top X|^k] \leq 2\sigma^k\}$. When $n\gtrsim (d\log(d/\delta))^{k/2}$ the projection algorithm $q = \Pi(\hat p_n; W_1 / \tW_1, \GG)$ or $q = \Pi(\hat p_n; W_1 / \tW_1, \GG, \epsilon)$    satisfies
\begin{align*}
    \|\bE_{q}[XX^{\top}] - \bE_{p^*}[XX^{\top}]\|_2 \leq \min(\sigma^2, C_1 \sigma^{1+1/(k-1)}  \epsilon^{1-1/(k-1)})+C_2\sigma^2 \max(\sqrt{\frac{ d\log(d/\delta)}{n\delta^{2/k}}},  \frac{ d\log(d/\delta)}{n\delta^{2/k}})
\end{align*}
with probability at least $1-\delta$, 
where $C_1, C_2$ are some universal constants.
\end{theorem}
The proof is deferred to Appendix~\ref{proof.efficient_sec_w1}. Compared to the sample complexity that is linear in $d$ in Theorem~\ref{thm.sec_tw_1_proj}, we sacrifice the sample complexity in the exchange of computational efficiency. 
When $k=4$, the MD functional is equivalent to minimizing $\frac{1}{n}\sum_{i=1}^n \|x_i - y_i\|_2$ over $y_1, \cdots, y_n$ subject to $ \sup_{v\in\bR^d} \frac{1}{n}\sum_{i=1}^n  (v^\top y_i)^4\leq 2\sigma^4.$  
The algorithm can be relaxed and made efficient via sum-of-squares program by assuming that the $k$-th moment bound on distribution is certifiable. This provides the first statistically and computationally efficient algorithm under $W_1$ corruption.
\section*{Acknowledgements}
The authors are grateful to Roman Vershynin for discussions that inspired Lemma~\ref{lem.empirical_psi_resilient}, Peter Bartlett for pointing out the VC dimension bound needed for Theorem~\ref{thm.linearregressiontvtildeproof}, and Adarsh Prasad for comments on results of mean estimation with bounded covariance assumption.

\bibliographystyle{plainnat} %
\bibliography{di}       %

\begin{thebibliography}{76}
\providecommand{\natexlab}[1]{#1}
\providecommand{\url}[1]{\texttt{#1}}
\expandafter\ifx\csname urlstyle\endcsname\relax
  \providecommand{\doi}[1]{doi: #1}\else
  \providecommand{\doi}{doi: \begingroup \urlstyle{rm}\Url}\fi

\bibitem[Acharya et~al.(2017)Acharya, Diakonikolas, Li, and
  Schmidt]{acharya2017sample}
Jayadev Acharya, Ilias Diakonikolas, Jerry Li, and Ludwig Schmidt.
\newblock Sample-optimal density estimation in nearly-linear time.
\newblock In \emph{Proceedings of the Twenty-Eighth Annual ACM-SIAM Symposium
  on Discrete Algorithms}, pages 1278--1289. SIAM, 2017.

\bibitem[Adrover and Yohai(2002)]{adrover2002projection}
Jorge Adrover and V{\'\i}ctor Yohai.
\newblock Projection estimates of multivariate location.
\newblock \emph{The Annals of Statistics}, 30\penalty0 (6):\penalty0
  1760--1781, 2002.

\bibitem[Anthony and Bartlett(2009)]{anthony2009neural}
Martin Anthony and Peter~L Bartlett.
\newblock \emph{Neural network learning: Theoretical foundations}.
\newblock cambridge university press, 2009.

\bibitem[Awasthi et~al.(2014)Awasthi, Balcan, and Long]{awasthi2014power}
Pranjal Awasthi, Maria~Florina Balcan, and Philip~M Long.
\newblock The power of localization for efficiently learning linear separators
  with noise.
\newblock In \emph{Proceedings of the forty-sixth annual ACM symposium on
  Theory of computing}, pages 449--458. ACM, 2014.

\bibitem[Bakshi and Prasad(2020)]{bakshi2020robust}
Ainesh Bakshi and Adarsh Prasad.
\newblock Robust linear regression: Optimal rates in polynomial time.
\newblock \emph{arXiv preprint arXiv:2007.01394}, 2020.

\bibitem[Bateni and Dalalyan(2019)]{bateni2019minimax}
Amir-Hossein Bateni and Arnak~S Dalalyan.
\newblock Minimax rates in outlier-robust estimation of discrete models.
\newblock \emph{arXiv preprint arXiv:1902.04650}, 2019.

\bibitem[Ben-Tal et~al.(2009)Ben-Tal, El~Ghaoui, and Nemirovski]{ben2009robust}
Aharon Ben-Tal, Laurent El~Ghaoui, and Arkadi Nemirovski.
\newblock \emph{Robust optimization}, volume~28.
\newblock Princeton University Press, 2009.

\bibitem[Borwein and Lewis(2010)]{borwein2010convex}
Jonathan Borwein and Adrian~S Lewis.
\newblock \emph{Convex analysis and nonlinear optimization: theory and
  examples}.
\newblock Springer Science \& Business Media, 2010.

\bibitem[Boucheron et~al.(2013)Boucheron, Lugosi, and
  Massart]{boucheron2013concentration}
St{\'e}phane Boucheron, G{\'a}bor Lugosi, and Pascal Massart.
\newblock \emph{Concentration inequalities: A nonasymptotic theory of
  independence}.
\newblock Oxford university press, 2013.

\bibitem[Catoni and Giulini(2017)]{catoni2017dimension}
Olivier Catoni and Ilaria Giulini.
\newblock Dimension-free pac-bayesian bounds for matrices, vectors, and linear
  least squares regression.
\newblock \emph{arXiv preprint arXiv:1712.02747}, 2017.

\bibitem[Chan et~al.(2014)Chan, Diakonikolas, Servedio, and
  Sun]{chan2014efficient}
Siu-On Chan, Ilias Diakonikolas, Rocco~A Servedio, and Xiaorui Sun.
\newblock Efficient density estimation via piecewise polynomial approximation.
\newblock In \emph{Proceedings of the forty-sixth annual ACM symposium on
  Theory of computing}, pages 604--613. ACM, 2014.

\bibitem[Chen et~al.(2018)Chen, Gao, Ren, et~al.]{chen2018robust}
Mengjie Chen, Chao Gao, Zhao Ren, et~al.
\newblock Robust covariance and scatter matrix estimation under huber’s
  contamination model.
\newblock \emph{The Annals of Statistics}, 46\penalty0 (5):\penalty0
  1932--1960, 2018.

\bibitem[Chen and Tyler(2002)]{chen2002influence}
Zhiqiang Chen and David~E Tyler.
\newblock The influence function and maximum bias of {T}ukey's median.
\newblock \emph{The Annals of Statistics}, 30\penalty0 (6):\penalty0
  1737--1759, 2002.

\bibitem[Delage and Ye(2010)]{delage2010distributionally}
Erick Delage and Yinyu Ye.
\newblock Distributionally robust optimization under moment uncertainty with
  application to data-driven problems.
\newblock \emph{Operations research}, 58\penalty0 (3):\penalty0 595--612, 2010.

\bibitem[Devroye and Lugosi(2012)]{devroye2012combinatorial}
Luc Devroye and G{\'a}bor Lugosi.
\newblock \emph{Combinatorial methods in density estimation}.
\newblock Springer Science \& Business Media, 2012.

\bibitem[Diakonikolas and Kane(2019)]{diakonikolas2019recent}
Ilias Diakonikolas and Daniel~M Kane.
\newblock Recent advances in algorithmic high-dimensional robust statistics.
\newblock \emph{arXiv preprint arXiv:1911.05911}, 2019.

\bibitem[Diakonikolas et~al.(2016)Diakonikolas, Kamath, Kane, Li, Moitra, and
  Stewart]{diakonikolas2016robust}
Ilias Diakonikolas, Gautam Kamath, Daniel~M Kane, Jerry Li, Ankur Moitra, and
  Alistair Stewart.
\newblock Robust estimators in high dimensions without the computational
  intractability.
\newblock In \emph{2016 IEEE 57th Annual Symposium on Foundations of Computer
  Science (FOCS)}, pages 655--664. IEEE, 2016.

\bibitem[Diakonikolas et~al.(2017)Diakonikolas, Kamath, Kane, Li, Moitra, and
  Stewart]{diakonikolas2017being}
Ilias Diakonikolas, Gautam Kamath, Daniel~M Kane, Jerry Li, Ankur Moitra, and
  Alistair Stewart.
\newblock Being robust (in high dimensions) can be practical.
\newblock In \emph{Proceedings of the 34th International Conference on Machine
  Learning-Volume 70}, pages 999--1008. JMLR. org, 2017.

\bibitem[Diakonikolas et~al.(2018{\natexlab{a}})Diakonikolas, Kamath, Kane, Li,
  Moitra, and Stewart]{diakonikolas2018robustly}
Ilias Diakonikolas, Gautam Kamath, Daniel~M Kane, Jerry Li, Ankur Moitra, and
  Alistair Stewart.
\newblock Robustly learning a gaussian: Getting optimal error, efficiently.
\newblock In \emph{Proceedings of the Twenty-Ninth Annual ACM-SIAM Symposium on
  Discrete Algorithms}, pages 2683--2702. Society for Industrial and Applied
  Mathematics, 2018{\natexlab{a}}.

\bibitem[Diakonikolas et~al.(2018{\natexlab{b}})Diakonikolas, Kamath, Kane, Li,
  Steinhardt, and Stewart]{diakonikolas2018sever}
Ilias Diakonikolas, Gautam Kamath, Daniel~M Kane, Jerry Li, Jacob Steinhardt,
  and Alistair Stewart.
\newblock Sever: A robust meta-algorithm for stochastic optimization.
\newblock \emph{arXiv preprint arXiv:1803.02815}, 2018{\natexlab{b}}.

\bibitem[Diakonikolas et~al.(2018{\natexlab{c}})Diakonikolas, Kane, and
  Stewart]{diakonikolas2018learning}
Ilias Diakonikolas, Daniel~M Kane, and Alistair Stewart.
\newblock Learning geometric concepts with nasty noise.
\newblock In \emph{Proceedings of the 50th Annual ACM SIGACT Symposium on
  Theory of Computing}, pages 1061--1073. ACM, 2018{\natexlab{c}}.

\bibitem[Diakonikolas et~al.(2019{\natexlab{a}})Diakonikolas, Kamath, Kane, Li,
  Moitra, and Stewart]{diakonikolas2019robust}
Ilias Diakonikolas, Gautam Kamath, Daniel Kane, Jerry Li, Ankur Moitra, and
  Alistair Stewart.
\newblock Robust estimators in high-dimensions without the computational
  intractability.
\newblock \emph{SIAM Journal on Computing}, 48\penalty0 (2):\penalty0 742--864,
  2019{\natexlab{a}}.

\bibitem[Diakonikolas et~al.(2019{\natexlab{b}})Diakonikolas, Kong, and
  Stewart]{diakonikolas2019efficient}
Ilias Diakonikolas, Weihao Kong, and Alistair Stewart.
\newblock Efficient algorithms and lower bounds for robust linear regression.
\newblock In \emph{Proceedings of the Thirtieth Annual ACM-SIAM Symposium on
  Discrete Algorithms}, pages 2745--2754. SIAM, 2019{\natexlab{b}}.

\bibitem[Diakonikolas et~al.(2020)Diakonikolas, Kane, and
  Pensia]{diakonikolas2020outlier}
Ilias Diakonikolas, Daniel~M Kane, and Ankit Pensia.
\newblock Outlier robust mean estimation with subgaussian rates via stability.
\newblock \emph{arXiv preprint arXiv:2007.15618}, 2020.

\bibitem[Donoho(1982)]{donoho1982breakdown}
David~L Donoho.
\newblock Breakdown properties of multivariate location estimators.
\newblock Technical report, Technical report, Harvard University, Boston. URL
  http://www-stat. stanford~…, 1982.

\bibitem[Donoho and Liu(1988{\natexlab{a}})]{donoho1988automatic}
David~L Donoho and Richard~C Liu.
\newblock The ``automatic'' robustness of minimum distance functionals.
\newblock \emph{The Annals of Statistics}, 16\penalty0 (2):\penalty0 552--586,
  1988{\natexlab{a}}.

\bibitem[Donoho and Liu(1988{\natexlab{b}})]{donoho1988pathologies}
David~L Donoho and Richard~C Liu.
\newblock Pathologies of some minimum distance estimators.
\newblock \emph{The Annals of Statistics}, pages 587--608, 1988{\natexlab{b}}.

\bibitem[Duchi and Namkoong(2018)]{duchi2018learning}
John Duchi and Hongseok Namkoong.
\newblock Learning models with uniform performance via distributionally robust
  optimization.
\newblock \emph{arXiv preprint arXiv:1810.08750}, 2018.

\bibitem[Dudley(1978)]{dudley1978central}
Richard~M Dudley.
\newblock Central limit theorems for empirical measures.
\newblock \emph{The Annals of Probability}, pages 899--929, 1978.

\bibitem[Dudley(1969)]{dudley1969speed}
Richard~Mansfield Dudley.
\newblock The speed of mean glivenko-cantelli convergence.
\newblock \emph{The Annals of Mathematical Statistics}, 40\penalty0
  (1):\penalty0 40--50, 1969.

\bibitem[Dvoretzky et~al.(1956)Dvoretzky, Kiefer, and
  Wolfowitz]{dvoretzky1956asymptotic}
Aryeh Dvoretzky, Jack Kiefer, and Jacob Wolfowitz.
\newblock Asymptotic minimax character of the sample distribution function and
  of the classical multinomial estimator.
\newblock \emph{The Annals of Mathematical Statistics}, 27\penalty0
  (3):\penalty0 642--669, 1956.

\bibitem[Foucart and Rauhut(2017)]{foucart2017mathematical}
Simon Foucart and Holger Rauhut.
\newblock A mathematical introduction to compressive sensing.
\newblock \emph{Bull. Am. Math}, 54:\penalty0 151--165, 2017.

\bibitem[Gao(2017)]{gao2017robust}
Chao Gao.
\newblock Robust regression via mutivariate regression depth.
\newblock \emph{arXiv preprint arXiv:1702.04656}, 2017.

\bibitem[Gao et~al.(2018)Gao, Liu, Yao, and Zhu]{gao2018robust}
Chao Gao, Jiyi Liu, Yuan Yao, and Weizhi Zhu.
\newblock Robust estimation and generative adversarial nets.
\newblock \emph{arXiv preprint arXiv:1810.02030}, 2018.

\bibitem[Gao et~al.(2019)Gao, Yao, and Zhu]{gao2019generative}
Chao Gao, Yuan Yao, and Weizhi Zhu.
\newblock Generative adversarial nets for robust scatter estimation: A proper
  scoring rule perspective.
\newblock \emph{arXiv preprint arXiv:1903.01944}, 2019.

\bibitem[Gneiting and Raftery(2007)]{gneiting2007strictly}
Tilmann Gneiting and Adrian~E Raftery.
\newblock Strictly proper scoring rules, prediction, and estimation.
\newblock \emph{Journal of the American Statistical Association}, 102\penalty0
  (477):\penalty0 359--378, 2007.

\bibitem[Golub and Van~Loan(1980)]{golub1980analysis}
Gene~H Golub and Charles~F Van~Loan.
\newblock An analysis of the total least squares problem.
\newblock \emph{SIAM journal on numerical analysis}, 17\penalty0 (6):\penalty0
  883--893, 1980.

\bibitem[Gromov(2007)]{gromov2007metric}
Mikhail Gromov.
\newblock \emph{Metric structures for Riemannian and non-Riemannian spaces}.
\newblock Springer Science \& Business Media, 2007.

\bibitem[Haagerup(1981)]{haagerup1981best}
Uffe Haagerup.
\newblock The best constants in the {K}hintchine inequality.
\newblock \emph{Studia Mathematica}, 70:\penalty0 231--283, 1981.

\bibitem[Huber(1973)]{huber1973robust}
Peter~J Huber.
\newblock Robust regression: asymptotics, conjectures and monte carlo.
\newblock \emph{The Annals of Statistics}, 1\penalty0 (5):\penalty0 799--821,
  1973.

\bibitem[Huber(2011)]{huber2011robust}
Peter~J Huber.
\newblock \emph{Robust statistics}.
\newblock Springer, 2011.

\bibitem[Joly et~al.(2017)Joly, Lugosi, Oliveira, et~al.]{joly2017estimation}
Emilien Joly, G{\'a}bor Lugosi, Roberto~Imbuzeiro Oliveira, et~al.
\newblock On the estimation of the mean of a random vector.
\newblock \emph{Electronic Journal of Statistics}, 11\penalty0 (1):\penalty0
  440--451, 2017.

\bibitem[Klivans et~al.(2018)Klivans, Kothari, and Meka]{klivans2018efficient}
Adam Klivans, Pravesh~K Kothari, and Raghu Meka.
\newblock Efficient algorithms for outlier-robust regression.
\newblock \emph{arXiv preprint arXiv:1803.03241}, 2018.

\bibitem[Klivans et~al.(2009)Klivans, Long, and Servedio]{klivans2009learning}
Adam~R Klivans, Philip~M Long, and Rocco~A Servedio.
\newblock Learning halfspaces with malicious noise.
\newblock \emph{Journal of Machine Learning Research}, 10\penalty0
  (Dec):\penalty0 2715--2740, 2009.

\bibitem[Kothari and Steurer(2017)]{kothari2017outlier}
Pravesh~K Kothari and David Steurer.
\newblock Outlier-robust moment-estimation via sum-of-squares.
\newblock \emph{arXiv preprint arXiv:1711.11581}, 2017.

\bibitem[Lai et~al.(2016)Lai, Rao, and Vempala]{lai2016agnostic}
Kevin~A Lai, Anup~B Rao, and Santosh Vempala.
\newblock Agnostic estimation of mean and covariance.
\newblock In \emph{2016 IEEE 57th Annual Symposium on Foundations of Computer
  Science (FOCS)}, pages 665--674. IEEE, 2016.

\bibitem[Lecu{\'e} and Depersin(2019)]{lecue2019robust}
Guillaume Lecu{\'e} and Jules Depersin.
\newblock Robust subgaussian estimation of a mean vector in nearly linear time.
\newblock \emph{arXiv preprint arXiv:1906.03058}, 2019.

\bibitem[Liu et~al.(2018)Liu, Shen, Li, and Caramanis]{liu2018high}
Liu Liu, Yanyao Shen, Tianyang Li, and Constantine Caramanis.
\newblock High dimensional robust sparse regression.
\newblock \emph{arXiv preprint arXiv:1805.11643}, 2018.

\bibitem[Lugosi(2017)]{lugosi2017lectures}
G{\'a}bor Lugosi.
\newblock Lectures on combinatorial statistics, 2017.

\bibitem[Lugosi and Mendelson(2019)]{lugosi2019sub}
G{\'a}bor Lugosi and Shahar Mendelson.
\newblock Sub-gaussian estimators of the mean of a random vector.
\newblock \emph{The Annals of Statistics}, 47\penalty0 (2):\penalty0 783--794,
  2019.

\bibitem[Luke{\v{s}} et~al.(2009)Luke{\v{s}}, Mal{\'y}, Netuka, and
  Spurn{\'y}]{lukevs2009integral}
Jaroslav Luke{\v{s}}, Jan Mal{\'y}, Ivan Netuka, and Jir{\'i} Spurn{\'y}.
\newblock \emph{Integral representation theory: applications to convexity,
  Banach spaces and potential theory}, volume~35.
\newblock Walter de Gruyter, 2009.

\bibitem[Marshall et~al.(1979)Marshall, Olkin, and
  Arnold]{marshall1979inequalities}
Albert~W Marshall, Ingram Olkin, and Barry~C Arnold.
\newblock \emph{Inequalities: theory of majorization and its applications},
  volume 143.
\newblock Springer, 1979.

\bibitem[Massart(2007)]{massart2007concentration}
Pascal Massart.
\newblock Concentration inequalities and model selection.
\newblock 2007.

\bibitem[Massey~Jr(1951)]{massey1951kolmogorov}
Frank~J Massey~Jr.
\newblock The {K}olmogorov-{S}mirnov test for goodness of fit.
\newblock \emph{Journal of the American statistical Association}, 46\penalty0
  (253):\penalty0 68--78, 1951.

\bibitem[Mitzenmacher and Upfal(2017)]{mitzenmacher2017probability}
Michael Mitzenmacher and Eli Upfal.
\newblock \emph{Probability and computing: randomization and probabilistic
  techniques in algorithms and data analysis}.
\newblock Cambridge university press, 2017.

\bibitem[Namkoong and Duchi(2016)]{namkoong2016stochastic}
Hongseok Namkoong and John~C Duchi.
\newblock Stochastic gradient methods for distributionally robust optimization
  with f-divergences.
\newblock In \emph{Advances in Neural Information Processing Systems}, pages
  2208--2216, 2016.

\bibitem[Okamoto(1959)]{okamoto1959some}
Masashi Okamoto.
\newblock Some inequalities relating to the partial sum of binomial
  probabilities.
\newblock \emph{Annals of the institute of Statistical Mathematics},
  10\penalty0 (1):\penalty0 29--35, 1959.

\bibitem[Penot(2012)]{penot2012calculus}
Jean-Paul Penot.
\newblock \emph{Calculus without derivatives}, volume 266.
\newblock Springer Science \& Business Media, 2012.

\bibitem[Prasad et~al.(2018)Prasad, Suggala, Balakrishnan, and
  Ravikumar]{prasad2018robust}
Adarsh Prasad, Arun~Sai Suggala, Sivaraman Balakrishnan, and Pradeep Ravikumar.
\newblock Robust estimation via robust gradient estimation.
\newblock \emph{arXiv preprint arXiv:1802.06485}, 2018.

\bibitem[Prasad et~al.(2019)Prasad, Balakrishnan, and
  Ravikumar]{prasad2019unified}
Adarsh Prasad, Sivaraman Balakrishnan, and Pradeep Ravikumar.
\newblock A unified approach to robust mean estimation.
\newblock \emph{arXiv preprint arXiv:1907.00927}, 2019.

\bibitem[Ren and Liang(2001)]{ren2001best}
Yao-Feng Ren and Han-Ying Liang.
\newblock On the best constant in {M}arcinkiewicz--{Z}ygmund inequality.
\newblock \emph{Statistics \& probability letters}, 53\penalty0 (3):\penalty0
  227--233, 2001.

\bibitem[Rothschild and Stiglitz(1978)]{rothschild1978increasing}
Michael Rothschild and Joseph~E Stiglitz.
\newblock Increasing risk: I. a definition.
\newblock In \emph{Uncertainty in Economics}, pages 99--121. Elsevier, 1978.

\bibitem[Steinhardt(2018)]{steinhardt2018robust}
Jacob Steinhardt.
\newblock \emph{Robust Learning: Information Theory and Algorithms}.
\newblock PhD thesis, Stanford University, 2018.

\bibitem[Steinhardt et~al.(2017{\natexlab{a}})Steinhardt, Charikar, and
  Valiant]{steinhardt2017resilience}
Jacob Steinhardt, Moses Charikar, and Gregory Valiant.
\newblock Resilience: A criterion for learning in the presence of arbitrary
  outliers.
\newblock \emph{arXiv preprint arXiv:1703.04940}, 2017{\natexlab{a}}.

\bibitem[Steinhardt et~al.(2017{\natexlab{b}})Steinhardt, Koh, and
  Liang]{steinhardt2017certified}
Jacob Steinhardt, Pang Wei~W Koh, and Percy~S Liang.
\newblock Certified defenses for data poisoning attacks.
\newblock In \emph{Advances in neural information processing systems}, pages
  3517--3529, 2017{\natexlab{b}}.

\bibitem[Steinhardt et~al.(2018)Steinhardt, Charikar, and
  Valiant]{steinhardt2018resilience}
Jacob Steinhardt, Moses Charikar, and Gregory Valiant.
\newblock Resilience: A criterion for learning in the presence of arbitrary
  outliers.
\newblock In \emph{9th Innovations in Theoretical Computer Science Conference
  (ITCS 2018)}, volume~94, page~45. Schloss Dagstuhl--Leibniz-Zentrum fuer
  Informatik, 2018.

\bibitem[Vapnik and Chervonenkis(2015)]{vapnik2015uniform}
Vladimir~N Vapnik and A~Ya Chervonenkis.
\newblock On the uniform convergence of relative frequencies of events to their
  probabilities.
\newblock In \emph{Measures of complexity}, pages 11--30. Springer, 2015.

\bibitem[Vershynin(2010)]{vershynin2010introduction}
Roman Vershynin.
\newblock Introduction to the non-asymptotic analysis of random matrices.
\newblock \emph{arXiv preprint arXiv:1011.3027}, 2010.

\bibitem[Vershynin(2018)]{vershynin2018high}
Roman Vershynin.
\newblock \emph{High-dimensional probability: An introduction with applications
  in data science}, volume~47.
\newblock Cambridge University Press, 2018.

\bibitem[Villani(2003)]{villani2003topics}
C{\'e}dric Villani.
\newblock \emph{Topics in optimal transportation}.
\newblock American Mathematical Soc., 2003.

\bibitem[Volpi et~al.(2018)Volpi, Namkoong, Sener, Duchi, Murino, and
  Savarese]{volpi2018generalizing}
Riccardo Volpi, Hongseok Namkoong, Ozan Sener, John~C Duchi, Vittorio Murino,
  and Silvio Savarese.
\newblock Generalizing to unseen domains via adversarial data augmentation.
\newblock In \emph{Advances in Neural Information Processing Systems}, pages
  5334--5344, 2018.

\bibitem[Wainwright(2019)]{wainwright2019high}
Martin~J Wainwright.
\newblock \emph{High-dimensional statistics: A non-asymptotic viewpoint},
  volume~48.
\newblock Cambridge University Press, 2019.

\bibitem[Yatracos(1985)]{yatracos1985rates}
Yannis~G Yatracos.
\newblock Rates of convergence of minimum distance estimators and
  {K}olmogorov's entropy.
\newblock \emph{The Annals of Statistics}, pages 768--774, 1985.

\bibitem[Young(1912)]{young1912classes}
William~Henry Young.
\newblock On classes of summable functions and their fourier series.
\newblock \emph{Proceedings of the Royal Society of London. Series A,
  Containing Papers of a Mathematical and Physical Character}, 87\penalty0
  (594):\penalty0 225--229, 1912.

\bibitem[Zhu et~al.(2019)Zhu, Jiao, and Tse]{zhu2019deconstructing}
Banghua Zhu, Jiantao Jiao, and David Tse.
\newblock Deconstructing generative adversarial networks.
\newblock \emph{arXiv preprint arXiv:1901.09465}, 2019.

\bibitem[Zhu et~al.(2020)Zhu, Jiao, and Steinhardt]{zhu2020robust}
Banghua Zhu, Jiantao Jiao, and Jacob Steinhardt.
\newblock Robust estimation via generalized quasi-gradients.
\newblock \emph{arXiv preprint arXiv:2005.14073}, 2020.

\end{thebibliography}
\newpage
\begin{appendix}
 {\bf  \Large Supplementary Material for Generalized Resilience and Robust Statistics}

\section{Two Approaches for Finite Sample Analysis}\label{sec.general_analysis}

We summarize two approaches to analyze the general projection algorithm~(Algorithm~\ref{alg.general_projection}) for oblivious corruption~(Definition~\ref{def.obliviouscorruption}) and adaptive corruption~(Definition~\ref{def.adaptivecont}). %

\begin{enumerate}
    \item \emph{Find $p^*$}: This approach aims at finding the true population distribution $p^*$ or its perturbations in the projection. Theorem~\ref{thm.admissible-1} and~\ref{thm.admissible-1-adaptive} present the general conditions under which this approach works. 
    \item \emph{Find $\hat p_n^*$}: This approach aims at finding the empirical distribution sampled from the true distribution $\hat p^*_n$ or its perturbations in the projection. Theorem~\ref{thm.admissible-2} presents the general conditions under which this approach works.  
\end{enumerate}

Among these two approaches there does not exist one analysis approach that strictly dominates the other, and in various cases one can apply both analysis approaches to obtain different bounds that are better in different parameter regimes.

\subsection{Find $p^*$}

\begin{theorem}[Find $p^*$ (Oblivious Corruption)]\label{thm.admissible-1}
Assume the oblivious corruption model of level $\epsilon$ under $D$. Denote the true distribution as $p^* \in \GG$ and the perturbed population distribution as $p$ with $D(p^*,p) \leq \epsilon$. Denote the cost function as $L(p^*, \theta)$ and the empirical distribution of observed data as $\hat{p}_n$. Assume the following conditions: 
\begin{enumerate}
    \item \textbf{Robust to perturbation:} there exists a function $\overline{D}$ such that for any $p_2, p_3$, 
    \begin{align}
      \sup_{p_1\in\cM} |\tD(p_1, p_2) - \tD(p_1, p_3) | \leq \overline{D}(p_2, p_3) \leq D(p_2, p_3). 
    \end{align}
    \item \textbf{Generalized Modulus of Continuity:}  $\mathcal{M} \supset \GG$, and for $\tilde{\epsilon} = 2\epsilon + 2\overline{D}(p, \hat{p}_n)$, we have
    \begin{align}\label{eqn.modulus_general}
    \sup_{p_1^*\in \cM, p_2^* \in \GG, \tD(p_1^*, p_2^*) - \tD(p_2^*, p_2^*) \leq \tilde{\epsilon}} L(p_2^*, \theta^*(p_1^*)) \leq \rho(\tilde{\epsilon}).
    \end{align}

\end{enumerate}
 
Then the projection algorithm $ q=\Pi(\hat p_n; \tD, \cM)$, $\theta^*(q) = \argmin_{\theta\in\Theta} L(q, \theta)$ satisfies
\begin{align}\label{eqn.ad1_robust}
    L(p^*, \theta^*(q)) \leq \rho(2\epsilon + 2\overline{D}(p, \hat{p}_n)).
\end{align}
\end{theorem}

\begin{proof}
By the `robust to perturbation' property of $D$, we have
\begin{align}
\tD(q,p^*) -  \tD(p^*,p^*) & \leq \tD(q, p) + \overline{D}(p^*, p) -  \tD(p^*,p^*) \nonumber \\
& \leq \tD(q, \hat p_n) + \overline{D}(p, \hat p_n) + \overline{D}(p^*, p) -  \tD(p^*,p^*)  \nonumber \\
& \leq \tD(p^*, \hat p_n) + \overline{D}(p, \hat p_n)+ \overline{D}(p^*, p) -  \tD(p^*,p^*) \nonumber \\
& \leq \tD(p^*, p) + 2\overline{D}(p, \hat p_n) + \overline{D}(p^*, p)  -  \tD(p^*,p^*) \nonumber \\
& \leq \tD(p^*, p^*)+ 2\overline{D}(p, \hat p_n)  + 2\overline{D}(p^*, p)  -  \tD(p^*,p^*) \nonumber \\
& =2\overline{D}(p, \hat p_n) + 2\overline{D}(p^*, p) \nonumber \\
& \leq 2\overline{D}(p, \hat p_n) + 2D(p^*, p) \nonumber \\
& \leq 2\overline{D}(p, \hat p_n) +2\epsilon.
\end{align}

We also know that $q\in \cM$. 
Hence, by the generalized modulus of continuity property: 
\begin{align}
L(p^*, \theta^*(q))\leq 
\sup_{p_1^*\in \cM, p_2^* \in \GG, \tD(p_1^*, p_2^*) - \tD(p^*_2,p^*_2) \leq 2\epsilon + 2\overline{D}(p, \hat p_n) } L(p_2^*, \theta^*(p_1^*)) \leq \rho(2\epsilon+ 2\overline{D}(p, \hat p_n)),
\end{align}
we can derive the conclusion.
\end{proof}

\begin{proposition}[Any $q$ suffices, not just the minimizer]\label{prop.any_q_suffices_oblivious}
Assume the conditions in Theorem~\ref{thm.admissible-1} and further $\tilde{D}(p,p) = 0$ for any $p$. Suppose that for any $p$ such that $D(p^*,p)\leq \epsilon, p^*\in \GG$, we have $\overline{D}(p, \hat{p}_n) \leq d_n$ uniformly over $p$ with probability at least $1-\delta$. Then, it follows from the proof of Theorem~\ref{thm.admissible-1} that for any $q \in \mathcal{M}$ such that
\begin{align}\label{eqn.qingeneral}
    \tD(q,\hat{p}_n) \leq \epsilon + d_n,
\end{align}
we have,
\begin{align}
    L(p^*, \theta^*(q)) \leq \rho(2\epsilon + 2 d_n),
\end{align}
and the the existence of $q$ satisfying~(\ref{eqn.qingeneral}) happens with probability at least $1-\delta$.  
\end{proposition}

\begin{proposition}[Solving robust inference under more general perturbations] \label{prop.obliviousmoregeneralperturbation}
It follows from the proof of Theorem~\ref{thm.admissible-1} that the final finite sample error bound still holds if we allow more general perturbations: instead of allowing any $p$ such that $D(p^*,p)\leq \epsilon$, we allow any $p$ such that $\sup_{r\in \mathcal{M}} |\tD(r, p^*) - \tD(r,p)|\leq 
\epsilon$. Hence, as long as the conditions in Theorem~\ref{thm.admissible-1} are satisfied, the projection algorithm performs well with this bigger set of arbitrary perturbations. 
\end{proposition}

\begin{corollary}
Consider the case of $n = \infty$ in Theorem~\ref{thm.admissible-1}, we know that if $q=\argmin\{\tD(q, p) \mid q \in \cM\}$, $\theta^*(q) = \argmin_{\theta\in\Theta} L(q, \theta)$, then
\begin{align}
    L(p^*, \theta^*(q)) \leq \rho(2\epsilon).
\end{align}
\end{corollary}

\begin{remark}
Setting $\tD = \overline{D} = \TV$ and $\mathcal{M} = \mathcal{G}^{\TV}(\rho_1,\rho_2, \tilde{\epsilon})$ in Theorem~\ref{thm.admissible-1} leads to the following bound~(Theorem~\ref{thm.G_fundamental_limit})
\begin{align}
    L(p^*, \theta^*(q)) \leq \rho_2(\rho_1(\tilde{\epsilon}), \tilde{\epsilon}),
\end{align}
where $\tilde{\epsilon} = 2\epsilon + 2\TV(p, \hat p_n)$. However, it would easily be a very loose bound if the contaminated distribution $p$ is a continuous distribution since in this case $\TV(p, \hat p_n) =1$ almost surely. To fully utilize the power of Theorem~\ref{thm.admissible-1}, one needs to design $\tD$ and $\overline{D}$ such that $\overline{D}(p, \hat{p}_n)$ vanishes fast enough. 
\end{remark}

Next theorem discusses the ``find $p^*$'' approach for adaptive corruption. 

\begin{theorem}[Find $p^*$ (Adaptive Corruption)]\label{thm.admissible-1-adaptive}
Assume the adaptive corruption model of level $\epsilon$ under $D$ (Definition~\ref{def.adaptivecont}). 
Denote the true distribution as $p^* \in \GG$ and the corresponding empirical distribution as $\hat p_n^*$. Denote the cost function as $L(p^*, \theta)$ and the empirical distribution of observed data as $\hat{p}_n$. We further assume the following conditions: 
\begin{enumerate}
    \item \textbf{Robust to perturbation:} there exists a pseudometric $\overline{D}$ such that for any $p_2, p_3$, 
    \begin{align}
      \sup_{p_1\in\cM} |\tD(p_1, p_2) - \tD(p_1, p_3) | \leq \overline{D}(p_2, p_3) \leq D(p_2, p_3). 
    \end{align}
      \item \textbf{Closeness between empirical distributions:}  $D(\hat p_n^*, \hat p_n) \leq \tilde \epsilon_\delta$ with probability at least $1-\delta$.
    \item \textbf{Generalized Modulus of Continuity:}  $\mathcal{M} \supset \GG$, and for $\tilde{\epsilon} = 2\tilde \epsilon_\delta + 2\overline{D}(p^*, \hat{p}_n^*)$, we have
    \begin{align}
    \sup_{p_1^*\in \cM, p_2^* \in \GG, \tD(p_1^*, p_2^*) - \tD(p_2^*, p_2^*) \leq \tilde{\epsilon}} L(p_2^*, \theta^*(p_1^*)) \leq \rho(\tilde{\epsilon}).
    \end{align}
\end{enumerate}
 
Then with probability at least $1-\delta$, the projection algorithm $ q=\Pi(\hat p_n; \tD, \cM)$, $\theta^*(q) = \argmin_{\theta\in\Theta} L(q, \theta)$ satisfies
\begin{align}
    L(p^*, \theta^*(q)) \leq \rho(2\tilde \epsilon_\delta + 2\overline{D}(p^*, \hat{p}_n^*)).
\end{align}
\end{theorem}

\begin{proof}
By the `robust to perturbation' property of $D$, we have
\begin{align}
\tD(q,p^*) -  \tD(p^*,p^*) & \leq \tD(q, \hat p_n) + \overline{D}(\hat p_n, p^*) -  \tD(p^*,p^*) \nonumber \\
& \leq \tD(p^*, \hat p_n) + \overline{D}(\hat p_n, p^*) -  \tD(p^*,p^*)  \nonumber \\
& \leq \tD(p^*, p^*) + 2\overline{D}(\hat p_n, p^*) -  \tD(p^*,p^*) \nonumber \\
& \leq 2\overline{D}(\hat p_n, \hat p^*_n) + 2\overline{D}(\hat p_n^*, p^*) \nonumber \\
& \leq 2D(\hat p_n, \hat p^*_n) + 2\overline{D}(\hat p_n^*, p^*) \nonumber \\ 
& \leq 2\tilde \epsilon_\delta + 2\overline{D}(p^*, \hat p_n^*)
\end{align}
with probability at least $1-\delta$. 

We also know that $q\in \cM$. 
Hence, by the generalized modulus of continuity property: 
\begin{align}
L(p^*, \theta^*(q))\leq 
\sup_{p_1^*\in \cM, p_2^* \in \GG, \tD(p_1^*, p_2^*) - \tD(p^*_2,p^*_2) \leq 2\tilde \epsilon_\delta + 2\overline{D}(p, \hat p_n) } L(p_2^*, \theta^*(p_1^*)) \leq \rho(2\tilde \epsilon_\delta + 2\overline{D}(p^*, \hat p_n^*)),
\end{align}
we can derive the conclusion.
\end{proof}

Theorem~\ref{thm.admissible-1} is summarized in the Figure~\ref{fig:oblivious}.
\begin{figure}[H]
    \centering
    \begin{tikzpicture}
    \node (p-star) at (-8, 0) {$p^*$};
    \node[above=0.15em of p-star] (p-star-label) {{\footnotesize true distribution $\in \GG$}};
    \node (p-tilde) at (-4, 0) {$p$};
    \node[above=0.15em of p-tilde] (p-tilde-label) {{\footnotesize corrupted distribution}};
    \draw (p-star) edge[->] node[above]{{\footnotesize $D(p^*, p) \leq \epsilon$}} (p-tilde);
    \node (samples) at (0, 0) {$\hat p_n$};
    \node[above=0.1em of samples] (samples-label) {{\footnotesize empirical distribution}};
    \draw (p-tilde) edge[->] (samples);
    \node (q) at (4, 0) {$q$};
    \node[above=0.1em of q] (q-label) {{\footnotesize projected $q\in \mathcal{M}$}};
    \draw (samples) edge[->] node[above]{\footnotesize projection} (q);
    \node (result) at (-2, -4.4) {$L(p^*, \theta^*(q))\leq \cdots$};
    \node (theta-hat) at (4, -4.4) {$\theta^*(q)$};
    \node[above=0.1em of theta-hat] (theta-hat-label) {{\footnotesize estimated parameters}};
    \draw (q)  edge[->]   (result);
    \draw (q)  edge[->] (theta-hat-label);
    \draw (p-star) edge[->]   (result);
    \draw (theta-hat) edge[->]  (result);
    \node[above=2.0em of result] (result-label) {{\footnotesize modulus of continuity }};
    \node[above=4.0em of result] (result-label2) {{\footnotesize $\tD(q, p^*) - \tD(p^*, p^*) \leq 2\epsilon + 2\overline{D}(p, \hat p_n)$ }};
    \end{tikzpicture}
    \caption{Framework for analysis in Theorem~\ref{thm.admissible-1}. }
    \label{fig:oblivious}
\end{figure}
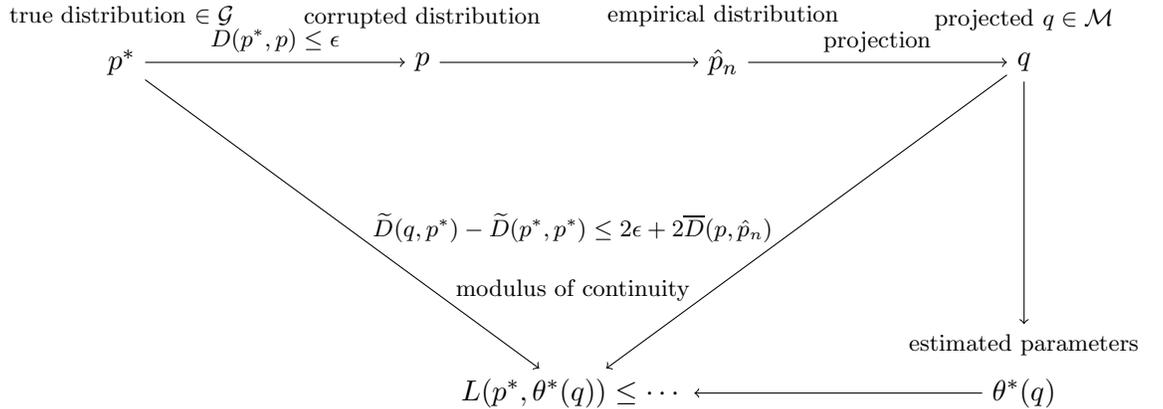

\subsection{Find $\hat p_n^*$}

One motivation for finding $\hat p_n^*$ instead of $p^*$ in the projection is that if $p^*$ is a continuous distribution and we use $\TV$ projection, then $\TV(p^*, \hat p_n)$ is always one, but $\TV(\hat p_n^*, \hat p_n)$ is small.

\begin{theorem}[Find $\hat p_n^*$]\label{thm.admissible-2}
Assume either oblivious corruption or adaptive corruption model of level $\epsilon$ under $D$. Denote the true distribution as $p^* \in \GG$, $\hat{p}_n^*$ as the empirical distribution sampled from $p^*$ and $\hat p_n$ as the empirical distribution of observed data. Denote the cost function as $L(p^*,\theta)$. Assume the following conditions hold:

\begin{enumerate}
     \item \textbf{Robust to perturbation:} there exists a function $\overline{D}$ such that for any $p_2, p_3$, 
    \begin{align}
      \sup_{p_1\in\cM} |\tD(p_1, p_2) - \tD(p_1, p_3) | \leq  \overline{D}(p_2, p_3). 
    \end{align}
    \item \textbf{Limited corruption:} $\overline{D}(\hat p_n, \hat p_n^*) \leq \epsilon_2$ with probability at least $1-\delta$.
     \item \textbf{Set for (perturbed) empirical distribution:} there exists a set $\GG' \subset \mathcal{M}$ such that there exists a distribution $\hat p'\in \GG'$ satisfying $\overline{D}(
    \hat p_n^*, \hat p')\leq \epsilon_1$ with probability at least $1-\delta$.
    \item \textbf{Generalized Modulus of Continuity:}  $\GG' \subset \mathcal{M}$, and for $\tilde{\epsilon} = 2(\epsilon_1 + \epsilon_2)$, we have
    \begin{align}
    \sup_{p_1^*\in \mathcal{M}, p_2^* \in \GG', \tD(p_1^*, p_2^*) - \tD(p_2^*, p_2^*)\leq \tilde{\epsilon}} L(p_2^*, \theta^*(p_1^*)) \leq \rho(\tilde{\epsilon}),
    \end{align}
    \item \textbf{Generalization bound:} for any $p^*\in \GG, \theta$, there exists some constant $C$ and some function $g$ such that $L(p^*, \theta) \leq C\cdot L(\hat p', \theta)+   g(\hat p', p^*) $. 
\end{enumerate}

 Then with probability at least $1-2\delta$, projection algorithm $ q=\Pi(\hat p_n; \tD, \cM), \theta^*(q) = \argmin_{\theta} L(q, \theta)$ satisfies
\begin{align}\label{eqn.admissible2_bound}
     L(p^*, \theta^*(q)) \leq  C \rho(2\epsilon_1 + 2\epsilon_2) + g(\hat p', p^*).
\end{align}
\end{theorem}

\begin{remark}
In the mean estimation setting where $L(p,\theta) = \| \mathbb{E}_p[X] - \theta\|_2$, a generalization bound may be of the form $C = 1$, 
\begin{align}
    g(\hat{p}', p^*) = \| \mathbb{E}_{\hat{p}'}[X] - \mathbb{E}_{p^*}[X] \|_2,
\end{align}
which can be shown using the triangle inequality. 
\end{remark}

\begin{proof}
It follows from the assumptions that with probability at least $1-2\delta$, there exists $\hat p' \in \GG', \overline{D}(\hat p', \hat p_n^*)\leq \epsilon_1, \overline{D}(\hat p_n^*, \hat p_n)\leq \epsilon_2$. Then, 
\begin{align}
  \tD(q, \hat p') - \tD(\hat p', \hat p') & \leq    \tD(q, \hat p_n^*) + \overline{D}(\hat p', \hat p_n^*)  - \tD(\hat p', \hat p') \nonumber \\
  & \leq    \tD(q, \hat p_n^*) + \epsilon_1  - \tD(\hat p', \hat p') \nonumber \\
  & \leq    \tD(q, \hat p_n) + \overline{D}(\hat p_n, \hat p_n^*)  + \epsilon_1  - \tD(\hat p', \hat p') \nonumber \\
      & \leq \tD(q, \hat{p}_n)   + \epsilon_1 + \epsilon_2 - \tD(\hat p', \hat p') \nonumber \\
    & \leq \tD(\hat p', \hat{p}_n)   + \epsilon_1 + \epsilon_2 - \tD(\hat p', \hat p') \nonumber \\
    & \leq \tD(\hat p', \hat{p}_n^*)   + \epsilon_1 + 2\epsilon_2 - \tD(\hat p', \hat p') \nonumber \\
& \leq \tD(\hat p', \hat p') + 2 \epsilon_1 + 2\epsilon_2 - \tD(\hat p', \hat p') \nonumber \\
& = 2 \epsilon_1 + 2\epsilon_2.
\end{align}

By the modulus of continuity condition and $q\in \mathcal{M}$, we know that with probability at least $1-2\delta$, we have
\begin{align}
     L(\hat p', \theta^*(q)) \leq \sup_{p_1^*\in \mathcal{M}, p_2^* \in \GG', \tD(p_1^*, p_2^*) - \tD(p_2^*, p_2^*)   \leq 2\epsilon_1 + 2\epsilon_2} L(p_2^*, p_1^*)  \leq \rho(2\epsilon_1 + 2\epsilon_2).
\end{align}

By the generalization bound condition, we have with probability at least $1-2 \delta$,
\begin{align}
    L(p^*, \theta^*(q)) & \leq C \rho(2\epsilon_1 + 2\epsilon_2) + g(\hat{p}', p^*). 
\end{align}

\end{proof}

\begin{proposition}[Any $q$ suffices, not just the minimizer]
Assume the conditions in Theorem~\ref{thm.admissible-2} and further $\tD(p,p) = 0$ for any $p$. Then, it follows from the proof of Theorem~\ref{thm.admissible-2} that for any $q\in \cM$ such that
\begin{align}\label{eqn.qingeneral2}
    \tD(q,\hat{p}_n) \leq \epsilon_1 + \epsilon_2,
\end{align}
we have
\begin{align}
    L(p^*, \theta^*(q)) \leq C \rho(2\epsilon_1 + 2\epsilon_2) + g(\hat{p}', p^*),
\end{align}
and the the existence of $q$ satisfying~(\ref{eqn.qingeneral2}) happens with probability at least $1-2\delta$.  
\end{proposition}

Theorem~\ref{thm.admissible-2} is summarized in Figure~\ref{fig:coupling}.

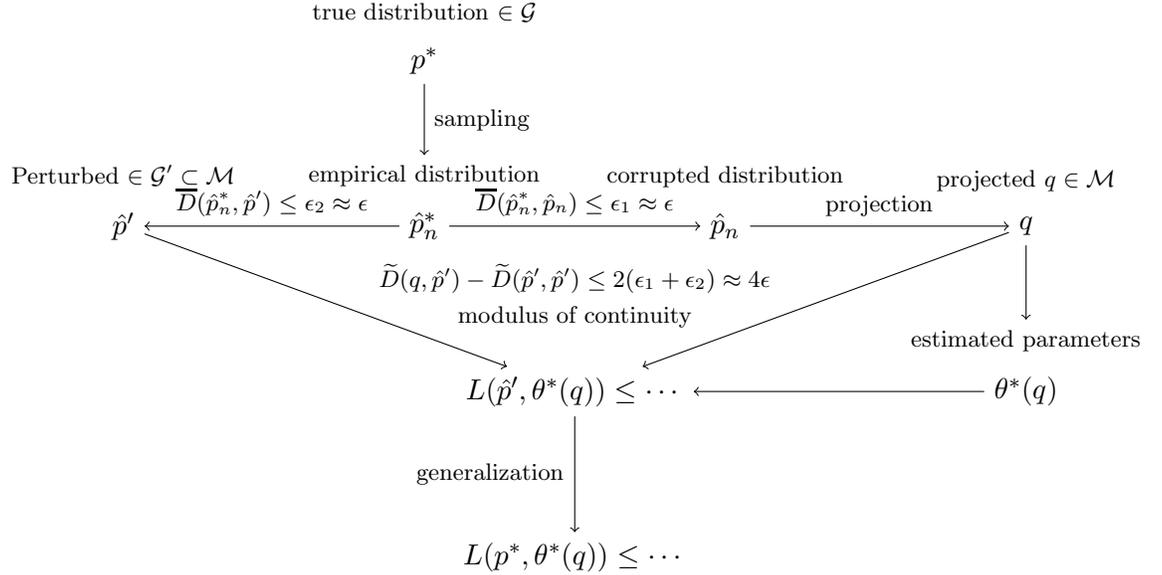
\begin{figure}[H]
    \centering
    \begin{tikzpicture}
    \node (p-star) at (-4, 0) {$p^*$};
    \node[above=0.2em of p-star] (p-star-label) {{\footnotesize true distribution $\in \GG$}};
    \node (p-tilde) at (-4, -2.2) {$\hat p^*_n$};
    \node[above=0.2em of p-tilde] (p-tilde-label) {{\footnotesize empirical distribution}};
    \draw (p-star) edge[->] node[right]{{\footnotesize sampling}} (p-tilde-label);
    \node (samples) at (0, -2.2) {$\hat p_n$};
    \node[above=0.2em of samples] (samples-label) {{\footnotesize corrupted distribution}};
    \draw (p-tilde) edge[->]  node[above]{{\footnotesize $\overline{D}(\hat p^*_n, \hat p_n)\leq \epsilon_1 \approx \epsilon$}} (samples);
    \node (perturbed) at (-8, -2.2) {$\hat p'$};
    \node[above=0.2em of perturbed] (perturbed-label) {{\footnotesize  Perturbed $\in \GG'\subset \MM$}};
    \draw (p-tilde) edge[->] node[above]{{\footnotesize $\overline{D}(\hat p_n^*, \hat p')\leq \epsilon_2 \approx \epsilon$}} (perturbed);
    \node (q) at (4, -2.2) {$q$};
    \node[above=0.2em of q] (q-label) {{\footnotesize projected $q\in \mathcal{M}$}};
    \draw (samples) edge[->] node[above]{\footnotesize projection} (q);
    \node (theta-hat) at (4, -4.4) {$\theta^*(q)$};
    \node[above=0.2em of theta-hat] (theta-hat-label) {{\footnotesize estimated parameters}};
    \draw (q)  edge[->] (theta-hat-label);
    \node (error) at (-2, -4.4) {$L(\hat p', \theta^*(q))\leq \cdots$};
    \draw (theta-hat) edge[->]    (error);
    \draw (perturbed) edge[->] (error);
     \draw (q) edge[->] node[above]{\footnotesize  } (error);
     \node (result) at (-2, -6.6) {$L(p^*, \theta^*(q))\leq \cdots$};
     \draw (error) edge[->] node[left]{\footnotesize generalization} (result);
     \node[above=1.0em of error] (result-label1) {{\footnotesize modulus of continuity}};  
     \node[above=2.2em of error] (result-label2) {{\footnotesize $\tD(q, \hat p') - \tD(\hat p', \hat p')  \leq 2(\epsilon_1 + \epsilon_2) \approx 4\epsilon$}};  
     \end{tikzpicture}
    \caption{Framework for analysis in Theorem~\ref{thm.admissible-2}. }
    \label{fig:coupling}
\end{figure}

\section{Connections with robust optimization and agnostic distribution learning}
\label{sec.connection_dro}

\subsection{Distributionally robust optimization (DRO)}
We provided two approaches to analyze the finite sample projection algorithm $q = \Pi(\hat p_n; \tD, \cM)$: ``Find $p^*$'' in Theorem~\ref{thm.admissible-1}, Theorem~\ref{thm.admissible-1-adaptive} and ``Find $\hat p_n^*$'' in Theorem~\ref{thm.admissible-2}. In this section, we build the connections between our projection algorithms and distributionally robust optimization (DRO). 

We first show that under appropriate conditions (Theorem~\ref{thm.droadmissible1}), the projection algorithm is approximately solving the following distributionally robust optimization problem, and any approximate solution of the DRO below produces a good robust estimate of $\theta$. 

\begin{align}\label{eqn.DRO1}
    \hat \theta = \argmin_{\hat \theta(\hat p_n)} \sup_{r: r\in\GG, 
    \tD(r, \hat{p}_n)\leq \epsilon + \tD(p, \hat p_n)} L(r, \hat \theta(\hat p_n))
\end{align}

Note that here we use $r$ to denote the dummy variable in the DRO. 

\begin{theorem}\label{thm.droadmissible1}
Assume the oblivious contamination model of level $\epsilon$ under $D$. Denote the true distribution as $p^* \in \GG$ and the perturbed population distribution as $p$ with $D(p,p^*)\leq \epsilon$. Denote the cost function as $L(p^*,\theta)$ and the empirical distribution of observed data as $\hat p_n$. Assume the following conditions. 
\begin{enumerate}
    \item \textbf{Robust to perturbation:} $\tD(p,q)$ is a pseudometric that satisfies
    \begin{align}
        \tD(p,q) \leq D(p,q). 
    \end{align}
    \item \textbf{Generalized Modulus of Continuity:}  there exists a set $\mathcal{M} \supset \GG$ and for $\tilde{\epsilon} = 2\epsilon + 2 \tD(p, \hat p_n)$, we have
    \begin{align}
    \sup_{p_1^*\in \cM, p_2^* \in \GG, \tD(p_1^*, p_2^*)  \leq \tilde{\epsilon}} L(p_2^*, \theta^*(p_1^*)) \leq \rho(\tilde{\epsilon}).
    \end{align}
\end{enumerate}

Then, 
\begin{enumerate}   
    \item \emph{Any approximate solution of DRO suffices:} for any $\theta, \rho$, if
    \begin{align}
   \sup_{r: r\in\GG, 
    \tD(r, \hat{p}_n) \leq \epsilon + \tD(p, \hat p_n)} L(r, \theta) \leq \rho,
    \end{align}
    then
    \begin{align}
        L(p^*, \theta) \leq  \rho. 
    \end{align}
    \item \emph{Projection approximately solves DRO:} for $q=\Pi(\hat p_n;\tD, \cM, \epsilon + \tD(p, \hat p_n))$, $\theta^*(q) = \argmin_{\theta\in\Theta} L(q, \theta)$, there is 
\begin{align}
\sup_{r: r\in\GG, 
    \tD(r, \hat{p}_n) \leq \epsilon + \tD(p, \hat p_n)} L(r, \theta^*(q))  \leq  \rho(2\epsilon + 2\tD(p, \hat p_n)).
\end{align}
\end{enumerate}
\end{theorem}

\begin{remark}
In practice, we need to know $\tilde D(p, \hat p_n)$ for the proposed DRO while we do not have the access to corrupted population distribution $p$. Thus we need to upper bound $\tilde D(p, \hat p_n)$ such that with probability at least $1-\delta$, $\bar D(p, \hat p_n) \leq d_n(\delta)$ for some $d_n(\delta)$. Then we can search over the larger $(\epsilon + d_n(\delta))$-ball and give the same guarantee.
\end{remark}

\begin{proof}
Regarding the first claim, it suffices to show that $\tD(p^*, \hat p_n) \leq \epsilon + \tD(p, \hat{p}_n)$. It is true since $\tD$ is a pseudometric and $\tD \leq D$:
\begin{align}
    \tD(p^*, \hat p_n) & \leq \tD(p^*, p) + \tD(p, \hat{p}_n) \\
    & \leq D(p^*, p) + \tD(p, \hat{p}_n) \\
    & \leq \epsilon + \tD(p, \hat{p}_n).
\end{align}

Now we verify the second claim. In order to apply the Generalized Modulus of Continuity property, it suffices to show that for any $r$ such that $r\in \GG, \tD(r, \hat p_n)  \leq \epsilon + \tD(p, \hat{p}_n)$ and $q \in \mathcal{M} \supset \GG$ being the output of the projection algorithm, we have
\begin{align}
    \tD(q,r) \leq 2\epsilon + 2 \tD(p, \hat{p}_n). 
\end{align}

Note that $q$ is the output of the projection algorithm implies either $\tD(q, \hat{p}_n) \leq \epsilon + \tD(p, \hat{p}_n)$ or $\tD(q, \hat{p}_n) \leq  \tD(r, \hat{p}_n) \leq \epsilon + \tD(p,\hat{p}_n)$. Hence, 
\begin{align}
    \tD(q,r) & \leq \tD(q, \hat{p}_n) + \tD(r, \hat{p}_n) \\
    & \leq 2 (\epsilon + \tD(p,\hat{p}_n)). 
\end{align}

\end{proof}

The above theorem shows that solving DRO in Equation (\ref{eqn.DRO1}) will also provide robustness guarantee. We remark that for $\GG$ as generalized resilience set, we can solve the following DRO:
\begin{align}
    \hat \theta = \argmin_{\hat \theta(\hat p_n)} \sup_{p^*: p^*\in\GG_{\uparrow}, 
    \tD(p^*, \hat{p}_n)  \leq \epsilon + \tD(p, \hat p_n)} L( p^*, \hat \theta(\hat p_n))
\end{align}

It is interesting to note that when solving DRO, we only need to consider the $p^*$ inside $\GG_{\uparrow} $. Meanwhile for projection algorithm, it suffices to project $\hat p_n$ onto $\GG_{\downarrow}$.

We now show that similar interpretations can also be made for the DRO formulation below. 
\begin{align}
    \hat{\theta} = \argmin_{\hat{\theta}(\hat{p}_n)} \sup_{r: r\in \GG', \tD(r, \hat{p}_n)  \leq \epsilon_1 + \epsilon_2} L(r, \hat{\theta}(\hat{p}_n))
\end{align}

\begin{theorem}\label{eqn.droadmissible2}
Assume either oblivious contamination or adaptive contamination model of level $\epsilon$ under $D$. Denote the true distribution as $p^* \in \GG$, $\hat{p}_n^*$ as the empirical distribution sampled from $p^*$ and $\hat p_n$ as the empirical distribution of observed data. Denote the cost function as $L(p^*,\theta)$. Assume the following conditions. 
\begin{enumerate}
     \item \textbf{Project via pseudometric:} $\tD(p,q)$ is a pseudometric. 
    \item \textbf{Limited contamination:} $\tD(\hat p_n, \hat p_n^*) \leq \epsilon_2$ with probability at least $1-\delta$.
     \item \textbf{Set for (perturbed) empirical distribution:} there exists a set $\GG' \subset \mathcal{M}$ such that there exists a distribution $\hat p'\in \GG'$ satisfying $\tD(
    \hat p_n^*, \hat p')\leq \epsilon_1$ with probability at least $1-\delta$.
    \item \textbf{Generalized Modulus of Continuity:}  $\GG' \subset \mathcal{M}$, and for $\tilde{\epsilon} = 2(\epsilon_1 + \epsilon_2)$, there is
    \begin{align}
    \sup_{p_1^*\in \mathcal{M}, p_2^* \in \GG', \tD(p_1^*, p_2^*) \leq \tilde{\epsilon}} L(p_2^*, \theta^*(p_1^*)) \leq \rho(\tilde{\epsilon}),
    \end{align}
    \item \textbf{Generalization bound:} for any $p^*\in \GG, \theta \in \Theta$, there exists some constant $C$ and some function $g$ such that $L(p^*, \theta) \leq C\cdot L(\hat p', \theta)+   g(\hat p', p^*) $. 
\end{enumerate}
Then, 
\begin{enumerate}
    \item \emph{Any approximate solution of DRO suffices:} for any $\theta, \rho$, if 
    \begin{align}
       \sup_{r: r\in \GG', \tD(r, \hat{p}_n) \leq \epsilon_1 + \epsilon_2} L(r, \theta) \leq \rho,
    \end{align}
    then, with probability at least $1-2\delta$, 
    \begin{align}
        L(p^*, \theta) \leq C\rho + g(\hat{p}',p^*). 
    \end{align}
    \item \emph{Projection approximately solves DRO:} for $q = \Pi(\hat{p}_n;\tD,\mathcal{M}, \epsilon_1 + \epsilon_2), \theta^*(q) = \argmin_{\theta \in \Theta}L(q,\theta)$, 
    \begin{align}
         \sup_{r: r\in \GG', \tD(r, \hat{p}_n)  \leq \epsilon_1 + \epsilon_2} L(r, \theta^*(q)) & \leq  \rho(2\epsilon_1 + 2\epsilon_2). 
    \end{align}
\end{enumerate}
\end{theorem}

\begin{proof}
We know that with probability at least $1-2\delta$, there exist $\hat p' \in \GG', \tD(\hat p', \hat p_n^*)\leq \epsilon_1, \tD(\hat p_n^*, \hat p_n)\leq \epsilon_2$. Hence
\begin{align}
    \tD(\hat{p}', \hat{p}_n) & \leq \tD(\hat{p}', \hat p_n^*) + \tD(\hat{p}_n^*, \hat{p}_n)\\
    & \leq \epsilon_1 + \epsilon_2. 
\end{align}

Hence, we know $L(\hat{p}',\theta)\leq \rho$, and it follows from the generalization bound that 
\begin{align}
    L(p^*,\theta) \leq C \rho + g(\hat{p}', p^*). 
\end{align}

Now we prove the second claim using the Generalized Modulus of Continuity property. It suffices to show that for any $r\in \GG', \tD(r,\hat{p}_n) \leq \epsilon_1 + \epsilon_2$, $q\in \mathcal{M} \supset \GG'$ being the output of the projection algorithm, we have
\begin{align}
    \tD(q,r) \leq 2\epsilon_1 + 2\epsilon_2. 
\end{align}

Note that $q$ is the output of the projection algorithm implies either $\tD(q, \hat{p}_n) \leq \epsilon_1 + \epsilon_2$ or $\tD(q, \hat{p}_n) \leq  \tD(r, \hat{p}_n) \leq \epsilon_1 + \epsilon_2$. Hence,
\begin{align}
    \tD(q,r) & \leq \tD(q, \hat{p}_n) + \tD(\hat{p}_n,r) \\
    & \leq 2 (\epsilon_1 + \epsilon_2). 
\end{align}

\end{proof}

\subsubsection*{Connection to distributionally robust optimization  }

Distributionally robust optimization (DRO) solves the min-max problem
\begin{equation}
\label{eq:dro-intro} \inf_{\theta} \sup_{q: \TV(q, p) \leq \epsilon} L(q; \theta).
\end{equation}
This is similar to our setting but omits the assumptions $\GG$ (or equivalently, takes $\GG$ to be all probability distributions). 
Consequently, \eqref{eq:dro-intro} is in many cases not defined. For instance, when  $L$ is any unbounded loss, 
the supremum is infinite for all $\theta$. As a consequence, DRO typically considers bounded loss functions~\citep{duchi2018learning}, replaces $\TV$ with some $f$ divergences that only allow a small family of perturbations 
\citep{delage2010distributionally,namkoong2016stochastic}, or takes $L$ to be Lipschitz and replace the discrepancy measure with Wasserstein distance 
\citep{volpi2018generalizing}.

More conceptually, the optimal $q$ in \eqref{eq:dro-intro} will typically push outlying points to be 
even more outlying, and thus \emph{magnifies} the influence of outliers, which is counter to our goal of resisting the effects of 
corruptions. On the other hand, if we restrict $q$ in \eqref{eq:dro-intro} to lie in $\GG$, we do resist outliers and in fact 
this modified DRO is the min-max optimal estimator in infinite samples under our framework, underscoring the importance of the 
assumptions $\GG$. We analyze a projection estimator below rather than DRO as it more readily admits bounds especially in finite samples, 
but study a DRO-based estimator in Appendix~\ref{sec.connection_dro}.

\subsection{Connection with agnostic distribution learning}\label{sec.agnostic_connection}
Agnostic distribution learning ~\citep{yatracos1985rates,devroye2012combinatorial, chan2014efficient,acharya2017sample, zhu2019deconstructing} concerns finding the distribution $p^*\in \GG$ that is \emph{closest} to distribution $p$ when only empirical samples from $p$ are observed. The ``Yatracos'' method in agnostic learning constructs a distance $\tTV \leq \TV$ such that $\tTV(p,q) \approx \TV(p,q)$ for all $p,q \in \GG$ and then projects the empirical distribution $\hat p_n$ to $\GG$ under $\tTV$.  A key difference from our work is that lost function is also $\TV$ in agnostic learning.
For large sets $\GG$ such as the resilient set, there does not exist some $\tTV$ weaker than $\TV$ such that $\tTV(p,q) = \TV(p,q)$ for all $p,q\in \GG$. However, one can in fact apply the Yatracos method on top of our $\tTV_{\mathcal{H}}$, by 
taking the loss in agnostic learning to be $\tTV_\sH$. Furthermore, our design of $\tTV_\sH$ depends on the \emph{representation} of the set $\GG$, but the Yatracos method is independent of the set representation.

\section{General Lemmas and Facts}\label{sec.appendixgeneralemmas}
\subsection{Notations}
We first collect the notations used throughout this paper. 
We use capital letter $X$ for random variable, lowercase letter $p, q$ for population distribution, and the corresponding empirical distributions with $n$ samples are denoted as ${\hat p_n}, {\hat q_n}$. Blackbold letter $\bbP$ is used for probability, e.g. $\bbP_q(A)$ represents the probability of event $A$ under distribution $q$, and blackbold letter $\bbE$ is used for expectation. 
We denote $\mu_p = \bE_p[X]$ and $\Sigma_p = \bE_p[(X-\mu_p)(X-\mu_p)^\top]$ as mean and covariance for distribution $p$.
We use $\mathsf{TV}(p,q) = \sup_{A} \bbP_p(A) - \bbP_q(A)$ to denote the total variation distance between $p$ and $q$. We use $\triangleq$ to make definition. We denote $\min\{a,b\}$ by $a \wedge b$, and $\max\{a,b\}$ by $a \vee b$. For non-negative sequences $a_\gamma$, $b_\gamma$, we use the notation $a_\gamma \lesssim_{\alpha} b_\gamma$ to denote that there exists a constant $C$
that only depends on $\alpha$ such that $\sup_\gamma \frac{a_\gamma}{b_\gamma}\leq C$, and $a_\gamma \gtrsim_{\alpha} b_\gamma$  is equivalent to $b_\gamma \lesssim_{\alpha} a_\gamma$.  When the constant $C$ is universal
we do not write subscripts for $\lesssim$ and $\gtrsim$. Notation $a_\gamma\asymp b_\gamma$ is equivalent to $a_\gamma \gtrsim b_\gamma$ and $a_\gamma \lesssim b_\gamma$. Notation $a_\gamma \gg b_\gamma$ means that $\liminf_{\gamma}\frac{a_\gamma}{b_\gamma}=\infty$, and $a_\gamma \ll b_\gamma$ is equivalent to $b_\gamma \gg a_\gamma$. We write $f(x) = O(g(x))$ for $x \rightarrow 0$ if there exists some non-negative constant $C_1, C_2$ such that $|f(x)| \leq C_1 |g(x)| $ for all $x \in [0, C_2]$.  We write $f(x) = \Omega(g(x))$ for $x \rightarrow 0$ if there exists some non-negative constant $C_1, C_2$ such that $|f(x)| \geq C_1 |g(x)| $ for all $x \in [0, C_2]$. We write $f(x) = \Theta(g(x))$ if $f(x) = O(g(x))$ and $f(x)= \Omega(g(x))$. For any norm $\|\cdot\|$, we use $\| \cdot\|_*$ to denote its dual norm, which is defined as $\| z\|_* = \sup\{z^\top x \mid \|x\|\leq 1\}$.

For a pseudometric $c$, the Wasserstein distance is the minimum-cost matching between $p$ and $q$ 
according to $c$:

\begin{definition}[{$W_{c,k}$ distance~\citep[Theorem 7.3]{villani2003topics}}]\label{def.W_ck}
Suppose $c(x,y)$ is a pseudometric. The Wasserstein-$k$ transportation distance for $k\geq 0$ is defined as
\begin{align}\label{eqn.def_wc}
    W_{c, k}(p, q) = \begin{cases} 
    \inf_{\pi_{p, q} \in \Pi(p, q) } \left (\int c^k(x, y) d \pi_{p, q} (x, y)\right )^{1/k} & k\geq 1 \\  \inf_{\pi_{p, q} \in \Pi(p, q) } \left (\int c^k(x, y) d \pi_{p, q} (x, y)\right ) & k \in [0,1) \end{cases},
\end{align}
where $\Pi(p, q)$ denotes the set of all couplings between $p$ and $q$.
\end{definition}
If $c$ is a pseudometric, then so is $W_{c, k}$~\citep[Page 209]{villani2003topics}. If $c(x,y)$ is the Euclidean distance $\|x-y\|_2$, we usually omit the subscript $c$ and write $W_k$. 

\subsection{Lemmas and facts}

In this section we present some general lemmas and facts that we use throughout the paper. 

\begin{lemma}[Non-decreasing property of function 
$x\psi^{-1}(\sigma/x)$]\label{lem.psi_increasing}
For any Orlicz function $\psi$, the function $x\psi^{-1}(\sigma/x)$ is non-decreasing for $x$ for the region $x\in[0, +\infty)$ for any constant $\sigma>0$, where $\psi^{-1}$ is  the (generalized) inverse function of $\psi$. 
\end{lemma}
\begin{proof}
Denote $t(x) = \psi^{-1}(\sigma/x) = \inf\{y \mid \psi(y) > \sigma/x\} $. Since $\psi(x)$ is non-decreasing, we know that $\psi^{-1}(\sigma/x)$ is a non-increasing function. Consider the function $\frac{\psi(t)}{t}$. From the property of convex functions, we know that for any $0<x_1<x_2$, 
\begin{align}
    \frac{\psi(x_1) - \psi(0)}{x_1} \leq     \frac{\psi(x_2) - \psi(0)}{x_2}. 
\end{align}
Thus we know $\frac{\psi(t)}{t}$ is an non-decreasing function. Thus the function $f(t) = \frac{t}{\psi(t)}$ is an non-increasing function. Since the function $\frac{x}{\sigma}\psi^{-1}(\sigma/x)$ is composition $(f \circ t)(x)$, we know that it is non-decreasing.
\end{proof}
\begin{lemma}[Generalized Holder's Inequality]\label{lem.multi_orlicz}
Define composition function as $(\psi \circ f) (x) = \psi(f(x))$. Given some Orlicz function $\psi$,  for any two random variables $X,Y$, any $p, q>0, 1/p+1/q = 1$,
\begin{align}
     \| XY \|_{\psi}\leq \|X \|_{\psi\circ x^p} \|Y\|_{\psi\circ x^q}. 
\end{align}
\end{lemma}
\begin{proof}
Denote $\|X \|_{\psi\circ x^p} = \sigma_1$, $\|Y\|_{\psi\circ x^q} = \sigma_2$. It follows from Young's inequality~\cite{young1912classes} that
\begin{align}
    |XY| \leq \frac{1}{p}|X|^p + \frac{1}{q}|Y|^q, \forall p, q >0, \frac{1}{p} + \frac{1}{q} = 1.
\end{align}
Thus 
\begin{align}
\bE \left[ \psi \left (\frac{|XY|}{\sigma_1\sigma_2}\right ) \right] &\leq\bE \left[  \psi \left (\frac{1}{p}\cdot \left |\frac{X}{\sigma_1}\right |^p+ \frac{1}{q}\cdot \left |\frac{Y}{\sigma_2}\right |^p \right )  \right] \nonumber \\
& \leq \bE \left[\frac{1}{p} \psi\left (\left |\frac{X}{\sigma_1}\right |^p\right ) \right] + \bE \left[\frac{1}{q} \psi\left (\left |\frac{X}{\sigma_2}\right |^q\right ) \right]\nonumber \\
& \leq \frac{1}{p}+ \frac{1}{q}  = 1.
 \end{align}
This shows that $ \| XY \|_{\psi}\leq \sigma_1 \sigma_2 $.
\end{proof}

The following lemma is a generalization of~\citep[Lemma 2.6.8]{vershynin2018high} and shows that if a distribution has its non-centered Orlicz norm bounded, then its centered Orlicz norm is also bounded.
\begin{lemma}[Centering]\label{lem.centering_psi}
For any Orlicz function $\psi$, then if  
\begin{align}
    \sup_{f\in\sF}\bE_p\left[\psi\left(\frac{|f(X) |}{\sigma}\right)\right]\leq 1,
\end{align}
then 
\begin{align}
     \sup_{f\in\sF}\bE_p\left[\psi\left(\frac{|f(X)-\bE_p[f(X)] |}{2\sigma}\right)\right]\leq 1.
\end{align}
\end{lemma}
\begin{proof}

For some fixed $f$, note that $\|\cdot\|_\psi$ satisfies the triangle inequality. Thus
\begin{align}
    \| f(X) - \bE_p[f(X)]\|_\psi \leq \| f(X) \|_\psi + \| \bE_p[f(X)]\|_\psi \leq \sigma + \| \bE_p[f(X)]\|_\psi \leq \sigma + |\bE_p[f(X)]|/\psi^{-1}(1),
\end{align}
where $\psi^{-1}$ is  the (generalized) inverse function of $\psi$. 
Now we show that $|\bE_p[f(X)]|/\psi^{-1}(1) \leq \| f(X)\|_\psi = \sigma$.
By Jensen's inequality, we have
\begin{align}
    \bE_p\left[\psi\left(\frac{\psi^{-1}(1)|f(X) |}{\bE_p[f(X)]}\right)\right] \geq \psi\left(\frac{\psi^{-1}(1)|\bE_p\left[f(X) \right]|}{\bE_p[f(X)]}\right) = \psi(\psi^{-1}(1)) = 1. 
\end{align}
This shows that $|\bE_p[f(X)]|/\psi^{-1}(1) \leq \| f(X)\|_\psi = \sigma$. Thus $ \| f(X) - \bE_p[f(X)]\|_\psi \leq 2\sigma$.
\end{proof}

\begin{lemma}[Convergence of mean under 2-norm for distribution with bounded second moment]\label{lem.kth_convergence}
Assume distribution $p$ has its second moment bounded, i.e.
\begin{align}
    \sup_{v\in\bR^d, \|v\|_2=1}\bE_{p}\left[|v^\top (X-\bE_p[X])|^2\right]& \leq \sigma^2.
\end{align}
Then
\begin{align}
    \bE_{p}  \left \| \frac{1}{n} \sum_{i = 1}^n  X_i - \mathbb{E}_{p}[X] \right \|_2 & \leq \sigma \sqrt{\frac{d}{n}}.
\end{align}
\end{lemma}
\begin{proof}
By Jensen's inequality,
\begin{align}
    \bE_{p}  \left \| \frac{1}{n} \sum_{i = 1}^n  X_i - \mathbb{E}_{p}[X] \right \|_2
    & \leq  \left(\bE_{p}  \left \| \frac{1}{n} \sum_{i = 1}^n  X_i - \mathbb{E}_{p}[X] \right \|_2^2\right)^{1/2} \nonumber \\
    & =  \left( \bE_{p} \mathsf{Tr} \left(\left(\frac{1}{n} \sum_{i = 1}^n  X_i - \mathbb{E}_{p}[X]\right)\left(\frac{1}{n} \sum_{i = 1}^n  X_i - \mathbb{E}_{p}[X]\right)^\top \right) \right)^{1/2} \nonumber \\
    & =  \mathsf{Tr} \left( \bE_{p}  \left[\left(\frac{1}{n} \sum_{i = 1}^n  X_i - \mathbb{E}_{p}[X]\right)\left(\frac{1}{n} \sum_{i = 1}^n  X_i - \mathbb{E}_{p}[X]\right)^\top \right] \right)^{1/2} \nonumber \\
    & \leq \sqrt{d} \left\|\bE_{p}  \left[\left(\frac{1}{n} \sum_{i = 1}^n  X_i - \mathbb{E}_{p}[X]\right)\left(\frac{1}{n} \sum_{i = 1}^n  X_i - \mathbb{E}_{p}[X]\right)^\top \right] \right\|_2^{1/2} \nonumber \\
    & = \sqrt{\frac{d}{n}} \|\bE_p[(X-\bE_p[X])(X-\bE_p[X])^\top] \|_2^{1/2} \nonumber \\
    & \leq \sigma \sqrt{\frac{d}{n}}.
\end{align}

\end{proof}

\section{Related discussions and remaining proofs in Section~\ref{sec.preliminaries}}

\subsection{Characterization of adaptive corruptions}\label{proof.lem_coupling}

The following lemma is useful for controlling the behavior of adaptive corruptions:
\begin{lemma}\label{lem.coupling}
Suppose $\hat p_n \mid \hat{p}_n^*$ is an allowed adaptive corruption with level $\epsilon$ under 
$D = W_{c,k},k\geq 1$ (Definition~\ref{def.W_ck}). Then
the perturbed and true empirical distribution are close in expectation:
    \begin{align}\label{eqn.cou2}
          \bE[W_{c, k}^k(\hat p_n, \hat p^*_n)]& \leq \epsilon^k.
    \end{align}
If additionally $0\leq c(x, y)\leq C$ for all $x, y$, then with probability at least $1-\delta$, 
    \begin{align}\label{eqn.cou3}
    W_{c, k}(\hat p_n, \hat p^*_n)\lesssim \left ( \epsilon^k + \frac{C^k\log(1/\delta)}{n} \right )^{\frac{1}{k}}.
\end{align}
\end{lemma}

\begin{proof}
Construct a coupling between $\hat p_n^* = \frac{1}{n} \sum_{i = 1}^n \delta_{X_i}$ and $\hat p_n = \frac{1}{n}\sum_{i = 1}^n \delta_{Y_i}$ as
\begin{align}
    \lambda_{\hat p_n^*, \hat p_n} = \frac{1}{n} \sum_{i = 1}^n \delta_{(X_i,Y_i)}. 
\end{align}

Then, it follows from the definition of $W_{c,k}(\hat p_n^*, \hat p_n)$ that
\begin{align}
    \left (W_{c,k}(\hat p_n^*, \hat p_n)\right )^k \leq \frac{1}{n} \sum_{i = 1}^n c^k(X_i,Y_i),
\end{align}
which proves the first two claims. Regarding the third part of the lemma, noting that $ |c^k(X,Y)| \leq C^k$ and $\mathbb{E}_{\pi}[c^{2k}(X,Y)] \leq C^k \epsilon^k$ and applying Bernstein's inequality, we have
\begin{align}
    \bP\left( \frac{1}{n}\sum_{i = 1}^n c^k(X_i,Y_i) - \bE_{\pi}[c^k(X,Y)] \geq t^k \right ) \leq \exp\left (  - \frac{n^2 t^{2k}}{2 \left( n C^k \epsilon^k + \frac{C^k n t^k}{3}  \right)} \right). 
\end{align}

It implies that with probability at least $1-\delta$, 
\begin{align}
    \frac{1}{n}\sum_{i = 1}^n c^k(X_i,Y_i) & \leq \epsilon^k + \max\left\{ \frac{4C^k \ln\left ( \frac{1}{\delta}\right)}{3n}, \sqrt{\frac{4C^k \epsilon^k \ln\left ( \frac{1}{\delta}\right) }{n}} \right \} \\
    & \lesssim \epsilon^k + \frac{C^k \ln\left ( \frac{1}{\delta}\right)}{n}. 
\end{align}
\end{proof}
Furthermore, when $W_{c, k} = \TV$, i.e. $c(x,y) = \mathbbm{1}(x\neq y), k = 1$, we can provide a stronger bound. 

\begin{lemma}\label{lem.coupling_TV}
For any $\hat p_n$ generated by adaptive corruption model with level $\epsilon$ under $D = \TV$, with probability at least $1-\delta$, we have
\begin{align}
    \TV(\hat p_n, \hat p^*_n)\leq \begin{cases} (\sqrt{\epsilon} + \sqrt{\frac{\log(1/\delta)}{2n}})^2 & \text{ for all }\epsilon \in [0,1],\delta\in (0,1],n\geq 1\\ 0 & \delta \geq 1-(1-\epsilon)^n \end{cases},
\end{align}
where $\hat{p}_n^*$ is the empirical distribution of $n$ \iid samples from the true distribution $p^*$. 
\end{lemma}
\begin{proof}
Note that $\frac{1}{n}\sum_{i = 1}^n \mathbbm{1}(X_i\neq Y_i)$ is being stochastically dominated by a binomial distribution $\mathsf{Bin}(n, \epsilon)$. 
The first result is a direct result of tail bound for binomial distribution~\cite{okamoto1959some}. To see the later result, note that the binomial distribution equals $0$ with probability $(1-\epsilon)^n$. Thus when $\delta \geq 1-(1-\epsilon)^n$, we have with probability at least $1-\delta$, $\TV(\hat p_n^*, \hat p_n) = 0. $
\end{proof}

\subsection{Discussions on the  population limit in Definition~\ref{def.populationlimit}}

Given distribution family $\mathcal{G}$, discrepancy $D$, loss $L$, and perturbation level $\epsilon$, we define
the population limit for the robust inference problem as
\begin{align}\label{eqn.appendix_population_limit}
  \gamma_1 =  \inf_{\theta(p)} \sup_{(p^*,p): D(p^*,p)\leq \epsilon, p^*\in \mathcal{G}} L(p^*, \theta(p)).
\end{align}

This definition is only considering the  limit for deterministic decision rule $\theta(p)$ given infinite number of samples. However, we claim that it is consistent with the minimax risk for randomized decision rule under certain level. The minimax risk in statistical literature~\cite{chen2018robust} is defined as 
\begin{align}\label{eqn.minimax_risk}
    \gamma_2(\delta) = \inf \{\gamma \mid \inf_{\theta(p)} \sup_{(p^*,p): D(p^*,p)\leq \epsilon, p^*\in \mathcal{G}} \mathbb{P}(L(p^*, \theta(p))> \gamma)\leq \delta\}.
\end{align}

Note that the upper bound for $\gamma_1$ is naturally an upper bound for $\gamma_2(0)$. To relate the lower bound of $\gamma_1$ and $\gamma_2(\frac{1}{2})$, we introduce the following lemma that shows any lower bound derived for $\gamma_1$ from Le Cam's two point method would also give a lower bound for $\gamma_2(\frac{1}{2})$. With this lemma, we can adapt all the lower bounds on $\gamma_1$ derived in this paper to the lower bound on $\gamma_2(\frac{1}{2})$.

\begin{lemma}\label{lem.minimax_random}
Suppose there exist $p_1, p_2$ such that
\begin{align}
    \inf_{\theta} L(p_1, \theta) + L(p_2, \theta) \geq 2\alpha.
\end{align}
Then for any randomized decision rule $\theta_r(p)$,
\begin{align}
    \bP(L(p_1, \theta_r(p))\geq {\alpha}) + \bP(L(p_2, \theta_r(p))\geq {\alpha}) \geq 1.
\end{align}
\end{lemma}
    
\begin{proof}
Consider the indicator function, we have for any fixed $\theta$, 
\begin{align}
    \mathbb{1}(L(p_1, \theta) \geq {\alpha}) +  \mathbb{1}(L(p_2, \theta) \geq {\alpha}) & \geq  \mathbb{1}(L(p_1, \theta) + L(p_2, \theta)\geq {2\alpha}) \nonumber \\ 
    & \geq \mathbb{1}(\inf_\theta L(p_1, \theta) + L(p_2, \theta)\geq {2\alpha}) \nonumber \\ 
    & = 1.
\end{align}
Now let $\theta_r(p)$ be randomized decision rule that is a function of $p$ and (possibly) outside randomness. Taking expectation with respect to $\theta_r(p)$ yields
\begin{align}
    \mathbb{P}(L(p_1, \theta_r(p)) \geq {\alpha}) +  \mathbb{P}(L(p_2, \theta_r(p)) \geq {\alpha}) \geq 1.
\end{align}

\end{proof}

\subsection{Proof of Lemma~\ref{lemma.population_limit} and related discussions}\label{proof.fundamental_limit}

We show a more general conclusion that for any distance $D$ with $D(p^*, p)\leq \epsilon$. As long as $D$ is a pseudometric, the conclusion on modulus holds. 
It follows from the assumption $D(p,p^*) \leq \epsilon$, and $p^* \in \mathcal{G}$ that the projection algorithm can find some $q \in \GG$ such that 
\begin{align}
D(p, q) \leq \epsilon. 
\end{align}
It follows from the triangle inequality of $D$ that
\begin{align}
D(p^*, q) \leq D(p^*, p) + D(p, q) \leq 2 \epsilon. 
\end{align}
Hence, 
\begin{align}
L(p^*, \theta^*(q)) \leq \sup_{p_1 \in \mathcal{G}, p_2 \in \mathcal{G}, D(p_1, p_2) \leq 2 \epsilon } L(p_1, \theta^*(p_2)). 
\end{align}

The following lemma provides lower bounds on the population limit. Note that the $\TV$ distance satisfies all the conditions required for $D$.

\begin{lemma}\cite{donoho1988automatic,chen2018robust}
\label{lemma.population_limit_opt}
Suppose that $D(p,q)$ is a pseudometric and $L(p,\theta^*(q))$ is a pseudometric over $(p,q)$. Then,
\begin{enumerate}
    \item \emph{projection algorithm is near-optimal}: for $q = \argmin_{q\in \GG} D(q,p)$ and the observed corrupted distribution $p$, 
    \begin{align}
   \sup_{p^* \in\GG, D(p^*, p)\leq \epsilon} L(p^*, \theta^*(q))  \leq    2\inf_{\theta(p)} \sup_{p^*: D(p^*,p)\leq \epsilon, p^*\in \mathcal{G}} L(p^*, \theta(p)) 
    \end{align}
    \item if the space induced by $D$ is a complete convex metric space~\citep[Theorem 1.97]{penot2012calculus}, then the population limit in (\ref{eqn.fundamental_limit}) is lower bounded by the modulus of continuity in (\ref{eqn.modulus}) up to a factor of $2$: 
\begin{align}
\label{eq.population_limit_opt}
\inf_{\theta(p)} \sup_{(p, p^*): p^* \in \GG, D(p^*, p)\leq \epsilon} L(p^*, \theta(p)) \geq \frac{1}{2}\modu(\GG, 2\epsilon, D, L).
\end{align}

It was discussed in detail in~\cite{donoho1988automatic} when the factor of $1/2$ is tight.  

\item For any (possibly random) decision rule $\theta_r(p)$, 
\begin{align}
    \sup_{(p^*, p): p^* \in \mathcal{G}, D(p^*, p)\leq \epsilon } \bP(L(p^*, \theta_r(p))\geq  \frac{1}{2}\modu(\GG, 2\epsilon, D, L))
     \geq  \frac{1}{2}.
\end{align}
\end{enumerate}
\end{lemma}
\begin{proof}
Now we show the near-optimality of projection algorithm when $L(p,\theta^*(q))$ is a pseudometric for $(p,q)$. Indeed, for $q = \argmin_{q\in \GG} D(q,p)$ and any estimator $\theta(p)$, we have 
\begin{align}
    \sup_{p^* \in\GG, D(p^*, p)\leq \epsilon} L(p^*, \theta^*(q)) & \leq \sup_{p_1, p_2 \in\GG, D(p, p_1)\leq \epsilon, D(p, p_2)\leq \epsilon} L(p_1, \theta^*(p_2)) \\
    & \leq \sup_{p_1, p_2 \in\GG, D(p, p_1)\leq \epsilon, D(p, p_2)\leq \epsilon} (L(p_1, \theta(p)) + L(p_2,\theta(p))) \\
    & \leq 2 \sup_{p^*: D(p^*,p)\leq \epsilon, p^*\in \mathcal{G}} L(p^*, \theta(p)).
\end{align}
Since the derivations above holds for any $\theta(p)$, we know
\begin{align}
   \sup_{p^* \in\GG, D(p^*, p)\leq \epsilon} L(p^*, \theta^*(q))  \leq 2   \inf_{\theta(p)} \sup_{p^*: D(p^*,p)\leq \epsilon, p^*\in \mathcal{G}} L(p^*, \theta(p)).
\end{align}

Last, we prove the near-optimality of modulus. For any $p_1 \in \mathcal{G}, p_2 \in \mathcal{G}, D(p_1, p_2) \leq 2\epsilon$, since we assumed that the space induced by $D$ is a complete convex metric space~\citep[Theorem 1.97]{penot2012calculus}, we can find some $r$ such that 
\begin{align}
D(p_1, r) & \leq \epsilon \\
D(p_2, r) & \leq \epsilon,
\end{align}
Hence, we get exactly the same observation $r$ for two different true distributions $p_1, p_2$. Setting $p = r$,
\begin{align}
\inf_{\theta} \sup_{(p^*, p): D(p^*, p)\leq \epsilon, p^* \in \mathcal{G}} L(p^*, \theta) & \geq \inf_{\theta} \frac{1}{2} \left( L(p_1, \theta) + L(p_2, \theta) \right) \\
& \geq \frac{1}{2} L(p_1, \theta^*(p_2)). 
\end{align}
The last inequality comes from the assumption that $L(p, \theta^*(q))$ is a pseudometric for $p, q$.
Since this inequality holds for \emph{any} $p_1 \in \mathcal{G}, p_2 \in \mathcal{G}, D(p_1, p_2) \leq 2\epsilon$, the result follows. 

By Lemma~\ref{lem.minimax_random}, we know that for random decision rule $\theta_r(p)$ and any $p_1, p_2$ satisfying the condition above, 
\begin{align}
    & \inf_{\theta_r(p)} \sup_{(p^*, p): D(p^*, p)\leq \epsilon, p^* \in \mathcal{G}} \bP(L(p^*, \theta_r(p))\geq \frac{1}{2} L(p_1, \theta^*(p_2))) \nonumber \\
     \geq & \inf_{\theta_r(p)} \frac{1}{2} \left( \bP( L(p_1, \theta_r(p))\geq \frac{1}{2} L(p_1, \theta^*(p_2))) + \bP( L(p_2, \theta_r(p))\geq \frac{1}{2} L(p_1, \theta^*(p_2))) \right) \nonumber \\
     \geq & \frac{1}{2}. \nonumber
\end{align}
\end{proof}

\subsection{Example when modulus is not a tight bound}

In Lemma~\ref{lemma.population_limit} and Lemma~\ref{lemma.population_limit_opt}, we show that modulus is a valid upper bound for the population limit when $D$ is a pseudometric and the bound is tight when $L$ is also a pseudometric and the space induced by $D$ is a complete convex metric space. Here we give a concrete example such that the modulus is not a tight bound when $D= \TV$ but $L$ is not a pseudometric. 

\begin{theorem}
Consider an one-dimensional classification problem. 
Take  $L(p, (\theta_1, \theta_2)) = \bE_p[\mathbb{1}(Y(\theta_1X- \theta_2) \leq 0)]$. For any constant $C > 0$, there exist two distributions $p_1, p_2$ such that 
\begin{align}
  \modu(\{p_1, p_2\}, 2\epsilon)\geq (C+1) \cdot \inf_{\theta(p)}\sup_{(p^*, p): \TV(p^*, p) \leq \epsilon, p^* \in \{p_1, p_2\}} L(p^*, \theta(p)).  
\end{align}
\end{theorem}

\begin{proof}
For given $C>0$, we design $p_1, p_2$ as follows. Consider two dimensional distributions $p_1, p_2$. We use $X$ to denote the covariate and $Y$ to denote the label. Here we assume $X$ is supported on $[0, 4]$ and $Y$ is supported on $\{-1, 1\}$. Let $\bP_{p_1}(Y = 1\mid X) = \mathbb{1}(X\geq  1)$, $\bP_{p_2}(Y = 1 \mid X) = \mathbb{1}(X\geq 3)$. We design marginal distribution of $p_1, p_2$ as follows
\begin{align}
    \bP_{p_1}(X=t) = \begin{cases} \frac{1}{2}-\epsilon, & t = 0.5\\
    \frac{2\epsilon}{C+1}, & t = 1.5\\
     \frac{2C \epsilon}{C+1}, & t = 2.5\\
    \frac{1}{2}-\epsilon, & t = 3.5
    \end{cases}, \quad  \bP_{p_2}(X=t) = \begin{cases} \frac{1}{2}-\epsilon, & t = 0.5\\
    \frac{2C\epsilon}{C+1}, & t = 1.5\\
     \frac{2 \epsilon}{C+1}, & t = 2.5\\
    \frac{1}{2}-\epsilon, & t = 3.5
    \end{cases}, 
\end{align}
Then we have $\TV(p_1, p_2) = 2\epsilon$. If we observe $p$ such that $\TV(p_1, p)<\epsilon$, then the true distribution $p^*$ must be $p_1$, we can take $\theta(p) = \theta^*(p_1)$, the induced cost $L(p_1, \theta(p)) = 0$.  Similarly, the cost is also $0$ when we observe $p$ such that $\TV(p_2, p)<\epsilon$. When $\TV(p_1, p) = \epsilon$ and $\TV(p_2, p) = \epsilon$. We have $p = \frac{p_1 + p_2}{2}$. Then we will output $\theta_1 = 1, \theta_2 = 2$, which gives  cost $L(p_1, (\theta_1, \theta_2)) = L(p_2, (\theta_1, \theta_2)) = \frac{2\epsilon}{C+1}$. Thus the population limit is upper bounded by $\frac{2\epsilon}{C+1}$.

However, the modulus is at least $L(p_1, \theta^*(p_2)) = 2\epsilon$. Thus we have 
\begin{align}
  \modu(\{p_1, p_2\}, 2\epsilon, \TV, L)\geq (C+1) \cdot \inf_{\theta(p)}\sup_{(p^*, p): \TV(p^*, p) \leq \epsilon, p^* \in \{p_1, p_2\}} L(p^*, \theta(p)).  
\end{align}
\end{proof}

\section{Related discussions and remaining proofs in Section~\ref{sec.population_TV}}\label{sec.appendix_population_TV}

\subsection{Resilience for pseudonorm loss: generalization of mean estimation}\label{sec.resilienceforwfdefinition}

We  present a straightforward generalization of the mean estimation example to the so-called $W_\sF$ (pseudo)norm \footnote{It was shown in~\citep[Lemma 1]{zhu2019deconstructing} that under some appropriate topology of distributions, any pseudonorm can be represented by $W_\sF$-norm for some symmetric family $\sF$ satisfying that $-\sF = \sF$. }. The $W_{\sF}(p, q)$ pseudonorm between two probability distributions $p,q$ is defined as
\begin{align}\label{eqn.def_WF}
    W_{\sF}(p, q) = \sup_{f\in\sF}\bE_p[f(X)] - \bE_q[f(X)],
\end{align}
where $\sF$ is symmetric, i.e. $\forall f\in \sF$, we have $-f\in\sF$. The corresponding resilient set is defined as
\begin{align}\label{eqn.def_G_TV_WF}
     \GG_{W_\sF}(\rho, \eta) = \{p  \mid  \sup_{r \leq \frac{p}{1-\eta}}  W_\sF(r, p)  \leq \rho\}.
\end{align}
With the same technique as mean estimation case, we can show 
\begin{align}\label{eqn.wfpopulationlimitmodulusbound}
    \modu(\GG_{W_\sF}(\rho, \eta), 2\epsilon) \leq 2\rho
\end{align}
if $2\epsilon \leq \eta < 1$.

\subsection{Key Lemmas}

\subsubsection{General properties of $\GG$}

For any two distributions $p,q$, a new distribution $r = \frac{\min(p, q)}{1-\TV(p, q)}$ is defined as follows. For any dominating measure $\nu$ satisfying $p \ll \nu, q\ll \nu$, we define $\frac{dr}{d\nu} = \min(\frac{dp}{d\nu}, \frac{dq}{d\nu})/(1-\TV(p,q))$. 

\begin{lemma}[Properties of deletion] \label{lem.deletionrproperties}
Denote by $\mathcal{P}$ the space of probability distributions. For any $\eta\in [0,1)$, the following statements are true.  
\begin{enumerate}
    \item \emph{$\eta$-deletion belongs to $\eta$-TV perturbation:} for any $r,p\in \mathcal{P}$, 
    \begin{align}
        r \leq \frac{p}{1-\eta} \Rightarrow \TV(r,p)\leq \eta
    \end{align}
    \item \emph{Existence of middle point:} for any $p\in \mathcal{P},q \in \mathcal{P},\TV(p,q)\leq \eta$, there exists some $r\in \mathcal{P}$ such that $r\leq \frac{p}{1-\eta}, r\leq \frac{q}{1-\eta}$. 
    \item \emph{Composition preserves being deletion:}
    If $r\leq \frac{p}{1-\eta}$, $r'\leq \frac{r}{1-\eta}$, then $r'\leq \frac{p}{(1-\eta)^2}$. 
    \item \emph{For any fixed $p\in \mathcal{P}$, the following three sets are equivalent:} 
\begin{itemize}
    \item $\mathcal{A}_1 = \{r  \mid r\leq \frac{p}{1-\eta}, r\in \mathcal{P} \}$,
    \item $\mathcal{A}_2 = \{r \mid \text{for all }A, \bP_r[X \in A] = \bP_p[X \in A|Z=0], Z \in \{0, 1\}, \bP(Z=0)\geq 1-\eta, \bP_p(X\in A) = p(A)\}$,
    \item $\mathcal{A}_3 = \{ \frac{\min(p, q)}{1-\TV(p, q)} \mid \TV(p, q)\leq \eta, q\in \mathcal{P} \}$. 
\end{itemize}
 \item If $r\in \mathcal{P}, p\in \mathcal{P}$ as distributions of $X$ satisfy $r\leq \frac{p}{1-\eta}$, then the induced distribution for $f(X)$ under both $r$ and $p$ satisfy the same relation for any measurable $f$. 
 \end{enumerate}
\end{lemma}

\begin{proof}
The first claim can be shown via the following inequalities:
\begin{align}
    r\leq \frac{p}{1-\eta}& \Rightarrow r - p \leq \eta r\nonumber \\
    & \Rightarrow \TV(r, p) = \sup_{A} \bP_r(A) - \bP_p(A)\leq \sup_{A}   \eta r(A) \leq \eta.
\end{align}
The second claim can be shown via taking $r = \frac{\min\{p,q\}}{1-\TV(p,q)}$.   From $\int_{\{x: p(x) > q(x) \}} (p(x) - q(x)) \nu(dx) = \TV(p, q)$ we can see that $r$ is a probability distribution. Furthermore, 
from $\TV(p, q)\leq \eta$, it is clear that
\begin{align}
    r \leq \frac{p}{1-\eta}, r \leq \frac{q}{1-\eta}. 
\end{align}
The third claim can be seen by
\begin{align}
    r'\leq \frac{r}{1-\eta} \leq \frac{p}{(1-\eta)^2}.
\end{align}
Now we show the equivalence of three sets in 
the fourth claim. 
We first show that $\mathcal{A}_1 \subset \mathcal{A}_2 $.
For any $r\leq \frac{p}{1-\eta}$, set distribution $q$ to satisfy that for any set $A$, $q(A) = \frac{p(A)  - (1-\eta)r(A) }{\eta}$. Then $q$ is a valid probability measure. We design the joint distribution of $X, Z$ such that
\begin{align*}
    X|(Z = 0) & \sim r, \\
    X|(Z = 1) & \sim q, \\
    \bP_p(Z = 0) & = 1-\eta, \\
    \bP_p(Z = 1) & = \eta.
\end{align*}
Then one can verify that $X\sim p$. We have found some $Z$ such that  $\bP_r[X\in A] = \bP_p[X\in A|Z=0], Z \in \{0, 1\}, \bP(Z=0)\geq 1-\eta$ for any measurable set $A$. This shows that $\mathcal{A}_1 \subset \mathcal{A}_2$.

We then show that $\mathcal{A}_2 \subset \mathcal{A}_3 $. Given a distribution $r\in \mathcal{A}_2$ and $p$, we choose a dominating measure $\nu$ such that $r\ll \nu, p\ll \nu$ and write the corresponding Radon--Nikodym derivatives as $r(x),p(x)$. Now the goal is to find some $q$ such that $r = \frac{\min(p,q)}{1-\TV(p,q)}$. We construct $q(x)$ as follows
\begin{align}
    q(x) = \begin{cases}  \bP(Z = 0| X=x)p(x), &\bP(Z = 0| X=x)< 1 \\ C\cdot p(x), & \bP(Z = 0| X=x) = 1 \end{cases}.
\end{align}
Here $C \geq 1$ is chosen such that $\int q(x) = 1$. Thus $\TV(p, q)$ can be computed as
\begin{align}
    \TV(p, q)  & = \int_{\{x: q(x)<p(x)\}} (p(x) - q(x)) dx \nonumber \\
    & =  \int_{\{x: \bP(Z = 0| X=x) < 1\}} (1- \bP(Z = 0| X=x))p(x) dx\nonumber \\
    & = \int_{\{x: \bP(Z = 0| X=x) < 1\}}  \bP(Z = 1| X=x)p(x) dx\nonumber\\
    & = \int_{\mathcal{X}}  \bP(Z = 1| X=x)p(x) dx\nonumber\\
    & = \bP(Z = 1) \leq \eta .
\end{align}
One can check that
\begin{align}
    \frac{\min(p(x), q(x))}{1-\TV(p, q)} = p(x \mid Z = 0) =  \begin{cases}  \frac{\bP(Z = 0| X=x) p(x)}{\bP(Z = 0)}, &\bP(Z = 0| X=x)< 1 \\ \frac{p(x)}{\bP(Z = 0)}, & \bP(Z = 0| X=x) = 1 \end{cases}
\end{align}
which shows that $\mathcal{A}_2 \subset \mathcal{A}_3$.

Lastly, we show that $\mathcal{A}_3 \subset \mathcal{A}_1$. This can be seen by the construction in the second claim. From $\int_{\{x: p(x) > q(x) \}} (p(x) - q(x)) = \TV(p, q)$ we can see that $r$ is a distribution. Furthermore, 
from $\TV(p, q)\leq \eta$, it is clear that
\begin{align}
    r \leq \frac{p}{1-\eta}, r \leq \frac{q}{1-\eta}. 
\end{align}

To show the fifth claim, if we know that for any measurable set $A$, $r_X(A) \leq \frac{p_X(A)}{1-\eta}$, then for any measurable function $f(X)$, and any measurable set $A$, we have
\begin{align}
    r_{f(X)}(A) = r( f^{-1}(A))\leq \frac{p(f^{-1}(A))}{1-\eta} = \frac{p_{f(X)}(A)}{1-\eta},
\end{align}
where $f^{-1}(A) = \{x \mid f(x)\in A\}$. 
\end{proof}

Now, we show that if a distribution has bounded Orlicz norm, then it is inside some resilient set $\GG_{W_\sF}$ defined in \eqref{eqn.def_G_TV_WF}. 
\begin{lemma}[Bounded Orlicz norm implies resilience]\label{lem.cvx_mean_resilience}
Given an Orlicz function $\psi$, assume 
\begin{align}
   \sup_{f\in\sF} \bE_p\left[\psi\left(\frac{\left|f(X) - \bE_p[f(X)]\right|}{\sigma}\right)\right] \leq 1
\end{align}
for some symmetric family $\sF$ and some $\sigma >0$.
For any $\eta \in [0, 1)$, we have
\begin{align}
        p \in \GG_{W_\sF}^{\TV}\left(    \frac{{\sigma\eta} \psi^{-1}(1/\eta)}{1-\eta} \wedge \sigma \psi^{-1}\left ( \frac{1}{1-\eta} \right), \eta\right ),
\end{align}
where $\GG_{W_\sF}^{\TV}$ is defined in~(\ref{eqn.def_G_TV_WF}),  $\psi^{-1}$ is the (generalized) inverse function of $\psi$.
\end{lemma}

\begin{proof}
The proof uses the property that for any $r\leq \frac{p}{1-\eta}$, there exists some event $E$ such that $\bP_p(E)\geq 1-\eta$ and $\bE_r[f(X)] = \bE_p[f(X)|E]$ for any measurable $f$~(Lemma~\ref{lem.deletionrproperties}).
    For any event $E$ with $\bP_p(E)\geq 1-\eta$, denote its compliment as $E^c$,  by the definition of conditional expectation, we have
\begin{align}
    \sup_{f\in\sF} \bE_p[f(X) | E] - \bE_p[f(X) ]  & = \sup_{f\in\sF} \frac{\bP_p(E^c)}{1-\bP_p( E^c)}\bE_p\left[f(X)  - \bE_p[f(X)]|E^c\right] 
\end{align}
By the bounded condition and convexity of $\psi$, one can see
\begin{align}
    1 & \geq \sup_{ f\in\sF} \bE_p\left[\psi\left(\frac{|f(X)-\bE_p[f(X)]|}{\sigma}\right)\right] \nonumber \\
    & \geq \sup_{ f\in\sF} \bP_p(E^c)\bE_p\left[\psi\left(\frac{|f(X)-\bE_p[f(X)]|}{\sigma}\right) \mid E^c\right] \nonumber \\
    & \geq  \sup_{f\in\sF} \bP_p(E^c)\psi\left(\frac{|\bE_p\left[f(X)-\bE_p[f(X)] \mid E^c\right] |}{\sigma}\right).
\end{align}
By definition of $\psi^{-1}$, this gives us 
\begin{align}
     \sup_{f\in\sF} \bE_p[f(X) | E] - \bE_p[f(X) ] & \leq  \sigma\frac{\bP_p(E^c)}{1-\bP_p( E^c)} \psi^{-1}(1/\bP_p(E^c)) \nonumber \\
     & \leq \frac{\sigma\eta}{1-\eta}\psi^{-1}(1/\eta).
     \end{align}
The last inequality uses the fact that $x\psi^{-1}(1/x)$ is a non-decreasing function from Lemma~\ref{lem.psi_increasing}. Thus we have $p \in \GG_{W_\sF}^{\TV}(\frac{\sigma\eta \psi^{-1}(1/\eta)}{1-\eta}, \eta)$ for any $\eta \in [0, 1)$.

Similarly, we have
\begin{align}
    \sup_{f\in \sF} \bE_p[f(X)|E] - \bE_p[f(X)] & \leq \sigma \psi^{-1}\left( \frac{1}{\bP_p(E)} \right) \\
    & \leq \sigma \psi^{-1}\left( \frac{1}{1-\eta} \right),
\end{align}
since $\bP_p(E)\geq 1-\eta$ and $\psi^{-1}(1/x)$ is a non-increasing function of $x$. It implies that
$p \in \GG_{W_\sF}^{\TV}(\sigma \psi^{-1}\left( \frac{1}{1-\eta} \right), \eta)$ for any $\eta \in [0, 1)$. This part can also be derived from~\citep[Lemma 10]{steinhardt2017resilience}.
\end{proof}

The results can be improved if we know a non-centered Orlicz norm bound.

\begin{lemma}[Lower bound on deleted distribution]\label{lem.noncentered_orlicz}
Given an Orlicz function $\psi$, assume 
\begin{align}
   \sup_{f\in\sF} \bE_p\left[\psi\left(\frac{\left|f(X)\right|}{\sigma}\right)\right] \leq 1
\end{align}
for some family $\sF$ and some $\sigma >0$. Assume $f(x)\geq 0$ for any $x\in\bR^d, f\in\sF$. 
For any $\eta \in [0, 1)$ and any $r \leq \frac{p}{1-\eta}$, we have
\begin{align}
        \bE_r[f(X)] \geq  \bE_p[f(X)] - \sigma\eta\psi^{-1}(1/\eta).%
\end{align}
where  $\psi^{-1}$ is the (generalized) inverse function of $\psi$.
\end{lemma}

\begin{proof}
The proof uses the property that for any $r\leq \frac{p}{1-\eta}$, there exists some event $E$ such that $\bP_p(E)\geq 1-\eta$ and $\bE_r[f(X)] = \bE_p[f(X)|E]$ for any measurable $f$~(Lemma~\ref{lem.deletionrproperties}).
    For any event $E$ with $\bP_p(E)\geq 1-\eta$, denote its compliment as $E^c$,  by the definition of conditional expectation, we have $\forall f\in\sF$,
\begin{align}
     \bE_p[f(X)]  & = \bE_p[f(X) | E] \cdot \bP_p[E] +  \bE_p[f(X) | E^c] \cdot \bP_p[E^c]
\end{align}
By the bounded condition and convexity of $\psi$, one can see
\begin{align}
    1 & \geq \sup_{ f\in\sF} \bE_p\left[\psi\left(\frac{|f(X)|}{\sigma}\right)\right] \nonumber \\
    & \geq \sup_{ f\in\sF} \bP_p(E^c)\bE_p\left[\psi\left(\frac{|f(X)|}{\sigma}\right) \mid E^c\right] \nonumber \\
    & \geq  \sup_{f\in\sF} \bP_p(E^c)\psi\left(\frac{|\bE_p\left[f(X) \mid E^c\right] |}{\sigma}\right).
\end{align}
By definition of $\psi^{-1}$, this gives us 
\begin{align}
   \bE_p[f(X)]  & = \bE_p[f(X) | E] \cdot \bP_p[E] +  \bE_p[f(X) | E^c] \cdot \bP_p[E^c] \nonumber \\ 
   & \leq  \bE_p[f(X) | E] \cdot \bP_p[E] +  \sigma{\bP_p(E^c)} \psi^{-1}(1/\bP_p(E^c)) \nonumber \\
     & \leq \bE_p[f(X) | E]  + {\sigma\eta}\psi^{-1}(1/\eta).
     \end{align}
The last inequality uses the fact that $x\psi^{-1}(1/x)$ is a non-decreasing function from Lemma~\ref{lem.psi_increasing}. Thus we have the desired bound. 
\end{proof}

Lemma~\ref{lem.cvx_mean_resilience} can be improved (usually by a constant) if each $f(X)$ has moment generating function. 
\begin{lemma}\citep[Lemma 2.3]{massart2007concentration}
Let $\psi$ be some convex and continuously differentiable function on $[0,b)$ with $0<b\leq \infty$, such that $\psi(0) = \psi'(0) = 0$. Assume that for $\lambda \in (0,b)$,
\begin{align}
    \sup_{f\in \sF} \ln(\bE_p[\exp(\lambda (f(X) - \bE_p[f(X)]))]) \leq \psi(\lambda). 
\end{align}
Then, for any $\eta \in [0,1)$, 
\begin{align}
    p \in \GG_{W_\sF}\left ( \frac{\eta}{1-\eta} \psi^{*-1}(\ln(1/\eta)) \wedge \psi^{*-1}\left( \ln \left (\frac{1}{1-\eta}\right ) \right ), \eta \right),
\end{align}
where $\GG_{W_\sF}^{\TV}$ is defined in~(\ref{eqn.def_G_TV_WF}), $\psi^{*-1}$ is the generalized inverse of the Fenchel--Legendre dual of $\psi$:
\begin{align}
    \psi^*(x) = \sup_{\lambda \in (0,b)} (\lambda x - \psi(\lambda)).
\end{align}
\end{lemma}
In particular, if $\psi(\lambda) = \frac{\lambda^2 \sigma^2}{2}$ for all $\lambda \in (0,\infty)$, then
\begin{align}
    p\in \GG_{W_\sF}\left ( \frac{\sigma\eta}{1-\eta} \sqrt{2 \ln(1/\eta)} \wedge \sigma \sqrt{2\ln(1/(1-\eta))} ,\eta \right).
\end{align}
\begin{proof}
Fix $f\in \sF$. It follows from~\citep[Lemma 2.3]{massart2007concentration} that for any set $E$ we have
\begin{align}
    \bE_p[f(X)|E] - \bE_p[f(X)] \leq \psi^{*-1}(\ln(1/\bP_p(E))). 
\end{align}
Hence, for any set $E$ such that $\bP_p(E)\geq 1-\eta$, we have
\begin{align}
    \bE_p[f(X)|E] - \bE_p[f(X)] \leq \psi^{*-1}(\ln(1/(1-\eta))),
\end{align}
where we used the fact that $\psi^*$ is a non-negative convex and non-decreasing function on $\mathbf{R}_+$. Regarding the second bound, we first write
\begin{align}
    \bE_p[f(X)] -\bE_p[f(X)|E] = \frac{\bP_p(E^c)}{1-\bP_p(E^c)} (\bE_p[f(X)|E^c] - \bE_p[f(X)]). 
\end{align}
Thus,
\begin{align}
     \bE_p[f(X)] -\bE_p[f(X)|E] & \leq \frac{\bP_p(E^c)}{1-\bP_p(E^c)} \psi^{*-1}(\ln(1/\bP_p(E^c))) \\
    & \leq \frac{\eta}{1-\eta} \psi^{*-1}(\ln(1/\eta)),
\end{align}
where in the last step we used the fact that $\frac{x}{1-x}\psi^{*-1}(\ln(1/x))$ is a non-decreasing function on $(0,1)$. 
\end{proof}

When $W_\sF(p, q) = \|\bE_p[XX^T] - \bE_q[XX^T] \|_2 = \sup_{v\in\bR^d, \|v\|_2\leq 1, \xi\in\{\pm 1 \}}  (\bE_p[\xi(v^TX)^2] - \bE_q[\xi(v^TX)^2])$, the two lemmas above provide a simple proof for the upper bound of population limit of second moment estimation under operator norm, which matches the results of Gaussian case in~\cite{gao2018robust} up to logarithmic factor. One can also design $\tTV_\mathcal{H}$ for this that achieves the same sample complexity as~\cite{gao2018robust} up to logarithmic factor, which is shown in Theorem~\ref{thm.tTV_joint_correct}. 
We further provide ways to give lower bound for the population limit in Lemma~\ref{lem.Wf_orlicz_lower_bound_limit}.

The resilient set is closely related to tail bounds, which is described in the following lemma. Similar results are also shown in literature~\citep[Lemma 2.4]{massart2007concentration},~\citep[Example 2.7]{steinhardt2018robust}.

\begin{lemma}[Resilience implies tail bounds]\label{lem.tail_bound}
If $p \in \GG_{W_\sF}^\TV(\rho,  \eta)$, then for every $f\in \sF$, 
\begin{align}
\bP_p\left( f(X)- \mathbb{E}_p[f(X)] \geq \frac{(1-\eta)\rho}{\eta} \right) & \leq \eta, \\
\bP_p\left( f(X)- \mathbb{E}_p[f(X)] \leq - \frac{(1-\eta)\rho}{\eta} \right) & \leq \eta.
\end{align} 
\end{lemma}
\begin{proof}
We adopt a similar proof as~\citep[Lemma 2.4]{massart2007concentration}. 
Note that the set  $ \GG_{W_\sF}(\rho,  \eta)$ can be alternatively written as
\begin{align}
      \GG_{W_\sF}^{\TV}(\rho, \eta
      ) = \{p  \mid  \sup_{p(E)\geq 1-\eta, f\in\sF} \bE_p[f(X)|E] - \bE_p[f(X)]  \leq \rho\}.
\end{align}
Note that for any event $E$, we have
\begin{align}
    \bP_p(E)(\bE_p[f(X)|E] - \bE_p[f(X)]) + \bP_p(E^c)( \bE_p[f(X)|E^c] - \bE_p[f(X)]) = 0.
\end{align}
It follows from the symmetry of $\sF$ that
\begin{align}
    \sup_{f\in\sF}\bP_p(E)( \bE_p[f(X)|E] - \bE_p[f(X)]) = \sup_{f\in\sF}  \bP_p(E^c)(\bE_p[f(X)|E^c] - \bE_p[f(X)]).
\end{align}
Thus we have
\begin{align}
     \sup_{p(E)\leq \eta, f\in\sF} \bE_p[f(X)  - \bE_p[f(X)] |E] \leq \frac{(1-\eta)\rho}{\eta}.
\end{align}
Taking $E = \{ f(X) - \bE_p[f(X)]\geq a\}$, by Markov's inequality, we have
\begin{align}
    a \leq \bE_p[f(X)  - \bE_p[f(X)] |E] \leq \frac{(1-\eta)\rho}{\bP_p(E)}.
\end{align}
Thus
\begin{align}
    \bP_p(f(X) - \bE_p[f(X)]\geq a)\leq \frac{(1-\eta)\rho}{a}.
\end{align}
The other side holds analogously. 

\end{proof}

\subsubsection{Key Lemmas for lower bound}

The below lemma shows that the population limit for resilient set is optimal up to constant under some topology assumption of $f$. We first show that if $p$ is inside some resilient set, then its deleted distribution $r$ is also inside some resilient set that has same population limit up to a multiplicative constant in many cases. 

\begin{lemma}[Resilience is approximately closed under deletion]\label{lem.resilience_closed}
Assume  $p\in  \GG_{W_\sF}^{\TV}(\rho, \eta(2-\eta))$ defined in Equation~\eqref{eqn.def_G_TV_WF}.
Then for any $r\leq \frac{p}{1-\eta}$, we have
\begin{align}
    r\in \GG_{W_\sF}^{\TV}(2\rho, \eta).
\end{align}
\end{lemma}
\begin{proof}
From $r\leq \frac{p}{1-\eta}$,  we know that for any $q \leq \frac{r}{1-\eta}$, we have
\begin{align}
    q \leq \frac{p}{(1-\eta)^2} =\frac{p}{1-2\eta + \eta^2}.
\end{align}
Thus from $p \in \GG_{W_\sF}^{\TV}(\rho, \eta(2-\eta))$, $\eta(2-\eta) \geq \eta$, we have
\begin{align}
    W_\sF(q, p) \leq \rho,\\
    W_\sF(r, p) \leq \rho.
\end{align}
Thus
\begin{align}
    \sup_{q\leq\frac{r}{1-\eta}} W_\sF(q, r) \leq  \sup_{q\leq\frac{r}{1-\eta}}  W_\sF(q, p) + W_\sF(p, r) \leq 2\rho.
\end{align}
\end{proof}

\begin{lemma}[Population limit for $\GG_{W_\sF}^{\TV}(\rho, \eta)$ is optimal]\label{lem.population_lower_bound} 
Assume $\epsilon \in [0, 1)$ is the perturbation level, $\eta \geq \epsilon$, and there exist a distribution $p_1 \in \GG_{W_\sF}(\rho/2, \epsilon(2-\epsilon))$ and some distribution $r_1\leq \frac{p_1}{1-\epsilon}$ such that 
\begin{align}
    W_\sF(p_1, r_1) \geq c\rho.
\end{align}
Then the population limit of the set $ \GG_{W_\sF}^{\TV}(\rho, \eta)$ under perturbation level $\epsilon$ is lower bounded by $c\rho/2$, i.e.,
\begin{align}
     \inf_{q(p)} \sup_{(p^*,p): \TV(p^*, p)\leq \epsilon, p^*\in \GG_{W_\sF}(\rho, \eta)} W_\sF(p^*, q) \geq \frac{1}{2}c\rho.
\end{align}
This matches the upper bound of population limit in~(\ref{eqn.wfpopulationlimitmodulusbound}) for $\GG_{W_\sF}^{\TV}$ up to a constant. Furthermore, for randomized decision rule $q_r(p)$, 
\begin{align}
     \inf_{q_r(p)} \sup_{(p^*,p): \TV(p^*, p)\leq \epsilon, p^*\in \GG_{W_\sF}(\rho, \eta)} \bP(W_\sF(p^*, q_r(p)) \geq \frac{c\rho}{2}) & \geq \frac{1}{2}.
\end{align}
\begin{proof}
From Lemma~\ref{lem.resilience_closed} and $p_1\in\GG_{W_\sF}^{\TV}(\rho/2, \epsilon(2-\epsilon))$, we know that $r_1\in \GG_{W_\sF}(\rho, \epsilon)$. From the assumption we also know that $p_1$ is inside the same set. Assume the observed corrupted distribution is $p = p_1$. Then, 
\begin{align}
     \inf_{q(p)} \sup_{(p^*,p): \TV(p^*, p)\leq \epsilon, p^*\in \GG_{W_\sF}(\rho, \eta)} W_\sF(p^*, q(p)) & \geq \inf_{q} \sup_{p^*: \TV(p^*, p_1)\leq \epsilon, p^*\in \GG_{W_\sF}(\rho, \eta)} W_\sF(p^*, q) \nonumber \\
     & \geq \frac{1}{2} \inf_{q} (W_\sF(p_1, q) + W_\sF(q, r_1)) \nonumber \\
     & \geq \frac{1}{2} W_\sF(p_1, r_1) \nonumber \\
     & \geq \frac{c \rho}{2}.
\end{align}
Thus we know that the population limit of  $\GG_{W_\sF}(\rho, \epsilon)$ is lower bounded by $c\rho/2$. From Lemma~\ref{lem.minimax_random}, we know that for randomized decision rule $q_r(p)$, 
\begin{align*}
     \inf_{q_r(p)} \sup_{(p^*,p): \TV(p^*, p)\leq \epsilon, p^*\in \GG_{W_\sF}(\rho, \eta)} \bP(W_\sF(p^*, q_r(p)) \geq \frac{c\rho}{2}) & \geq \frac{1}{2} \inf_{q_r(p)} (\bP(W_\sF(p_1, q_r(p)) \geq \frac{c\rho}{2}) \nonumber \\ &  + \bP(W_\sF(r_1, q_r(p))  \geq \frac{c\rho}{2})) \nonumber \\ 
     &
      \geq \frac{1}{2}.
\end{align*}
\end{proof}
\end{lemma}
This lemma combined with Theorem~\ref{thm.G_fundamental_limit} empowers us to show that under appropriate choice of $W_\sF$ and $\rho$, one can tightly bound the information-theoretic limit of $\GG_{W_\sF}$ within universal constant factors.

The next two lemmas shows that we can also show similar results for Orlicz norm bounded set under some topological assumption of $f$. 
\begin{lemma}[Orlicz norm bounded set is approximately closed under deletion]\label{lem.Wf_orlicz_closed}
For some Orlicz function $\psi$, define
\begin{align}
   \GG_\psi(\sigma) = \{ p \mid \sup_{f\in\sF} \bE_p\left[\psi\left(\frac{|f(X) - \bE_p[f(X)]|}{\sigma}\right)\right] \leq 1 \}.
\end{align}
for some symmetric family $\sF$. Assume that there exist some distribution $p\in \GG_\psi(\sigma)$ and $\epsilon \leq 1/2$, then
\begin{align}
    r\in \GG_{(1-\epsilon)\psi}\left(5\sigma\right) = \left\{ p \mid \sup_{f\in\sF} \bE_p\left[(1-\epsilon)\psi\left(\frac{|f(X) - \bE_p[f(X)]|}{5\sigma}\right)\right] \leq 1 \right\}
\end{align}
\end{lemma}
\begin{proof}

We use $\|\cdot \|_{\psi, r}$ to represent the $\psi-$norm of $X\sim r$. Denote $\tilde \psi = (1-\epsilon)\psi$.
By triangle inequality,
\begin{align}\label{eqn.proof_closed_bound}
    \sup_{f\in\sF} \|f(X) - \bE_r[f(X)]\|_{\tilde \psi, r} & \leq    \sup_{f\in\sF} \|f(X) - \bE_p[f(X)]\|_{\tilde \psi, r} + \sup_{f\in\sF}\| \bE_p[f(X)] - \bE_r[f(X)]\|_{\tilde \psi, r}.
\end{align}
Now we bound the first term of RHS, note that $p\in\GG_\psi(\sigma)$ implies that
\begin{align}
    \sup_{f\in\sF} \bE_p\Big[\tilde \psi\big(\frac{|f(X) - \bE_p[f(X)]|}{\sigma}\big)\Big] = (1-\epsilon)\sup_{f\in\sF} \bE_p\Big[\psi\big(\frac{|f(X) - \bE_p[f(X)]|}{\sigma}\big)\Big] \leq  1-\epsilon.
\end{align}
Thus from $r\leq \frac{p}{1-\epsilon}$, we have 
\begin{align}
        \sup_{f\in\sF} \bE_r\Big[\tilde \psi\big(\frac{|f(X) - \bE_p[f(X)]|}{\sigma}\big)\Big] \leq  \frac{1}{1-\epsilon} \sup_{f\in\sF} \bE_p\Big[\tilde \psi\big(\frac{|f(X) - \bE_p[f(X)]|}{\sigma}\big)\Big] \leq 1.
\end{align}

From
Lemma~\ref{lem.cvx_mean_resilience} and the definition of $\psi^{-1}$,
the second term of RHS in Equation (\ref{eqn.proof_closed_bound}) is
\begin{align}
    \frac{\sup_{f\in\sF}|\bE_p[f(X)] - \bE_r[f(X)]|}{\tilde \psi^{-1}(1)} \leq   \frac{2\sigma\epsilon \psi^{-1}(1/\epsilon)}{\psi^{-1}(\frac{1}{1-\epsilon})}.
\end{align}

Combining the upper bound of two terms together,   we have
\begin{align}
    \sup_{f\in\sF} \|f(X) - \bE_r[f(X)]\|_{\tilde \psi, r} \leq \sigma\left(\frac{2\epsilon \psi^{-1}(1/\epsilon)}{\psi^{-1}(\frac{1}{1-\epsilon})}+1\right).
\end{align}

If we further assume that   that $\epsilon \leq 1/2$, by Lemma~\ref{lem.psi_increasing}, we know $\epsilon\psi^{-1}(1/\epsilon)\leq (1-\epsilon)\psi^{-1}(1/(1-\epsilon)) $. Thus
\begin{align}
     \sup_{f\in\sF} \|f(X) - \bE_r[f(X)]\|_{\tilde \psi, r} \leq \sigma\left(\frac{2\epsilon \psi^{-1}(1/\epsilon)}{\psi^{-1}(\frac{1}{1-\epsilon})}+1\right) \leq \sigma(\frac{2}{1-\epsilon}+1) \leq 5 \sigma. 
\end{align}
\end{proof}
Based on the above Lemma, we are able to show that under mild topological condition of the range of $f$, the upper bound we derive in Lemma~\ref{lem.cvx_mean_resilience} is optimal up to a constant. 
\begin{lemma}[Population limit  for Orlicz norm bounded set  is optimal] \label{lem.Wf_orlicz_lower_bound_limit}
For some Orlicz function $\psi$, define
\begin{align}
   \GG_\psi(\sigma) = \{ p \mid \sup_{f\in\sF} \bE_p\left[\psi\left(\frac{|f(X) - \bE_p[f(X)]|}{\sigma}\right)\right] \leq 1 \}.
\end{align}
for some symmetric family $\sF$. Assume that $\epsilon \leq 1/2$
and there exists some  function $f\in\sF$ that has range as a superset of $(-\infty, 0]$ or $ [0, +\infty)$. 
Then the population limit of  $\GG_{\psi}(\sigma)$ is lower bounded below:
\begin{align}
      \inf_{q(p)} \sup_{(p^*, p): \TV(p^*, p)\leq \epsilon, p^*\in \GG_{\psi}(\sigma)} W_\sF(p^*, q) \geq \frac{\sigma\epsilon\psi^{-1}(1/\epsilon)}{20},
\end{align}
where $\psi^{-1}$ is the (generalized) inverse function of $\psi$. This matches the upper bound in Lemma~\ref{lem.cvx_mean_resilience} up to a constant. Furthermore, for random decision rule $q(p)$, we also have
\begin{align}
     \inf_{q_r(p)} \sup_{(p^*,p): \TV(p^*, p)\leq \epsilon, p^*\in \GG_{W_\sF}(\rho, \eta)} \bP(W_\sF(p^*, q_r(p)) 
     \geq  \frac{ \sigma\epsilon\psi^{-1}(1/\epsilon)}{20}) 
     \geq  \frac{1}{2}.
\end{align}
\end{lemma}
\begin{proof}

From Lemma~\ref{lem.Wf_orlicz_closed}, we know that 
\begin{align}
    p_1\in \GG_{\psi / (1-\epsilon)}(\sigma/5) \Rightarrow  \forall r_1 \leq \frac{p_1}{1-\epsilon}, p_1, r_1 \in \GG_{\psi}(\sigma).
\end{align}
Thus here we would like to show that 
there exist a distribution $p_1 \in \GG_{\psi / (1-\epsilon)}(\sigma/5)$ and some distribution $r_1\leq \frac{p_1}{1-\epsilon}$ such that
\begin{align}
    W_\sF(p_1, r_1) \geq C_1\sigma \epsilon\psi^{-1}(1/\epsilon),
\end{align}

Assume the range of some $f\in\sF$ is a superset of $[0, +\infty)$. We construct the distribution $p_1$ as follows,
\begin{align}
     \bP_{p_1}(f(X) =  t) &= \left\{ \begin{array}{cl}
     \epsilon,    & t  = \sigma {\psi^{-1}((1-\epsilon)/\epsilon)}/5   \\
    1-\epsilon,   & t   = 0\\
    0 ,    & \text{otherwise}
    \end{array}\right. \\
\end{align}
Then we have
\begin{align}
    \bE_{p_1}\left[\psi\left(\frac{f(X)}{\sigma/5}\right)/ (1-\epsilon)\right] = 1, 
\end{align}
which means that $p_1 \in \GG_{\psi / (1-\epsilon)}(\sigma/5)$, furthermore, we can design $r\leq \frac{p_1}{1-\epsilon}$ by deleting the non-zero part of $p_1$, thus for $\epsilon \leq 1/2$, we have
\begin{align}
    W_\sF(p_1, r_1)& \geq |\bE_{p_1}[f(X)] - \bE_r[f(X)]| \geq \frac{\sigma}{5} \epsilon \psi^{-1}(\frac{1-\epsilon}{\epsilon}) \geq \frac{\sigma}{5} \epsilon \psi^{-1}(\frac{1}{2\epsilon})\nonumber \\ 
    & \geq \frac{\sigma}{10} \epsilon (\psi^{-1}(\frac{1}{\epsilon})+\psi^{-1}(0)) = \frac{\sigma}{10} \epsilon \psi^{-1}(\frac{1}{\epsilon}).
\end{align}
The last two inequality is from Jensen's inequality and the fact that $\psi^{-1}$ is a concave function, $\psi^{-1}(0) =0$. 
Thus if the observed corrupted population distribution $p = p_1$, 
\begin{align}
    \inf_{q(p)} \sup_{(p^*, p): \TV(p^*, p)\leq \epsilon, p^*\in \GG_{\psi}(\sigma)} W_\sF(p^*, q) & \geq \inf_{q(p_1)} \sup_{p^*: \TV(p^*, p_1)\leq \epsilon, p^*\in \GG_{\psi}(\sigma)} W_\sF(p^*, q) \nonumber \\
     & \geq \frac{1}{2} \inf_{q} \left ( W_\sF(p_1, q) + W_\sF(q, r_1) \right )\nonumber \\
     & \geq \frac{1}{2} W_\sF(p_1, r_1) \nonumber \\
     & \geq \frac{ \sigma\epsilon\psi^{-1}(1/\epsilon)}{20}.
\end{align}

From Lemma~\ref{lem.minimax_random}, we know that for randomized decision rule $q_r(p)$, 
\begin{align}
     & \inf_{q_r(p)} \sup_{(p^*,p): \TV(p^*, p)\leq \epsilon, p^*\in \GG_{W_\sF}(\rho, \eta)} \bP(W_\sF(p^*, q_r(p)) 
     \geq  \frac{ \sigma\epsilon\psi^{-1}(1/\epsilon)}{20}) \nonumber  \\  \geq & \frac{1}{2} \inf_{q_r(p)} (\bP(W_\sF(p_1, q_r(p)) \geq \frac{ \sigma\epsilon\psi^{-1}(1/\epsilon)}{20}) + \bP(W_\sF(r_1, q_r(p)) \geq \frac{ \sigma\epsilon\psi^{-1}(1/\epsilon)}{20})) \nonumber \\ 
     \geq & \frac{1}{2}.
\end{align}

\end{proof}

\section{Related discussions and remaining proofs in Section~\ref{sec.finite_sample_algorithm_weaken}}

\subsection{Proof of Lemma~\ref{lemma.vcinequality}}\label{Appendix.proof_vcinequality}

\begin{proof}
The first statement follows from the VC inequality~\citep[Chap 2, Chapter 4.3]{devroye2012combinatorial}. Now we prove the second statement. Fix $f \in \mathcal{H}$ and denote $M = |\mathcal{H}|$. By the Dvoretzky-Kiefer-Wolfowitz inequality~\cite{dvoretzky1956asymptotic}, with probability $1 - 2\exp(-2n\epsilon^2)$ we have $|\bP_{\hat{p}_n}[f(x) \geq t] - \bP_{p}[f(x) \geq t]| \leq \epsilon$ for all $t \in \bR$.
Union bounding over $f \in \mathcal{H}$, we have that 
$\tTV_{\mathcal{H}}(\hat{p}_n, p) \leq \epsilon$ with probability at least $1 - 2M\exp(-2n\epsilon^2)$. 
Solving for $\delta$, we obtain $\epsilon = \sqrt{\log(2M/\delta)/2n}$, which proves the 
lemma.
\end{proof}

\subsection{Proof of Theorem~\ref{thm.tTV_mean}}\label{proof.mean_psi}

We rely on the combination of the two lemmas. First, we have the population result that $\GG(\psi)\subset \GG_{\mathsf{mean}}^{\TV}$ in the following lemma.
\begin{lemma}
For any Orlicz function $\psi$, assume that
\begin{align}  
   \sup_{v\in\bR^d, \|v \|_*=1} \bE_p\left[\psi\left(\frac{|\langle v, X - \bE_p[X]\rangle|}{\sigma}\right)\right] \leq 1,
\end{align}
where $\| \cdot\|_*$ is the dual norm of $\|\cdot\|$. Then:
\begin{enumerate}
    \item For any $\eta \in [0, 1)$, we have 
$p \in \GG^\TV_{\mathsf{mean}}\left(\frac{\sigma\eta \psi^{-1}(1/\eta)}{1-\eta}, \eta\right )$;
\item The population limit for the set of distributions satisfying~(\ref{eqn.boundedpsimeanexample}) is  $\Theta(\sigma\epsilon\psi^{-1}(\frac{1}{2\epsilon}))$ for $\epsilon < 0.499$. 
\end{enumerate}
\end{lemma}

Note here that the conclusion $p\in\GG_{\mathsf{mean}}(\frac{\sigma\epsilon\psi^{-1}(1/\epsilon)}{1-\eta}, \eta)$ for any $\eta \in [0, 1)$ is a corollary of Lemma~\ref{lem.cvx_mean_resilience}. Thus the by Lemma~\ref{lem.G_TV_mean_modulus} the population limit when the perturbation level is $\epsilon$ for some $\epsilon<1/4$ is upper bounded by $C\sigma\epsilon\psi^{-1}(1/\epsilon)$ for some universal constant $C$. Furthermore, we show in Lemma~\ref{lem.Wf_orlicz_lower_bound_limit} that the population limit for Orlicz norm bounded set is lower bounded by $C\sigma\epsilon\psi^{-1}(1/\epsilon)$ when $\epsilon\leq \frac{1}{2}$.

Second, we have the finite-sample result for any distribution in the generalized resilience set.
\begin{lemma}
Denote $\tilde \epsilon = 2\epsilon + 2C^{\mathsf{vc}}\sqrt{\frac{d+1+\log(1/\delta)}{n}}$, where $C^{\mathsf{vc}}$ is from Lemma~\ref{lemma.vcinequality}. Assume $p^* \in \GG^\TV_{\mathsf{mean}}(\rho(\tilde \epsilon), \tilde \epsilon)$. For $\mathcal{H} = \{v^\top X \mid v\in \bR^d\}$, let $q $ denote the output of the projection algorithm $ \Pi(\hat{p}_n; \tTV_{\mathcal{H}}, \GG^\TV_{\mathsf{mean}}(\rho(\tilde \epsilon), \tilde \epsilon))$. Then, with probability at least $1-\delta$, 
\begin{align} 
    \|\bE_{p^*}[X] - \bE_q[X]\|\leq 2\rho(\tilde \epsilon) = 2\rho\Bigg(2\epsilon + 2C^{\mathsf{vc}}\sqrt{\frac{d+1+\log(1/\delta)}{n}}\Bigg).
\end{align} 
\end{lemma}

\begin{proof}
We provide the first half finite-sample results here and defer the population results to By Proposition~\ref{prop.tTV_general} it suffices to bound $\tTV_\mathcal{H}(p, \hat p_n)$ and to bound the modulus of continuity for $\GG^\TV$ under $\tTV_\mathcal{H}$. 

Since the VC dimension of hyper-planes in $\bR^d$ is $d+1$, it follows from Lemma~\ref{lemma.vcinequality} that $\tTV_\mathcal{H} \leq C^{\mathsf{vc}}\sqrt{\frac{d+1+\log(1/\delta)}{n}}$ with probability at least $1-\delta$. Now we upper bound the modulus, which equals
\begin{align}
    \sup_{p_1, p_2\in \GG^\TV_\mathsf{mean}(\rho(\tilde \epsilon),\tilde \epsilon): \tTV_\mathcal{H}(p_1, p_2) \leq \tilde \epsilon}  \|\bE_{p_1}[X] - \bE_{p_2}[X]\|.
\end{align}
The condition that $\tTV_\mathcal{H}(p_1, p_2) \leq \tilde \epsilon$ implies that for any $v\in\bR^d$,  $\|v\|_* = 1$, where $\|\cdot \|_*$ is the dual norm of $ \|\cdot \|$,
\begin{align}
    \sup_{t\in \bR}|\bP_{p_1}[v^\top X\geq t] - \bP_{p_2}[v^\top X \geq t]| \leq \tilde \epsilon.
\end{align}

It follows from Lemma~\ref{lem.tvt_midpoint} and  $p_1, p_2\in \GG^\TV_\mathsf{mean}(\rho(\tilde \epsilon), \tilde \epsilon)$ that there exist some distribution $r$ with $r(A) \leq \frac{p(A)}{1-\epsilon}$ and $r(A) \leq \frac{q(A)}{1-\epsilon}$ for any event $A\in\mathcal{A} = \{v^\top X\geq t, v^\top X\leq t \mid t\in\bR\}$, and   
\begin{align}
   | \bE_{p_1}[v^\top X]  - \bE_{r}[v^\top X]| \leq \rho(\tilde \epsilon), 
   | \bE_{r}[v^\top X]  - \bE_{p_2}[v^\top X]| \leq \rho(\tilde \epsilon).
\end{align}
This yields $\bE_{p_1}[v^{\top}X] - \bE_{p_2}[v^{\top}X] \leq 2\rho(\tilde \epsilon)$, or 
$v^{\top}(\mu_{p_1} - \mu_{p_2}) \leq 2\rho(\tilde \epsilon)$. Taking the maximum over $\|v\|_* = 1$ yields 
$\|\mu_{p_1} - \mu_{p_2}\| \leq 2\rho(\tilde \epsilon)$, where $\|\cdot \|_*$ is the dual norm of $ \|\cdot \|$.
which shows the modulus is small. The final conclusion follows from Proposition~\ref{prop.tTV_general}. 
\end{proof}

\subsection{Proof of Theorem~\ref{thm.linearregressiontvtildeproof}}\label{appendix.proof_linreg_tvt}

Similar to the case of mean estimation, the proof can be decomposed into the two lemmas below:
\begin{lemma}
 Assume the second moments of $X$ and $Z$ exist 
and satisfy the following conditions:
    \begin{align}
        \bE_{p^*}\bigg[\psi \bigg(\frac{(v^{\top}X)^2}{\sigma_1^2 \bE_{p^*}[(v^{\top}X)^2]}\bigg)\bigg] &\leq 1 \text{ for all } v \in \bR^d, \text{ and} \\
          \bE_{p^*}\left[\psi \left(Z^2/\sigma_2^2\right)\right] &\leq 1.
    \end{align}
Then $p^*\in  \GG^\TV( \rho, 8 \rho, \eta)$ for $\rho = 2(\frac{\sigma_1\sigma_2\eta \psi^{-1}(1/\eta)}{1-\eta})^2$, 
for all $ \eta$ satisfying $\sigma_1^2\eta\psi^{-1}(\frac{1}{\eta}) < \frac{1}{2}$. Here the bridge function and cost function  in Definition~\ref{def.G_TV} are
$B(p, \theta) = L(p, \theta) = \bE_p[(Y-X^{\top}\theta)^2 - (Y-X^{\top}\theta^*(p))^2]$.
The  population limit for the set satisfying the two conditions is $ \Theta( ({\sigma_1\sigma_2\epsilon \psi^{-1}(1/\epsilon)})^2)$ when the perturbation level $\epsilon$ satisfies  $\epsilon < 1/2$ and $ 2\sigma_1^2\epsilon\psi^{-1}(\frac{1}{2\epsilon}) < 1/2$.
\end{lemma}
\begin{lemma}\label{lem.appendix_finite_linreg_TV}
Denote $\tilde \epsilon =  2\epsilon + 2 C^{\mathsf{vc}}\sqrt{\frac{10d +\log(1/\delta)}{n} }$.
Assume $p^* \in \GG^\TV_\downarrow(\rho_1(\tilde \epsilon), \tilde \epsilon)\cap \GG^\TV_\uparrow(2\rho_1(\tilde \epsilon), \rho_2(\tilde \epsilon), \tilde \epsilon)$. For $\mathcal{H}$ designed in (\ref{eqn.H_linreg}), let $q$ denote the output of the projection algorithm $\Pi(\hat{p}_n; \tTV_{\mathcal{H}}, \GG^\TV)$. Then, with probability at least $1-\delta$, 
\begin{align*}
    \bE_{p^*}[(Y-X^\top\theta^*(q))^2 - (Y-X^\top\theta^*(p^*))^2]\leq \rho_2(\tilde \epsilon) = \rho_2\Bigg( 2\epsilon + 2 C^{\mathsf{vc}}\sqrt{\frac{10d +\log(1/\delta)}{n} }\Bigg).
\end{align*} 
\end{lemma}

For linear regression, we need to slightly shrink the resilient set from $
    \GG^{\TV}_\downarrow(\rho_1(\tilde \epsilon), \tilde \epsilon)\cap \GG^{\TV}_\uparrow(\rho_1(\tilde \epsilon), \rho_2(\tilde \epsilon), \tilde \epsilon)$
to $
    \GG^{\TV}_\downarrow(\rho_1(\tilde \epsilon), \tilde \epsilon)\cap \GG^{\TV}_\uparrow(2\rho_1(\tilde \epsilon), \rho_2(\tilde \epsilon), \tilde \epsilon).$ 
The conditions in the first lemma still imply that $p^*$ is inside this 
smaller set with appropriate parameters.
The shrinkage comes from the following consideration: 
when using the mean cross lemma, we can only cross the mean of the same function $f(X)$ for $p $ and $q$ while the excess predictive loss in the original $\GG^\TV$ requires the mean cross for two different functions.

Now we begin with proving the first lemma:
\paragraph*{Upper bound}
We first show the upper bound. 
Denote $Z = Y -X^{\top}\theta^*(p) $.   Since the second moment of $X$ and $Z$ exist, we can denote $\mathbb{E}_r[XZ] = \mu_r$, $\mathbb{E}_r[XX^{\top}] = M_r$, $\mathbb{E}_{p^*}[XZ] = \mu_{p^*}$, $\mathbb{E}_{p^*}[XX^{\top}] = M_{p^*}$.
The optimal $\theta$ in both cases can be written as a closed form solution\footnote{If either $M_r$ or $M_{p^*}$ is not invertible, we use $M^{-1} = M^{\dagger}$ as its pseudoinverse.}:
\begin{align}
    \theta^*(r) & = \bE_r[XX^{\top}]^{-1}\bE_r[XY] \nonumber \\
    & = \bE_r[XX^{\top}]^{-1}\bE_r[XX^{\top}\theta^*({p^*})+XZ] \nonumber \\
    & = \theta^*({p^*}) + \bE_r[XX^{\top}]^{-1}\bE_r[XZ] \nonumber \\
    & = \theta^*({p^*}) + M_r^{-1}\mu_r.
\end{align}

Then $p^* \in \GG_{\downarrow}$ is equivalent to that $\forall  r \leq \frac{p^*}{1-\eta}$, we have
\begin{align}
  \rho_1 & \geq \mathbb{E}_{r}[\ell(\theta^*({p^*}), X) - \ell(\theta^*(r), X)] \nonumber \\
  & = \mathbb{E}_{r}[(Y - X^{\top}\theta^*({p^*}))^2 - (Y - X^{\top}\theta^*(r))^2] \nonumber \\
  & = \mathbb{E}_{r}[2YX^{\top}(\theta^*(r) - \theta^*({p^*})) - \theta^*(r)^{\top}XX^{\top}\theta^*(r) + \theta^*({p^*})^{\top}XX^{\top}\theta^*({p^*})] \nonumber \\
  & = \mathbb{E}_{r}[2(X^{\top}\theta^*({p^*})+Z)X^{\top}(\theta^*(r) - \theta^*({p^*})) - \theta^*(r)^{\top}XX^{\top}\theta^*(r) + \theta^*({p^*})^{\top}XX^{\top}\theta^*({p^*})] \nonumber \\
  & = \mathbb{E}_{r}[2(X^{\top}\theta^*({p^*})+Z)X^{\top}M_r^{-1}\mu_r - (\theta^*({p^*}) + M_r^{-1}\mu_r)^{\top}XX^{\top}(\theta^*({p^*}) + M_r^{-1}\mu_r) \nonumber \\
  &\quad  + \theta^*({p^*})^{\top}XX^{\top}\theta^*({p^*})] \nonumber \\
  & =  2 \theta^*({p^*})^{\top} \mu_r + 2\mu_r^{\top}M_r^{-1}\mu_r - 2 \theta^*({p^*})^{\top} \mu_r -  \mu_r^{\top}M_r^{-1}\mu_r -  \theta^*({p^*})^{\top}M_r\theta^*({p^*})+  \theta^*({p^*})^{\top}M_r\theta^*({p^*}) \nonumber \\
  & = \mu_r^{\top}M_r^{-1}\mu_r \nonumber \\
  & = \|M_r^{-1/2}\mu_r \|_2^2
\end{align}

By similar calculation, we can see that $p^* \in \GG_{\uparrow}^{\TV}$ is equivalent to  
\begin{align}
     \forall \theta, \text{if } \forall r \leq \frac{{p^*}}{1-\eta}, 
    & \|M_r^{1/2}(\theta - \theta^*({p^*})) - M_r^{-1/2}\mu_r) \|_2^2 \leq \rho_1 \nonumber \\
    \Rightarrow & \|M_{p^*}^{1/2}(\theta - \theta^*({p^*})) \|_2^2 \leq \rho_2.
\end{align}

We first check that $p^* \in  \GG_{\downarrow}$. Note that $M_r^{-1/2}\mu_r = M_r^{-1/2}M_{p^*}^{1/2}M_{p^*}^{-1/2}\mu_r$. We bound the term $M_r^{-1/2}M_{p^*}^{1/2} $ and $M_{p^*}^{-1/2}\mu_r $ separately.

From the first condition, we know that the $\psi$ norm of $\frac{(v^{\top}X)^2}{\bE_{p^*}[(v^{\top}X)^2]}$ is upper bounded by $\sigma_1^2$, thus by Lemma~\ref{lem.centering_psi} we know the centered $\psi$ norm is bounded by $2\sigma_1^2$. By
Lemma~\ref{lem.noncentered_orlicz}, we have for any $v\in\bR^d$,
\begin{align}\label{eqn.linreg_TV_proof_used1}
\forall r\leq \frac{p^*}{1-\eta},
     \bE_r[(v^{\top}X)^2]  \geq (1-{\sigma_1^2\eta\psi^{-1}(\frac{1}{\eta})) \bE_{p^*}[(v^{\top}X)^2]}. 
\end{align}
Thus when ${\sigma_1^2\eta\psi^{-1}(\frac{1}{\eta}) } < \frac{1}{2}$, we have for any $v\in\bR^d$,
\begin{align}
    & v^{\top}M_rv \geq (1 - \sigma_1^2\eta\psi^{-1}(\frac{1}{\eta}) ) v^{\top}M_{p^*}v \nonumber \\
\Rightarrow & M_r \succeq  \frac{1}{2}  M_{p^*} \nonumber \\
\Rightarrow & M_r^{-1} \preceq 2  M_{p^*}^{-1} \label{eqn.matrix_operator_monotone} \\
\Rightarrow & M_{p^*}^{1/2} M_r^{-1}M_{p^*}^{1/2} \preceq 2 I \label{eqn.matrix_multiply_lem} \\
\Rightarrow & \|M_{p^*}^{1/2}M_r^{-1/2} \|_2 \leq \sqrt{2}.
\end{align}

Equation (\ref{eqn.matrix_operator_monotone}) comes from the monotone property of matrix operator. Equation (\ref{eqn.matrix_multiply_lem}) comes from the fact that $A \preceq B$ would lead to $C^{\top}AC \preceq C^{\top}BC$. 
From the two conditions in $\GG(\psi)$
and Lemma~\ref{lem.multi_orlicz}, we have
\begin{align}\label{eqn.linreg_eq1}
    \forall v\in \bR^d, \|v^{\top}XZ \|_\psi \leq \sigma_1 \sigma_2 \bE_{p^*}[(v^{\top}X)^2]^{1/2}.
\end{align}
Taking $v = v' M_{p^*}^{-1/2}$, where $\|v'\|_2 = 1$,  we can see that this is equivalent to 
\begin{align}\label{eqn.linreg_eq2}
    \forall v\in\bR^d, \|v\|_2 =1, \| v^{\top}M^{-1/2}_{p^*}XZ\|_\psi \leq \sigma_1 \sigma_2.
\end{align}
Note that $\bE_{p^*}[v^{\top}M^{-1/2}_{p^*}XZ ]=0$. By Lemma~\ref{lem.cvx_mean_resilience}, this gives us that for any $\eta$,
\begin{align}\label{eqn.linreg_TV_proof_used2}
      \forall v\in\bR^d, \|v\|_2 =1, \forall r\leq \frac{p^*}{1-\eta}, \bE_r[ v^{\top}M^{-1/2}_{p^*}XZ] \leq \frac{\sigma_1\sigma_2\eta \psi^{-1}(\frac{1}{\eta})}{1-\eta}.
\end{align}
Thus we have $\| M_{p^*}^{-1/2}\mu_r\|_2\leq \frac{\sigma_1\sigma_2\eta \psi^{-1}(\frac{1}{\eta})}{1-\eta}$.

From the above results, we have
\begin{align}
    \| M_r^{-1/2}\mu_r\|_2 & =  \| M_r^{-1/2}M_{p^*}^{1/2} M_{p^*}^{-1/2}\mu_r\|_2  \nonumber \\
    & \leq \| M_r^{-1/2}M_{p^*}^{1/2}\|_2 \| M_{p^*}^{-1/2}\mu_r\|_2 \nonumber \\
    & \leq \frac{\sqrt{2}\sigma_1\sigma_2\eta \psi^{-1}(\frac{1}{\eta})}{1-\eta}. 
\end{align}
Thus we can conclude that $p^* \in \GG_{\downarrow}(\rho, \eta)$ with $\rho = 2(\frac{\sigma_1\sigma_2\eta \psi^{-1}(\frac{1}{\eta})}{1-\eta})^2$ when $\sigma_1^2\eta\psi^{-1}(\frac{1}{\eta})  < \frac{1}{2}$.
Then we check that $p^*\in \GG_{\uparrow}^{\TV}$. 
If $\|M_r^{1/2}(\theta - \theta^*({p^*})) - M_r^{-1/2}\mu_r) \|_2^2 \leq \rho $ holds,
 we have
\begin{align}
    \|M_{p^*}^{1/2}(\theta^*({p^*}) - \theta) \|_2 & = \|M_{p^*}^{1/2}M_r^{-1/2}M_r^{1/2}(\theta^*({p^*}) - \theta)  \|_2 \nonumber \\
    & \leq \|M_{p^*}^{1/2}M_r^{-1/2}\|_2 \|M_r^{1/2}(\theta^*({p^*}) - \theta)   \|_2  \nonumber \\
    & \leq \|M_{p^*}^{1/2}M_r^{-1/2}\|_2 \left(\|M_r^{1/2}(\theta^*({p^*}) - \theta) - M_r^{-1/2}\mu_r \|_2 + \| M_r^{-1/2}\mu_r\|_2 \right) \nonumber \\
    & \leq 2\sqrt{2}\sqrt{\rho}.
\end{align}
This gives that $p^*\in \GG_{\uparrow}^{\TV}(\rho, 8\rho, \eta)$.

\begin{remark}
As the proof shows one may weaken the assumptions to 
   \begin{align}
       \forall v\in\bR^d, \|v\|_2=1, r \leq \frac{p^*}{1-\eta}, \bE_r[(v^{\top}X)^2] &\geq f(\eta) \cdot \bE_{p^*}[(v^{\top}X)^2], \text{ and} \\
       \sup_{v\in\bR^d, \|v\|_2=1}   \bE_{p^*}\left[\psi \left(\bE_{p^*}[XX^\top]^{-1/2}XZ/\sigma\right)\right] &\leq 1.
    \end{align}
Here $f(\eta) \in (0, 1), \forall \eta\in [0, 1)$. The excess predictive loss is $O(\frac{\sigma \eta \psi^{-1}(1/\eta)}{f(\eta)(1-\eta)})$ under this set of assumptions. One can verify that Gaussian distribution satisfies the first condition.

Note that when $\psi(x) = x^{k/2}$ and $X, Z$ are independent, we can bound the term 
\begin{align}\sup_{ v\in\bR^d, \|v\|_2 =1}  \bE_{p^*}[|v^{\top}M^{-1/2}_{p^*}XZ|^{k}]\leq \sup_{ v\in\bR^d, \|v\|_2 =1} \bE_{p^*}[|v^{\top}M^{-1/2}_{p^*}X|^{k}] \bE_{p^*}[|Z|^{k}] \leq \sigma_1^{k}\sigma_2^{k}.
\end{align}
Thus the final bound would be of the order  $O(\sigma_1^2\sigma_2^2\eta^{2-2/k})$ when $\eta$ is small enough. 
However, without independence assumption we can only bound the $k/2$-th moment of $v^{\top}M^{-1/2}_{p^*}XZ $, thus the result becomes $O(\sigma_1^2\sigma_2^2\eta^{2-4/k})$.
\end{remark}

\subsubsection*{Lower bound}
Then we show the lower bound for the population limit.  Consider the set $\tilde \GG_{\mathsf{LinReg}}$ that is smaller than $\GG(\psi)$:
\begin{align}
   \tilde \GG_{\mathsf{LinReg}}= \{ p \mid \bE_p[X^2] = 1, \bE_p[\psi(X^2)] \leq C, \bE_p[\psi((Y-\theta^*(p)X)^2)] \leq C\}.
\end{align}
Here $C$ is some universal constant that may depend on $\psi$. Then it suffices to show the population limit of the set $\tilde \GG_{\mathsf{LinReg}}$ is lower bounded by $(\epsilon \psi^{-1}(1/\epsilon))^2$, i.e.
we need to show that for any estimator $\theta(p)$,
\begin{align}
    \inf_{\theta(p)}\sup_{(p^*, p): p^*\in \tilde \GG_{\mathsf{LinReg}}, \TV(p^*, p)\leq \epsilon} \bE_{p^*}[(Y-X^{\top}\theta(p))^2 - (Y-X^{\top}\theta^*(p^*))^2] \geq C(\epsilon \psi^{-1}(\frac{1}{\epsilon}))^2.
\end{align} 

Consider the case of one-dimensional
distribution $X$. If  $\bE_{p^*}[X^2] = \bE_{q}[X^2]  = 1$, 
the cost $L(p^*,\theta)$ can be written as
\begin{align}\label{eqn.linreg_lower_bound_pseudo_metric}
     L(p^*, \theta) = \bE_{p^*}[(Y-X^{\top}\theta)^2 - (Y-X^{\top}\theta^*(p^*))^2] & =  \bE_{p^*}[(X\theta^*(p^*) + Z - X\theta)^2 - Z^2] \nonumber \\
     & = \bE_{p^*}[X^2](\theta - \theta^*(p^*))^2\nonumber \\
     & = (\theta - \theta^*(p^*))^2 \nonumber \\
     & = (\bE_{p^*}[XY] - \theta)^2.
\end{align}
Here we use the fact that  $\theta^*(p^*) = \bE_{p^*}[XX^T]^{-1} \bE_{p^*}[XY] = \bE_{p^*}[XY]$.

Now we construct distribution $p_1, p_2 \in \tilde \GG_{\mathsf{LinReg}}$ with the same marginal distribution on $X$:
\begin{align}
    \bP_{p_1}(X=t) = \bP_{p_2}(X=t) &= \left\{ \begin{array}{cl}
     \frac{1-\epsilon}{2},    & t  = 0  \\
     \frac{1-\epsilon}{2},    & t  = \sqrt{\frac{2(1-\epsilon \psi^{-1}(1/\epsilon))}{1-\epsilon}}  \\
    \epsilon,   & t   = \sqrt{\psi^{-1}(1/\epsilon)}\\
    0 ,    & \text{otherwise}
    \end{array}\right. \\
\end{align}
We have $\bE_{p_1}[X^2]=\bE_{p_2}[X^2]=1$, and
\begin{align}
   \bE_{p_1}[\psi(X^2)] = \frac{1-\epsilon}{2} \cdot \psi({\frac{2(1-\epsilon \psi^{-1}(1/\epsilon))}{1-\epsilon}})  + \epsilon \psi(\psi^{-1}(1/\epsilon)) \leq \frac{1}{2} \cdot \psi(2) + 1 \leq C,
\end{align}
where $C$ is some constant. 
Now we construct the conditional distribution $Y | X$ under $p_1$ as follows.
\begin{align}
    Y |_{X = t} = \begin{cases}  0, & t\neq \sqrt{\psi^{-1}(1/\epsilon)} \\ X, & t= \sqrt{\psi^{-1}(1/\epsilon)}  \end{cases}.
\end{align}
The conditional distribution $Y \mid X$ under $p_2$ is
\begin{align}
    Y |_{X = t} = \begin{cases}  0, & t\neq \sqrt{\psi^{-1}(1/\epsilon)} \\ -X, & t= \sqrt{\psi^{-1}(1/\epsilon)}  \end{cases}.
\end{align}
Then $\theta^*(p_1) = \bE_{p_1}[XY] \in [0, 1]$, $\theta^*(p_2) = \bE_{p_2}[XY] \in [-1, 0]$. For $Z = Y - \theta^*X$, in both cases we have $\bE[\psi(Z^2)] \leq \bE[\psi(X^2)] \leq C$. Furthermore, we have
\begin{align}
    |\theta^*(p_1) - \theta^*(p_2)| = |\bE_{p_1}[XY] - \bE_{p_2}[XY]| = 2\epsilon \psi^{-1}(1/\epsilon).
\end{align}
Then the population limit of the set $\tilde \GG_{\mathsf{LinReg}}$ is lower bounded by $\epsilon \psi^{-1}(1/\epsilon)$ once we assume the observed corrupted distribution $p = p_1$:
\begin{align}
    \inf_{\theta(p)} \sup_{(p^*,p): \TV(p^*, p)\leq \epsilon, p^*\in   \tilde \GG_{\mathsf{LinReg}}} L(p^*, \theta(p)) & \geq \inf_{\theta} \sup_{p^*: \TV(p^*, p_1)\leq \epsilon, p^*\in \tilde \GG_{\mathsf{LinReg}}} L(p^*, \theta) \nonumber \\
     & = \inf_{\theta} \sup_{p^*: \TV(p^*, p_1)\leq \epsilon, p^*\in \tilde \GG_{\mathsf{LinReg}}} (\theta^*(p^*)- \theta)^2 \nonumber \\
     & \geq \frac{1}{2}\inf_{\theta} ( (\theta^*(p_1)- \theta)^2 +   (\theta^*(p_2)- \theta)^2)\nonumber \\
     & \geq \frac{1}{4} (\theta^*(p_1)- \theta^*(p_2))^2 \nonumber \\
     & \geq  (\epsilon\psi^{-1}(1/\epsilon))^2.
\end{align}

From Lemma~\ref{lem.minimax_random}, we know that for randomized decision rule $\theta_r(p)$, 
\begin{align}
     & \inf_{\theta_r(p)} \sup_{(p^*,p): \TV(p^*, p)\leq \epsilon, p^*\in \tilde \GG_{\mathsf{LinReg}}} \bP(L(p^*, \theta_r(p)) 
     \geq  (\epsilon\psi^{-1}(1/\epsilon))^2) \nonumber  \\  \geq & \frac{1}{2} \inf_{q_r(p)} (\bP(L(p_1, \theta_r(p)) \geq (\epsilon\psi^{-1}(1/\epsilon))^2) + \bP(L(p_2,  \theta_r(p)) \geq (\epsilon\psi^{-1}(1/\epsilon))^2)) \nonumber \\ 
     \geq & \frac{1}{2}.
\end{align}

Now we prove Lemma~\ref{lem.appendix_finite_linreg_TV}.

\begin{proof}
By Proposition~\ref{prop.tTV_general} it suffices to bound $\tTV_\mathcal{H}(p, \hat p_n)$ and to bound the modulus of continuity for $\GG$ under $\tTV_\mathcal{H}$.

Denote $\tilde X = (X, Y)$, then $\tTV_\mathcal{H}(p, q)$ can be upper bounded in the following way:
\begin{align}
\tTV_\mathcal{H}(p, q) 
& \leq \sup_{ v_1, v_2 \in \bR^{d+1}, t\in\bR} \left| \bP_{p}[(v_1^{\top}\widetilde X)^2 - (v_2^{\top}\widetilde X)^2 \geq t] - \bP_{q}[(v_1^{\top}\widetilde X)^2 - (v_2^{\top}\widetilde X)^2 \geq t] \right|,
\end{align}
From~\citep[Theorem 8.3]{anthony2009neural} we know that the VC dimension
of the collection of sets $\{\{x\in\bR^{d+1} \mid (v_1^\top x)^2 - (v_2^\top x)^2\geq t\} \mid v_1, v_2 \in \bR^{d+1}, t\in \bR\}$ is at most $10d$. Thus  from Lemma~\ref{lemma.vcinequality}  we have $\tTV_\mathcal{H}(p, q)\leq C^{\mathsf{vc}}\sqrt{\frac{10d +\log(1/\delta)}{n} }$.

We now show the modulus of continuity is upper bounded by $\rho_2(\tilde \epsilon)$. We still apply mean cross lemma on the function $f = \ell(\theta^*(p_2), X) - \ell(\theta^*(p_1), X)$.
Define $\ell(\theta, X) = (\theta^{\top}X-Y)^2$, then the bridge function is $B(p, \theta) = \bE_{p}[\ell(\theta, X) - \ell(\theta^*(p), X)]$.

From Lemma~\ref{lem.tvt_cross_mean}, we know that for any $f \in \mathcal{H}$,  there exists $r_{p_1}\leq \frac{p_1}{1-\tilde \epsilon}, r_{p_2}\leq \frac{p_2}{1-\tilde \epsilon}$ such that the mean under $f$ of $r_{p_1}$ and $r_{p_2}$ can cross.
Taking $f = \ell(\theta^*(p_2), X) - \ell(\theta^*(p_1), X)$, we have
\begin{align}
    \bE_{r_{p_1}}[\ell(\theta^*(p_2), X) - \ell(\theta^*(p_1), X)] & \leq  \bE_{r_{p_2}}[\ell(\theta^*(p_2), X) - \ell(\theta^*(p_1), X)]  \nonumber \\
    & \leq \bE_{r_{p_2}}[\ell(\theta^*(p_2), X) - \ell(\theta^*(r_{p_2}), X)] = B(r_{p_2}, \theta^*(p_2)) \leq \rho_1(\tilde \epsilon).
\end{align}
The last inequality comes from the fact that $p_2\in\GG_{\downarrow}^{\TV}(\rho_1(\tilde \epsilon), \tilde \epsilon)$.  
Combining the above equation with the fact that  $p_1\in\GG_{\downarrow}^{\TV}(\rho_1(\tilde \epsilon), \tilde \epsilon)$, we know
\begin{align}
    \bE_{r_{p_1}}[\ell(\theta^*(p_2), X) - \ell(\theta^*(r_{p_1}), X)] & =  \bE_{r_{p_1}}[\ell(\theta^*(p_2), X)- \ell(\theta^*(p_1), X) + \ell(\theta^*(p_1), X) - \ell(\theta^*(r_{p_1}), X)] \nonumber \\ 
    & \leq 2\rho_1(\tilde \epsilon).
\end{align}
From $p \in \GG_{\uparrow}^{\TV}(2\rho_1(\tilde \epsilon), \rho_2(\tilde \epsilon), \tilde \epsilon) $, this implies that  $
    L(p_1, \theta^*(p_2)) \leq \rho_2(\tilde \epsilon)$,  
which implies the final conclusion once we take 
$  B(p,\theta)= L(p,\theta) = \bE_p[(Y - X^\top \theta)^2 - (Y - X^\top \theta^*(p))^2]$. This proof actually works for any $\ell$ and $B(p, \theta) = \bE_{p}[\ell(\theta, X) - \ell(\theta^*(p), X)]$ as excess predictive loss.
\end{proof}

\subsection{Another set of sufficient conditions for linear regression under $\TV$ perturbation}\label{appendix.another_proof_TV_linreg}

We show here that the hyper-contractivity condition in $\GG(\psi)$ can be dropped if we assume that the radius of $\theta$ is bounded by $R$, i.e. $\theta \in \Theta = \{ \theta \mid \|\theta\|_2 \leq R \}$. 

\begin{example}[Linear Regression with bounded parameter assumption]\label{example.TV_linreg_boundedR}
Let $(X, Y)\sim p^*$ and take
$B(p, \theta) = L(p, \theta) = \bE_p[(Y-X^{\top}\theta)^2 - (Y-X^{\top}\theta^*(p))^2]$.
Let $Z = Y-X^{\top}\theta^*(p)$ denote the residual error. Assume that $\theta \in \Theta = \{ \theta \mid \|\theta\|_2 \leq R \}$, the second moments of $X$, $Z$ exist
and satisfy the following conditions:
    \begin{align}
    \bE_{p^*}\bigg[\psi \bigg(\frac{(v^{\top}X)^2}{\sigma_1^2}\bigg)\bigg] &\leq 1, \bE_{p^*}[(v^\top X)^2 ] \geq \sigma_2^2 \text{ for all } v \in \bR^d, \text{ and} \\
    \bE_{p^*}\left[\psi \left(\frac{Z^2}{\sigma_3^2}\right)\right] &\leq 1.
    \end{align}
Then $p^* \in \GG(\frac{\rho^2}{4\sigma_2^2}, \frac{3\rho^2}{2\sigma_2^2}, \eta)$ for $\rho = (4\sigma_1^2 R+2\sigma_1\sigma_3)\eta\psi^{-1}(1/\eta)$, 
and any $ \eta<1/2$. 
The  population limit for the set satisfying the three conditions is $ O_{\sigma_1, \sigma_2, \sigma_3, R}( ({\epsilon \psi^{-1}(1/\epsilon)})^2)$ when the perturbation level $\epsilon$ is less than $1/4$. 
 
\end{example}

\begin{proof}
We show that under these assumptions, the gradient of $L(p, \theta)$ can be robustly estimated. Thus by applying Lemma~\ref{lem.robust_gradient_estimation} we can show the final results. 

From the two assumptions
and Lemma~\ref{lem.multi_orlicz}, we have
\begin{align}
    \forall v\in \bR^d, \|v^{\top}XZ \|_\psi \leq \sigma_1 \sigma_2.
\end{align}

Lemma~\ref{lem.robust_gradient_estimation} requires the gradient to be inside resilient set for all $\theta$, i.e. $\sup_{\theta\in \Theta, r\leq \frac{p^*}{1-\eta}} \|\bE_r[XX^\top(\theta^*(p)-\theta)+ XZ] -  \bE_{p^*}[XX^\top(\theta^*(p)-\theta)+ XZ] \|_2\leq \rho$. From the two conditions in assumption and that the radius of $\Theta$ is upper bounded by $R$, we can derive $\rho = (4\sigma_1^2 R+2\sigma_1\sigma_3)\eta\psi^{-1}(1/\eta)$ if $\eta<1/2$. Furthermore, we know that $\bE_{p^*}[(v^\top X)^2 ] \geq \sigma_2^2$ for any $v$. Thus $\bE_{p^*}[\ell(X, \theta)]$ is $2\sigma_2^2$-strongly convex. 
Then it follows from the second statement of Lemma~\ref{lem.robust_gradient_estimation} that $p^* \in \GG(\frac{\rho^2}{4\sigma_2^2}, \frac{3\rho^2}{2\sigma_2^2}, \eta)$.

\end{proof}

\subsection{Necessity of hyper-contractive condition for linear regression}\label{proof.delete_dimension_linreg}

\paragraph*{Lower bounds for linear regression }
One might wonder whether a simpler condition such as sub-Gaussianity of $X$ and $Z$ 
would also guarantee a finite population limit. Even if $Z \equiv 0$, sub-Gaussianity of $X$ is \emph{not} sufficient. 
Here we exhibit a univariate sub-Gaussian $X$ for which an adversary can perturb 
$X$ to be zero almost surely, thereby destroying all information between $X$ and $Y$. We also illustrate the construction in Figure~\ref{fig:tv-delete—dimension}. When $p^*$ has most of its mass concentrating on a degenerate subspace and at most $\epsilon$ mass outside, the adversary is able to completely delete all the information outside the subspace. Thus inferring $\theta$ outside the subspace is impossible. 

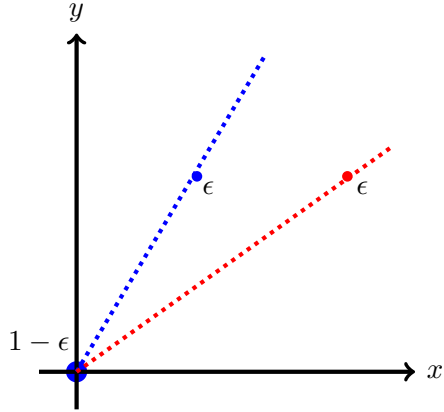
\begin{figure}[ht!]
\begin{center}
\begin{tikzpicture}[
  implies/.style={double,double equal sign distance,-implies},
]
 \def\x{0.8};
  \def\y{0.5};

 \draw (1.8, 0.45) node {$\epsilon$};
  \draw (-2.5,-1.6) node {$1-\epsilon$};
  \fill[red] (1.6, 0.6) circle[radius=2pt];
 \fill[blue] (-2, -2) circle[radius=4pt];
 
   \fill[blue] (-0.4, 0.6) circle[radius=2pt];
 \draw (-0.25, 0.45) node {$\epsilon$};

\draw[->,ultra thick] (-2.5,-2)--(2.5,-2) node[right]{$x$};
\draw[->,ultra thick] (-2,-2.5)--(-2,2.5) node[above]{$y$};
\draw (-2, -2) edge[ultra thick, blue, dotted] (0.5,2.2);
\draw (-2,-2) edge[ultra thick, red, dotted]  (2.2, 1.0);
\end{tikzpicture}
\end{center}
\caption{Dimension deletion phenomenon. Here $X$ follows some sub-Gaussian distribution with $1-\epsilon$ mass on $0$ and $\epsilon$ mass outside (on the blue point). By deleting the blue point and add the red point at arbitrary position, the adversary is able to completely delete the information for identifying $\theta$. Thus recovery of $\theta$ becomes impossible. }
\label{fig:tv-delete—dimension}
\end{figure}

\begin{theorem}\label{thm.delete_dimension_linreg}
Let 
\begin{align}
    \GG_{\mathsf{LinReg}}' = \{ p \mid \sup_{v\in\bR^d, \|v\|_2 =1} \bE_p[\exp((v^\top(X-\bE_p[X]))^2/\sigma^2)] \leq 2, Y = X^{\top}\theta, \theta \in \bR^d\}
\end{align} be the family of distributions with sub-Gaussian $X$ and no noise in $Y$. 
For any estimator $\theta(p)$, there is a pair of distributions $(p, p^*)$ such that $ p^*\in \GG_{\mathsf{LinReg}}', \TV(p, p^*) \leq \epsilon$, and
\begin{align}
     \bE_{p^*}[(Y-X^{\top}\theta(p))^2 - (Y-X^{\top}\theta^*(p^*))^2] = \infty.
\end{align}    
In other words, any estimator achieves arbitrarily large error in the worst case.
\end{theorem}

\begin{proof}

We show a stronger result than Theorem~\ref{thm.delete_dimension_linreg} here: for any $f: \bR \mapsto \bR$ that is convex and non-negative, $f(0) = 0$, $f(x)\to \infty$ as $|x|\to \infty$ and any fixed estimator $\theta(p)$, there is a pair of distributions $(p, p^*)$ such that $ p^*\in \GG_{\mathsf{LinReg}}', \TV(p, p^*) \leq \epsilon$, and 
\begin{align}
     \bE_{p^*}[f(Y-X^{\top}\theta(p)) - f(Y-X^{\top}\theta^*(p^*))] = \infty.
\end{align}    

We consider the case that both $X$ and $Y$ are scalar random variables. We first construct the marginal distributions for two distributions $p^*_1, p^*_2$ as follows
\begin{align}
    \bP_{p^*_1}[X = x] & = \left\{ \begin{array}{cl}
     \epsilon,    & x  = b(\sigma, \epsilon)   \\
    1-\epsilon,   & x   = 0\\
    0 ,    & \text{otherwise}
    \end{array}\right. \\
    \bP_{p^*_2}[X = x] & = \left\{ \begin{array}{cl}
     \epsilon,    & x  = -b(\sigma, \epsilon)   \\
    1-\epsilon,   & x   = 0\\
    0 ,    & \text{otherwise}
    \end{array}\right. 
\end{align}
Here $b(\sigma, \epsilon)$ is the largest value such that $\|X\|_{\psi_2} \leq \sigma$, where $\psi_2$ is the Orlicz function for sub-Gaussian distributions.  We design the joint distribution between $X, Y$ for $p_1^*$ as $Y = \theta^{(1)}X$, where $\theta^{(1)} = \frac{t}{ b(\sigma, \epsilon)}$, and the joint distribution between $X, Y$ for $p_2^*$ as $Y = \theta^{(2)}X$, where $\theta^{(2)} = -\frac{t}{ b(\sigma, \epsilon)}$. Here $t>0$ is an arbitrary number that later will be taken to approach $\infty$. 

Now we define the observed distribution $\tilde p$. Define
\begin{align}
   \bP_{\tilde p}[X = t] = \left\{ \begin{array}{cl}
    1,   & t   = 0\\
    0 ,    & \text{otherwise}
    \end{array}\right. 
\end{align}
The distribution of $Y$ is also $0$ with probability $1$. One can see that $\TV(p^*_1, \tilde p)\leq \epsilon, \TV(p^*_2, \tilde p)\leq \epsilon$, $p^*_1 \in \GG_{\mathsf{LinReg}}', p^*_2 \in \GG_{\mathsf{LinReg}}'$. So we have
\begin{align}
    &  \inf_{\theta(\tilde p)} \sup_{p^*\in \GG_{\mathsf{LinReg}}', \TV(p^*, \tilde{p})\leq \epsilon} \bE_{p^*}[f(Y-X\theta) - f(Y-X\theta^*(p))] \nonumber \\
    \geq & \inf_{\theta} \max_{p^*\in\{p^*_1, p^*_2\}}  \bE_{p^*}[f(Y-X\theta)] \nonumber \\
     \geq & \frac{1}{2} \inf_{\theta} (\bE_{p^*_1}[f(Y-X\theta)] + \bE_{p^*_2}[f(Y-X\theta)]) \nonumber \\
     = & \frac{1}{2} \inf_{\theta} (\bE_{p^*_1}[f(X(\theta^{(1)}-\theta))] + \bE_{p^*_2}[f(X(\theta^{(2)}-\theta))]) \nonumber \\
     \geq & \frac{1}{2}\inf_{\theta} \left ( \epsilon f( b(\sigma, \epsilon)(\theta^{(1)} - \theta)) + \epsilon f( - b(\sigma, \epsilon)(\theta^{(2)} - \theta))  \right ) \nonumber \\
     \geq  &  \inf_{\theta} \epsilon f( b(\sigma, \epsilon)(\theta^{(1)} - \theta^{(2)})/2) \nonumber \\
     = & \epsilon f(t),
\end{align}
where the last inequality is due to Jensen's inequality. Taking $t\to \infty$ finishes the proof. 
\end{proof}

The hyper-contractivity condition prevents the deletion of dimension (the dimension preserving property is also known as anti-concentration in literature). When  $\frac{2\sigma_1^2\eta\psi^{-1}(1/\eta)}{1-\eta} < 1$, for some $f(\eta, \kappa)>0$, hyper-contractivity guarantees the following holds:
\begin{align}
   \forall v\in\bR^d, \bE_{q}[(v^\top X)^2]  \geq f(\eta, \kappa)  \bE_{p^*}[(v^\top X)^2],
\end{align}
where $q$ is any distribution such that $\TV(p^*, q)\leq \eta$.

\subsection{Proof of Theorem~\ref{thm.tTV_joint_multiplicative}}\label{proof.G_TV_joint}

Similar to the previous cases, we decompose the proof into two lemmas, one showing that $\GG(\psi)$ is a subset of generalized resilience, one showing the finite-sample rate for generalized resilience set. 

For the choice of $B, L$ in the generalized resilience set (Definition~\ref{def.G_TV}), 
due to the non-linear dependence of $\Sigma_p$ on $p$, 
$L$ turns out to be unsuitable as a bridge function. For bridge function  $B$ we need to use $\Sigma$ and $\mu$ rather than $\Sigma_p$ and $\mu_p$ so that $B$ is convex as a function of $p$, 
thus we instead define
\begin{align}
   B(p, (\mu, \Sigma)) & =  \max\big(\|\Sigma^{-1/2}(\mu_p-\mu)\|_2^2/\eta, \|I_d - \Sigma^{-1/2}\bE_p[(X-\mu)(X-\mu)^{\top}]\Sigma^{-1/2}\|_2 \big). \label{eqn.joint_B}
\end{align}
With this choice, we are ready to prove the following two lemmas, which give the proof for Theorem~\ref{thm.tTV_joint_multiplicative} once combined together.
 \begin{lemma}

Consider $B$ and $L$ defined in (\ref{eqn.joint_B}) and (\ref{eqn.joint_L}). 
For $X \sim p$ if we have 
\begin{align} 
    \sup_{v\in\bR^d, \|v\|_2 = 1}\bE_p\left[\psi\left(\frac{({v^{\top}(X-\mu_p)})^2}{\kappa^2\bE_{p}[(v^\top (X-\mu_p))^2]}\right)\right] \leq 1,
\end{align}
then $p \in \GG^\TV(\rho, 6\rho, \eta)$ with $\rho = 4\kappa^2\eta \psi^{-1}(1/\eta)$, assuming $\eta \leq \frac{1}{2}$ and $(1+\eta)\rho < \frac{1}{3}$.
Thus when the perturbation level is $\epsilon \leq \frac{\eta}{2}$, we can recover $\mu, \Sigma$ such that
$\|\Sigma_p^{-1/2}(\mu_p - \mu)\|_2 = O(\kappa\epsilon\sqrt{\psi^{-1}({1}/{\epsilon})})$ and 
$\|I_d - \Sigma_p^{-1/2}\Sigma\Sigma_p^{-1/2}\|_2 = O(\kappa^2\epsilon{\psi^{-1}({1}/{\epsilon})})$.
 \end{lemma}
  \begin{lemma}
Denote $\tilde \epsilon = 2\epsilon + 2C^{\mathsf{vc}}\sqrt{\frac{d+1+\log(1/\delta)}{n}}, \GG^\TV = \bigcap_{\epsilon \in[0, 1/2)}\GG^\TV_\downarrow(\rho(\epsilon),  \epsilon)$. Assume $\tilde \epsilon < 1/2$, $p^* \in \GG^\TV$ .  For $\mathcal{H} = \{v^\top x \mid v\in\bR^d, \|v\|_2=1\}$, let $q$ denote the output of the projection algorithm $\Pi(\hat{p}_n; \tTV_{\mathcal{H}}, \GG^\TV)$ or $\Pi(\hat{p}_n; \tTV_{\mathcal{H}}, \GG^\TV, \tilde{\epsilon}/2)$. Then there exist some $C_1$ such that when $\tilde \epsilon \leq C_1$, with probability at least $1-\delta$, 
\begin{align}
     \|\Sigma_{p^*}^{-1/2}(\mu_{p^*} -\mu_q) \|_2 & \lesssim \sqrt{\tilde \epsilon\rho(3\tilde \epsilon)}, \\
    \| I_d - \Sigma_{p^*}^{-1/2}\Sigma_q\Sigma_{p^*}^{-1/2}\|_2 & \lesssim \rho(3\tilde \epsilon).
\end{align}
 \end{lemma}
 
 We provide the proof for the first lemma as below.
 \begin{proof}
With the choice of $B, L$, we have
\begin{align*}
    \GG_{\downarrow}(\rho_1, \eta) & = \{ p \mid \forall r\leq \frac{p}{1-\eta}, \max\big(\|\Sigma^{-1/2}_p(\mu_r-\mu_p)\|_2^2/\eta, \nonumber \\
  &\qquad  \qquad \|I_d - \Sigma_p^{-1/2}\bE_r[(X-\mu_p)(X-\mu_p)^{\top}]\Sigma_p^{-1/2}\|_2 \big) \leq \rho_1\}, \\
    \GG_{\uparrow}(\rho_1, \rho_2, \eta) & = \Big\{ p \mid  \forall (\mu, \Sigma), \forall r\leq \frac{p}{1-\eta},\nonumber \\
    & \qquad \qquad  \Big(\max\big(\|\Sigma^{-1/2}(\mu_r-\mu)\|_2^2/\eta, \|I_d - \Sigma^{-1/2}\bE_r[(X-\mu)(X-\mu)^{\top}]\Sigma^{-1/2}\|_2 \big) \leq \rho_1 \nonumber \\
& \qquad \qquad \Rightarrow  \max\big(\|\Sigma_p^{-1/2}(\mu_p-\mu)\|_2^2/\eta, \|I_d - \Sigma_p^{-1/2}\Sigma\Sigma_p^{-1/2}\|_2 \big) \leq \rho_2\Big) \Bigg\}.
\end{align*}

We first show that with appropriate choices of $\rho_1$ and $\rho_2$, $\GG_{\downarrow}(\rho_1, \eta)$ is a subset of $\GG_{\uparrow}(\rho_1, \rho_2, \eta)$. 

It suffices to show that for any $p\in \GG_{\downarrow} $, any $\mu, \Sigma, r\leq \frac{p}{1-\eta}$ satisfying $\max\big(\|\Sigma^{-1/2}(\mu_r-\mu)\|_2^2/\eta, \|I_d - \Sigma^{-1/2}\bE_r[(X-\mu)(X-\mu)^{\top}]\Sigma^{-1/2}\|_2 \big) \leq \rho_1$, we have
\begin{align}
    \max\big(\|\Sigma_p^{-1/2}(\mu_p-\mu)\|_2^2/\eta, \|I_d - \Sigma_p^{-1/2}\Sigma\Sigma_p^{-1/2}\|_2 \big) \leq \rho_2. 
\end{align}

We first note that
\begin{align}
    \|I_d- \Sigma^{-1/2}\Sigma_r\Sigma^{-1/2} \|_2 = &
    \|I_d- \Sigma^{-1/2}\bE_r[(X-\mu_r)(X-\mu_r)^{\top}]\Sigma^{-1/2} \|_2 \nonumber \\
    = &\|I_d- \Sigma^{-1/2}\bE_r[(X-\mu)(X-\mu)^{\top}]\Sigma^{-1/2} \nonumber \\
    & + \Sigma^{-1/2}(\mu_r-\mu)(\mu_r-\mu)^{\top}\Sigma^{-1/2} \|_2 \nonumber \\
     \leq &\|I_d- \Sigma^{-1/2}\bE_r[(X-\mu)(X-\mu)^{\top}]\Sigma^{-1/2} \|_2 \nonumber \\
    & + \| \Sigma^{-1/2}(\mu_r-\mu)(\mu_r-\mu)^{\top}\Sigma^{-1/2} \|_2 \nonumber \\
    = &  \|I_d- \Sigma^{-1/2}\bE_r[(X-\mu)(X-\mu)^{\top}]\Sigma^{-1/2} \|_2 \nonumber \\
    & + \| \Sigma^{-1/2}(\mu_r-\mu) \|_2^2 \nonumber \\
    \leq & (1+\eta)\rho_1. 
\end{align}
Thus we have
\begin{align}
    (1-(1+\eta)\rho_1)I_d\preceq \Sigma^{-1/2}\Sigma_r\Sigma^{-1/2}\preceq  (1+(1+\eta)\rho_1)I_d.
\end{align}
From the fact  that $A \preceq B$ leads to $C^{\top}AC \preceq C^{\top}BC$ and taking $C = \Sigma^{1/2}$, we have
\begin{align}\label{eqn.proof_joint_sigma_re}
    (1-(1+\eta)\rho_1)\Sigma\preceq \Sigma_r\preceq  (1+(1+\eta)\rho_1)\Sigma.
\end{align}

Similarly, 
from $p \in \GG_{\downarrow}(\rho_1, \eta)$, we have
\begin{align}
    \|I_d- \Sigma_p^{-1/2}\Sigma_r\Sigma_p^{-1/2} \|_2 = &
    \|I_d- \Sigma_p^{-1/2}\bE_r[(X-\mu_r)(X-\mu_r)^{\top}]\Sigma_p^{-1/2} \|_2 \nonumber \\
    = &\|I_d- \Sigma_p^{-1/2}\bE_r[(X-\mu_p)(X-\mu_p)^{\top}]\Sigma_p^{-1/2} \nonumber \\
    & + \Sigma_p^{-1/2}(\mu_r-\mu_p)(\mu_r-\mu_p)^{\top}\Sigma_p^{-1/2} \|_2 \nonumber \\
     \leq &\|I_d- \Sigma_p^{-1/2}\bE_r[(X-\mu_p)(X-\mu_p)^{\top}]\Sigma_p^{-1/2} \|_2 \nonumber \\
    & + \| \Sigma_p^{-1/2}(\mu_r-\mu_p)(\mu_r-\mu_p)^{\top}\Sigma_p^{-1/2} \|_2 \nonumber \\
    = &  \|I_d- \Sigma_p^{-1/2}\bE_r[(X-\mu_p)(X-\mu_p)^{\top}]\Sigma_p^{-1/2} \|_2 \nonumber \\
    & + \| \Sigma_p^{-1/2}(\mu_r-\mu_p) \|_2^2 \nonumber \\
    \leq & (1+\eta)\rho_1. 
\end{align}
Thus we have
\begin{align}\label{eqn.proof_joint_sigma_rp}
    (1-(1+\eta)\rho_1)\Sigma_p\preceq \Sigma_r\preceq  (1+(1+\eta)\rho_1)\Sigma_p.
\end{align}
Combining Equation (\ref{eqn.proof_joint_sigma_re}) and (\ref{eqn.proof_joint_sigma_rp}), we know that
\begin{align}
    \frac{1-(1+\eta)\rho_1}{1+(1+\eta)\rho_1}\Sigma_p \preceq \Sigma \preceq \frac{1+(1+\eta)\rho_1}{1-(1+\eta)\rho_1} \Sigma_p
\end{align}
When $(1+\eta)\rho_1 \leq \frac{1}{3}$, we have
\begin{align}
(1-3(1+\eta)\rho_1)\Sigma_p \preceq \Sigma \preceq  (1+3(1+\eta)\rho_1)\Sigma_p 
\end{align}
Thus
\begin{align}
    \|I_d- \Sigma_p^{-1/2}\Sigma\Sigma_p^{-1/2} \|_2 &\leq 3(1+\eta)\rho_1.
\end{align}
Furthermore, we know that
\begin{align}
    \| \Sigma_p^{-1/2}(\mu_r -\mu)\|_2 \leq \| \Sigma_p^{-1/2}\Sigma^{1/2}\|_2 \| \Sigma^{-1/2}(\mu_r-\mu) \|_2 \leq \sqrt{(1+3(1+\eta)\rho_1)\eta\rho_1}.
\end{align}
From $\|\Sigma_p^{-1/2}(\mu_r-\mu_p) \|_2 \leq \sqrt{\eta\rho_1}$, by triangle inequality, we know that
\begin{align}
    \| \Sigma_p^{-1/2}(\mu_p-\mu)\|_2 \leq & \|\Sigma_p^{-1/2}(\mu_p-\mu_r) \|_2 + \| \Sigma_p^{-1/2}(\mu_r -\mu)\|_2 \nonumber \\
    \leq & (\sqrt{1+3(1+\eta)\rho_1)}+1)\sqrt{\eta\rho_1}.
\end{align}
Thus
\begin{align}
    \| \Sigma_p^{-1/2}(\mu_p-\mu)\|_2^2/\eta \leq (\sqrt{1+3(1+\eta)\rho_1)}+1)^2\rho_1 < 6\rho_1
\end{align}
assuming $(1+\eta)\rho_1 \leq \frac{1}{3}$.
Therefore $L(p, (\mu, \Sigma)) = \max\big(\|\Sigma_p^{-1/2}(\mu_p-\mu)\|_2^2/\eta, \|I_d - \Sigma_p^{-1/2}\Sigma\Sigma_p^{-1/2}\|_2 \big) \leq 6\rho_1$ if $(1+\eta)\rho_1 \leq \frac{1}{3} $. By taking $\rho_2 = 6\rho_1$, we know that $\GG_{\downarrow}(\rho_1,  \eta)\subset \GG_{\uparrow}(\rho_1, \rho_2, \eta) $. 

Now we only need to show that for any $p$ that satisfies 
\begin{align}
    \sup_{v\in\bR^d, \|v\|_2 = 1}\bE_p\bigg[\psi\big(\frac{({v^{\top}\Sigma_p^{-1/2}(X-\mu_p)})^2}{\kappa^2}\big)\bigg] \leq 1,
\end{align}
we have $p\in \GG_{\downarrow}(\rho_1, \eta)$ for some $\rho_1$.
We view $\Sigma_p^{-1/2}X$ as a random variable. Note that $\psi \circ x^2$ is also a Orlicz function. 
From Lemma~\ref{lem.centering_psi}, we know that bounded raw $\psi$ norm can imply bounded central $\psi$ norm. Thus from Lemma~\ref{lem.cvx_mean_resilience} and centering Lemma~\ref{lem.centering_psi}, for any $\eta < 1/2$,
\begin{align}
    \|\Sigma^{-1/2}_p(\mu_r-\mu_p)\|_2 &\leq 2\kappa\eta \sqrt{\psi^{-1}(1/\eta)}, \\
    \|I_d - \Sigma_p^{-1/2}\bE_r[(X-\mu_p)(X-\mu_p)^{\top}]\Sigma_p^{-1/2}\|_2 & \leq 4\kappa^2\eta \psi^{-1}(1/\eta).
\end{align}

We have shown that $p\in \GG_\downarrow(\rho_1, \eta)$ for $\rho_1 = 4\kappa^2\eta\psi^{-1}(1/\eta)$. Thus $p\in\GG(\rho, 6\rho, \eta)$ for $\rho = 4\kappa\eta\psi^{-1}(1/\eta)$ assuming $(1+\eta)\rho \leq \frac{1}{3}$.
\end{proof}

Now we provide a proof for the second lemma on the finite sample results:
\begin{proof}
The bound on $\tilde\epsilon$ is the same as in the proof of Theorem~\ref{thm.tTV_mean}. It suffices to show the modulus of continuity.

We first show that when $\tTV_\mathcal{H}(p_1, p_2)\leq \tilde \epsilon$, $p_1, p_2\in \GG$, we have $\|I_d - \Sigma_{p_1}^{-1/2}\Sigma_{p_2}\Sigma_{p_1}^{-1/2}\|_2\lesssim \rho(3\tilde \epsilon)$.  Without loss of generality, we assume $\Sigma_{p_1}$ is invertible. 
Consider any fixed direction $v\in\bR^d, \|v\|_2=1$, from $p_1 \in \GG_\downarrow$, we know that for any $r\leq \frac{p_1}{1-\epsilon}$, $\|\Sigma_{p_1}^{-1/2}(\bE_r[X] - \bE_p[X])\| \leq \rho(\epsilon)$.  By taking $\mathcal{F} = \{f(X) = v^\top \Sigma_{p_1}^{-1/2} X \mid v\in\bR^d, \|v\|_2 = 1\}$, this is equivalent to the condition $p_1 \in \GG_{W_\mathcal{F}}$. Thus from Lemma~\ref{lem.tail_bound}, we have
\begin{align}
   \bP_{p_1}\left(|v^\top \Sigma_{p_1}^{-1/2}(X-\mu_{p_1})| \geq \sqrt{\frac{\rho(\tilde\epsilon)}{\tilde\epsilon}}\right) \leq \tilde\epsilon.
\end{align}
From $\tTV_\mathcal{H}(p_1, p_2)\leq \tilde \epsilon$, we know that 
\begin{align}
     \bP_{p_2}\left(|v^\top \Sigma_{p_1}^{-1/2}(X-\mu_{p_1})| \geq \sqrt{\frac{\rho(\tilde\epsilon)}{\tilde\epsilon}}\right) \leq 3\tilde\epsilon.
\end{align}
We truncate $p_1, p_2$ by deleting all the mass the satisfies $|v^\top \Sigma_{p_1}^{-1/2}(X-\mu_{p_1})| \geq \sqrt{\frac{\rho(\tilde\epsilon)}{\tilde\epsilon}} $ to get deleted distribution $r_{p_1}, r_{p_2}$. Then we know that $\TV(p_1, r_{p_1})\leq \tilde \epsilon, \TV(p_2, r_{p_2})\leq 3\tilde \epsilon$. From $p_1, p_2\in\GG$, we know that
\begin{align}
     \|\Sigma_{p_1}^{-1/2}(\mu_{p_1} -\mu_{r_{p_1}}) \|_2  \leq \sqrt{\tilde \epsilon\rho(\tilde \epsilon)}, 
    \| I_d - \Sigma_{p_1}^{-1/2}\bE_{r_{p_1}}[(X-\mu_{p_1})(X-\mu_{p_1})^\top]\Sigma_{p_1}^{-1/2}\|_2  \leq \rho(\tilde \epsilon) \\
     \|\Sigma_{p_2}^{-1/2}(\mu_{p_2} -\mu_{r_{p_2}}) \|_2  \leq \sqrt{3\tilde \epsilon\rho(3\tilde \epsilon)}, 
    \| I_d - \Sigma_{p_2}^{-1/2}\bE_{r_{p_2}}[(X-\mu_{p_2})(X-\mu_{p_2})^\top]\Sigma_{p_2}^{-1/2}\|_2  \leq \rho(3\tilde \epsilon). \nonumber 
\end{align}
From the above inequality we also have
\begin{align}
    \| I_d - \Sigma_{p_2}^{-1/2}\Sigma_{r_{p_2}}\Sigma_{p_2}^{-1/2}\|_2 & \leq \| I_d - \Sigma_{p_2}^{-1/2}\bE_{r_{p_2}}[(X-\mu_{p_2})(X-\mu_{p_2})^\top]\Sigma_{p_2}^{-1/2}\|_2  + \|\Sigma_{p_2}^{-1/2}(\mu_{p_2} -\mu_{r_{p_2}}) \|_2^2 \nonumber \\ 
    & \leq \rho(3\tilde \epsilon) + 3\tilde \epsilon\rho(3\tilde \epsilon). \nonumber
\end{align}
This is equivalent to
\begin{align}\label{proof.joint_dominance}
    (1-\rho(3\tilde \epsilon) - 3\tilde \epsilon\rho(3\tilde \epsilon)) \Sigma_{r_{p_2}} \preceq \Sigma_{{p_2}} \preceq (1+\rho(3\tilde \epsilon) + 3\tilde \epsilon\rho(3\tilde \epsilon)) \Sigma_{r_{p_2}}.
\end{align}
Now we know that the random variable $v^\top \Sigma_{p_1}^{-1/2}X$ under $r_{p_1}, r_{p_2}$ has bounded support, and $\tTV_\mathcal{H}(r_{p_1}, r_{p_2})\leq 5\tilde \epsilon$, thus we have
\begin{align}\label{eqn.proof_mean_close_joint}
    |v^\top \Sigma_{p_1}^{-1/2}(\mu_{r_{p_1}} - \mu_{r_{p_2}})| & = |v^\top \Sigma_{p_1}^{-1/2}(\mu_{r_{p_1}}-\mu_{p_1}+ \sqrt{\frac{\rho(\tilde\epsilon)}{\tilde\epsilon}} -( \mu_{r_{p_2}}-\mu_{p_1}+ \sqrt{\frac{\rho(\tilde\epsilon)}{\tilde\epsilon}} ))|  \\ 
    & \stackrel{(i)}{=} |\int_{t=\mu_{p_1}- \sqrt{\frac{\rho(\tilde\epsilon)}{\tilde\epsilon}}}^{\mu_{p_1}+ \sqrt{\frac{\rho(\tilde\epsilon)}{\tilde\epsilon}}} (\bP_{r_{p_1}}(v^\top \Sigma_{p_1}^{-1/2}X \geq t) -  \bP_{r_{p_2}}(v^\top \Sigma_{p_1}^{-1/2}X \geq t)) dt | \nonumber \\
    & \stackrel{(ii)}{\leq} \int_{t=\mu_{p_1}- \sqrt{\frac{\rho(\tilde\epsilon)}{\tilde\epsilon}}}^{\mu_{p_1}+ \sqrt{\frac{\rho(\tilde\epsilon)}{\tilde\epsilon}}} 5\tilde \epsilon dt \nonumber \\
    & \leq 10\sqrt{\tilde \epsilon\rho(\tilde \epsilon)}.\nonumber 
\end{align}
Here (i) utilizes the integral representation of mean, (ii) uses the triangle inequality $\tTV_\sH(r_{p_1}, r_{p_2}) \leq \tTV_\sH(r_{p_1}, r_{p_2}) + \tTV_\sH(r_{p_1}, {p_1}) + \tTV_\sH({p_1}, {p_2}) +\tTV_\sH({p_2}, r_{p_2}) \leq 5\tilde \epsilon$.
\begin{align}\label{eqn.proof_cov_close_joint}
    & |\bE_{r_{p_1}}[(v^\top \Sigma_{p_1}^{-1/2}(X-\mu_{r_{p_1}}))^2] - \bE_{r_{p_2}}[(v^\top \Sigma_{p_1}^{-1/2}(X-\mu_{r_{p_1}}))^2]|  \\
   \stackrel{(i)}{=} &\left|\int_{t=0}^{ \sqrt{\frac{\rho(\tilde\epsilon)}{\tilde\epsilon}}}(t\bP_{r_{p_1}}(|v^\top \Sigma_{p_1}^{-1/2}(X-\mu_{r_{p_1}})|\geq t) -  t\bP_{r_{p_2}}(|v^\top \Sigma_{p_1}^{-1/2}(X-\mu_{r_{p_1}})|\geq t) dt\right| \nonumber \\ 
    \stackrel{(ii)}{\leq} & \left|\int_{t=0}^{ \sqrt{\frac{\rho(\tilde\epsilon)}{\tilde\epsilon}}} 10\tilde \epsilon  t dt \right|= 5\rho(\tilde \epsilon).\nonumber 
\end{align}
Here (i) utilizes the integral representation of covariancce, (ii) is a result of $\bP_{r_{p_1}}(|X|\geq t) -  \bP_{r_{p_2}}(|X|\geq t) = \bP_{r_{p_1}}(X\geq t) -  \bP_{r_{p_2}}(X\geq t) + \bP_{r_{p_1}}(X\leq -t) -  \bP_{r_{p_2}}(X\leq -t)$, and then apply the similar triangle inequality to both. 
Based on all the results above, we have
\begin{align}
  &|1 -\bE_{r_{p_2}}[(v^\top \Sigma_{p_1}^{-1/2}(X-\mu_{r_{p_2}}))^2]|  \nonumber \\ 
  \stackrel{(i)}{\leq} & |\bE_{r_{p_1}}[(v^\top \Sigma_{p_1}^{-1/2}(X-\mu_{{p_1}}))^2] -\bE_{r_{p_2}}[(v^\top \Sigma_{p_1}^{-1/2}(X-\mu_{r_{p_2}}))^2]| + \rho(\tilde \epsilon) \nonumber \\ 
  \stackrel{(ii)}{\leq} & |\bE_{r_{p_1}}[(v^\top \Sigma_{p_1}^{-1/2}(X-\mu_{{p_1}}))^2] - \bE_{r_{p_1}}[(v^\top \Sigma_{p_1}^{-1/2}(X-\mu_{r_{p_1}}))^2] \nonumber \\ 
  & + \bE_{r_{p_1}}[(v^\top \Sigma_{p_1}^{-1/2}(X-\mu_{r_{p_1}}))^2]  -\bE_{r_{p_2}}[(v^\top \Sigma_{p_1}^{-1/2}(X-\mu_{r_{p_1}}))^2] \nonumber \\
  &  + \bE_{r_{p_2}}[(v^\top \Sigma_{p_1}^{-1/2}(X-\mu_{r_{p_1}}))^2]  -\bE_{r_{p_2}}[(v^\top \Sigma_{p_1}^{-1/2}(X-\mu_{r_{p_2}}))^2] | + \rho(\tilde \epsilon) \nonumber \\ 
  \stackrel{(iii)}{\leq} & (v^\top \Sigma_{p_1}^{-1/2}(\mu_{r_{p_1}} - \mu_{{p_1}} ))^2 + 5\rho(\tilde \epsilon) +  (v^\top \Sigma_{p_1}^{-1/2}(\mu_{r_{p_1}} - \mu_{r_{p_2}} ))^2 + \rho(\tilde \epsilon) \nonumber \\ 
  \stackrel{(iv)}{\leq} & 6\rho(\tilde \epsilon) + 11\tilde \epsilon \sqrt{\rho(\tilde \epsilon)}\lesssim \rho(\tilde \epsilon).\nonumber
\end{align}
Here (i) comes from the fact that $p_1\in\GG$, thus $\rho(\tilde \epsilon)\geq \|I - \bE_{r_{p_1}}[\Sigma_{p_1}^{-1/2}(X-\mu_{{p_1}})(X-\mu_{{p_1}})^\top \Sigma_{p_1}^{-1/2}]\|$. This gives that $|\bE_{r_{p_1}}[(v^\top \Sigma_{p_1}^{-1/2}(X-\mu_{{p_1}}))^2] -1|\leq \rho(\tilde \epsilon)$. (ii) is a result of triangle inequality. (iii) is a result of \eqref{eqn.proof_cov_close_joint} and (iv) is a result of~\eqref{eqn.proof_mean_close_joint} and the assumption that $p_1\in\GG$. 
On the other hand, we have
\begin{align}
    |1 -\bE_{r_{p_2}}[(v^\top \Sigma_{p_1}^{-1/2}(X-\mu_{r_{p_2}}))^2]| = \left|1 - \frac{\bE_{r_{p_2}}[(v^\top \Sigma_{p_1}^{-1/2}(X-\mu_{r_{p_2}}))^2]}{\bE_{{p_1}}[(v^\top \Sigma_{p_1}^{-1/2}(X-\mu_{{p_1}}))^2]}\right|.  \nonumber 
\end{align}
Denote $v' = \frac{\Sigma_{p_1}^{-1/2}v}{\|\Sigma_{p_1}^{-1/2}v\|_2}$,  from~\eqref{proof.joint_dominance} we have
\begin{align}
(1-\rho(3\tilde \epsilon) -3\tilde \epsilon\rho(3\tilde \epsilon))\frac{\bE_{r_{p_2}}[(v'^\top(X-\mu_{r_{p_2}}))^2]}{\bE_{{p_1}}[(v'^\top (X-\mu_{{p_1}}))^2]} & \leq \frac{\bE_{{p_2}}[(v'^\top(X-\mu_{{p_2}}))^2]}{\bE_{{p_1}}[(v'^\top (X-\mu_{{p_1}}))^2]} \nonumber \\ 
& \leq (1+\rho(3\tilde \epsilon) + 3\tilde \epsilon\rho(3\tilde \epsilon))\frac{\bE_{r_{p_2}}[(v'^\top(X-\mu_{r_{p_2}}))^2]}{\bE_{{p_1}}[(v'^\top (X-\mu_{{p_1}}))^2]}.\nonumber
\end{align}
Thus
\begin{align}
    \left|1 - \frac{\bE_{{p_2}}[(v'^\top(X-\mu_{{p_2}}))^2]}{\bE_{{p_1}}[(v'^\top (X-\mu_{{p_1}}))^2]}\right| \lesssim \rho(3\tilde \epsilon).\nonumber
\end{align}
Now we have shown for any $v'\in\bR^d$, the above inequality holds. 
Taking $v' = {\Sigma_{p_1}^{-1/2}v} $, we have
\begin{align}
    \|I_d - \Sigma_{p_1}^{-1/2} \Sigma_{p_2}\Sigma_{p_1}^{-1/2}\|_2 \lesssim \rho(3\tilde \epsilon).\nonumber
\end{align}
This gives multiplicative bound for covariance. Now we only need to bound the difference between mean.
From above proof, we already know that when $\rho(3\tilde \epsilon)\leq 1$, for the fixed $v\in\bR^d, \|v\|_2=1$,
\begin{align}
     |v^\top \Sigma_{p_1}^{-1/2}(\mu_{r_{p_1}} - \mu_{r_{p_2}})|\lesssim 10\sqrt{\tilde \epsilon \rho(\tilde \epsilon)}, \\
     \|\Sigma_{p_1}^{-1/2}(\mu_{p_1} -\mu_{r_{p_1}}) \|_2  \leq \sqrt{\tilde \epsilon\rho(\tilde \epsilon)},  \\
     \|\Sigma_{p_1}^{-1/2}(\mu_{p_2} -\mu_{r_{p_2}}) \|_2  \lesssim \sqrt{\tilde \epsilon\rho(3\tilde \epsilon)}. \nonumber
\end{align}
Thus we have
\begin{align}
    |v^\top \Sigma_{p_1}^{-1/2}(\mu_{{p_1}} - \mu_{{p_2}})|\lesssim \sqrt{\tilde \epsilon\rho(3\tilde \epsilon)}. \nonumber
\end{align}
This shows that 
\begin{align}
    \|\Sigma_{p_1}^{-1/2}(\mu_{{p_1}} - \mu_{{p_2}})\|_2 \lesssim \sqrt{\tilde \epsilon\rho(3\tilde \epsilon)}.\nonumber
\end{align}
\end{proof}

We also justify the choice of $\mathcal{H}$ is consistent with the general design described in Section~\ref{subsec.discussion_tTV}. 
We first write down the dual representation of $B$. Recall that we take $B$ as 
\begin{align}
   B(p, (\mu, \Sigma)) & =  \max\big(\|\Sigma^{-1/2}(\mu_p-\mu)\|_2^2/\eta, \|I_d - \Sigma^{-1/2}\bE_p[(X-\mu)(X-\mu)^{\top}]\Sigma^{-1/2}\|_2 \big) \nonumber \\
   & = \max(\sup_{v_1 \in \bR^d, \|v_1\|_2=1} (v_1^\top \Sigma^{-1/2}(\mu_p-\mu))^2/\eta, \nonumber \\
   & \quad  \sup_{v_2 \in \bR^d, \|v_2\|_2=1} |1-v_2\Sigma^{-1/2}\bE_p[(X-\mu)(X-\mu)^{\top}]\Sigma^{-1/2} v_2|)\nonumber \\
   & = \max(\sup_{y\in\bR, v_1\in\bR^d,\|v_1\|_2 = 1} {y(v_1^\top \Sigma^{-1/2}(\mu_p-\mu)) - y^2)}/{\eta}, \nonumber \\
   & \quad \sup_{v_2\in\bR^d, \|v_2\|_2 =1}  |1-v_2\Sigma^{-1/2}\bE_p[(X-\mu)(X-\mu)^{\top}]\Sigma^{-1/2} v_2|)\nonumber \\
   & = \sup_{f\in \sF_1\bigcup \sF_2} \bE_p[f(X)],
\end{align}
where $\sF_1 = \{ (y(v_1^\top \Sigma^{-1/2}(x-\mu)) - y^2)/\eta \mid y\in\bR, v_1\in\bR^d, \|v_1\|_2=1 \}$, $\sF_2 = \{ \xi(1-v_2\Sigma^{-1/2}(x-\mu)(x-\mu)\Sigma^{-1/2}v_2) \mid \xi\in\{\pm 1\}, v_2\in\bR^d, \|v_2\|_2=1 \}$. This can also be viewede as $B(p, q)$ where $\mu = \mu_q, \Sigma = \Sigma_q$. 

We show that by taking $\mathcal{H} = \{ v^\top x \mid v\in\bR^d\}$, $\tTV_\mathcal{H}$ small can imply $\tTV_{\sF_1\bigcup \sF_2}$ small.
\begin{align}
\tTV_{\sF_1}(p, q) = & \sup_{y\in\bR, v_1\in\bR^d, \|v_1\|_2=1} |\bP_p((y(v_1^\top \Sigma_q^{-1/2}(x-\mu_q)) - y^2)/\eta \geq t) \nonumber \\
&  - \bP_q((y(v_1^\top \Sigma_q^{-1/2}(x-\mu_q)) - y^2)/\eta \geq t)| \nonumber \\
\leq &\sup_{v\in\bR^d, \|v\|_2=1, t\in\bR} |\bP_p(|v^\top \Sigma_q^{-1/2}(X-\mu_q)| \geq t) - \bP_q(|v^\top \Sigma_q^{-1/2}(X-\mu_q)| \geq t)| \nonumber \\
\leq &\sup_{v\in\bR^d, \|v\|_2=1, t\in\bR} |\bP_p(v^\top X \geq t) - \bP_p(v^\top X\geq t)|  \nonumber \\
\leq & \tTV_\mathcal{H}(p, q).
\end{align}
Furthermore, for $\tTV_{\sF_2}$, we have
\begin{align}
\tTV_{\sF_2}(p, q) = & \sup_{\xi \in\{\pm 1\}, v\in\bR^d, \|v\|_2=1, t\in\bR} |\bP_p(\xi v^\top \Sigma_q^{-1/2}(X-\mu_p)(X-\mu_p)^{\top}\Sigma_q^{-1/2}v \geq t) \nonumber \\
&  - \bP_q(\xi v^\top \Sigma_q^{-1/2}(X-\mu_p)(X-\mu_p)^{\top}\Sigma_q^{-1/2}v \geq t)| \nonumber \\
\leq &\sup_{v\in\bR^d, \|v\|_2=1, t\in\bR} |\bP_p(|v^\top \Sigma_q^{-1/2}(X-\mu_p)| \geq t) - \bP_q(|v^\top \Sigma_q^{-1/2}(X-\mu_p)| \geq t)|\nonumber \\
\leq &\sup_{v\in\bR^d, \|v\|_2=1, t>0} |\bP_p(v^\top(X-\mu_p) \geq t) - \bP_q(v^\top(X-\mu_p) \geq t)| \nonumber \\
& +  |\bP_p(v^\top(X-\mu_p) \leq -t) - \bP_q(v^\top(X-\mu_p) \leq -t)|\nonumber \\
\leq &\sup_{v\in\bR^d, \|v\|_2=1, t\in\bR} |\bP_p(v^\top X \geq t) - \bP_p(v^\top X\geq t)| \nonumber \\
& + \sup_{v\in\bR^d, \|v\|_2=1, t\in\bR} |\bP_p(v^\top X \geq t) - \bP_p(v^\top X\geq t)|  \nonumber \\
\leq & 2\tTV_\mathcal{H}(p, q).
\end{align}
Thus we have $\tTV_{\sF_1 \bigcup \sF_2}\leq 2\tTV_\mathcal{H}$. 

\subsection{Joint mean and covariance estimation under different cost function}
Also, we are able to guarantee robustness for joint mean and covariance estimation by choosing a different set of $B, L$ pairs in $\GG$:
\begin{align}\label{eqn.weaker_joint}
   B(p, (\mu, \Sigma)) & =  \max\big(\|\mu_p-\mu\|_2^2/\eta, \|\Sigma - \bE_p[(X-\mu)(X-\mu)^{\top}]\|_2 \big), \\ \label{eqn.weaker_joint_L}
    L(p, (\mu, \Sigma)) & = \max\big(\|\mu_p-\mu\|_2^2/\eta, \|\Sigma - \Sigma_p\|_2 \big).
\end{align}
We have the following finite sample error bounds, which generalizes~\citep[Theorem 4.1]{gao2019generative} to nonparametric classes. 
 \begin{theorem}\label{thm.tTV_joint_correct}
 
Denote $\tilde \epsilon = 2\epsilon + 2C^{\mathsf{vc}}\sqrt{\frac{d+1+\log(1/\delta)}{n}}, \GG = \GG_\downarrow(\rho(\tilde \epsilon), \tilde \epsilon) \bigcap \GG_\downarrow(\rho(2\tilde \epsilon), 2\tilde \epsilon)$ when $B, L $ are chosen as~\eqref{eqn.weaker_joint} and~\eqref{eqn.weaker_joint_L}. Assume $\tilde \epsilon < 1/2$, $p^* \in \GG$ . Assume the oblivious corruption model of level $\epsilon$. Denote the empirical distribution of observed data as $\hat p_n$. For $\mathcal{H} = \{v^\top x \mid v\in\bR^d, \|v\|_2=1\}$, let $q$ denote the output of the projection algorithm $\Pi(\hat{p}_n; \tTV_{\mathcal{H}}, \GG)$ or $\Pi(\hat{p}_n; \tTV_{\mathcal{H}}, \GG, \tilde{\epsilon}/2)$. Then, with probability at least $1-\delta$, 
\begin{align}
     \|\mu_{p^*} -\mu_q \|_2 & \leq 2\sqrt{\tilde \epsilon\rho(\tilde \epsilon)}, \\
    \| \Sigma_{p^*}-\Sigma_q\|_2 & \leq 7\rho(2\tilde \epsilon).
\end{align}
\end{theorem}
\begin{proof}
The bound on $\tilde\epsilon$ is the same as in the proof of Theorem~\ref{thm.tTV_mean}. It suffices to show the modulus of continuity.

Assume $\tTV_\mathcal{H}(p_1, p_2) \leq \tilde \epsilon$, and $p_1, p_2\in\GG_\downarrow(\rho(\tilde \epsilon), \tilde \epsilon)$. Following the same argument as mean estimation in Theorem~\ref{thm.tTV_mean}, we know that
\begin{align}
    \|\mu_{p_1} - \mu_{p_2}\|_2\leq 2\sqrt{\tilde \epsilon \rho(\tilde \epsilon)}.
\end{align}
Thus it suffices to bound the modulus of continuity for covariance estimation. 
Note that
\begin{align}
 &\sup_{v\in\bR^d, \|v\|_2=1, t\in\bR^d} |\bP_{p_1}(v^\top  (X-\mu_{p_2})(X-\mu_{p_2})^{\top}v \geq t) - \bP_{p_2}(v^\top  (X-\mu_{p_2})(X-\mu_{p_2})^{\top}v \geq t)| \nonumber \\
= &\sup_{v\in\bR^d, \|v\|_2=1, t\in\bR^d} |\bP_{p_1}(|v^\top (X-\mu_{p_2})| \geq t) - \bP_{p_2}(|v^\top (X-\mu_{p_2})| \geq t)| \nonumber \\
\leq &\sup_{v\in\bR^d, \|v\|_2=1, t\in\bR^d} |\bP_{p_1}(v^\top(X-\mu_{p_2}) \geq t) - \bP_{p_2}(v^\top(X-\mu_{p_2}) \geq t)| \nonumber \\
& + |\bP_{p_1}(v^\top(X-\mu_{p_2}) \leq -t) - \bP_{p_2}(v^\top(X-\mu_{p_2}) \leq -t)|  \nonumber \\
\leq & 2\tTV_\mathcal{H}({p_1}, {p_2}) \nonumber \\
\leq & 2\tilde \epsilon. \nonumber
\end{align}
Without loss of generality, assume that there exists some $v^*$ such that 
\begin{align}
    v^{*\top}(\Sigma_{p_1}-\Sigma_{p_2})v^* = \|\Sigma_{p_1}-\Sigma_{p_2}\|_2. \nonumber
\end{align}
Thus from $\tTV_\mathcal{H}(p_1, p_2)\leq \tilde \epsilon$ and Lemma~\ref{lem.tvt_cross_mean},  there exist some $r_{p_1}\leq \frac{p_1}{1-2\tilde \epsilon}$, $r_{p_2}\leq \frac{p_2}{1-2\tilde \epsilon}, $ such that 
\begin{align}
    v^{*\top} \bE_{r_{p_1}}[(X-\mu_{p_2})(X-\mu_{p_2})^{\top}]v^* \leq  v^{*\top} \bE_{r_{p_2}}[(X-\mu_{p_2})(X-\mu_{p_2})^{\top}]v^*.  \nonumber
\end{align}
From $p_1, p_2 \in\GG_\downarrow(\rho(\tilde \epsilon), \tilde \epsilon)$, we know that 
\begin{align*}
\sup_{r_{p_1}\leq \frac{p_1}{1-2\tilde \epsilon}} v^{*\top}(\Sigma_{p_1} - \bE_{{r_{{p_1}}}}[(X-\mu_{p_2})(X-\mu_{p_2})^{\top}])v^* \leq \rho(2\tilde \epsilon),\\
\sup_{r_{p_2}\leq \frac{p_2}{1-2\tilde \epsilon}} v^{*\top}( \bE_{{r_{{p_2}}}}[(X-\mu_{p_2})(X-\mu_{p_2})^{\top}] - \Sigma_{p_2})v^* \leq \rho(2\tilde \epsilon).
\end{align*}
Thus overall, we have
\begin{align*}
     \|\Sigma_{p_1}-\Sigma_{p_2}\|_2  = & v^{*\top}(\Sigma_{p_1}-\Sigma_{p_2})v^* \nonumber \\
      = & v^{*\top}(\Sigma_{p_1} - \bE_{r_{p_1}}[(X-\mu_{p_1})(X-\mu_{p_1})^{\top}] \nonumber \\
     & + \bE_{r_{p_1}}[(X-\mu_{p_1})(X-\mu_{p_1})^{\top}] - \bE_{r_{p_1}}[(X-\mu_{r_{p_1}})(X-\mu_{r_{p_1}})^{\top}] \nonumber \\
     & + \bE_{r_{p_1}}[(X-\mu_{r_{p_1}})(X-\mu_{r_{p_1}})^{\top}] - \bE_{r_{p_1}}[(X-\mu_{p_2})(X-\mu_{p_2})^{\top}] \nonumber \\
     & + \bE_{r_{p_1}}[(X-\mu_{p_2})(X-\mu_{p_2})^{\top}] - \bE_{r_{p_2}}[(X-\mu_{p_2})(X-\mu_{p_2})^{\top}]  \nonumber \\ 
     & + \bE_{r_{p_2}}[(X-\mu_{p_2})(X-\mu_{p_2})^{\top}] -\Sigma_{p_2})v^* \nonumber \\ 
     \leq & \rho(2\tilde \epsilon) + \|\mu_{p_1}-\mu_{r_{p_1}}\|_2^2 + \|\mu_{p_2}-\mu_{r_{p_1}}\|_2^2 + 0 + \rho(2\tilde \epsilon) \nonumber \\ 
     \leq & 2\rho(2\tilde \epsilon) + 10\tilde \epsilon \rho(2\tilde \epsilon) \nonumber\\ 
     \leq & 7\rho(2\tilde \epsilon).
\end{align*}
\end{proof}

\subsection{Midpoint lemma for $\tTV$}

In Section~\ref{sec.finite_sample_algorithm_weaken},  we control  the modulus of continuity by mean-cross lemma. Another way to bound the modulus is via the midpoint lemma. We can show that for two 1-dimensional distributions that are close under $\tTV_\mathcal{H}$, we can also find a midpoint that is close to both distributions. This is formally proved in the following lemma.

\begin{lemma}[Midpoint for $\tTV$]
\label{lem.tvt_midpoint} 
Suppose two distributions $p,q$ on the real line satisfy
\begin{align}\label{eqn.tTV_XY}
\sup_{t\in \bR} |\bP_p(X\geq t) - \bP_q(Y\geq t)| \leq \epsilon.
\end{align}

Then one can find some distribution $r$ with $r(A) \leq \frac{p(A)}{1-\epsilon}$ and $r(A) \leq \frac{q(A)}{1-\epsilon}$ for any event $A\in\mathcal{A} = \{X\geq t, X\leq t \mid t\in\bR\}$.
Furthermore, $\sup_{r(A)\leq \frac{p(A)}{1-\epsilon}, \forall A \in\mathcal{A}} |\bE_r[X] - \bE_p[X]| = \sup_{r\leq \frac{p}{1-\epsilon}} |\bE_r[X] - \bE_p[X]|$.
\end{lemma}
\begin{proof}
To guarantee that $r(A) \leq \frac{p(A)}{1-\epsilon}$ and $r(A) \leq \frac{q(A)}{1-\epsilon}$ for any event $A\in\mathcal{A} = \{X\geq t, X\leq t \mid t\in\bR\}$, it is equivalent to the following constraints:
\begin{align*}
    \frac{\max( \bP_p(X\leq t), \bP_q(X\leq t) ) - \epsilon}{1-\epsilon}\leq \bP_r(X\leq t) \leq \frac{\min(\bP_p(X\leq t),\bP_q(X\leq t))}{1-\epsilon}, \forall t\in\bR. 
\end{align*}
Such distribution $r$  exists since it is always true that $ \min(\bP_p(X\leq t),\bP_q(X\leq t)) \geq \max( \bP_p(X\leq t), \bP_q(X\leq t) ) - \epsilon$.

Now it suffices to argue that $\sup_{r(A)\leq \frac{p(A)}{1-\epsilon}, \forall A \in\mathcal{A}} |\bE_r[X] - \bE_p[X]| = \sup_{r\leq \frac{p}{1-\epsilon}} |\bE_r[X] - \bE_p[X]|$. We consider the case when $\bE_r[X] \geq \bE_p[X]$. Then the left-hand side is equivalent to the following constrained optimization problem:
\begin{align}
    \sup_{F} & -\int_{-\infty}^0 F(t)dt + \int_{0}^{+\infty} (1-F(t))dt
    \nonumber \\ 
     s.t. &  \frac{\bP_p(X\leq t)-\epsilon}{1-\epsilon} \leq F(t) \leq \frac{\bP_p(X\leq t)}{1-\epsilon}, F(-\infty) =0, F(+\infty) =1, F:\text{non-decreasing}  \nonumber
\end{align}
One can see that it is maximized at the case when $F(t)$ is minimized everywhere, i.e. $F(t) =  \frac{\bP_p(X\leq t)-\epsilon}{1-\epsilon}, \forall \bP_p(X\leq t)\geq \epsilon, F(t) = 0$ otherwise, which achieves the same supremum when we restrict $r\leq p/(1-\epsilon)$. Similarly when $\bE_r[X]< \bE_p[X]$, the distance is maximized when $F(t)$ is maximized everywhere, which is equivalent to deleting the largest $\epsilon$ mass. This finishes the proof.

\end{proof}

The key observation in this lemma is that the worst-case perturbation under $\tTV$ distance is to delete $\epsilon$ mass from the smallest points (or the largest points), which coincides with the worst-case perturbation under $\TV$ distance. This motivates the design of the new mid-point.

 We remark here that both the mean cross lemma above and the mid-point lemma (Lemma~\ref{lem.tvt_midpoint}) can bound the modulus under $\tTV$ distance for all tasks we considered. In the main text, we have shown that mid-point lemma can bound the modulus for mean. Here we use other examples to illustrate the power of mean-cross lemma. In all the analyses below, the mean-cross lemma can be substituted with mid-point lemma. 

\subsection{Finite sample analysis for $\tTV_{
\mathcal{H}}$ projection to $\GG_{W_\sF}$ in~(\ref{eqn.def_G_TV_WF})}
 
We now present a theorem for robust learning with loss function $W_{\sF}$. 

\begin{theorem}
Choose some symmetric $\mathcal{H} \subset \sF$ such that $W_{\mathcal{H}}(p,q)\geq \frac{1}{2}W_{\sF}(p,q), \forall p,q \in \GG_{W_{\sF}}(\rho(\tilde \epsilon), \tilde \epsilon)$. Denote $\tilde \epsilon = 2\epsilon + \sqrt{\frac{2\ln(2|\mathcal{H}|/\delta)}{n}}$. Assume $p^* \in \GG_{W_{\sF}}(\rho(\tilde \epsilon), \tilde \epsilon)$ and the oblivious corruption model of level $\epsilon$ with $\TV$ perturbation. Denote the empirical distribution of observed data as $\hat p_n$. Let $q$ denote the output of the projection algorithm $\Pi(\hat{p}_n; \tTV_{\mathcal{H}}, \GG)$ or $\Pi(\hat{p}_n; \tTV_{\mathcal{H}}, \GG, \tilde{\epsilon}/2)$. Then, with probability at least $1-\delta$, 
\begin{align}
    W_{\sF}(p^*,q) \leq 4\rho\left ( 2\epsilon + \sqrt{\frac{2\ln(2|\mathcal{H}|/\delta)}{n}}\right ).
\end{align} 
\end{theorem}

\begin{proof}
The $\tilde \epsilon$ bound follows from Proposition~\ref{prop.tTV_general} and Lemma~\ref{lemma.vcinequality}. Now we upper bound the modulus:
\begin{align}
\sup_{p_1, p_2\in \GG_{W_{\sF}}(\rho(\tilde \epsilon),\tilde \epsilon): \tTV_\mathcal{H}(p_1, p_2) \leq \tilde \epsilon}  W_{\sF}(p_1,p_2) \leq 4\rho(\tilde \epsilon).
\end{align}

The condition that $\tTV_\mathcal{H}(p_1, p_2) \leq \tilde \epsilon$ implies that for any $h\in \mathcal{H}$,
\begin{align}
    \sup_{t\in \bR}|\bP_{p_1}[h(X)\geq t] - \bP_{p_2}[h(X) \geq t]|\leq \tilde \epsilon.
\end{align}
 
Take $h = \argmax_{h\in \mathcal{H}} \mathbb{E}_{p_1}[h(X)] - \mathbb{E}_{p_2}[h(X)]$, hence  $\mathbb{E}_{p_1}[h(X)] - \mathbb{E}_{p_2}[h(X)] = W_{\mathcal{H}}(p_1,p_2)$. It follows from Lemma~\ref{lem.tvt_cross_mean} that there exist some $r_{p_1} \leq \frac{p_1}{1-\tilde \epsilon}, r_{p_2} \leq \frac{p_2}{1-\tilde \epsilon}$ such that
  \begin{align}
      \bE_{r_{p_1}}[h(X)] \leq \bE_{r_{p_2}}[h(X)].
  \end{align}
Furthermore, from $p_1, p_2\in \GG_{W_{\sF}}(\rho(\tilde \epsilon), \tilde \epsilon), \mathcal{H} \subset \sF$, we have
\begin{align}
    \bE_{p_1}[h(X)]  - \bE_{r_{p_1}}[h(X)] \leq \rho(\tilde \epsilon), \\
    \bE_{r_{p_2}}[h(X)]  - \bE_{p_2}[h(X)] \leq \rho(\tilde \epsilon).
\end{align}
Then,
\begin{align}
W_{\mathcal{H}}(p_1,p_2) & =  \bE_{p_1}[h(X)] - \bE_{p_2}[h(X)] \\ 
& = \bE_{p_1}[h(X)] - \bE_{r_{p_1}}[h(X)] + \bE_{r_{p_1}}[h(X)] - \bE_{r_{p_2}}[h(X)] + \bE_{r_{p_2}}[h(X)] - \bE_{p_2}[h(X)]  \\
& \leq 2\rho(\tilde \epsilon),
\end{align}
which implies that $W_{\sF}(p_1,p_2)\leq 2 W_{\mathcal{H}}(p_1,p_2) \leq 4 \rho(\tilde \epsilon)$, which shows the modulus is small. The final conclusion follows from Proposition~\ref{prop.tTV_general}. 
\end{proof}

\subsection{Modulus bound on $\tTV_\mathcal{H}$}\label{appendix.general_sHandtTV}

It follows from Proposition~\ref{prop.tTV_general} that it suffices to check the following modulus
\begin{align}
    \sup_{p_1, p_2\in \GG: \tTV_\mathcal{H}(p_1, p_2) \leq \eta} L(p_2, \theta^*(p_1))
\end{align}
to guarantee the finite sample error of $\tTV_{\mathcal{H}}$ projection algorithms. The following lemma generalizes the modulus bound via assuming $B(p,\theta)$ is convex in $p$ and minimax theorem.   

\begin{lemma}\label{lemma.minimaxmodulus}
Assume $B(p, \theta)$ is convex in $p$ for all $\theta$ in $\GG(\rho_1, \rho_2, \eta)$, consider the dual representation of $B$:
\begin{align}
    B(p, \theta) = \sup_{f\in\sF_\theta} \bE_p[f(X)] - B^*(f, \theta).
\end{align}
Here $\sF_\theta = \{f \mid B^*(f,\theta) <\infty\}$. We take ${\mathcal{H}} =  \bigcup_{\theta\in \Theta} \sF_\theta$ and assume that for any $\theta \in \Theta$, the minimax theorem holds:
\begin{align}
       \min_{r \leq \frac{p}{1-\epsilon}} \sup_{f\in\sF_{\theta}} \bE_{r}[f(X)] - B^*(f, \theta) = \sup_{f\in\sF_{\theta}} \min_{r \leq \frac{p}{1-\epsilon}} \bE_{r}[f(X)] - B^*(f, \theta).
\end{align}
Then, the modulus is being controlled by $\rho_2$:
\begin{align}
    \sup_{p_1, p_2\in \GG: \tTV_\mathcal{H}(p_1, p_2) \leq \eta} L(p_2, \theta^*(p_1)) \leq \rho_2. 
\end{align}
\end{lemma}
\begin{proof}
Recall that
\begin{align}
    \tTV_{\mathcal{H}}(q, p) & = \sup_{f\in\bigcup_{\theta\in \Theta} \sF_\theta,  t\in \bR}|\bP_p[f(X)\geq t] - \bP_q[f(X) \geq t]|.
\end{align}
From $p, q\in\GG_{\downarrow}^{\TV}$,
\begin{align}
    & \forall r_q \leq \frac{q}{1-\eta},  \sup_{f\in\sF_{\theta^*(q)}} \bE_{r_q}[f(X)] - B^*(f, \theta^*(q)) \leq \rho_1
\end{align}
From Lemma~\ref{lem.tvt_cross_mean}, we know that for any $f \in \mathcal{H}$,  there exists $r_p, r_q$ such that 
\begin{align}
    \bE_{r_p}[f(X)] \leq  \bE_{r_q}[f(X)].
\end{align}
Thus for any $f\in\sF_{\theta^*(q)}$, we have
\begin{align}
    \bE_{r_p}[f(X)] - B^*(f, \theta^*(q)) \leq \bE_{r_q}[f(X)] - B^*(f, \theta^*(q)) \leq \rho_1.
\end{align}
Therefore we know
\begin{align}
    \sup_{f\in\sF_{\theta^*(q)}} \min_{r_p \leq \frac{p}{1-\eta}} \bE_{r_p}[f(X)] - B^*(f, \theta^*(q)) \leq \rho_1.
\end{align}
Since we assumed minimax theorem holds, one has 
\begin{align}
    \min_{r_p \leq \frac{p}{1-\eta}} \sup_{f\in\sF_{\theta^*(q)}} \bE_{r_p}[f(X)] - B^*(f, \theta^*(q)) = \sup_{f\in\sF_{\theta^*(q)}} \min_{r_p \leq \frac{p}{1-\eta}} \bE_{r_p}[f(X)] - B^*(f, \theta^*(q)) \leq \rho_1.
\end{align}
The LHS is exactly the condition in $\GG_{\uparrow}^{\TV}$. Thus from $p \in \GG_{\uparrow}^{\TV}$, we know that
\begin{align}
    L(p, \theta^*(q))\leq \rho_2.
\end{align}
This finishes the proof.
\end{proof}

\subsection{Robust gradient estimation implies robust regression (not necessarily optimally)}

One approach for robust learning is through robust estimation of the gradient of the loss function~\cite{diakonikolas2018sever, prasad2018robust}. In this example, we first show that if the gradient $\mathbb{E}_{p^*}\nabla \ell(\theta,X)$ can be estimated robustly for all $\theta$, the distribution $p^*$ is inside $\GG$ for both $B(p,\theta)$ and $L(p,\theta)$ being the \emph{excess predictive loss} $\bE_p[\ell(\theta, X) -  \ell(\theta^*(p), X)]$, where $\theta^*(p) = \argmin_\theta \bE_p[\ell(\theta,X)]$. Note that linear regression in Theorem~\ref{thm.linearregressiontvtildeproof}  is a special case of $\GG$ for excess predictive loss\footnote{We would like to point out it was shown in literature that any excess predictive loss function can be written as a Bregman divergence and any Bregman divergence can be represented as some excess predictive loss function from proper scoring function construction~\cite{gneiting2007strictly}. Thus we can define the set $\GG$ for Bregman divergence and guarantee the population limit similarly. We omit the details here. }.

\begin{lemma}\label{lem.robust_gradient_estimation}
For any $B(p, \theta) = L(p, \theta) = \bE_p[\ell(\theta, X) -  \ell(\theta^*(p), X)]$,
and any $X\sim p^*$, suppose that the distribution of  random variable $ \nabla \ell(\theta, X)$ is inside $\GG_{\mathsf{mean}}(\rho, \eta)$ for \emph{all} $\theta$, i.e. $\sup_{\theta\in \Theta, r\leq \frac{p^*}{1-\eta}} \|\bE_r[\nabla \ell(\theta, X)] -  \bE_{p^*}[\nabla \ell(\theta, X)] \|_2\leq \rho$. Then 
\begin{enumerate}
    \item If the radius
    of $\Theta$ is upper bounded by $R$, i.e. $\|\theta\|_2\leq R$ for all $\theta \in \Theta$,  then $p^*\in \mathcal{G}^{\TV}(2R \rho, 4R \rho, \eta) $. Thus the population limit is at most $4R\rho$ if $2\epsilon \leq \eta < 1$.
    \item If $\bE_{p^*}[\ell(\theta, X)]$ is $\xi$-strongly convex in $\theta$, i.e.
    \begin{align}
   \bE_{p^*}[ \ell(\theta_1, x)] - \bE_{p^*}[\ell(\theta_2, x)] \geq \nabla \bE_{p^*}[\ell(\theta_2, x)]^\top (\theta_1-\theta_2) + \frac{\xi}{2}\|\theta_1-\theta_2\|_2^2, \forall \theta_1, \theta_2\in\Theta .
    \end{align}
    then $p^*\in  \mathcal{G}^{\TV}(\frac{\rho^2}{2\xi}, \frac{3\rho^2}{\xi}, \eta)$.  Thus the population limit is at most $ \frac{3\rho^2}{\xi}$ if $2\epsilon \leq \eta <1$.
\end{enumerate}
\end{lemma}
\begin{remark}
Note that when we assume $\nabla \ell(\theta, X)$ has bounded covariance, i.e. $\mathsf{Cov}_{p^*}[\nabla \ell(\theta, X)] \preceq \sigma^2 I$ for all $\theta$, we have the distribution of $\nabla \ell(\theta, X) $ inside  $\GG_{\mathsf{mean}}(\rho, \eta)$ with $\rho  = 2\sigma\sqrt{\eta}$ when $\eta < 1/2$ from Theorem~\ref{thm.linearregressiontvtildeproof}. Thus when $\TV(p^*, q)\lesssim \epsilon$ and  both $p^*$ and $q$ has the covariance of gradient $\nabla \ell(\theta, X)$ bounded for all $\theta$, we can control $L(p^*, \theta^*(q))$ using Lemma~\ref{lem.robust_gradient_estimation}. 

If one aims to find some $\hat{\theta}$ such that $\| \mathbb{E}_{p^*}[\nabla \ell(\hat{\theta},X)]\|$ is small, then it suffices to find some $\hat \theta$ and $q$ such that $\| \bE_q[\nabla \ell (\hat \theta, X)]\|$ is small,  $\mathsf{Cov}_q[\nabla \ell (\hat \theta, X)] \preceq \tilde \sigma^2 I$ and $\TV(p^*, q)\lesssim \epsilon$. Indeed, it follows from the modulus of continuity that $\bE_{p^*}[\nabla \ell (\hat \theta, X)]$ is close to $\bE_{q}[\nabla \ell (\hat \theta, X)]$, which implies $\|\bE_{p^*}[\nabla \ell (\hat \theta, X)]\|$ is small. The algorithm in~\cite{diakonikolas2018sever} can be justified with this argument. 
\end{remark}

\begin{proof}
We first check all the distributions that satisfy the first condition are inside $\mathcal{G}_{\downarrow}^{\TV} \bigcap \mathcal{G}_{\uparrow}^{\TV}$. For $\mathcal{G}_{\downarrow}^{\TV}$ and any $r\leq \frac{p}{1-\eta}$, we have 
\begin{align}
     \mathbb{E}_{r}[\ell(\theta^*(p^*), X) - \ell(\theta^*(r), X)] & =  \mathbb{E}_{r}[\int_0^1 \nabla \ell(\theta^*(r) + t(\theta^*(p^*)-\theta^*(r)), X)^{\top}  (\theta^*(p^*)-\theta^*(r)) dt]  \nonumber \\
     & = \int_0^1 \mathbb{E}_{r}[ \nabla \ell(\theta^*(r) + t(\theta^*(p^*)-\theta^*(r)), X)] ^{\top}  (\theta^*(p^*)-\theta^*(r)) dt \nonumber \\
     & \leq  \int_0^1 (\mathbb{E}_{p^*}[ \nabla \ell(\theta^*(r) + t(\theta^*(p^*)-\theta^*(r)), X)] ^{\top}  (\theta^*(p^*)-\theta^*(r)) \nonumber \\
     & \quad + \rho \| \theta^*(p^*) - \theta^*(r)\|_2) dt \nonumber \\
     & \leq \mathbb{E}_{p^*}[\ell(\theta^*(p^*), X) - \ell(\theta^*(r), X)] + 2R \rho \nonumber \\
     & \leq 2R \rho.
\end{align}
The last equation comes from $\mathbb{E}_{p}[\ell(\theta^*(p^*), X) - \ell(\theta^*(r), X)] \leq 0$. This shows that $p^* \in \mathcal{G}_{\downarrow}^{\TV}(\rho, \eta)$.

Then we verify that the set is also inside $\mathcal{G}_{\uparrow}^{\TV}$, which is defined such that  for any $r\leq \frac{p^*}{1-\eta}$ and any $\theta$,
\begin{align}
& \mathbb{E}_{r}[\ell(\theta, X) - \ell(\theta^*(r), X)] \leq 2R\rho \nonumber \\
\Rightarrow  &
    \mathbb{E}_{p^*}[\ell(\theta, X) - \ell(\theta^*(p^*), X)]\leq 4R\rho.
\end{align}
From $\|\bE_r[\nabla \ell(\theta, X)] -  \bE_{p^*}[\nabla \ell(\theta, X)] \|_2\leq \rho$, we have 
\begin{align}
     \mathbb{E}_{p^*}[\ell(\theta, X) - \ell(\theta^*(p^*), X)] & =  \mathbb{E}_{p^*}[\int_0^1 \nabla \ell(\theta^*(p^*) + t(\theta-\theta^*(p^*)), X)^{\top}  (\theta-\theta^*(p^*)) dt]  \nonumber \\
     & = \int_0^1 \mathbb{E}_{p^*}[ \nabla \ell(\theta^*(p^*) + t(\theta-\theta^*(p^*)), X)] ^{\top}  (\theta-\theta^*(p^*)) dt \nonumber \\
     & \leq  \int_0^1 (\mathbb{E}_{r}[ \nabla \ell(\theta^*(p^*) + t(\theta-\theta^*(p^*)), X)] ^{\top}  (\theta-\theta^*(p^*)) + \rho \| \theta - \theta^*(p^*)\|_2) dt \nonumber \\
     & \leq  \mathbb{E}_{r}[\ell(\theta, X) - \ell(\theta^*(p), X)]  + 2R \rho \nonumber \\
     & \leq  \mathbb{E}_{r}[\ell(\theta, X) - \ell(\theta^*(r), X)]  + 2R \rho \nonumber \\
     & \leq 4R \rho.
\end{align}
This shows all distributions that satisfy condition 1 are inside $\mathcal{G}^{\TV}(2R \rho, 4R \rho, \eta)$. Thus the population information-theoretic limit is $4R\rho$ if $2\epsilon \leq \eta < 1$.

Now we check that all the distributions that satisfy the second condition are inside $\mathcal{G}^{\TV}(\frac{\rho^2}{2\xi}, \frac{3\rho^2}{\xi}, \eta)$. For  $\mathcal{G}_{\downarrow}^{\TV}$, similar to the previous proof, we have
\begin{align}
     &\mathbb{E}_{r}[\ell(\theta^*(p^*), X) - \ell(\theta^*(r), X)] \nonumber \\ & =  \mathbb{E}_{r}[\int_0^1 \nabla \ell(\theta^*(r) + t(\theta^*(p^*)-\theta^*(r)), X)^{\top} \cdot (\theta^*(p^*)-\theta^*(r)) dt]  \nonumber \\
     & = \int_0^1 \mathbb{E}_{r}[ \nabla \ell(\theta^*(r) + t(\theta^*(p^*)-\theta^*(r)), X)] ^{\top} \cdot (\theta^*(p^*)-\theta^*(r)) dt \nonumber \\
     & \leq  \int_0^1 (\mathbb{E}_{p^*}[ \nabla \ell(\theta^*(r) + t(\theta^*(p^*)-\theta^*(r)), X)] ^{\top} \cdot (\theta^*(p^*)-\theta^*(r)) \nonumber \\
     & \quad + \rho \| \theta^*(p) - \theta^*(r)\|_2) dt \nonumber \\
     & \leq  \int_0^1 (\mathbb{E}_{p^*}[ \nabla \ell(\theta^*(r) + t(\theta^*(p^*)-\theta^*(r)), X)] ^{\top} \cdot (\theta^*(p^*)-\theta^*(r))) dt \nonumber \\
     & \quad + \rho\sqrt{\frac{2}{\xi}}  \sqrt{ \mathbb{E}_{p^*}[\ell(\theta^*(r), X) - \ell(\theta^*(p^*), X)]} \nonumber \\
     & = - \mathbb{E}_{p^*}[\ell(\theta^*(r), X) - \ell(\theta^*(p^*), X)]  + \rho\sqrt{\frac{2}{\xi}} \sqrt{ \mathbb{E}_{p^*}[\ell(\theta^*(r), X) - \ell(\theta^*(p^*), X)]} \nonumber \\
     & \leq \frac{\rho^2}{2\xi}.
\end{align}

Here we use the property of strong convexity to get 
\begin{align}
    \mathbb{E}_{p^*}[\ell(\theta^*(r), X) - \ell(\theta^*(p^*), X)] & \geq \nabla \mathbb{E}_{p^*}[\ell(\theta^*(p), X)](\theta^*(r)- \theta^*(p)) + \frac{\xi}{2} \|\theta^*(r) - \theta^*(p)\|_2^2 \nonumber \\
   & = \frac{\xi}{2} \|\theta^*(r) - \theta^*(p)\|_2^2. 
\end{align}

For $\mathcal{G}_{\uparrow}^{\TV}$, if $\mathbb{E}_{r}[\ell(\theta, X) - \ell(\theta^*(r), X)] \leq \frac{\rho^2}{2\xi}$,  we have
\begin{align}
     &\mathbb{E}_{p^*}[\ell(\theta, X) - \ell(\theta^*(p^*), X)]\nonumber \\ & =  \mathbb{E}_{p^*}[\int_0^1 \nabla \ell(\theta^*(p^*) + t(\theta-\theta^*(p^*)), X)^{\top} \cdot (\theta-\theta^*(p^*)) dt]  \nonumber \\
     & = \int_0^1 \mathbb{E}_{p^*}[ \nabla \ell(\theta^*(p^*) + t(\theta-\theta^*(p^*)), X)] ^{\top} \cdot (\theta-\theta^*(p^*)) dt \nonumber \\
     & \leq  \int_0^1 (\mathbb{E}_{r}[ \nabla \ell(\theta^*(p^*) + t(\theta-\theta^*(p^*)), X)] ^{\top} \cdot (\theta-\theta^*(p^*)) + \rho(\epsilon) \| \theta - \theta^*(p^*)\|_2) dt \nonumber \\
     & \leq \frac{\rho^2}{2\xi} + \rho  \sqrt{\frac{2}{\xi}}  \sqrt{ \mathbb{E}_{p^*}[\ell(\theta^*(r), X) - \ell(\theta^*(p^*), X)]}.
\end{align}
Solving this inequality, we have 
\begin{align}
    \mathbb{E}_{p^*}[\ell(\theta, X) - \ell(\theta^*(p^*), X)] & \leq \frac{(\rho\sqrt{2/\xi}+2\rho\sqrt{1 / \xi})^2}{4} \nonumber \\
    & <\frac{3\rho^2}{\xi}.
\end{align}
This shows any distribution $p^*$ that satisfies condition 2 are inside $  \mathcal{G}_{\mathsf{pred}}^{\TV}(\frac{\rho^2}{2\xi}, \frac{3\rho^2}{\xi}, \eta)$.  Thus the population information-theoretic limit is $ \frac{3\rho^2}{\xi}$ assuming $2\epsilon  \leq \eta < 1 $.
\end{proof}

\subsection{Further discussion on robust classification}\label{appendix.classification}

We study the sufficient conditions that can ensure a distribution is in $\GG$ in Definition~\ref{def.G_TV} where $L(p,\theta) = \bE_p[\mathbbm{1}(YX^\top \theta \leq 0)]$ is the zero-one loss function for linear classification, where $(X,Y) \in \mathbf{R}^d \times \{-1,1\}$. We always assume that we have augmented $X$ by an additional dimension of constant $1$ to avoid the non-zero offset term in the classifier. 

\subsubsection{Bridge function $B(p,\theta)$ is zero-one loss and linearly separable }

We first consider the setting that $B(p,\theta) =L(p,\theta) = \bE_p[\mathbbm{1}(Y X^\top \theta \leq  0)]$. Assume $\theta\in\Theta = \{ \theta\in\bR^d: \| \theta\|_2\leq 1\}$. We consider the special case of $\rho_1 = \rho_2 = 0$ in Definition~\ref{def.G_TV}:
\begin{align} 
    \GG_{\downarrow}^{\TV}(0, \eta) & \triangleq \{ p \mid   \sup_{r\leq\frac{p}{1-\eta}}
    \bE_r[\mathbbm{1}(Y X^\top \theta^*(p) \leq  0)]\leq 0\}, \label{eqn.gdownclassexa1} \\
    \GG_{\uparrow}^{\TV}(0, 0, \eta) & \triangleq \{ p \mid  \forall \theta\in\Theta, \forall r \leq \frac{p}{1-\eta}, \left( \bE_r[\mathbbm{1}(Y X^\top \theta \leq  0)]\leq 0 \Rightarrow  \bE_p[\mathbbm{1}(Y X^\top \theta \leq  0)]\leq 0\right) \}, \label{eqn.gupclassexa1}
\end{align}
where $\theta^*(p) = \argmin_{\theta\in\Theta} \bE_p[\mathbbm{1}(Y X^\top \theta \leq  0)]$. 

We investigate the sufficient conditions that imply $p\in  \GG_{\downarrow}^{\TV}(0, \eta) \cap  \GG_{\uparrow}^{\TV}(0, 0, \eta)$. 

\begin{proposition}
Suppose distribution $p$ of $(X,Y) \in \mathbf{R}^d \times \{-1,1\}$ satisfies the following properties:
\begin{itemize}
    \item $(X,Y)$ is linearly separable under $p$;
    \item There  does not exist $\theta \in \Theta$ such that $\bP_p(YX^\top \theta \leq 0)\in(0, \eta]$.
\end{itemize}

Then, $p\in  \GG_{\downarrow}^{\TV}(0, \eta) \cap  \GG_{\uparrow}^{\TV}(0, 0, \eta)$ defined in~(\ref{eqn.gdownclassexa1}) and~(\ref{eqn.gupclassexa1}). 
\end{proposition}
\begin{proof}
If $(X,Y)$ is linearly separable under $p$, then 
\begin{align}
    \bE_p[\mathbbm{1}(Y X^\top \theta^*(p) \leq  0)] = 0. 
\end{align}
Thus, for any $r\leq \frac{p}{1-\eta}$, 
\begin{align}
    \bE_r[\mathbbm{1}(Y X^\top \theta^*(p) \leq  0)] \leq \frac{1}{1-\eta}   \bE_p[\mathbbm{1}(Y X^\top \theta^*(p) \leq  0)] = 0,
\end{align}
which shows that $p$ being linearly separable implies that $p\in\GG_\downarrow(0, \eta)$. It suffices to check that $p\in\GG_{\uparrow}$. We need to show that for any $\theta$ and any $r\leq \frac{p}{1-\eta}$, $ \bE_r[\mathbbm{1}(Y X^\top \theta \leq  0)]\leq 0 $ implies $ \bE_p[\mathbbm{1}(Y X^\top \theta \leq  0)]\leq 0$. 

For any $\theta\in\bR^d$, $\|\theta\|_2=1$, by assumption only two situations will occur: $\bE_p(\mathbbm{1}(YX^\top \theta \leq 0))=0$ or $\bE_p(\mathbbm{1}(YX^\top \theta \leq 0))>\eta$. If we observe $\bE_r[\mathbbm{1}(YX^\top \theta \leq 0)] = 0$, 
then there must be $\bE_p[\mathbbm{1}(YX^\top \theta \leq 0)] = 0$ since deletion would at most decrease the cost by $\eta$. 

\end{proof}

Intuitively, the second sufficient condition guarantees that deleting $\eta$ fraction of mass cannot decrease the loss from non-zero to zero. It is also necessary: if there exists some $\theta$ such that $\bE_p(\mathbbm{1}(YX^\top \theta \leq 0))\in(0, \eta]$, then the adversary can delete all the mass with $\mathbbm{1}(YX^\top \theta \leq 0)$ to construct a linearly separable distribution, then the implication in $\GG_\uparrow$ would fail.

However, the second sufficient condition is in general hard to be satisfied for continuous distributions in high dimensions: one can always rotate $\theta$ to satisfy $\bP_p(YX^\top \theta \leq 0)\in(0, \eta]$ since we know $\bP_p(YX^\top \theta^*(p) \leq 0)=0$ and $\bP_p(YX^\top (-\theta^*(p)) \leq 0)=1$. In next section, we change $B$ to hinge loss and show that the corresponding sufficient condition can be easier to satisfy.

\subsubsection{Bridge function $B(p,\theta)$ is hinge loss}

Take $B(p, \theta) = \bE_p[\max(0, 1- YX^{\top}\theta)]$, $L(p, \theta) = \bE_p[\mathbb{1}[YX^{\top}\theta \leq 0]]$. Assume $\theta\in\Theta = \{ \theta\in\bR^d: \| \theta\|_2\leq 1\}$. The set $\GG(\rho_1, \rho_2,\eta)$ is defined as  $\mathcal{G}^{\TV}_{\downarrow}(\rho_1,\eta) \cap \mathcal{G}^{\TV}_{\uparrow}(\rho_1,\rho_2,\eta)$, where
\begin{align}
    \GG_{\downarrow}^{\TV}(\rho_1,\eta) & = \{ p \mid \sup_{r\leq\frac{p}{1-\eta}}  \bE_r[\max(0, 1-YX^{\top}\theta^*(p))] \leq \rho_1\}. \label{eqn.gdownhingedefin1}\\
    \GG_{\uparrow}^{\TV}(\rho_1,\rho_2,\eta) & = \{ p \mid \forall \theta \in \Theta, \forall r \leq \frac{p}{1-\eta},  \bigg( \bE_r[\max(0, 1-YX^{\top}\theta)] \leq \rho_1  \Rightarrow   \bP_p(YX^{\top}\theta \leq 0) \leq \rho_2\}\bigg), \label{eqn.guphindgedefin2}
\end{align}
where $\theta^*(p) = \argmin_{\theta \in \Theta} \bE_p[\max(0, 1- Y X^\top \theta )]$. 

We investigate the sufficient conditions for a distribution to be inside $\GG(\rho_1, \rho_2, \eta)$ as follows,
\begin{proposition}
Suppose distribution $p$ of $(X,Y) \in \mathbf{R}^d \times \{-1,1\}$ satisfies the following properties:
\begin{itemize}
    \item $ \bE_p[\max(0, 1-YX^{\top}\theta^*(p))]\leq (1-\eta)\rho_1.$
   \item $\forall \theta\in \Theta,  \left( \bP_{p}(Y X^{\top}\theta\leq \frac{1}{2}) \leq  \eta + 2(1-\eta)\rho_1  \Rightarrow \bP_p(Y X^{\top}\theta \leq 0) \leq \rho_2\right).$
\end{itemize}
Then, $p\in  \GG(\rho_1, \rho_2, \eta)$ defined in~(\ref{eqn.gdownhingedefin1}) and~(\ref{eqn.guphindgedefin2}).  
\end{proposition}
\begin{proof}
We first show that $p\in \GG_{\downarrow}^{\TV}$. For any $r\leq \frac{p}{1-\eta}$, we have
\begin{align}
    \bE_r[\max(0, 1-X^{\top}\theta^*(p))] & \leq \frac{\bE_p[\max(0, 1-X^{\top}\theta^*(p))]}{1-\eta} \\
    & \leq \frac{(1-\eta)\rho_1}{1-\eta} \\
    & = \rho_1,
\end{align}
which implies that $p\in \GG_{\downarrow}^{\TV}(\rho_1,\eta)$. 

Next we show that $p\in \GG_{\uparrow}^{\TV}$. Assume there exists some $\theta\in\Theta$ and $r\leq \frac{p}{1-\eta}$ such that $\bE_r[\max(0, 1-YX^{\top}\theta)] \leq  \rho_1 $. Then we claim that $\bP_{p}(YX^{\top}\theta\leq \frac{1}{2}) \leq  \eta + 2(1-\eta)\rho_1$ must hold. If it does not hold, then $\bP_{p}(YX^{\top}\theta\leq \frac{1}{2}) >  \eta + 2(1-\eta)\rho_1$ would imply that $\bP_{r}(YX^{\top}\theta\leq \frac{1}{2}) >  2\rho_1$.
Therefore $ \bE_{r}[\max(0, 1-YX^{\top}\theta)] >  \rho_1$,
which contradicts the assumption. 
\end{proof}

Here the first condition is a standard assumption on the margin, and the second condition can be verified by the following condition: for any fixed $\theta \in \Theta$, denote $Z =  YX^\top \theta$, and the left $\rho_2$ quantile of $Z$ as $q$, then as long as the following inequality holds, the second condition holds:
\begin{align}
    \bP_p(Z\in[q, q+1/2]) \geq \eta + 2(1-\eta)\rho_1.
\end{align}
When $XY$ is isotropic Gaussian distribution, the above conditions are satisfied for certain parameters.

In robust classification case, since the target loss function is zero-one loss, the robust error would be at most $\epsilon$ given $\epsilon$ perturbation in total variation. Thus one needs more stringent results for robustness to make the guarantee meaningful. Here our condition on concentration allows $\rho_2$ to be $0$, which provides strong guarantee for the classification error.

\begin{remark}
In the literature of robust classification~\cite{klivans2009learning, awasthi2014power, diakonikolas2018learning}, it is usually assumed that $P_X$ instead of $P_{X|Y}$ satisfies some nice concentration properties. However, one can easily create a toy example where $P_{X|Y=1}$ and $P_{X|Y=-1}$ are well separated and satisfy our sufficient conditions, but $P_X$ does not have good concentration property.

\end{remark}

\subsubsection{Estimating Chow-parameters implies robust classification under polynomial threshold function}

As another example under our framework of $\GG$, it is proposed in~\citep[Lemma 3.4]{diakonikolas2018learning} that with appropriate estimate of Chow-parameters, one can guarantee certain level of classification accuracy if the classification function is the sign of some degree-$d$ polynomial threshold functions $f(x)$. Thus the classification problem can be reduced to robustly estimating the mean of $\bE_{p^*}[f(X)t(X)]$ where  $t(x)$ is one of the polynomial functions with degree at most $d$. Thus the result can be incorporated into our framework when $B$ measures the loss in estimating the Chow parameters, and $L$ is the classification zero-one loss, and the conditions in~\cite{diakonikolas2018learning} serves as sufficient conditions for $p^*$ being inside $\GG$.

 \section{Related discussions and remaining proofs in Section~\ref{sec.expand}}

\subsection{Key Lemmas}

\subsubsection{Generalized Modulus of Continuity}
The following Lemma is essentially the same  as~\citep[Corollary A.25]{diakonikolas2017being}. It shows that the generalized modulus of continuity between bounded covariance set and resilient set can be controlled.  For completeness we present the proof here.
\begin{lemma}[]\label{lem.empirical_identity_cov_modulus}
For some constant non-negative $\rho_1, \rho_2, \tau$, assume $\tau \geq \epsilon$,
denote $\mu_p = \bE_p[X]$. Define 
\begin{align}
    \mathcal{G}_1 &= \{p:   \forall r \leq \frac{p}{1-\epsilon},  \| \mu_r - \mu_p\|_2 \leq  \rho_1 ,   \lambda_{\mathsf{min}}( \mathbb{E}_r[(X - \mu_p)(X - \mu_p)^{\top}]) \geq  1- \rho_2 \}\\ 
    \mathcal{G}_2 &= \{p:  \| \mathbb{E}_p[(X - \mu_p)(X - \mu_p)^{\top}]\|_2 \leq 1+ \tau \}.
\end{align}
Here $\lambda_{\mathsf{min}}(A)$ is the smallest eigenvalue of symmetric matrix $A$. Then, for any $\epsilon \in [0,1)$ we have
\begin{align}
    \sup_{p\in\GG_1, q\in\GG_2, \TV(p, q)\leq \epsilon} \|\bE_p[X] - \bE_q[X]\|_2   \leq  C\cdot \left( \sqrt{\frac{(\tau+\rho_2)\epsilon}{1-\epsilon}} + \max(1, \frac{\epsilon}{1-\epsilon})\rho_1\right).
\end{align}
Here $C$ is some universal constant.
\end{lemma}
\begin{proof}
Assume $p\in \GG_1, q\in\GG_2$, $p\neq q$. 
Without loss of generality, we assume $\mu_p=0$. From $\TV(p, q)= \epsilon_0 \leq \epsilon$, we construct distribution $r = \frac{\min(p, q)}{1-\epsilon_0}$. By Lemma~\ref{lem.deletionrproperties} we know that 
$r\leq \frac{p}{1-\epsilon_0}$, $r\leq \frac{q}{1-\epsilon_0}$.
Denote $\tilde r = (1-\epsilon_0)r$. Consider measure $p-\tilde r, q-\tilde  r$. We have $\mu_q = \mu_p - \mu_{p-\tilde r} + \mu_{q-\tilde r} =  -\mu_{p-\tilde r} + \mu_{q-\tilde r}$. Note that $\| \mu_{p-\tilde r}\|_2 = \| \mu_p - \mu_{\tilde r}\|_2  = \| \mu_{\tilde r} \|_2 \leq (1-\epsilon_0)\rho_1 \leq \rho_1$.
For any $v\in\bR^d, \|v\|_2=1$, we have
\begin{align}
    v^{\top}\Sigma_qv^{\top} & = v^{\top}(\bE_q[XX^{\top}] -\mu_q\mu_q^{\top})v \nonumber \\
    & = v^{\top}(\bE_{\tilde r}[XX^{\top}] + \bE_{q-\tilde r}[XX^{\top}] - (\mu_{q-\tilde r}-\mu_{p-\tilde r})(\mu_{q-\tilde r}-\mu_{p-\tilde r})^{\top})v \nonumber \\
    & \geq (1-\rho_2)(1-\epsilon_0) + \bE_{q-\tilde r}[(v^{\top}X)^2] - (v^{\top}\mu_{q-\tilde r})^2 + 2 v^{\top}\mu_{q-\tilde r}v^{\top}\mu_{p-\tilde r} - (v^{\top}\mu_{p-\tilde r})^2 \nonumber \\
    & \geq  1-\rho_2-\tau  + \bE_{q-\tilde r}[(v^{\top}X)^2] - (v^{\top}\mu_{q-\tilde r})^2 - 2 \|\mu_{q-\tilde r}\|_2 \|\mu_{p-\tilde r}\|_2 - \|\mu_{p-\tilde r}\|^2 \nonumber \\
    & \geq  1-\rho_2-\tau  + \bE_{q-\tilde r}[(v^{\top}X)^2] - (v^{\top}\mu_{q-\tilde r})^2 - 2\rho_1 \|\mu_{q-\tilde r}\|_2 - \rho_1^2.\nonumber
\end{align}
Here we use the fact that $\epsilon_0 \leq \epsilon \leq \tau$. Denote $b_q = \frac{q-\tilde r}{\epsilon_0}$. Then $b_q$ is a distribution. If $\mu_{b_q} = 0$, then we already know that $\|\mu_q - \mu_r\|\leq \epsilon_0 \|\mu_{b_q}\|_2 =0$. Otherwise
we take $v = \frac{\mu_{b_q}}{\|\mu_{b_q}\|_2}$.
Then we can see $ \bE_{q-\tilde r}[(v^{\top}X)^2]= \epsilon_0\bE_{b_q}[(v^{\top}X)^2]\geq \epsilon_0 \|\mu_{b_q}\|_2^2$. 
From $q\in \GG_2$, we know  that $v^{\top}\Sigma_qv\leq 1 + \tau$. Thus
\begin{align}
    (\epsilon_0 - \epsilon^2_0)\|\mu_{b_q}\|^2_2 -  2 \epsilon_0 \rho_1 \|\mu_{b_q}\|_2 \leq \rho_1^2 + \rho_2  + 2\tau.
\end{align}
Solving the inequality, we derive that
\begin{align}
  \|\mu_q - \mu_r\|_2 \leq \epsilon_0 \|\mu_{b_q}\|_2 \leq C(\sqrt{\frac{\epsilon_0}{1-\epsilon_0}}( \sqrt{\tau+\rho_2} + \rho_1)+ \frac{\epsilon_0}{1-\epsilon_0}\rho_1)
\end{align}
where $C$ is some universal constant.
Thus we can conclude
\begin{align}
    \| \mu_p - \mu_q\|_2\leq \| \mu_p - \mu_r\|_2 + \| \mu_r - \mu_q\|_2 \leq  C\cdot \left( \sqrt{\frac{(\tau+\rho_2)\epsilon}{1-\epsilon}} + \max(1, \frac{\epsilon}{1-\epsilon})\rho_1\right).\nonumber
\end{align}
\end{proof}
\begin{remark}
The lemma is key for proving near-optimal modulus when $\GG_1$ is empirical distribution of sub-Gaussian (or bounded $k$-th moment) with identity covariance, and $\GG_2$ is bounded covariance matrix. 

When both $p$ and $q$ have bounded covariance, we have $\rho_1 = \frac{\sigma_1 \sqrt{\epsilon}}{1-\epsilon} $, $\rho_2 = 0$, $\tau = \sigma_2-1$. The result recovers the population modulus for bounded covariance set in Theorem~\ref{thm.tTV_mean}. In the meantime the coefficient of $\sigma_2\sqrt{\epsilon}$ improves from $1/(1-\epsilon)$ to $1/\sqrt{1-\epsilon}$, which can be much better when $\epsilon$ is large. 
\end{remark}
Note that compared with 
Furthermore, we show a stronger lemma that the generalized modulus of continuity for the same sets under $\tTV_\mathcal{H}$ distance is also bounded. 
The Lemma is critical in showing the generalized modulus of continuity for both bounded $k$-th moment distribution and sub-Gaussian distribution with identity covariance assumption under $\tTV_\mathcal{H}$ distance. 
\begin{lemma}\label{lem.tTV_identitycov_modulus}
For some non-negative constant $\rho_1, \rho_2, \tau$, assume $\tau \geq \epsilon$.
We denote $\mu_p = \bE_p[X]$. Define 
\begin{align}
     \mathcal{G}_1 & =  \{p: \forall r \leq \frac{p}{1-2\epsilon},  \|\mu_r - \mu_p\|_2 \leq \rho_1, 
     \lambda_{\mathsf{min}}(\mathbb{E}_r[(X - \mu_p)(X - \mu_p)^{\top}]) \geq 1- \rho_2\}, \\
    \GG_2 &= \{p:  \| \mathbb{E}_p[(X - \mu_p)(X - \mu_p)^{\top}]\|_2 \leq 1+ \tau\}, \\
    \tTV_\mathcal{H}(q, p) & = \sup_{v\in\bR^d, t\in\bR} |\bP_p(v^\top X \geq t) - \bP_q(v^\top X \geq t) |.\nonumber
\end{align}
Here $\lambda_{\mathsf{min}}(A)$ is the smallest eigenvalue of symmetric matrix $A$. Assume $\epsilon < 1/3$,
then there exists a universal constant $C$ such that
\begin{align}
\sup_{q\in \GG_2, p\in \GG_1, \tTV_\mathcal{H}(q,p)\leq \epsilon} \|\mu_{p} - \mu_q\|_2 \leq C\cdot (\sqrt{\tau\epsilon } + \sqrt{\rho_2 \epsilon } + \rho_1),
\end{align}
where $C$ is some universal constant.
\end{lemma}

\begin{proof}
From Lemma~\ref{lem.cvx_mean_resilience} and $\tTV_\mathcal{H}(q, p)\leq \epsilon <2\epsilon$,
 for $v = \frac{\mu_p-\mu_q}{\|\mu_p-\mu_q\|_2}$, there exist $r_p\leq \frac{p}{1-2\epsilon}, r_q\leq \frac{q}{1-2\epsilon}$,
\begin{align}
    \bE_{r_p}[\langle X-\mu_q, v\rangle^2] \leq  \bE_{r_q}[\langle X-\mu_q, v\rangle^2]
\end{align}
Thus we have
\begin{align}
    \bE_q[\langle X-\mu_q, v \rangle^2] &  \geq (1-2\epsilon)\bE_{r_q}[\langle X-\mu_q, v \rangle^2] \nonumber \\
    & \geq (1-2\epsilon)\bE_{r_p}[\langle X-\mu_q, v \rangle^2] \nonumber \\
    & = (1-2\epsilon)\bE_{r_p}[\langle X-\mu_p, v \rangle^2 + \langle \mu_p-\mu_q, v \rangle^2 + 2\langle X-\mu_p, v \rangle \cdot \langle \mu_p-\mu_q, v \rangle ] \nonumber \\
    & \geq (1-2\epsilon)(1 - \rho_2 + \|  \mu_p-\mu_q\|_2^2 - 2\rho_1 \| \mu_p - \mu_q\|_2).\nonumber
\end{align}
The last inequality comes from  the fact that $p\in\GG_1$.
From $q\in \GG_2$, we know
\begin{align}
   & \|\bE_q[\langle X-\mu_q, v \rangle^2] \leq 1 + \tau \nonumber \\
    \Rightarrow & \| \mu_p-\mu_q\|_2^2 - 2\rho_1 \| \mu_p - \mu_q\|_2 \leq  \frac{1 + \tau}{1-2\epsilon} - 1 + \rho_2 < \rho_2 + 9\tau. \nonumber
\end{align}
Here we use the assumption that $\epsilon < 1/3$.
Solving the equation, we can derive
\begin{align}
    \| \mu_p - \mu_q\|_2 \leq 5\rho_1 + 4\sqrt{\rho_2 + 9\tau}. 
\end{align}

Next we show that if $\|\mu_p - \mu_q\|_2 \geq 11\rho_1 + 9\sqrt{ (\rho_2 + 9\tau)\epsilon}$, we have 
\begin{align}
    \| \bE_q[(X-\mu_q)(X-\mu_q)^{\top}] \|_2 > 1 + \tau.
\end{align}

Consider the unit vector $v = \frac{\mu_q - \mu_p}{\|\mu_q - \mu_p \|_2}$.
For $a\in[0, 1], b\in[0, 1]$, 
consider the random variable $\frac{av^{\top}(X-\mu_p)}{\rho_1}+\frac{b(1-(v^{\top}(X-\mu_p))^2)}{\rho_2}$. we have
\begin{align}
   & |\bP_p(\frac{av^{\top}(X-\mu_p)}{\rho_1}+\frac{b(1-(v^{\top}(X-\mu_p))^2)}{\rho_2}\geq t) - \bP_q(\frac{av^{\top}(X-\mu_p)}{\rho_1}+\frac{b(1-(v^{\top}(X-\mu_p))^2)}{\rho_2}\geq t)|
  \leq  2\epsilon. \nonumber
\end{align}
To see this, indeed, for any given $t\in\bR^d$, when the equation $\frac{av^{\top}(X-\mu_p)}{\rho_1}+\frac{b(1-(v^{\top}(X-\mu_p))^2)}{\rho_2} = t$ has two solutions (denoted as $x_0,x_1$), we have
\begin{align}
   & |\bP_p(\frac{av^{\top}(X-\mu_p)}{\rho_1}+\frac{b(1-(v^{\top}(X-\mu_p))^2)}{\rho_2}\geq t) - \bP_q(\frac{av^{\top}(X-\mu_p)}{\rho_1}+\frac{b(1-(v^{\top}(X-\mu_p))^2)}{\rho_2}\geq t)| \nonumber \\
   & \leq |\bP_p(v^{\top}X\leq x_0) - \bP_q(v^{\top}X\leq x_0) |+ |\bP_p(v^{\top}X\geq x_1) - \bP_q(v^{\top}X\geq x_1) | \nonumber \\
   & \leq 2\epsilon.\nonumber
\end{align}
When the equation has one or zero solution and $b\neq 0$, the difference is 0. When the equation has one solution and $b=0$, we know that the difference is bounded by $\epsilon$.
For any $a, b \in [0, 1]$, from Lemma~\ref{lem.cvx_mean_resilience}, we know that there exists $r_p \leq \frac{p}{1-2\epsilon}$, $r_q \leq \frac{q}{1-2\epsilon}$, such that 
\begin{align}
    \bE_{r_q}[\frac{av^{\top}(X-\mu_p)}{\rho_1}+\frac{b(1-(v^{\top}(X-\mu_p))^2)}{\rho_2}] & \leq \bE_{r_p}[\frac{av^{\top}(X-\mu_p)}{\rho_1}+\frac{b(1-(v^{\top}(X-\mu_p))^2)}{\rho_2}] \nonumber \\
    & \leq a + b \leq 2.
\end{align}
This is from that $p\in \GG_1$.
Thus we have
\begin{align}
    \max_{a\in[0, 1], b\in[0, 1]} \min_{r_q\leq \frac{q}{1-2\epsilon}} \bE_{r_q}[\frac{av^{\top}(X-\mu_p)}{\rho_1}+\frac{b(1-(v^{\top}(X-\mu_p))^2)}{\rho_2}] \leq 2.\nonumber
\end{align}
By minimax theorem, we can see that 
\begin{align}\nonumber
    \min_{r_q\leq \frac{q}{1-2\epsilon}}  \max_{a\in[0, 1], b\in[0, 1]} \bE_{r_q}[\frac{av^{\top}(X-\mu_p)}{\rho_1}+\frac{b(1-(v^{\top}(X-\mu_p))^2)}{\rho_2}] \leq 2.
\end{align}
Thus there exists some $r_q\leq \frac{q}{1-2\epsilon}$, such that for any $ a, b \in [0, 1]$,  
\begin{align}\nonumber
 \bE_{r_q}[\frac{av^{\top}(X-\mu_p)}{\rho_1}+\frac{b(1-(v^{\top}(X-\mu_p))^2)}{\rho_2}] \leq 2.
\end{align}
By taking $a=0, b=1$ and $a=1, b=0$, we have
\begin{align}\label{eqn.proof_tTV_modulus_rq1}
 \bE_{r_q}[(v^{\top}(X-\mu_p))^2)] & \geq 1 - 2\rho_2,\\\label{eqn.proof_tTV_modulus_rq2}
 \bE_{r_q}[{v^{\top}(X-\mu_p)}] & \leq 2\rho_1.
\end{align}

Denote $\tilde r_q = (1-2\epsilon)r_q$, then we have $q \geq \tilde r_q$.
To bound from below the maximum eigenvalue, it is sufficient to lower bound the term
\begin{align}\nonumber
    v^{\top} \bE_q[(X-\mu_q)(X-\mu_q)^{\top}] v = v^{\top} \bE_{\tilde r_q}[(X-\mu_q)(X-\mu_q)^{\top}] v + v^{\top} \bE_{q-\tilde r_q}[(X-\mu_q)(X-\mu_q)^{\top}] v. 
\end{align}
Now we bound the two terms separately, first note that
\begin{align}
    \bE_{q-\tilde r_q}[\langle X - \mu_q, v\rangle] & \geq \bE_{q-\tilde r_q}[\langle X - \mu_p, v\rangle]  - \int  (q-\tilde r_q)\langle \mu_p - \mu_q, v\rangle \nonumber \\
    & =  \bE_{q}[\langle X - \mu_p, v\rangle] - \bE_{\tilde r_q}[\langle X - \mu_p, v\rangle]  - \int (q-\tilde r_q) \langle \mu_p - \mu_q, v\rangle \nonumber \\
     & \geq \|\mu_q - \mu_p \|_2 - \bE_{\tilde r_q}[\langle X - \mu_p, v\rangle]  - 2\epsilon \|\mu_q - \mu_p \|_2 \nonumber \\
    & \geq (1-2\epsilon)\|\mu_q - \mu_p \|_2 - (1-2\epsilon)2\rho_1   \nonumber \\
    & > (1-2\epsilon)(11\rho_1 + 9\sqrt{(\rho_2+9\tau)\epsilon} - 2\rho_1) \nonumber \\
    & = (1-2\epsilon)(9\rho_1 + 9\sqrt{(\rho_2+9\tau)\epsilon}) \nonumber \\
    & > \frac{1}{3}(9\rho_1 + 9\sqrt{(\rho_2+9\tau)\epsilon}) \nonumber \\
    & > 3(\rho_1 +\sqrt{(\rho_2+9\tau)\epsilon})) \nonumber
\end{align}
By Cauchy-Schwarz inequality,
\begin{align}
    \bE_{q-\tilde r_q}[\langle X - \mu_q, v\rangle^2] \int (q-\tilde r_q) \geq  \Big(\bE_{q-\tilde r_q}[\langle X - \mu_q, v\rangle]\Big)^2 > 9(\rho_1+\sqrt{(\rho_2 + 9 \tau)\epsilon})^2.\nonumber
\end{align}
Thus 
\begin{align}
    \bE_{q-\tilde r_q}[\langle X - \mu_q, v\rangle^2]  > \frac{9(\rho_1+\sqrt{(\rho_2 + 9 \tau)\epsilon})^2}{\epsilon}.\nonumber
\end{align}
Now we bound the term $v^{\top} \bE_{\tilde r_q}[(X-\mu_q)(X-\mu_q)^{\top}] v $. We can see from Equation~\eqref{eqn.proof_tTV_modulus_rq1} and~\eqref{eqn.proof_tTV_modulus_rq2} that
\begin{align}
    \bE_{\tilde r_q}[\langle X-\mu_q, v  \rangle^2] & =   \bE_{\tilde r_q}[\langle X-\mu_p, v  \rangle^2 + \langle \mu_p - \mu_q, v\rangle^2 + 2\langle X-\mu_p, v\rangle \langle \mu_p - \mu_q, v \rangle] \nonumber \\
    & \geq (1-2\epsilon)(1-2\rho_2   - 4\rho_1 \cdot  (5\rho_1 + 4\sqrt{\rho_2+9\tau} ) \nonumber \\
    & > 1 - 2\epsilon - 2\rho_2 - 4\rho_1 \cdot  (5\rho_1 + 4\sqrt{\rho_2+9\tau} ) \nonumber \\
    & > 1 - 2\tau - 2\rho_2 - 20\rho_1^2 - 16\rho_1\sqrt{\rho_2+9\tau}. \nonumber
\end{align}
Combining two inequalities together, we see that when $\epsilon < 1/3$,
\begin{align}
     & v^{\top} \bE_q[(X-\mu_q)(X-\mu_q)^{\top}] v \nonumber \\ 
     & = v^{\top} \bE_{p-r_-}[(X-\mu_q)(X-\mu_q)^{\top}] v + v^{\top} \bE_{r_+}[(X-\mu_q)(X-\mu_q)^{\top}] v \nonumber \\
     & >1 - 2\tau - 2\rho_2 - 20\rho_1^2 - 16\rho_1\sqrt{\rho_2+9\tau} + \frac{9(\rho_1+\sqrt{\epsilon(\rho_2 + 9 \tau)})^2}{\epsilon} , \nonumber \\
     & = 1 -2\tau - 2\rho_2 - 20\rho_1^2 -16\rho_1\sqrt{\rho_2+9\tau}  + \frac{9\rho_1^2}{\epsilon}  + 9(\rho_2+9\tau) + \frac{18\rho_1\sqrt{\rho_2+9\tau}}{\sqrt{\epsilon}} \nonumber \\
     & > 1 +\tau,
\end{align}
which contradicts with the fact that $q\in \GG_2$. Thus we have
\begin{align}
\sup_{q\in \GG_2, p\in \GG_1, \tTV_\mathcal{H}(q,p)\leq \epsilon} \|\mu_{p} - \mu_q\|_2 \leq 11\rho_1 + 9\sqrt{ (\rho_2 + 9\tau)\epsilon}.
\end{align}
\end{proof}

\subsubsection{General Convergence and Concentration results}

\begin{lemma}[{Convergence of  mean for empirical distribution with bounded support and bounded second moment~\citep[Corollary 8.45]{foucart2017mathematical}}]\label{lem.bounded_support_sub_gaussian_tail} 

Given distribution $p$ satisfying the following conditions:
\begin{align}
    \sup_{v\in\bR^d, \|v\|_2=1}\bE_{p}\left[(v^\top (X-\bE_p[X]))^2\right]& \leq \sigma^2, \\
    \|X - \bE_p[X]\|_2 & \leq R \text{ a.s.}.
\end{align}
Denote the empirical distribution of $n$ \iid samples from $p$ as $\hat p_n$.  Then with probability at least $1-\delta$, there is some constant $C$ such that 
\begin{align}
    \|\bE_{p}[X] - \bE_{\hat p_n}[X]\|_2 \leq C\left(\sigma\sqrt{\frac{d}{n}} + \sigma \sqrt{\frac{\log(1/\delta)}{n}}+ \frac{R\log(1/\delta)}{n}\right).
\end{align}
\end{lemma}

\begin{lemma}[{Convergence of covariance for empirical distribution with bounded support~\citep[Thoerem 5.44]{vershynin2010introduction}}]\label{lem.matrix_bernstein}
Given distribution $p$, denote $\Sigma = \bE_p[XX^\top]$. Assume that $\|X\|_2\leq R$ almost surely. Denote $\Delta_1 = \sqrt{\frac{R^2\log(d/\delta)}{n}}$. Then with probability at least $1-\delta$, there exists some constant $C_1$ such that 
\begin{align}
   \|\frac{1}{n}\sum_{i=1}^n X_iX_i^\top - \Sigma\|_2 \leq C_1\max(\|\Sigma\|_2^{1/2}\Delta_1, \Delta_1^2).
\end{align}
Denote $\Delta_2 = \sqrt{\frac{R^2\log(d)}{n}}$.
By integrating over $\delta$, we know that for some constant $C_2$
\begin{align}
    \bE_{p}\left[\|\frac{1}{n} \sum_{i=1}^n X_iX_i^\top - \bE_{p}[XX^\top]\|_2 \right ] \leq C_2\max(\|\Sigma\|_2^{1/2}\Delta_2, \Delta_2^2).
\end{align}
\end{lemma}

The below lemma controls the tail of $\|X\|_2$ when $X$ has bounded $k$-th moment. 
\begin{lemma}[Tail bound for the norm of bounded $k$-th moment distribution]\label{lem.kth_population_deletion_R}
Assume distribution $p$ has its $k$-th moment bounded for $k\geq 2$, i.e.
\begin{align}
    \sup_{v\in\bR^d, \|v\|_2=1}\bE_{p}\left[|v^\top X|^k\right]& \leq \sigma^k.
\end{align}
Then,
\begin{align}\label{eqn.kth_moment_norm_tail_bound}
    \bP_p \left(\left \| X \right \|_2\geq t\right) & \leq  \left(\frac{\sigma\sqrt{d}}{t}\right)^k.
\end{align}
\end{lemma}
\begin{proof}
Since the $k$-th moment is bounded, by Chebyshev's inequality and Khinchine's inequality~\cite{haagerup1981best},
\begin{align}
   \bP_p \left(\left \| X \right \|_2\geq t\right)
    & \leq \frac{\bE_{p}  \left \| X \right \|_2^k}{ t^k} \nonumber \\
    & \leq  \frac{ \bE_{X\sim p, \xi\sim \{\pm 1\}^d} \left|  \xi^{\top} X \right|^k}{t^k}.
\end{align}
From the fact that $X$ has bounded $k$-th moment, we have
\begin{align}
     \bE_{X\sim p, \xi\sim \{\pm 1\}^d} |   \xi^{\top}X|^k  & \leq \sigma^k\bE_{ \xi\sim \{\pm 1\}^d} \|\xi \|_2^k  \nonumber \\
    & = \sigma^k d^{k/2}.
\end{align}

Thus overall, 
\begin{align}
    \bP_p \left(\left \| X  \right \|_2\geq t\right) & \leq \left(\frac{\sigma\sqrt{d}}{t}\right)^k.
\end{align}
\end{proof}

\begin{remark}
If we know the distribution $p$ is sub-Gaussian with paramter $\sigma$,
then we know the $k$-th moment of $p$ is bounded by $(C\sigma \sqrt{k})^k$ for some constant $C$,  i.e.
\begin{align}
\sup_{v\in\bR^d, \|v\|_2=1}\bE_p[|v^\top X|^k] \leq (C\sigma\sqrt{k})^k.    
\end{align}
Then we have a better bound from the Hanson-Wright inequality~\citep[
Exercise 6.3.5]{vershynin2018high}, 
\begin{align}
    \bP(\|X\|_2 \geq C_1\sigma \sqrt{d} + t) \leq \exp(-\frac{C_2t^2}{\sigma^2}).
\end{align}
One can see that by taking $t = C_3\sigma \sqrt{d}$, the above bound gives tail of $\exp(-Cd)$ while~\eqref{eqn.kth_moment_norm_tail_bound} only gives constant tail bound. It is an open problem whether we can do better given bounded $k$-th moment condition.
\end{remark}

\subsubsection{Some negative results on empirical distribution from bounded $\psi$-norm distribution}\label{sec.negativeresultresilient}

\begin{lemma}\label{lem.psinormboundednegative}
Suppose zero mean distribution $p$ on $\mathbf{R}^d$ satisfies
\begin{align}
    \bP_p( \| X\| \geq C \sqrt{d} )\geq \frac{1}{2},
\end{align}
where $C$ is some universal constant. Denote the empirical distribution of $n$ \iid\ samples from $p$ as $\hat{p}_n$. Then, for any Orlicz function $\psi$, the Orlicz norm of random vector $Y\sim \hat p_n$, defined as 
\begin{align}
    \| Y\|_\psi \triangleq  \sup_{v\in \mathbf{R}^d, \|v\|_2 = 1} \| v^\top Y\|_\psi,
\end{align}
satisfies 
\begin{align}
    \bP_p \left ( \|Y\|_\psi \geq  \frac{C \sqrt{d}}{ \psi^{-1}(n)}\right ) \geq \frac{1}{2}. 
\end{align}
\end{lemma}
In particular, if we take $\psi(x) = x^{k}$, then Lemma~\ref{lem.psinormboundednegative} shows that it requires at least $\Omega(d^{k/2})$ samples to ensure the empirical distribution $\hat p_n$ has constant $k$-th moment with probability $1/2$. Similarly, $\psi(x) = \exp(x^2)-1$ corresponds to sub-Gaussian, which implies we would need at least $\exp(\Omega(d))$ number of samples to guarantee the empirical distribution has constant sub-Gaussian norm with probability $1/2$.  
\begin{proof}
Denote the samples in $\hat p_n$ as $X_1,X_2,\ldots, X_n$. We have
\begin{align}
    \sup_{v\in \mathbf{R}^d, \|v\|_2 \leq 1} \bE_{\hat p_n}[\psi(|v^\top X|/\sigma)]   & \geq \bE_{\hat p_n}[\psi(|  (\frac{X_1 }{\| X_1 \|})^\top X|/\sigma)]  \\
    &\geq  \frac{1}{n} \psi( \|X_1\| /\sigma),
\end{align}
where in the first inequality we have taken $v = \frac{X_1}{\|X_1\|}$. Hence, with probability at least $1/2$, we have
\begin{align}
    \sup_{v\in \mathbf{R}^d, \|v\|_2 \leq 1} \bE_{\hat p_n}[\psi(|v^\top X|/\sigma)] & \geq \frac{1}{n} \psi(C \sqrt{d}/\sigma). 
\end{align}
If 
\begin{align}
    \sigma = \frac{C \sqrt{d}}{ \psi^{-1}(n)},
\end{align}
then $ \frac{1}{n} \psi(C \sqrt{d}/\sigma) \geq 1$. 
\end{proof}

The next lemma shows that the lower bound in Lemma~\ref{lem.psinormboundednegative} in the case of $\psi(x) = x^2$ (bounded second moment) is not tight: even if the distribution has bounded support ($\|x\|\leq \sqrt{d}$ almost surely), the sample size needed to ensure either resilience with the right parameter or bounded second moment is superlinear in $d$. 
\begin{lemma}\label{lem.couponlowerbound}
There exist a distribution $p$ on $\mathbf{R}^d$ with the following properties:
\begin{enumerate}
    \item \emph{bounded support:} for $X\sim p$, $\|X\|\leq \sqrt{d}$ almost surely;
    \item \emph{identity population second moment:} \begin{align}
        \bE_p[X X^\top] = I_d        
    \end{align}
    \item \emph{growing empirical second moment:} let $\hat p_n$ denote the empirical distribution of $n$ \iid\ samples $X_1,X_2,\ldots,X_n$ from $p$. Then, if $n = d$, 
    \begin{align}
        \bE_p[ \| \frac{1}{n}\sum_{i = 1}^n X_iX_i^\top \|] \asymp \frac{\ln d}{\ln \ln d};
    \end{align}
    \item \emph{empirical distribution not resilient:} if $n=d$, then there exists an absolute constant $C_1>0$ such that the following does not hold for any absolute constant $C_2>0$: $\hat p_n \in \GG_{\mathsf{mean}}(C_2 \sqrt{\eta},\eta)$ for $\eta = \frac{C_1 
    \ln n}{n \ln \ln n}$ with probability at least $1/2$. 
\end{enumerate}
\end{lemma}
\begin{proof}
Let $p$ be the distribution of $X$ that takes value $\sqrt{d} e_i$ with probability $1/d$, where $\{e_i\}_{i =1}^d$ is the standard basis in $\mathbf{R}^d$. Then, the third part follows from~\citep[Exercise 5.4.14]{vershynin2018high}. 

Regarding the last statement, we have
\begin{align}
    \frac{1}{n}\sum_{i = 1}^n X_i = \frac{\sqrt{d}}{n} \begin{bmatrix} Z_1 \\Z_2 \\ \ldots \\ Z_d \end{bmatrix},
\end{align}
where $Z_i = \sum_{k = 1}^n \mathbbm{1}(X_k = \sqrt{d} e_i)$, and $(Z_1,Z_2,\ldots,Z_d) \sim \mathsf{mult}(n, (1/d,1/d,\ldots,1/d))$. It follows from~\citep[Chapter 5]{mitzenmacher2017probability} that if $n = d$ then $\max_{k\in [d]} Z_k$ tightly concentrates on $\frac{\ln n}{\ln \ln n}$. 

Without loss of generality assume $Z_1 = \max_{k\in [d]} Z_k$. Consider the deletion operation that deletes all samples that contribute to the counts in $Z_1$. Hence the deletion fraction is $\frac{Z_1}{n}$. Since $Z_1$ tightly concentrates around $\frac{\ln n}{\ln \ln n} \gg 1$ as $n\to \infty$, there exist two absolute constants $c_1,c_2$ such that $\frac{c_2 \ln n}{\ln \ln n}\leq Z_1 \leq \frac{c_1 \ln n}{\ln \ln n}$ with probability at least $1/2$ for $n$ large enough. From now on we condition on this event. Denote the empirical distribution of remaining samples as $r$. Then, take $n = d$, 
\begin{align}
    \| \bE_{\hat p_n} [X] - \bE_r[X]\| & \geq |e_1^\top ( \bE_{\hat p_n} [X] - \bE_r[X])| = \frac{\sqrt{d}}{n}Z_1 
     = \frac{Z_1}{\sqrt{n}}  \asymp \frac{\ln n}{\sqrt{n}\ln \ln n} \gg \sqrt{\frac{Z_1}{n}} \asymp \sqrt{\frac{\ln n}{n \ln \ln n}}. 
\end{align}
\end{proof}

\subsubsection{Empirical distribution from bounded $\psi$-norm distributions is resilient for all $\eta$}

In Lemma~\ref{lem.couponlowerbound}, it is shown that given $n = d$ samples from a distribution with bounded second moment and bounded support, the empirical distribution is not resilient for small $\eta$ with the right rate. However, we show in the lemma below  that the empirical distribution is in fact resilient with the right rate for $\eta$ large enough. In particular when $\eta \gtrsim \frac{c d}{n}$ for some $c>0$ under bounded second moment and bounded support assumptions.

\begin{lemma}\label{lem.empirical_psi_resilient}
For a given Orlicz function $\psi$, we define
\begin{align}
    \GG & =  \{ p \mid \sup_{f\in\sF} \bE_p\left[\psi\left(\frac{| f(X)-\bE_p[f(X)]|}{\sigma}\right)\right]  \leq  1 \}. 
\end{align}
Assume $p\in\GG$, and the empirical distribution $\hat p_n$ of $n$ \iid samples from $p$ satisfies
\begin{align}
    \bE_{p}[W_\sF(p, \hat p_n)] & \leq \xi_n.
\end{align}
Assume the following equations have solutions, denoted as  $x_0, t$: 
\begin{align}
    \sigma x_0\psi^{-1}(1/x_0) & = \xi_n, \\
    4\psi'(\psi^{-1}(\frac{t}{x_0}))x_0\psi^{-1}(\frac{1}{x_0}) & = t.
\end{align}
We  assume that for any $\epsilon,t>0$, there exists some $C_t$ that only depends on $\psi$ and $t$ such that  
\begin{align}\label{eqn.proof_generalf_long_assumption}
    \psi^{-1}(t/\epsilon) \leq C_t \psi^{-1}(1/\epsilon).
\end{align}
Define
\begin{align}
   \rho_\delta(\eta) & =\frac{C_t+2}{1-\eta}\left(\sigma\eta {\psi}^{-1}(\frac{1}{\delta \eta})  + \frac{\xi_n}{\delta} \right). 
\end{align}
Then with probability at least $1-2\delta$, 
\begin{align}
    \hat p_n \in  \bigcap_{\eta\in[0, 1)}\GG_{W_\sF}(\rho_\delta(\eta), \eta ) = \{ p \mid \forall \eta \in [0, 1),  \sup_{f\in\sF, r\leq \frac{p}{1-\eta}} |\bE_p[f(X)] - \bE_r[f(X)]| \leq  \rho_\delta(\eta)\}.
\end{align}
\end{lemma}

Here the first term $\sigma \eta\psi^{-1}(1/\sigma\eta)$ is close to the population limit in Lemma~\ref{lem.cvx_mean_resilience}, and the second term $\xi_n/\delta$ is similar to the finite sample error bound without corruption.

\begin{proof}
 We use the similar technique as Lemma~\ref{lem.cvx_mean_resilience} to show that $\hat p_n \in \GG'$ with high probability.
Note that $x_0>0$ is defined as the solution to the  following equation:
\begin{align}
    \sigma x\psi^{-1}(1/x) = \xi_n.
\end{align}
We then define a convex function $\tilde \psi$ for $t>0$ as
\begin{align}
    \tilde{\psi}(x) =  \left\{  
    \begin{array}{ll}  
    \psi(x), & 0\leq x \leq \psi^{-1}(\frac{t}{x_0}), \\  
     \psi'(\psi^{-1}(\frac{t}{x_0}))(x-\psi^{-1}(\frac{t}{x_0}))+\frac{t}{x_0}, & x > \psi^{-1}(\frac{t}{x_0}).
     \end{array}  \right.  
\end{align}
One can see that $\tilde{\psi}$ is convex, non-negative, non-decreasing and $\tilde{\psi}(|x|) \leq \psi(|x|)$.  Hence,
\begin{align}
    \tilde{\psi}^{-1}(\frac{1}{x}) =  \left\{ 
    \begin{array}{ll}  
   \frac{1}{x\psi'(\psi^{-1}(t/x_0))}-\frac{t}{x_0\psi'(\psi^{-1}(t/x_0))}+\psi^{-1}(\frac{t}{x_0})& 0 \leq x \leq \frac{x_0}{t}, \\ 
     \psi^{-1}(\frac{1}{x}), & x > \frac{x_0}{t}.
     \end{array}  \right.  
\end{align}
Note that from Lemma~\ref{lem.psi_increasing}, we know that $x\tilde{\psi}^{-1}_t(1/x)$ is non-decreasing, and 

\begin{align}\label{eqn.bound_tilde_psi}
    x\tilde{\psi}_t^{-1}(\frac{t}{x}) \leq \left\{ 
    \begin{array}{ll}  
   x_0\psi^{-1}(\frac{t}{x_0})& 0 \leq x \leq {x_0}, \\  
     x\psi^{-1}(\frac{t}{x}), & x > {x_0}.
     \end{array}  \right.  
\end{align}
Now we bound the term: 
\begin{align}
     \sup_{f\in\sF} \bE_{\hat p_n}\left[\tilde{\psi}\left(\frac{|f(X)-\bE_{p}[f(X)]|}{\sigma}\right)\right] =    \sup_{f\in\sF}\frac{1}{n}\sum_{i=1}^n \tilde{\psi}\left(\frac{|f(X_i)-\bE_{p}[f(X)]|}{\sigma}\right)
\end{align}
By $p \in \GG$,  $\sup_{f\in\sF} \bE_{p}\left[{\psi}\left(\frac{|f(X)-\bE_{p}[f(X)]|}{\sigma}\right)\right]  \leq 1$, we have
\begin{align}
     &\bE_{p} \left[ \sup_{f\in\sF} \frac{1}{n} \sum_{i=1}^n \tilde{\psi}\left(\frac{|f(X_i)-\bE_{p}[f(X)]|}{\sigma}\right) \right] \nonumber \\
     \leq &  \bE_{p} \left[ \sup_{f\in\sF} \frac{1}{n} \sum_{i=1}^n \tilde{\psi}\left(\frac{|f(X_i)-\bE_{p}[f(X)]|}{\sigma}\right) - \bE_{p} \Big[\tilde{\psi}\left(\frac{|f(X)-\bE_{p}[f(X)]|}{\sigma}\right)\Big]\right]  \\
     &   + \sup_{f\in\sF} \bE_{p} [\tilde{\psi}\left(\frac{|f(X)-\bE_{p}[f(X)]|}{\sigma}\right)]\ \label{eqn.bdd_kth_proof1} \\
     \leq & 2\bE_{p*, \epsilon\sim\{\pm 1\}^n} \left[\sup_{f\in\sF} \frac{1}{n} \sum_{i=1}^n \epsilon_i \tilde{\psi}\left(\frac{|f(X_i)-\bE_{p}[f(X)]|}{\sigma}\right)   \right] + 1 \label{eqn.bdd_kth_proof2} \\
     \leq &   2\psi'(\psi^{-1}(\frac{t}{x_0})) \bE_{p*, \epsilon \sim \{\pm 1\}^n}\left[ \sup_{f\in\sF} \frac{1}{n} \sum_{i=1}^n \epsilon_i  \left(\frac{f(X_i)-\bE_{p}[f(X)]}{\sigma}\right)  \right] + 1 \label{eqn.bdd_kth_proof3}\\
      \leq & 4\psi'(\psi^{-1}(\frac{t}{x_0}))   \bE_{p*}\left[ \sup_{f\in\sF} \frac{1}{n} \sum_{i=1}^n  \left(\frac{f(X_i)-\bE_{p}[f(X)]}{\sigma}\right) \right] + 1 \label{eqn.bdd_kth_proof4}\\
    = & 4\psi'(\psi^{-1}(\frac{t}{x_0}))\frac{\bE_{p}\left[W_\sF(\hat p_n, p) \right]}{\sigma} + 1 \label{eqn.bdd_kth_proof5} \\
     \leq & 4\psi'(\psi^{-1}(\frac{t}{x_0}))\frac{\xi_n}{\sigma}+1
\end{align}
Here Equation (\ref{eqn.bdd_kth_proof1}) is from triangle inequality of $\sup$. Equation (\ref{eqn.bdd_kth_proof2}) and Equation (\ref{eqn.bdd_kth_proof4}) are from symmetrization inequality~\citep[Proposition 4.11]{wainwright2019high}. Equation (\ref{eqn.bdd_kth_proof3}) is from Talagrand contraction inequality~\citep[Exercise 6.7.7]{vershynin2018high}.

Now we apply a similar argument in Lemma~\ref{lem.cvx_mean_resilience} to show that $\hat p_n$ is in the resilient set induced by $\tilde{\psi}$.

For any event $E$, denote its compliment as $E^c$,  by the definition of conditional expectation and symmetry of $\sF$, 
\begin{align}
     \sup_{f\in\sF} \bP_{\hat p_n}(E)(\bE_{\hat p_n}[f(X) | E] - \bE_{\hat p_n}[f(X) ]) & =  \sup_{f\in\sF} \bP_{\hat p_n}(E^c)( \bE_{\hat p_n}[f(X)-\bE_{\hat p_n}[f(X) | E^c]).
\end{align}
Thus we have
\begin{align}
 \sup_{f\in\sF} \bE_{\hat p_n}[f(X) | E] - \bE_{\hat p_n}[f(X) ]  & =  \sup_{f\in\sF} \left\{ \frac{\bP_{\hat p_n}(E^c)}{1-\bP_{\hat p_n}( E^c)}\bE_{\hat p_n}\Big[f(X)  - \bE_{\hat p_n}[f(X)]\mid E^c\Big] \right\} \nonumber \\
 & \leq \sup_{f\in\sF} \left\{ \frac{\bP_{\hat p_n}(E^c)}{1-\bP_{\hat p_n}( E^c)}\bE_{\hat p_n}\Big[f(X)  - \bE_{ p}[f(X)]\mid E^c\Big] \right\} \nonumber \\
 & \quad + \sup_{f\in\sF} \left\{  \frac{\bP_{\hat p_n}(E^c)}{1-\bP_{\hat p_n}( E^c)} | \bE_{\hat p_n}[f(X)] - \bE_{p}[f(X)] |\right\} \nonumber \\
 & = \sup_{f\in\sF} \left\{ \frac{\bP_{\hat p_n}(E^c)}{1-\bP_{\hat p_n}( E^c)}\bE_{\hat p_n}\Big[f(X)  - \bE_{ p}[f(X)]\mid E^c\Big]\right\} \nonumber \\
 & \quad +  \frac{\bP_{\hat p_n}(E^c)}{1-\bP_{\hat p_n}( E^c)} W_\sF(p, \hat p_n).
\end{align}
We then control the first term in RHS. From Equation (\ref{eqn.bdd_kth_proof4}), by Markov's inequality, we know that for any $\eta\in[0, 1)$, with probability at least $1-\delta$
\begin{align}
    \frac{4\psi'(\psi^{-1}(\frac{t}{x_0}))\frac{\xi_n}{\sigma}+1}{\delta} & \geq \sup_{ f\in\sF} \bE_{\hat p_n} \Big[\tilde{\psi}\left(\frac{|f(X)-\bE_{p}[f(X)]|}{\sigma}\right)\Big] \nonumber \\
    & \geq  \sup_{\bP_{\hat p_n}(E) \geq 1- \eta, f\in\sF} \bP_{\hat p_n}(E^c) \bE_{\hat p_n} \Big[\tilde{\psi}\left(\frac{|f(X)-\bE_{p}[f(X)]|}{\sigma}\right) \mid E^c \Big] \nonumber \\
    & \geq  \sup_{\bP_{\hat p_n}(E) \geq 1- \eta, f\in\sF} \bP_{\hat p_n}(E^c)\tilde{\psi} \Bigg(\Big|\bE_{\hat p_n}\Big[\frac{|f(X)-\bE_{p}[f(X)]|}{\sigma} \mid E^c\Big] \Big|\Bigg) \label{eqn.proof_empirical_psi_inverse}.
\end{align}
This gives us with probability at least $1-\delta$, 
\begin{align}
    &  \sup_{\bP_{\hat p_n}(E) \geq 1- \eta, f\in\sF} \bE_{\hat p_n}[f(X) | E] - \bE_{\hat p_n}[f(X) ] \nonumber \\
     \leq  & \sup_{\bP_{\hat p_n}(E) \geq 1- \eta} \frac{\sigma \bP_{\hat p_n}(E^c)}{1-\bP_{\hat p_n}( E^c)} \tilde{\psi}^{-1}(\frac{4\psi'(\psi^{-1}(\frac{t}{x_0}))\frac{\xi_n}{\sigma}+1}{  \delta \bP_{\hat p_n(E^c)}})+ \frac{\eta}{1-\eta}W_\sF(p, \hat p_n)\nonumber \\
     \leq & \frac{\sigma\eta}{1-\eta} \tilde{\psi}^{-1}(\frac{4\psi'(\psi^{-1}(\frac{t}{x_0}))\frac{\xi_n}{\sigma}+1}{ \delta \eta}) + \frac{\eta}{1-\eta}W_\sF(p, \hat p_n) \label{eqn.proof_general_psi1} \\
     = & \frac{\sigma\eta}{1-\eta} \tilde{\psi}^{-1}(\frac{4\psi'(\psi^{-1}(\frac{t}{x_0}))x_0\psi^{-1}(\frac{1}{x_0})}{  \delta  \eta}) + \frac{\sigma\eta}{1-\eta} \tilde{\psi}^{-1}(\frac{1}{\delta \eta}) + \frac{\eta}{1-\eta}W_\sF(p, \hat p_n).\label{eqn.proof_general_psi2}
\end{align}

Equation (\ref{eqn.proof_general_psi1}) 
uses the fact that $x\psi^{-1}(b/x)$ is a non-decreasing function in $[0, 1)$ for any $b>0$ in Lemma~\ref{lem.psi_increasing}. Equation (\ref{eqn.proof_general_psi2}) is from the concave and non-negative property of $\tilde{\psi}^{-1}$. By Markov's inequality, we know that with probability at least $1-\delta$, we have
$W_\sF(p, \hat p_n) \leq \frac{\xi_n}{\delta}$. By union bound we have for any $\eta\in[0, 1)$, with probability at least $1-2\delta$, 
\begin{align}
      \sup_{\bP_{\hat p_n}(E) \geq 1- \eta, f\in\sF} \bE_{\hat p_n}[f(X) | E] - \bE_{\hat p_n}[f(X) ]& \leq \frac{\sigma\eta}{1-\eta} \tilde{\psi}^{-1}(\frac{4\psi'(\psi^{-1}(\frac{t}{x_0}))x_0\psi^{-1}(\frac{1}{x_0})}{  \delta  \eta}) \nonumber \\ 
    & \quad + \frac{\sigma\eta}{1-\eta} \tilde{\psi}^{-1}(\frac{1}{\delta \eta}) + \frac{\eta}{1-\eta} \frac{\xi_n}{\delta}.
\end{align}
Note that $t$ is the solution to 
\begin{align}
    t = 4\psi'(\psi^{-1}(\frac{t}{x_0}))x_0\psi^{-1}(\frac{1}{x_0}).
\end{align}
Denote $\tilde \rho_\delta(\eta)$ as
\begin{align}
    \tilde \rho_\delta(\eta) =\frac{1}{1-\eta}({\sigma\eta}  \tilde{\psi}^{-1}(\frac{t}{\delta\eta}) + {\sigma\eta} \tilde{\psi}^{-1}(\frac{1}{\delta\eta})+ \frac{\eta\xi_n}{\delta}).
\end{align}
So far we have shown  that $\hat p_n \in \bigcap_{\eta\in[0, 1)}\GG_{W_\sF}(\tilde \rho_\delta(\eta), \eta)$
Now we show  that  $ \bigcap_{\eta\in[0, 1)}\GG_{W_\sF}(\tilde \rho_\delta(\eta), \eta) \subset \bigcap_{\eta\in[0, 1)}\GG_{W_\sF}(\rho_\delta(\eta), \eta )$, where
\begin{align}
   \rho_\delta(\eta) & =\frac{C_t+2}{1-\eta}\left(\sigma\eta {\psi}^{-1}(\frac{1}{\delta \eta})  + \frac{\xi_n}{\delta} \right). 
\end{align}
From Equation (\ref{eqn.bound_tilde_psi}),
\begin{align}
    \frac{\eta}{1-\eta}\tilde{\psi}^{-1}(\frac{t}{\delta \eta}) \leq \left\{ 
    \begin{array}{ll}  
   \frac{x_0}{(1-\eta)\delta}\psi^{-1}(\frac{t}{x_0})& 0 \leq \eta \leq \frac{x_0}{\delta }, \\  
     \frac{\eta}{1-\eta}\psi^{-1}(\frac{t}{\delta \eta}), & \eta > \frac{x_0}{\delta }.
     \end{array}  \right. 
\end{align}
From  Equation (\ref{eqn.proof_generalf_long_assumption}), we know that for any $\eta$,
\begin{align}
    \psi^{-1}(\frac{t}{\delta \eta}) & \leq C_t\psi^{-1}(\frac{1}{\delta\eta}).
\end{align}
These two equations combined show that
\begin{align}
    \tilde \rho_\delta(\eta) & \leq  \frac{(C_t+1)}{1-\eta}({\sigma\eta}  \tilde{\psi}^{-1}(\frac{1}{\delta\eta}) + \frac{\sigma x_0 \psi^{-1}(\frac{1}{x_0})}{\delta})+ \frac{\eta\xi_n}{(1-\eta)\delta} \nonumber \\
    & = \frac{(C_t+1)}{1-\eta}({\sigma\eta}  \tilde{\psi}^{-1}(\frac{1}{\delta\eta}) + \frac{\xi_n}{\delta})+ \frac{\eta\xi_n}{(1-\eta)\delta} \nonumber \\
    & \leq \frac{(C_t+2)}{1-\eta}({\sigma\eta}  \tilde{\psi}^{-1}(\frac{1}{\delta\eta}) + \frac{\xi_n}{(1-\eta)\delta}) \nonumber \\
    & = \rho_\delta(\eta).
\end{align}
\end{proof}

As a corollary, we can easily show that the empirical distribution for bounded $k$-th moment distribution is inside resilience family:
\begin{corollary}\label{cor.bounded_kth_empirical_resilience}
Suppose for $k\geq 2$, 
\begin{align}
    \GG & =  \{ p \mid \sup_{v\in\bR^d, \| v\|_2 = 1} \bE_p\left[|\langle X-\bE_p[X], v\rangle|^k\right]  \leq  \sigma^k \}. 
\end{align}
Define
\begin{align}
   \rho_\delta(\eta) & =\frac{Ck\sigma}{1-\eta}\left(\frac{\eta^{1-1/k}}{\delta^{1/k}}  + \frac{1}{\delta} \sqrt{\frac{d}{n}}\right). 
\end{align}
Then with probability at least $1-2\delta$, the empirical distribution $\hat p_n$ of $n$ \iid samples from $p$ satisfies
\begin{align}
    \hat p_n \in \bigcap_{\eta\in[0, 1)}\GG_{\mathsf{mean}}(\rho_\delta(\eta), \eta ) = \{ p \mid \forall \eta \in [0, 1),  \sup_{r\leq \frac{p}{1-\eta}} \|\bE_p[X] - \bE_r[X]\|_2 \leq  \rho_\delta(\eta)\}.
\end{align}
\end{corollary}
\begin{proof}

 We first check the conditions required in Lemma~\ref{lem.empirical_psi_resilient}. Here $\psi(x) = x^k$ for $k\geq 2$. From Lemma~\ref{lem.kth_convergence} we have
\begin{align}
    \bE_{p^*}[\|\bE_{p^*}[X] -\bE_{\hat p_n^*}[X]\|_2] & \leq \sigma\sqrt{\frac{d}{n}}.
\end{align}
One can solve $x_0 = (\frac{d}{n})^{k/2(k-1)}, t = k^k$ from the following equations:
\begin{align}
    \sigma x_0\psi^{-1}(1/x_0) & = \xi_n, \\
    4\psi'(\psi^{-1}(\frac{t}{x_0}))x_0\psi^{-1}(\frac{1}{x_0}) & = t.
\end{align}
Then for any $\eta \in [0, 1)$, we have
\begin{align}
    \psi^{-1}(\frac{t}{\eta}) \leq k\psi^{-1}(\frac{1}{\eta})
\end{align}
for some universal constant $C$. 
Then when $\delta > x_0$, with probability at least $1-2\delta$,
\begin{align}
    \hat p_n^* \in \GG‘.
\end{align}
Then by Lemma~\ref{lem.empirical_psi_resilient}, we know that  with probability at least $1-2\delta$, for any $r\leq \frac{\hat p_n^*}{1-\eta}$, 
\begin{align}
    \| \mu_r - \mu_{\hat p_n^*}\|_2 
    & \leq \frac{Ck\sigma}{1-\eta}\left(\frac{\eta^{1-1/k}}{\delta^{1/k}} + {\frac{1}{\delta}\sqrt{\frac{d}{n}}}\right)\label{eqn.bounded_kth_empirical_resilience}.
\end{align}
\end{proof}

\subsubsection{Empirical distribution from distributions with moment generating function is resilient with better rate for some $\eta$}

The next lemma shows that the empirical distribution of \iid samples from distributions with moment generating functions is resilient with a \emph{fixed} $\eta$, whose dependence of the parameters on $\delta$ is better than that in Lemma~\ref{lem.empirical_psi_resilient}. It is a generalization of~\citep[Lemma  4.4]{diakonikolas2019robust}.

\begin{lemma}\label{lem.subgaussian_empirical_resilience}
Let $\psi$ be some convex and continuously differentiable function on $[0,b)$ with $0<b\leq \infty$, such that $\psi(0) = \psi'(0) = 0$. Assume that for $\lambda \in (0,b)$,
\begin{align}
    \sup_{v\in \mathbf{R}^d, \|v\|_2 = 1} \ln(\bE_p[\exp(\lambda (v^\top X - \bE_p[v^\top X]))]) \leq \psi(\lambda). 
\end{align}
Denote by $\psi^*(x)$ the Fenchel--Legendre dual of $\psi$:
\begin{align}
    \psi^*(x) = \sup_{\lambda \in (0,b)} (\lambda x - \psi(\lambda)).
\end{align}

Fix $\eta \in [0,1)$. Then, there exists an absolute constant $C>0$ such that with probability at least $1-\delta$, 
\begin{align}
    \hat p_n^* \in \{ p \mid \sup_{r\leq \frac{p}{1-\eta}} \|\bE_p[X] - \bE_r[X]\| \leq  \rho\},
\end{align}
where 
\begin{align}
    \rho = \max\left\{\frac{4\eta}{1-\eta} \psi^{*-1} \left( \frac{Cd + \ln(2/\delta) + n h(\eta)}{n\eta}  \right) , \frac{4\eta}{1-\eta} \psi^{*-1} \left(\frac{Cd + \ln(2/\delta)}{n} \right) \right\},
\end{align}
$h(p) = p\ln(1/p) + (1-p)\ln(1/(1-p))$ is the binary entropy function, and $\psi^{*-1}$ is the generalized inverse of $\psi^{*}$. In particular, if $\psi(\lambda) = \frac{\lambda^2 \sigma^2}{2}$, and $\eta\in [0,1/2]$, then one can take
\begin{align}
    \rho = C \sigma \cdot \left( \sqrt{\eta} \sqrt{\frac{d + \ln(1/\delta)}{n}} +  \eta \sqrt{\ln(1/\eta)} \right),
\end{align}
where $C$ is some universal constant. 
\end{lemma}
\begin{proof}
Throughout this proof the constant $C$ may be different from line by line, but is always an absolute constant. It follows from the Chernoff method~\citep[Page 24]{boucheron2013concentration} that for any $v\in \mathbf{R}^d, \|v\|_2 = 1$, we have for any $t\geq 0$, 
\begin{align}
    \bP\left( \frac{1}{n} \sum_{i = 1}^n v^\top X_i - v^\top \mathbb{E}_p[X] \geq t \right) \leq \exp(-n \psi^*(t)). 
\end{align}
It follows from~\citep[Corollary 4.2.13]{vershynin2018high} that we can take a net of vectors $\mathcal{N} \subset \{v\mid v\in \mathbf{R}^d, \|v\|_2 = 1\}$ such that $|\mathcal{N}|\leq C^d$ and that for any $x$, $\sup_{v\in \mathcal{N}} v^\top x \geq \frac{1}{2} \|x\|_2$. Then it follows from the union bound that 
\begin{align}
   &    \bP\left( \sup_{v\in \mathbf{R}^d, \|v\|_2 = 1} \frac{1}{n} \sum_{i = 1}^nv^\top X_i - v^\top \mathbb{E}_p[X] \geq t \right) \nonumber \\ 
    & \quad \leq \bP\left( \sup_{v\in \mathcal{N}} \frac{1}{n} \sum_{i = 1}^nv^\top X_i - v^\top \mathbb{E}_p[X] \geq \frac{t}{2} \right) \\
   & \quad \leq \sum_{v\in \mathcal{N}}  \bP\left(  \frac{1}{n} \sum_{i = 1}^nv^\top X_i - v^\top \mathbb{E}_p[X] \geq \frac{t}{2} \right) \\
   & \quad \leq |\mathcal{N}| \exp(-n \psi^*(t/2)) \\
   & \quad \leq \exp(C d - n \psi^*(t/2)). 
\end{align}

Denote $\hat{\mu}_n = \frac{1}{n}\sum_{i = 1}^n X_i$. From now on we assume $\eta n$ is an integer. Our goal is to find the parameter $\rho$ such that
\begin{align}
    \bP\left( \sup_{J\subset [n], |J|\geq (1-\eta)n} \left \| \frac{1}{|J|}\sum_{i\in J}X_i - \hat{\mu}_n \right \| \geq \rho \right ) \leq \delta. 
\end{align}
It follows from a replacement argument that it suffices to consider only those $J$ such that $|J| = (1-\eta) n$. Noting that for any $J\subset [n], |J| = (1-\eta)n$, 
\begin{align}
\frac{1}{|J|} \sum_{i\in J} X_i - \hat{\mu}_n = \frac{\eta}{1-\eta} (\hat{\mu}_n - \frac{1}{\eta n} \sum_{i\notin J}X_i). 
\end{align}
Hence, 
\begin{align*}
     & \bP\left( \sup_{J\subset [n], |J|\geq (1-\eta)n} \left \| \frac{1}{|J|}\sum_{i\in J}X_i - \hat{\mu}_n \right \| \geq \rho \right ) \nonumber \\
     & \quad = \bP\left( \sup_{J\subset [n], |J|= (1-\eta)n} \left \| \frac{1}{|J|}\sum_{i\in J}X_i - \hat{\mu}_n \right \| \geq \rho \right )  \\
     & \quad =\bP\left( \sup_{J\subset [n], |J|= (1-\eta)n} \left \|\frac{\eta}{1-\eta} (\hat{\mu}_n - \frac{1}{\eta n} \sum_{i\notin J}X_i)\right \| \geq \rho \right )  \\
     & \quad \leq \bP( \frac{\eta}{1-\eta}\| \hat{\mu}_n - \bE_p[X]\|\geq \rho/2) + \bP\left( \sup_{J\subset [n], |J|= (1-\eta)n} \frac{\eta}{1-\eta} \left \| \frac{1}{\eta n} \sum_{i\notin J}X_i - \mathbb{E}_p[X]\right \| \geq \rho/2 \right )  \\
     & \quad \leq  \bP\left ( \| \hat{\mu}_n - \bE_p[X]\|\geq \frac{\rho(1-\eta)}{2\eta} \right ) + 
     {n \choose (1-\eta)n} \bP\left(  \left \| \frac{1}{\eta n} \sum_{i\notin J}X_i - \mathbb{E}_p[X]\right \| \geq \frac{\rho(1-\eta)}{2\eta} \right )  \\
     & \quad \leq \exp(Cd - n \psi^{*}(\rho(1-\eta)/(4\eta))) + \exp(n h(\eta)+ Cd - \eta n \psi^*(\rho(1-\eta)/(4\eta))),
\end{align*}
where in the last step we used the inequality ${n \choose k} \leq \exp(n h(k/n))$, where $h(p) = p\ln(1/p) + (1-p)\ln(1/(1-p))$ is the binary entropy function. It now suffices to guarantee that
\begin{align}
    n \psi^{*}(\rho(1-\eta)/(4\eta)) - Cd & \geq \ln(2/\delta) \\
    \eta n \psi^*(\rho(1-\eta)/(4\eta)) - n h(\eta)- Cd & \geq \ln(2/\delta).  
\end{align}
It we choose $\rho$ such that
\begin{align}
    \rho \geq \frac{4\eta}{1-\eta} \psi^{*-1} \left( \frac{Cd + \ln(2/\delta) + n h(\eta)}{n\eta} \vee \frac{Cd + \ln(2/\delta)}{n} \right),
\end{align}
the two bounds above would be satisfied. 
\end{proof}

\subsubsection{Empirical distribution from distributions with bounded $k$-th moment and bounded support has bounded covariance and is resilient for all $\eta$}

The lemma below shows that if we delete the distribution with bounded $k$-th moment, we are able to guarantee good properties of the deleted distribution that is required in the modulus of continuity in Lemma~\ref{lem.empirical_identity_cov_modulus}, which holds for all $\eta\in[0, 1/k]$.
\begin{lemma}[Properties of deleted distribution of bounded $k$-th moment]\label{lem.kth_deletion_set}
Assume the distribution $p^*$ has identity covariance and has its $k$-th moment bounded by $\sigma$ for $k\geq 2$, i.e.
\begin{align}
    \sup_{v\in\bR^d, \|v\|_2=1}\bE_{p^*}[|v^\top (X-\bE_{p^*}[X])|^k]\leq \sigma^k, \bE_{p^*} [(X-\bE_{p^*}[X])(X-\bE_{p^*}[X])^\top] = I_d.
\end{align}
It implies $\sigma\geq 1$. For any fixed $\eta \in [0, 1/2]$, we define a new distribution $p'$ satisfying that for any event $A$, 
\begin{align}
    \bP_{p'}(X\in A) = \bP_p(X\in A \mid \|X - \mu_{p^*} \| \leq \sigma \sqrt{d}/ \eta^{1/k}). 
\end{align}
Denote the empirical distribution of $n$ \iid\ samples from $p'$ as $\hat p'$.
Denote $\Delta_1 = \sqrt{\frac{\sigma^2d(\log(d/\delta))}{n\eta^{2/k}}}$, $\Delta_2 = \sqrt{\frac{\sigma^2d\log(d)}{n\eta^{2/k}}}$, and 
\begin{align}  
     \mathcal{G}_1(\eta) & =  \Bigg\{p \mid \forall r \leq \frac{p}{1-\eta}, \| \mathbb{E}_r[X] - \bE_p[X]\|_2 \leq \frac{C_1k\sigma}{1-\eta}\left(\frac{\eta^{1-1/k}}{\delta^{1/k}} + {\frac{1}{\delta}\sqrt{\frac{d}{n}}}\right) \Bigg \}, \nonumber \\
       \mathcal{G}_2(\eta)  & = \Bigg\{p \mid \|\bE_p[(X-\mu_p)(X-\mu_p)^\top] \|_2 \leq \frac{1}{1-\eta}+ C_3\max(\Delta_1, \Delta_1^2)
     \Bigg \}, \\
      \mathcal{G}_3(\eta) & = \Bigg\{p \mid  \forall r \leq \frac{p}{1-\eta},   \| \bE_r[(X-\bE_{ p}[X])(X-\bE_{p}[X])^T] - I_d\|_2 
    \leq  \frac{C_2k}{1-\eta}\bigg(\frac{k\sigma^2\eta^{1-2/k}}{\delta^{2/k}}  \nonumber \\ 
    & \quad +\frac{\max(\Delta_2, \Delta_2^2)}{\delta} + \frac{k\sigma^2d}{n\delta^2}\bigg)\Bigg \}.\nonumber
\end{align}
where $C_1, C_2, C_3$ are some universal constants. Then the following holds:
\begin{enumerate}
\item $\TV(p^*, p')\leq \eta$,
    \item With probability at least $1-2\delta$, $\hat p' \in \GG_1(\eta) $,
    \item With probability at least $1-\delta$, $\hat p' \in \GG_2(\eta) $,
    \item If $k > 2$, with probability at least $1-6\delta$, $\hat p' \in \GG_3(\eta)$.
\end{enumerate}
\end{lemma}
\begin{proof}

We show the four conclusions separately. 
\begin{enumerate}
\item 
From Lemma~\ref{lem.kth_population_deletion_R}, we know that
\begin{align}
    \bP_{p^*}(\|X - \bE_{p^*}[X]\|_2 \geq t ) \leq \frac{\sigma^k d^{k/2}}{t^k}.
\end{align}
By taking $t = \sigma \sqrt{d}/ \eta^{1/k}$, we have
\begin{align}
    \bP_{p^*}(\|X - \bE_{p^*}[X]\|_2 \geq \sigma \sqrt{d}/ \eta^{1/k} ) \leq \eta.
\end{align}

Thus we know $\TV(p^*, p')\leq \eta$.

\item 
From Lemma~\ref{lem.Wf_orlicz_closed}, we know that Orlicz-norm bounded function is approximately closed under deletion. From now we condition on the event that $\TV(p', p^*)\leq \eta$. Since $p'$ is a deletion of $p^*$, by Lemma~\ref{lem.Wf_orlicz_closed} we have for any $\eta\in[0, 1/2)$
\begin{align}
    \sup_{v\in\bR^d, \|v\|_2 = 1} \bE_{p'}\left[ \left(\frac{|v^\top (X-\mu_{p^*})|}{5\sigma}\right)^{k}\right]\leq \frac{1}{1-\eta}.
\end{align}
From Lemma~\ref{lem.centering_psi} (Centering) and Corollary~\ref{cor.bounded_kth_empirical_resilience}, we know that with probability at least $1-2\delta$, for any $\eta\in[0, 1/2)$ and any $r\leq \frac{\hat p'}{1-\eta}$, 
\begin{align}\label{eqn.proof_kth_mean_closeness}
    \| \mu_r - \mu_{\hat p'}\|_2 
    & \leq \frac{C_1k\sigma}{1-\eta}\left(\frac{\eta^{1-1/k}}{\delta^{1/k}} + {\frac{1}{\delta}\sqrt{\frac{d}{n}}}\right),
\end{align}
where $ C_1$ is some universal constant. This shows that $\hat p' \in \GG_1(\eta)$ with probability at least $1-2\delta$.

\item 
Since $p^*$ is inside $\GG_{\mathsf{mean}}$ and $p'$ is a $\eta$ deletion, by definition of resiliense~\ref{eqn.oldresilience} and Lemma~\ref{lem.cvx_mean_resilience} we have
\begin{align}\label{eqn.proof_kth_mean_close_deletion}
    \|\mu_{p'} - \mu_{p^*}\|_2\leq 2 \sigma \eta^{1-1/k}.
\end{align}
Furthermore, by $p'$ having bounded $k$-th moment and bounded support, from Lemma~\ref{lem.bounded_support_sub_gaussian_tail}, we have with probability at least $1-\delta$,
\begin{align}\label{eqn.proof_kth_mean_empirical_close}
     \|\mu_{p'} - \mu_{\hat p'}\|_2\leq C\left(\sigma\sqrt{\frac{d}{n}} + \sigma \sqrt{\frac{\log(1/\delta)}{n}}+ \frac{\sigma\sqrt{d}\log(1/\delta)}{n\eta^{1/k}}\right).
\end{align}

Denote $\Delta_1 = \sqrt{\frac{\sigma^2d(\log(d)+\log(1/\delta))}{n\eta^{2/k}}}$.   
From Lemma~\ref{lem.matrix_bernstein}, we know that with probability at least $1 - \delta$,
\begin{align}
    \|\bE_{\hat p'}[(X-\mu_{p^*})(X - \mu_{p^*})^{\top}] - \bE_{p'}[(X-\mu_{p^*})(X - \mu_{p^*})^{\top}]\|_2 \leq  C\max( \Delta_1, \Delta_1^2).
    \end{align}
Thus we know that with probability at least $1-\delta$,
\begin{align}
     \|\bE_{\hat p'}[(X-\mu_{\hat p'})(X - \mu_{\hat p'})^{\top}] \|_2 & \leq   \|\bE_{\hat p'}[(X-\mu_{p^*})(X - \mu_{p^*})^{\top}] \|_2 \nonumber \\
     &\leq \| \bE_{p'}[(X-\mu_{p^*})(X - \mu_{p^*})^{\top}]\|_2 +C\max( \Delta_1, \Delta_1^2) \nonumber \\
     & \leq \frac{1}{1-\eta} + C\max( \Delta_1, \Delta_1^2).
\end{align}

This shows that with probability at least $1-\delta$, $\hat p' \in \GG_2(\eta)$.
\item 
When $k>2$, 
denote $\Delta_2 = \sqrt{\frac{\sigma^2d\log(d)}{n\eta^{2/k}}}$. From Lemma~\ref{lem.matrix_bernstein}, we also know that with probability at least $1 - \delta$,
    \begin{align}\label{eqn.proof_vershynin_expectation}
     \bE_{p'}[ \|\bE_{\hat p'}[(X-\mu_{p^*})(X - \mu_{p^*})^{\top}] - \bE_{p'}[(X-\mu_{p^*})(X - \mu_{p^*})^{\top}]\|_2] \leq C \max(\Delta_2, \Delta_2^2).
\end{align}
for some constant $C$. 
By Lemma~\ref{lem.Wf_orlicz_closed} we know that the distribution of $(v^\top(X-\bE_{p^*}[X]))^2$ has its $(1-\eta)x^{k/2}$-norm bounded by $5\sigma$ under $X\sim p'$. From centering lemma in Lemma~\ref{lem.centering_psi} we know that 
\begin{align}
    \sup_{v\in\bR^d, \|v\|_2 = 1} (1-\eta)\bE_{p'}\left|(v^\top(X-\bE_{p^*}[X]))^2 - \bE_{p'}{(v^\top(X-\bE_{p^*}[X]))^2}\right|^{k/2} \leq (5\sigma)^k.
\end{align}
This combined with~\eqref{eqn.proof_vershynin_expectation} and  Lemma~\ref{lem.empirical_psi_resilient} gives that when $k>2$, with probability at least $1-2\delta$,  for any $r\leq \frac{\hat p'}{1-\eta}$, 
\begin{align}
    & \| \bE_r[(X-\bE_{p^*}[X])(X-\bE_{ p^*}[X])^T] - \bE_{\hat p'}[(X-\bE_{p^*}[X])(X-\bE_{p^*}[X])^T]\|_2 \nonumber \\
    \leq & \frac{C_2k}{1-\eta}\left(\frac{\sigma^2\eta^{1-2/k}}{\delta^{2/k}} + \frac{\max(\Delta_2, \Delta_2^2)}{\delta} \right).
\end{align}

We also have with probability at least $1-\delta$, 
\begin{align}
      \|\bE_{\hat p'}[(X-\mu_{p^*})(X - \mu_{p^*})^{\top}] - I_d\|_2  \leq & \|\bE_{\hat p'}[(X-\mu_{p^*})(X - \mu_{p^*})^{\top}] - \bE_{p'}[(X-\mu_{p^*})(X - \mu_{p^*})^{\top}]\|_2
     \nonumber \\
      & + \| \bE_{p'}[(X-\mu_{p^*})(X - \mu_{p^*})^{\top}] - I_d\|_2 \nonumber \\
      \leq & C(\max( \Delta_1, \Delta_1^2) + \sigma^2 \eta^{1-2/k}),
\end{align}
and from Equation  \eqref{eqn.proof_kth_mean_close_deletion} and \eqref{eqn.proof_kth_mean_empirical_close} , we have with probability at least $1-3\delta$,
\begin{align}
    & \|\bE_r[(X-\bE_{p^*}[X])(X-\bE_{ p^*}[X])^T] - \bE_r[(X-\bE_{\hat p'}[X])(X-\bE_{ \hat p'}[X])^T]\|_2\nonumber \\
     \leq  & \|\bE_r[(X-\bE_{p^*}[X])(X-\bE_{ p^*}[X])^T] - \bE_r[(X-\bE_{r}[X])(X-\bE_{r}[X])^T]\|_2 \nonumber \\
     & + \|\bE_r[(X-\bE_{r}[X])(X-\bE_{r}[X])^T] - \bE_r[(X-\bE_{\hat p'}[X])(X-\bE_{ \hat p'}[X])^T]\|_2\nonumber \\
     \leq & \|\mu_{p^*} -\mu_r\|_2^2 +  \|\mu_{\hat p'} -\mu_r\|_2^2 \nonumber \\
     \leq & (\|\mu_{p^*} - \mu_{\hat p'}\|_2 + \|\mu_{\hat p'} -\mu_r\|_2)^2 + \|\mu_{\hat p'} -\mu_r\|_2^2 \nonumber \\
     \lesssim & k^2\sigma^2(\frac{\eta^{2-2/k}}{\delta^{2/k}} + \frac{d}{n\delta^2})+\frac{\sigma^2  d\log^2(1/\delta) }{n^2\eta^{2/k}}.
\end{align}
Combining above three inqualities,  we know that when $\eta < 1/k$, with probability at least $1-6\delta$, there exists some constant $C_2$ such that
\begin{align}
     \| \bE_r[(X-\bE_{\hat p'}[X])(X-\bE_{ \hat p'}[X])^T] - I_d\|_2 
    \leq  \frac{C_2k^2}{1-\eta}\left(\frac{\sigma^2\eta^{1-2/k}}{\delta^{2/k}} + \frac{\max(\Delta_2, \Delta_2^2)}{\delta} + \frac{k\sigma^2d}{n\delta^2}\right).
\end{align}
Then we know that with probability at least $1-6\delta$, $\hat p' \in \GG_3(\eta)$. We remark here in fact we have shown that $\hat p' \in \GG_1(\eta)\cap \GG_2(\eta)\cap \GG_3(\eta)$ with probability at least $1-6\delta$.
    \end{enumerate}
 \end{proof}

\subsection{Mean estimation with sub-Gaussian distributions}\label{proof.general_subgaussian_tv_projection}

Our first observation is that even if $\GG$ is small, we can take $\MM$ to be the family of resilient 
distributions while maintaining similarly small modulus. 
Thus we only need $\hat{p}_n^*$ to be {resilient},
which is easier to satisfy than e.g.~bounded moments or sub-Gaussianity.
For distributions with moment generating functions, a union bound leads to the following typical result~(Lemma~\ref{lem.subgaussian_empirical_resilience}): if $p^*$ is sub-Gaussian with parameter $\sigma$, then for any fixed $\eta$, the empirical distribution $\hat p_n^*$ 
is $(\rho, \eta)$-resilient with probability $1-\delta$, for $\rho = O(\sigma (\sqrt{\eta \frac{d+\log(1/\delta)}{n}} + \eta \sqrt{\log(1/\eta)}))$, which gives tighter bound for resilience paramter than~\citet[Lemma  4.4]{diakonikolas2019robust}. 
We thus obtain:
\begin{theorem}[Sub-Gaussian]\label{thm.general_subgaussian_tv_projection}
Denote $\tilde \epsilon = 2(\sqrt{\epsilon} + \sqrt{\frac{\log(1/\delta)}{2n}})^2$.
There exist some constants $C_1, C_2$ such that the following statement is true. Take $\GG$ as family of sub-Gaussian with parameter $\sigma$ and $\cM$ as resilient set, i.e.
\begin{align}
    \GG & =  \bigg\{ p \mid \sup_{v\in\bR^d, \|v\|_2=1} \bE_p\bigg[\exp\bigg(\left( \frac{|v^{\top}(X-\bE_p[X])|}{\sigma}\right)^2\bigg) \bigg] \leq  2 \bigg\},\\
  \cM & = \GG^{\TV}_{\mathsf{mean}}\bigg( C_1\sigma\cdot \bigg(\epsilon \sqrt{\log(1/\epsilon)} + \sqrt{ \frac{d+\log(1/\delta)}{n}}\bigg) , \tilde \epsilon \bigg),
\end{align}
where $\GG^{\TV}_{\mathsf{mean}}$ is defined in~\eqref{eqn.oldresilience}.  If $p^*\in\GG$ and $\tilde \epsilon\leq 1/2$,
then the projection $q = \Pi(\hat p_n; \TV / \tTV_\mathcal{H}, \cM)$ of $\hat{p}_n$ onto $\MM$ satisfies:
\begin{align}
 \|\bE_{p^*}[X]- \bE_q[X] \|_2 & \leq   C_2\sigma \cdot \bigg(\epsilon \sqrt{\log(1/\epsilon)} + \sqrt{\frac{d+\log(1/\delta)}{n}} \bigg)
\end{align}
 with probability at least $1-3\delta$. Moreover, 
this bound holds for any $q \in \MM$ within $\TV$ (or $\tTV$) distance $\tilde \eps/2$ of $\hat{p}_n$.

\end{theorem}

\begin{proof}
We verify the five conditions in Theorem~\ref{thm.admissible-2}. 

\begin{enumerate}
    \item \textbf{Robust to perturbation:} True since $\TV$ satisfies triangle inequality.
    \item \textbf{Limited Corruption:} It follows from Lemma~\ref{lem.coupling_TV} that  with probability at least $1-\delta$, 
    \begin{align}
        \TV(\hat p_n, \hat p_n^*)\leq \left(\sqrt{\epsilon} + \sqrt{\frac{\log(1/\delta)}{2n}}\right)^2 = \frac{\tilde\epsilon}{2}.
    \end{align}
\item \textbf{Set for (perturbed) empirical distribution:} It can be seen from Lemma~\ref{lem.subgaussian_empirical_resilience}  that for some fixed $\eta\in[0, 1/2]$, with probability at least $1-\delta$,  there exists some constant $C$ such that
\begin{align}
  \hat p_n^* \in   \GG^{\TV}_{\mathsf{mean}}\left( C \sigma \cdot \left( \sqrt{\frac{d + \log(1/\delta)}{n}} +  \eta \sqrt{\log(1/\eta)} \right), \eta \right).
\end{align}

\item \textbf{Generalized Modulus of Continuity:}

 We construct
$\cM  = \GG'=  \GG^{\TV}_{\mathsf{mean}}\left(C\sigma \cdot \left(  \sqrt{\frac{d + \log(1/\delta)}{n}} +  \tilde{\epsilon} \sqrt{\log(1/\tilde{\epsilon})} \right), \tilde{\epsilon}\right)$. Thus it follows from the population limit of $\GG^{\TV}_{\mathsf{mean}}$ in Lemma~\ref{lem.G_TV_mean_modulus} that for some constant $C_1$, we have
\begin{align}
    \sup_{p_1^*\in \mathcal{M}, p_2^* \in \GG', \TV(p_1^*, p_2^*)\leq \tilde \epsilon} \|\bE_{p_1^*}[X]-\bE_{p_2^*}[X]\|_2 & \leq   C_1\sigma\left(\sqrt{\frac{d+\log(1/\delta)}{n}} + \tilde{\epsilon} \sqrt{\log(1/\tilde{\epsilon})}\right).
\end{align}
Since $f(x) = x\sqrt{\log(1/x)}$ is a concave function, and $f(0)=0$, we have 
\begin{align}
    f(a+b) = \frac{a}{a+b}f(a+b) + \frac{b}{a+b}f(a+b) \leq f(a)+f(b).
\end{align}
From the assumption in theorem statement we know that  $n \geq \log(1/\delta)$, we have
\begin{align}\label{eqn.concave_upper_bound}
    \tilde{\epsilon} \sqrt{\log(1/\tilde{\epsilon})} \lesssim {\epsilon} \sqrt{\log(\frac{1}{\epsilon})} + \frac{\log(1/\delta)}{n}\sqrt{\log(\frac{n}{\log(1/\delta)})} \lesssim {\epsilon} \sqrt{\log(\frac{1}{\epsilon})} + \sqrt{\frac{\log(1/\delta)}{n}}.
\end{align}
From Theorem~\ref{thm.tTV_mean}, we know that the generalized modulus for $\tTV_\mathcal{H}$ is the same as $\TV$ for resilient set.
\item \textbf{Generalization bound:} Since we take $\hat p' = \hat p_n^*$, we have
\begin{align}
    \|\bE_{p^*}[X]- \bE_q[X] \|_2 \leq  \|\bE_{p^*}[X]- \bE_{\hat p_n^*}[X] \|_2 +   \|\bE_{\hat p_n^*}[X] - \bE_{q}[X]\|_2 .
\end{align}
Thus by with probability at least $1-\delta$, 
\begin{align}
        \|\bE_{p^*}[X]- \bE_q[X] \|_2 \leq \|\bE_{\hat p_n^*}[X] - \bE_{q}[X]\|_2  + C\sigma\sqrt{\frac{d+\log(1/\delta)}{n}}.
\end{align}
The convergence of $\|\bE_{\hat p_n^*}[X] - \bE_{p^*}[X]\|_2  $ is from~\citep[Equation (5.5)]{lugosi2017lectures}.
\end{enumerate}

Combining the five conditions, from Theorem~\ref{thm.admissible-2}, for projection algorithm $q = \argmin\{\TV(q, \hat p_n) \mid q \in \cM\}$, we have with probability at least $1-3\delta$,
\begin{align}
 \|\bE_{p^*}[X]- \bE_q[X] \|_2 & \leq C\sigma \cdot \left(\sqrt{\frac{d+\log(1/\delta)}{n}} + \epsilon\sqrt{\log(1/\epsilon)}\right),
\end{align}
where $C$ is some universal constant.
\end{proof}

\subsection{Mean estimation with bounded $k$-th moment}\label{proof.general_kth_tv_projection}

Taking $\cM$ to be the set of resilient distributions as before, we obtain:
\begin{theorem}[Bounded $k$-th moment]\label{thm.general_kth_tv_projection}
Denote $\tilde \epsilon = 2(\sqrt{\epsilon} + \sqrt{\frac{\log(1/\delta)}{2n}})^2$.
There exist some constants $C_1, C_2$ such that the following statement is true.
Take $\GG^\TV$ as bounded $k$-th moment set for $k\geq 2$ and $\cM$ as resilient set, i.e.
\begin{align}
    \GG^\TV & =  \{ p \mid \sup_{v\in\bR^d, \|v\|_2=1} \bE_p[ |v^{\top}(X-\bE_p[X])|^k] \leq  \sigma^k \} \\
  \cM & =  \GG^\TV_{\mathsf{mean}}\bigg({C_1k\sigma}\bigg(\frac{\epsilon^{1-1/k}}{\delta^{1/k}}  + \frac{1}{\delta} \sqrt{\frac{d}{n}}\bigg),  \tilde \epsilon \bigg),
\end{align}
If $p^*\in\GG^\TV$ and $\tilde \epsilon \leq 1/2$, 
then the projection $q = \Pi(\hat p_n; \TV / \tTV_\mathcal{H}, \cM)$ of $\hat{p}_n$ onto $\MM$ satisfies: 
\begin{align}
    \|\bE_{p^*}[X] - \bE_q[X]\|_2 & \leq  C_2k\sigma\cdot\left(\frac{\epsilon^{1-1/k}}{\delta^{1/k}} +\frac{1}{\delta}\sqrt{\frac{d}{n}}\right)
\end{align}
with probability at least $1-4\delta$.  Moreover, 
this bound holds for any $q \in \MM$ within $\TV$ (or $\tTV$) distance $\tilde \eps/2$ of $\hat{p}_n$.
\end{theorem}

\citet{steinhardt2018resilience}
presented an analysis for the same projection algorithm that requires $d^{3/2}$ samples, which our result improves to $d$. 

\begin{proof}

Among the five conditions in Theorem~\ref{thm.admissible-2}, 
we only need to verify the set for (perturbed) empirical distribution,  generalized modulus of continuity and generalization bound. Other two conditions are identical to the proof in  Appendix~\ref{proof.general_subgaussian_tv_projection}.

\begin{enumerate}
    \item \textbf{Set for (perturbed) empirical distribution:}
     From Corollary~\ref{cor.bounded_kth_empirical_resilience}, we know that with probability at least $1-2\delta$, 
\begin{align}
\hat p_n^* \in \bigcap_{\eta\in[0, 1)}\GG_{\mathsf{mean}}\left(\frac{Ck\sigma}{1-\eta}\left(\frac{\eta^{1-1/k}}{\delta^{1/k}}  + \frac{1}{\delta} \sqrt{\frac{d}{n}}\right), \eta \right) = \GG'.
\end{align}
for some constant $C$. 
\item \textbf{Generalized modulus of continuity:}
Denote $\tilde{\epsilon} = \left(\sqrt{\epsilon} + \sqrt{\frac{\log(1/\delta)}{2n}}\right)^2$, from assumption  we know that $\tilde \epsilon <1/4$. 
Since $f(x) = x^{1-1/k}$ is a concave function, following a similar analysis as~\eqref{eqn.concave_upper_bound}, we know that 
\begin{align}
    \tilde \epsilon^{1-1/k} \leq 2\epsilon^{1-1/k} + 2\sqrt{\frac{\log(1/\delta)}{n}}.
\end{align}
Thus with appropriate choice of $C_1$, we can  make $$\GG' \subset \cM = \GG_{\mathsf{mean}}\left({C_1k\sigma}\left(\frac{\epsilon^{1-1/k}}{\delta^{1/k}}  + \frac{1}{\delta} \sqrt{\frac{d}{n}}\right), \left(\sqrt{\epsilon} + \sqrt{\frac{\log(1/\delta)}{2n}}\right)^2 \right) $$.  Therefore the generalized modulus of continuity for perturbation level $\tilde \epsilon$ is upper bounded by ${C_1k\sigma}\left(\frac{\epsilon^{1-1/k}}{\delta^{1/k}}  + \frac{1}{\delta} \sqrt{\frac{d}{n}}\right)$. From Theorem~\ref{thm.tTV_mean}, we know that the generalized modulus for $\tTV_\mathcal{H}$ is the same as $\TV$ for resilient set.

\item \textbf{Generalization bound:}
From Lemma~\ref{lem.kth_convergence}, we know that with probability at least $1-\delta$, 
\begin{align}
        \|\bE_{p^*}[X]- \bE_q[X] \|_2 \leq \|\bE_{\hat p_n^*}[X] - \bE_{q}[X]\|_2  + \|\bE_{p^*}[X] - \bE_{\hat p_n^*}[X] \|_2.
\end{align}
By Chebyshev's inequality, we have
\begin{align}
    \bP_{p^*}(\|\bE_{p^*}[X] - \bE_{\hat p_n^*}[X] \|_2 \geq t)\leq \frac{\bE_{p^*}[\|\bE_{p^*}[X] - \bE_{\hat p_n^*}[X] \|_2^k ]}{t^k}
\end{align}
Since  the $k$-th moment is bounded, by Khinchine's inequality~\cite{haagerup1981best}, there is
\begin{align}
   \bE_{p^*}  \left \| \frac{1}{n} \sum_{i = 1}^n  X_i - \mathbb{E}_{p^*}[X] \right \|_2^k \leq    \bE_{X_i\sim p^*, \xi\sim \{\pm 1\}^d} \left|  \frac{1}{n} \xi^{\top}( \sum_{i = 1}^n  X_i - \mathbb{E}_{p^*}[X] )\right|^k  .
\end{align}
By Marcinkiewicz-Zygmund inequality~\cite{ren2001best} there exists some $C_2$, such that for any $v\in\bR^d$,
\begin{align}
  \bE_{X_i\sim {p^*}}\left| \frac{1}{n} v^{\top}(  \sum_{i = 1}^n  X_i - \mathbb{E}_{p^*}[X] )\right|^k \leq  \frac{(C_2\sigma \sqrt{k})^k}{n^{k/2}}\| v\|_2^k.
\end{align}
Therefore by first conditioning on $\xi$, we have
\begin{align}
    \bE_{X_i\sim p, \xi\sim \{\pm 1\}^d} \left|  \frac{1}{n} \xi^{\top}( \sum_{i = 1}^n  X_i - \mathbb{E}_p[X] )\right|^k   & \leq (\frac{C_2\sigma\sqrt{k}}{n^{1/2}})^k  \bE_{ \xi\sim \{\pm 1\}^d} \|\xi \|_2^k   \nonumber \\
    & = \left(C_2\sigma\sqrt{k}\sqrt{\frac{d}{n}}\right)^k.
\end{align}

Thus overall, we know that with probability at least $1-\delta$, there exists some constant $C$, such that 
\begin{align}
   \left \| \frac{1}{n} \sum_{i = 1}^n  X_i - \mathbb{E}_{p}[X] \right \|_2 \leq \frac{C \sigma\sqrt{k}}{\delta^{1/k}} \sqrt{\frac{d}{n}}.
\end{align}
We have with probability at least $1-\delta$,
\begin{align}
    \|\bE_{p^*}[X]- \bE_q[X] \|_2 \leq \|\bE_{\hat p_n^*}[X] - \bE_{q}[X]\|_2  + \frac{C \sigma\sqrt{k}}{\delta^{1/k}} \sqrt{\frac{d}{n}}.
\end{align}
\end{enumerate}

Combining the five conditions, from Theorem~\ref{thm.admissible-2}, for projection algorithm $q = \argmin\{\TV(q, \hat p_n) \mid q \in \cM\}$, we have with probability at least $1-4\delta$,
\begin{align}
    \|\mu_{p^*}-\mu_q\|_2 & \leq   Ck\sigma\cdot\left(\frac{\epsilon^{1-1/k}}{\delta^{1/k}} +\frac{1}{\delta}\sqrt{\frac{d}{n}}\right).
\end{align}
\end{proof}

\subsection{Mean estimation via projecting to bounded covariance set}\label{proof.projection_bounded_covariance}

The below theorem is following a similar flow of proof as in~\citep{prasad2019unified}, which shows the performance guarantee for filtering algorithm. Here we use our framework to give a proof of the performance guarantee for general projection algorithm. 

\begin{theorem}[Bounded covariance, $\TV$ projection]\label{thm.projection_bounded_covariance}
Denote
\begin{align}
\epsilon_1 = \max\left(\frac{d\log(d/\delta)}{n}, \epsilon + \frac{\log(1/\delta)}{n}\right),
    \tilde \epsilon = 4\left(\sqrt{\epsilon_1} + \sqrt{\frac{\log(1/\delta)}{2n}}\right)^2.
\end{align}
We take both $\GG$ and $\cM$ to be the set of bounded covariance set as below:
\begin{align}
    \GG &= \{p \mid \|\Sigma_p \|_2 \leq \sigma^2 \}, \\
     \cM & = \left\{p \mid  \| \Sigma_p\|_2 \leq C_2\sigma^2\left(1 + \frac{d\log(d/\delta)}{n \epsilon_1}\right) \right\}.
\end{align}

If $p^*\in\GG$ and $\tilde \epsilon<1/2$, then the projection $q = \Pi(\hat p_n; \TV / \tTV_\mathcal{H}, \cM)$ of $\hat{p}_n$ onto $\MM$ satisfies:
\begin{align}
 \|\bE_{p^*}[X]- \bE_q[X] \|_2 & \leq C_3\sigma\cdot \left(\sqrt{\epsilon} + \sqrt{\frac{d\log(d/\delta)}{n}}\right)
\end{align}
with probability at least $1-3\delta$. Moreover,
this remains true for any $q \in \MM$ within $\TV$ (or $\tTV$) distance $\tilde \eps/2$ of $\hat{p}_n$.
\end{theorem}

\begin{proof}

Among the five conditions in Theorem~\ref{thm.admissible-2}, the `robust to perturbation' and `limited corruption' conditions are identical to the proof in  Appendix~\ref{proof.general_subgaussian_tv_projection}. 
we only need to verify the other three conditions. 

\begin{enumerate}

\item \textbf{Set for (perturbed) empirical distribution:}

For any fixed $\epsilon$ as perturbation level, we show that there exists some distribution $\hat p'$ that has bounded covariance and $\TV(\hat p', \hat p_n)$ is small.

We truncate the distribution $p^*$ by removing all $X$  with $\|X - \mu_{p^*} \| \geq \sigma \sqrt{d}/ \sqrt{\epsilon_1}$ to get a new distribution $p'$, where $\epsilon_1$ is some parameter to be specified later.  Denote the empirical distribution of $p'$ with $n$ samples as $\hat p'$.
From Lemma~\ref{lem.kth_deletion_set}, we know that  $\TV(p^*, p')\leq \epsilon_1$.  It follows from Lemma~\ref{lem.coupling_TV} that with probability at least $1-\delta$, 
    \begin{align}
        \TV(\hat p', \hat p_n^*)& \leq \left(\sqrt{\epsilon_1} + \sqrt{\frac{\log(1/\delta)}{2n}}\right)^2.
    \end{align}
Denote  $\Delta_1 = \sqrt{\frac{\sigma^2d(\log(d/\delta))}{n\epsilon_1}}$ and 
\begin{align}  
     \mathcal{G}' & =  \{p \mid \|\bE_p[(X-\mu_p)(X-\mu_p)^\top] \|_2 \leq \sigma^2 + C\max(\Delta_1, \Delta_1^2) \},
\end{align}
where $C$ is some universal constant to be specified.
From Lemma~\ref{lem.matrix_bernstein}, we know that with probability at least $1 - \delta$,
\begin{align}
    \|\bE_{\hat p'}[(X-\mu_{p^*})(X - \mu_{p^*})^{\top}] - \bE_{p'}[(X-\mu_{p^*})(X - \mu_{p^*})^{\top}]\|_2 \leq  C\max( \sigma \Delta_1, \Delta_1^2).
    \end{align}
Thus we know that with probability at least $1-\delta$, 
\begin{align}
     \|\bE_{\hat p'}[(X-\mu_{\hat p'})(X - \mu_{\hat p'})^{\top}] \|_2 & \leq   \|\bE_{\hat p'}[(X-\mu_{p^*})(X - \mu_{p^*})^{\top}] \|_2 \nonumber \\
     &\leq \| \bE_{p'}[(X-\mu_{p^*})(X - \mu_{p^*})^{\top}]\|_2 +C\max(\sigma  \Delta_1, \Delta_1^2) \nonumber \\
     & \leq \frac{\sigma^2}{1-\epsilon} + C\max( \sigma \Delta_1, \Delta_1^2).
\end{align}
Here we use the fact that $p'$ is a deletion of $p^*$, thus $\bE_{p'}[(v^\top (X-\mu_{p^*}))^2] \leq \frac{1}{1-\epsilon} \bE_{p^*}[(v^\top (X-\mu_{p^*}))^2]$.
Thus we have  $\hat p' \in \GG'$ with probability at least $1-\delta$.
\item \textbf{Generalized Modulus of Continuity:}
Since $\epsilon + \log(1/\delta)/n \leq \epsilon_1$, the perturbation level for modulus of continuity is $2\left(\sqrt{\epsilon_1} + \sqrt{\frac{\log(1/\delta)}{2n}}\right)^2 + 2\left(\sqrt{\epsilon} + \sqrt{\frac{\log(1/\delta)}{2n}}\right)^2 \leq 4\left(\sqrt{\epsilon_1} + \sqrt{\frac{\log(1/\delta)}{2n}}\right)^2$. Denote the right hand side as $\tilde \epsilon$. 
From $\hat p' \in \GG'$, and that both $\GG'$ and $\cM$ guarantees the covariance to be upper bounded, we know from Lemma~\ref{lem.cvx_mean_resilience} that for any $\tilde \epsilon \in [0, 1/2]$, there exist constants $C_1, C_2$ such that
\begin{align}
      \sup_{p\in\GG', q\in\cM, \TV(p, q)\leq \tilde \epsilon} \|\bE_p[X] - \bE_q[X]\|_2   & \leq \sqrt{ \frac{\sigma^2}{1-\epsilon} +C_1\max(\sigma \Delta_1, \Delta_1^2)}\sqrt{ \tilde \epsilon} \nonumber \\
      & \leq C_2\sigma\cdot \left(1 +  \frac{d\log(d/\delta)}{n  \epsilon_1}\right)\epsilon_1.
\end{align}
From Theorem~\ref{thm.tTV_mean}, we know that the generalized modulus for $\tTV_\mathcal{H}$ is the same as $\TV$ for bounded covariance set.
\item \textbf{Generalization bound:} Note that $p'$ is a $\epsilon_1$-deletion of $p^*$. By triangle inequality and the resilient condition for $p^*$, we have
\begin{align}\label{eqn.generalization_bound_bounded_cov}
    \|\bE_{p^*}[X]- \bE_q[X] \|_2 & \leq \|\bE_{p^*}[X]- \bE_{p'}[X] \|_2+  \|\bE_{p'}[X]- \bE_{\hat p'}[X] \|_2 +   \|\bE_{\hat p'}[X] - \bE_{q}[X]\|_2  \nonumber \\
    & \leq \sigma \sqrt{\epsilon_1} + \|\bE_{p'}[X]- \bE_{\hat p'}[X] \|_2 +   \|\bE_{\hat p'}[X] - \bE_{q}[X]\|_2.
\end{align}
From Lemma~\ref{lem.bounded_support_sub_gaussian_tail} and the assumption that $\epsilon + \log(1/\delta)/n \leq \epsilon_1$, we know that with probability at least $1-\delta$, there exists some constant $C_3$ such that 
\begin{align}
    \|\bE_{p'}[X]- \bE_{\hat p'}[X] \|_2 & \leq C_3\left(\sigma\sqrt{\frac{d}{n}} + \sigma \sqrt{\frac{\log(1/\delta)}{n}}+ \frac{\sigma\sqrt{d}\log(1/\delta)}{n\sqrt{\epsilon + \log(1/\delta)/n}}\right)\nonumber \\
    & \leq C_3\left(\sigma\sqrt{\frac{d}{n}} + \sigma \sqrt{\frac{\log(1/\delta)}{n}}+ \frac{\sigma\sqrt{d}\log(1/\delta)}{n\sqrt{  \log(1/\delta)/n}}\right) \nonumber \\
    & =  C_3\left(\sigma\sqrt{\frac{d}{n}} + \sigma \sqrt{\frac{\log(1/\delta)}{n}}+ \sigma \sqrt{\frac{{d\log(1/\delta)}}{{n}}}\right)
\end{align}
Thus with probability at least $1-\delta$, there exists some constant $C_4$ such that
\begin{align}
        \|\bE_{p^*}[X]- \bE_q[X] \|_2 \leq \|\bE_{\hat p'}[X] - \bE_{q}[X]\|_2  + \sigma \sqrt{\epsilon_1}+ C_4 \sigma \sqrt{\frac{d\log(1/\delta)}{n}}.
\end{align}
\end{enumerate}

Combining the five conditions, from Theorem~\ref{thm.admissible-2}, for projection algorithm $q = \Pi(\hat p_n; \TV, \cM, \tilde \epsilon/2)$, there exists some constnat $C$ such that with probability at least $1-3\delta$,
\begin{align}
 \|\bE_{p^*}[X]- \bE_q[X] \|_2 & \leq  C\sigma\cdot ((1 + \frac{d\log(d/\delta)}{n\epsilon_1})\sqrt{ \epsilon_1} +\sqrt{\frac{d\log(1/\delta)}{n}}).
\end{align}

By taking $\epsilon_1  = \max(\frac{d\log(d/\delta)}{n}, \epsilon+\frac{\log(1/\delta)}{n})$, we can see that $ 1 + \frac{d\log(d/\delta)}{n\epsilon_1}\leq 2$. Thus we can get the following bound:
\begin{align}
 \|\bE_{p^*}[X]- \bE_q[X] \|_2 & \leq C \sigma\cdot \left (\sqrt{\epsilon} + \sqrt{\frac{d\log(d/\delta)}{n}}\right ).
\end{align}

\end{proof}

\subsection{Proof of Theorem~\ref{thm.projection_kth_moment_identity_covariance}}\label{proof.bounded_kth_identity}
\begin{proof}
Among the five conditions in Theorem~\ref{thm.admissible-2}, the robust to perturbation and limited corruption condition is identical to the proof in  Appendix~\ref{proof.general_subgaussian_tv_projection}. 
We only need to verify the other three conditions. 

\begin{enumerate}

    \item \textbf{Set for (perturbed) empirical distribution:}

For any fixed $\epsilon$ as perturbation level, we show that there exists some distribution $\hat p'$ that satisfies conditions required in Lemma~\ref{lem.tTV_identitycov_modulus} for modulus of continuity and $\TV(\hat p', \hat p_n)$ is small.

We truncate the distribution $p^*$ by removing all $X$  with $\|X - \mu_{p^*} \| \geq \sigma (kd)^{1/2}/ \epsilon_1^{1/k}$ to get a new distribution $p'$, where $\epsilon_1$ is some paramter to be specified later and we assume $\epsilon_1 \geq \epsilon+\log(1/\delta)/n$. Denote the empirical distribution of $p'$ with $n$ samples as $\hat p'$.
From Lemma~\ref{lem.kth_deletion_set}, we know that $\TV(p^*, p')\leq \epsilon_1$.  It follows from Lemma~\ref{lem.coupling_TV} that with probability at least $1-\delta$, 
    \begin{align}
        \TV(\hat p', \hat p_n^*)\leq \left(\sqrt{\epsilon_1} + \sqrt{\frac{\log(1/\delta)}{2n}}\right)^2.
    \end{align}
Then we apply the result in  Lemma~\ref{lem.kth_deletion_set}.    
Denote  $\Delta_1 = \sqrt{\frac{\sigma^2d\log(d/\delta)}{n \epsilon_1^{2/k}}}$, $\Delta_2 = \sqrt{\frac{\sigma^2d\log(d)}{n \epsilon_1^{2/k}}}$, and 
\begin{align}  
     \mathcal{G}' & =  \Bigg\{p:  \forall r \leq \frac{p}{1-\epsilon_1}, \| \mathbb{E}_r[X] - \bE_p[X]\|_2 \leq \frac{C_1k\sigma}{1-\epsilon_1}\left(\frac{\epsilon_1^{1-1/k}}{\delta^{1/k}} + {\frac{1}{\delta}\sqrt{\frac{d}{n}}}\right), \nonumber \\
     & \qquad \|\bE_p[(X-\mu_p)(X-\mu_p)^\top] \|_2 \leq \frac{1}{1-\epsilon_1}+ C_3\max(\Delta_1, \Delta_1^2), \nonumber \\
     & \qquad  \| \bE_r[(X-\bE_{ p}[X])(X-\bE_{p}[X])^T] - I_d\|_2  \leq  \frac{C_2k}{1-\epsilon_1}\left(\frac{k\sigma^2\epsilon_1^{1-2/k}}{\delta^{2/k}} + \frac{\max(\Delta_2, \Delta_2^2)}{\delta} + \frac{k\sigma^2d}{n\delta^2}\right)\Bigg \}
\end{align}
where $C_1, C_2, C_3$ are some universal constants. Then when $k>2$, from Lemma~\ref{lem.kth_deletion_set}, under appropriate choice of constants $C_1, C_2, C_3$, for any $ \epsilon_1 <1$, we have $\hat p' \in \GG'$ with probability at least $1-6\delta$.

\item \textbf{Generalized Modulus of Continuity:}

Under the assumption that $\epsilon + \log(1/\delta)/n \leq \epsilon_1$, 
the perturbation level for modulus of continuity is $2\left(\sqrt{\epsilon_1} + \sqrt{\frac{\log(1/\delta)}{2n}}\right)^2 + 2\left(\sqrt{\epsilon} + \sqrt{\frac{\log(1/\delta)}{2n}}\right)^2 \leq 4\left(\sqrt{\epsilon_1} + \sqrt{\frac{\log(1/\delta)}{2n}}\right)^2$. Denote the right hand side as $\tilde \epsilon$. Then we also have $\tilde \epsilon \lesssim \epsilon_1$. Now we take $\cM$ as 
\begin{align}
    \cM = \left\{ p \mid \bE_p[(X-\mu_p)(X-\mu_p)]\leq 1 + \frac{C_2k}{1-\epsilon_1}\left(\frac{k\sigma^2\epsilon_1^{1-2/k}}{\delta^{2/k}} + \frac{\max(\Delta_2, \Delta_2^2)}{\delta} + \frac{k\sigma^2d}{n\delta^2}\right) \right\}.
\end{align}

When $k>2$, in Lemma~\ref{lem.empirical_identity_cov_modulus}, take $\rho_1 = \frac{C_1k\sigma}{1-\epsilon_1}\left(\frac{\epsilon_1^{1-1/k}}{\delta^{1/k}} + {\frac{1}{\delta}\sqrt{\frac{d}{n}}}\right)$, $\rho_2 = \tau = \frac{C_2k}{1-\epsilon_1}\left(\frac{k\sigma^2\epsilon_1^{1-2/k}}{\delta^{2/k}} + \frac{\max(\Delta_2, \Delta_2^2)}{\delta} + \frac{k\sigma^2d}{n\delta^2}\right)$. We have for  any $ \tilde \epsilon \in [0, 1)$, 
\begin{align}
      \sup_{p\in\GG', q\in\cM, \TV(p, q)\leq  \tilde \epsilon} \|\bE_p[X] - \bE_q[X]\|_2   \leq \frac{C_6}{1- \epsilon_1}\left(\frac{k\sigma \epsilon_1^{1-1/k}}{\delta^{1/k}} + \sqrt{\frac{k\max(\Delta_2, \Delta_2^2) \epsilon_1}{\delta}} + \frac{k\sigma}{\delta}\sqrt{\frac{d}{n}}\right).
\end{align}
From Lemma~\ref{lem.tTV_identitycov_modulus}, we know that the generalized modulus for $\tTV_\mathcal{H}$ is the same as $\TV$ for $\GG'$ and $\cM$.
\item \textbf{Generalization bound:}

From the same argument as~\eqref{eqn.generalization_bound_bounded_cov}, we know that with probability at least $1-\delta$, 
\begin{align}
    \|\bE_{p^*}[X]- \bE_q[X] \|_2 \leq \|\bE_{\hat p'}[X] - \bE_{q}[X]\|_2  + C_7\left(\sigma\tilde \epsilon^{1-1/k} + \sigma\sqrt{\frac{d}{n}} + \sigma \sqrt{\frac{\log(1/\delta)}{n}}+ \frac{\sigma\sqrt{d}\log(1/\delta)}{n\tilde \epsilon^{1/k}}\right).
\end{align}
\end{enumerate}

By taking $\epsilon_1 = \epsilon + \frac{\log(1/\delta)}{n}$, we can see that $\tilde \epsilon \asymp \epsilon_1 \asymp \epsilon + \frac{\log(1/\delta)}{2n}$.
Combining the five conditions, from Theorem~\ref{thm.admissible-2}, for projection algorithm $q = \Pi(\hat p_n; \TV, \cM, \tilde \epsilon/2)$ where $\tilde \epsilon < 1/2$, we have with probability at least $1-8\delta$, 
\begin{align}
    \|\mu_{p^*}-\mu_q\|_2 & \lesssim \frac{k\sigma \tilde \epsilon^{1-1/k}}{\delta^{1/k}}  + \sqrt{\frac{k\max(\Delta_2, \Delta_2^2)\tilde \epsilon}{\delta}} + \frac{k\sigma}{\delta}\sqrt{\frac{d}{n}}+ \sigma\sqrt{\frac{\log(1/\delta)}{n}}+ \frac{\sigma\sqrt{d}\log(1/\delta)}{n\tilde \epsilon^{1/k}} \nonumber \\ 
    & \lesssim \frac{k\sigma \tilde \epsilon^{1-1/k}}{\delta^{1/k}}  + \sqrt{\frac{k\max(\Delta_2, \Delta_2^2)\tilde \epsilon}{\delta}} + \frac{k\sigma}{\delta}\sqrt{\frac{d}{n}}. \nonumber \\ 
     & \lesssim \frac{k\sigma \tilde \epsilon^{1-1/k}}{\delta^{1/k}}  + \sqrt{\frac{k\Delta_2\tilde \epsilon}{\delta}} + \sqrt{\frac{k\Delta_2^2\tilde \epsilon}{\delta}}+ \frac{k\sigma}{\delta}\sqrt{\frac{d}{n}} \nonumber \\
     & \lesssim \frac{k\sigma \tilde \epsilon^{1-1/k}}{\delta^{1/k}}  + \sqrt{\frac{k\tilde \epsilon^{1-1/k} }{\delta} \sqrt{\frac{\sigma^2d\log(d)}{n}}} + \sqrt{\frac{\sigma^2kd\log(d)\tilde \epsilon^{1-2/k}}{n\delta}} + \frac{k\sigma}{\delta}\sqrt{\frac{d}{n}} \nonumber \\
     & \lesssim \frac{k\sigma \tilde \epsilon^{1-1/k}}{\delta^{1/k}}  + \frac{1}{\delta}\sqrt{\frac{\sigma^2kd\log(d)}{n}} + \sqrt{k}\tilde \epsilon^{1-1/k} + \sqrt{\frac{\sigma^2kd\log(d)\tilde \epsilon^{1-2/k}}{n\delta}}+ \frac{k\sigma}{\delta}\sqrt{\frac{d}{n}}\nonumber \\
     & \lesssim \frac{k\sigma \tilde \epsilon^{1-1/k}}{\delta^{1/k}} + \frac{k\sigma}{\delta}\sqrt{\frac{d\log(d)}{n}} \nonumber \\ 
     & \lesssim \frac{k\sigma  \epsilon^{1-1/k}}{\delta^{1/k}} + \frac{k\sigma}{\delta}\sqrt{\frac{d\log(d)}{n}}.  \nonumber
\end{align} 

Thus we have with probability at least $1-\delta$,
\begin{align}
     \|\mu_{p^*}-\mu_q\|_2 & \leq   C_3k\sigma\cdot \left(\frac{ \epsilon^{1-1/k}}{\delta^{1/k}} + \frac{1}{\delta}\sqrt{\frac{d\log(d)}{n}}\right).
\end{align}
for some constant $C_3$. 

We remark here that 
by appropriate choice of constants, we can make the projection set $\cM$ smaller than the set $\cM$ in Theorem~\ref{thm.projection_bounded_covariance}, and we know that bounded $k$-th moment would imply bounded second moment. Thus the bound in Theorem~\ref{thm.projection_bounded_covariance} also applies to the projection algorithm here, the final bound shall be the minimum of the two terms. 
\end{proof}

\subsection{Improving the sample complexity by reducing high dimensional mean estimation to one dimension  }\label{appendix.highd_to_oned}

We show that by adopting a different algorithm and analysis technique, we can improve the sample complexity in bounded Orlicz norm distributions discussed in Section~\ref{sec.expand}. We first that this results in good sample complexity in one-dimensional mean estimation, and show that high dimensional mean estimation can be reduced to a one-dimension problem. 

\begin{theorem}[One dimensional mean estimation for bounded Orlicz norm distribution under $\TV$ projection]\label{thm.oned_projection}
Assume the corruption model is either oblivious corruption (Definition~\ref{def.obliviouscorruption}) or adaptive corruption model (Definition~\ref{def.adaptivecont}) of level $\epsilon$. For some Orlicz function $\psi$ that satisfies $\psi(x)\geq x$ when $x\geq 1$, 
we take both $\GG$ and $\cM$ to be the set of one-dimensional bounded Orlicz norm distributions as below:
\begin{align}
    \GG &= \left\{p\mid \mathbb{E}_p \left[  \psi\left(\frac{(X-\mu_p)^2}{\sigma^2}\right) \right]\leq 1 \right\}, \\
     \cM & = \left\{p \mid  \mathbb{E}_p \left[ \psi\left(\frac{(X-\mu_p)^2}{4\sigma^2}\right) \right] \leq 4 \right\}.
\end{align}
Denote
\begin{align}
    \tilde \epsilon = 4\left(\sqrt{\epsilon + \frac{\log(1/\delta)}{n}} + \sqrt{\frac{\log(1/\delta)}{2n}}\right)^2.
\end{align}
If $p^*\in\GG$ and $\tilde \epsilon<1/2$, then  the projection algorithm $q = \Pi\left(\hat p_n; \TV, \cM\right)$  or $q = \Pi\left(\hat p_n; \TV, \cM, \tilde \epsilon/2\right)$ satisfies the following with probability at least $1-3\delta$:
\begin{align}
 \|\bE_{p^*}[X]- \bE_q[X] \|_2 & \leq C\sigma\cdot \left(\epsilon\sqrt{\psi^{-1}(1/\epsilon)} + \sqrt{\frac{\log(1/\delta)}{n}}\right),
\end{align}
where $C$ is some universal constant.
\end{theorem}

\begin{proof}

Among the five conditions in Theorem~\ref{thm.admissible-2}, the `robust to perturbation' and `limited corruption' conditions are identical to the proof in  Appendix~\ref{proof.general_subgaussian_tv_projection}. 
we only need to verify the other three conditions. Denote $\epsilon_1 = \epsilon + \frac{\log(1/\delta)}{n}$.

\begin{enumerate}

\item \textbf{Set for (perturbed) empirical distribution:}

For any fixed $\epsilon$ as perturbation level, we show that there exists some distribution $\hat p'$ that has bounded Orlicz norm (inside $\cM$)
and $\TV(\hat p', \hat p_n)$ is small.

We truncate the distribution $p^*$ by removing all $X$  with $|X - \mu_{p^*} | \geq \sigma \sqrt{\psi^{-1}(1/\epsilon_1)}$ to get a new distribution $p'$, where $\epsilon_1$ is some parameter to be specified later.  Denote the empirical distribution of $p'$ with $n$ samples as $\hat p'$. By Markov's inequality, we have for any $t\geq 0$,
\begin{align}
    \bP_{p^*}(|X-\mu_{p^*}|\geq t) = \bP_{p^*}\left(\psi(\frac{(X-\mu_{p^*})^2}{\sigma^2})\geq \psi(\frac{t^2}{\sigma^2})\right) \leq \frac{\bE_{p^*}[\psi((X-\mu_{p^*})^2/\sigma^2)]}{\bE_{p^*}[\psi(t^2/\sigma^2)]}  \leq \frac{1}{\bE_{p^*}[\psi(t^2/\sigma^2)]}.\nonumber
\end{align}

By taking $t = \sigma\sqrt{\psi^{-1}(1/\epsilon_1)}$, we know that  $\TV(p^*, p')\leq \epsilon_1$.  It follows from Lemma~\ref{lem.coupling_TV} that with probability at least $1-\delta$, 
    \begin{align}
        \TV(\hat p', \hat p_n^*)& \leq \left(\sqrt{\epsilon_1} + \sqrt{\frac{\log(1/\delta)}{2n}}\right)^2.
    \end{align}
Consider the random variable $\psi(\frac{(X-\mu_{p^*})^2}{\sigma^2})$, where $X\sim p'$, we know that 
\begin{align}
     \psi(\frac{(X-\mu_{p^*})^2}{\sigma^2}) & \leq \frac{1}{\epsilon_1} \quad \text{a.s.}, \\ 
    \bE_{p'}[\psi(\frac{(X-\mu_{p^*})^2}{\sigma^2})]& \leq 1, \\
    \mathsf{Var}(\psi(\frac{(X-\mu_{p^*})^2}{\sigma^2})) & = \bE_{p'}[\psi^2(\frac{(X-\mu_{p^*})^2}{\sigma^2})] - \bE_{p'}[\psi(\frac{(X-\mu_{p^*})^2}{\sigma^2})]^2 \leq \frac{1}{\epsilon_1}.
\end{align}
By Bernstein's inequality, we know that
\begin{align}
    \bP_{p'}[\sum_{i=1}^n (\psi(\frac{(X_i-\mu_{p^*})^2}{\sigma^2}) -\bE_{p'}[\psi(\frac{(X-\mu_{p^*})^2}{\sigma^2})]) \geq t ]\leq \exp\left(-\frac{t^2}{2n/\epsilon_1 + 2t/3\epsilon_1}\right).
\end{align}
Let RHS be $\delta$ and solving $t$, we know that with probability at least $1-\delta$,
\begin{align}
    \bE_{\hat p'}[\psi(\frac{(X-\mu_{p^*})^2}{\sigma^2})] & \leq \bE_{p'}[\psi(\frac{(X-\mu_{p^*})^2}{\sigma^2})] + \left(\sqrt{\frac{2\log(1/\delta)}{n\epsilon_1}} + \frac{2\log(1/\delta)}{3n\epsilon_1}\right) \nonumber \\ 
    & \leq 1 + \left(\sqrt{\frac{2\log(1/\delta)}{n\epsilon_1}} + \frac{2\log(1/\delta)}{3n\epsilon_1}\right) < 4.
\end{align}
From centering lemma in~\ref{lem.centering_psi} we know that with probability at least $1-\delta$,
\begin{align}
     \bE_{\hat p'}\left[\psi\left(\frac{(X-\mu_{\hat p'})^2}{4\sigma^2}\right)\right] \leq 4. 
\end{align}
By taking $\GG' = \cM$, we know that   $\hat p' \in \GG'$ with probability at least $1-\delta$.
\item \textbf{Generalized Modulus of Continuity:}
Since $\epsilon + \log(1/\delta)/n \leq \epsilon_1$, the perturbation level for modulus of continuity is $2\left(\sqrt{\epsilon_1} + \sqrt{\frac{\log(1/\delta)}{2n}}\right)^2 + 2\left(\sqrt{\epsilon} + \sqrt{\frac{\log(1/\delta)}{2n}}\right)^2 \leq 4\left(\sqrt{\epsilon_1} + \sqrt{\frac{\log(1/\delta)}{2n}}\right)^2$. Denote the right hand side as $\tilde \epsilon$. Then we also know $\tilde \epsilon \lesssim \epsilon_1$. 
We know from Lemma~\ref{lem.cvx_mean_resilience} that for any $\tilde \epsilon \in [0, 1/2]$, there is
\begin{align}\label{eqn.proof_kth_id_tv1}
      \sup_{p\in\GG', q\in\cM, \TV(p, q)\leq \tilde \epsilon} \|\bE_p[X] - \bE_q[X]\|_2  
      & \lesssim \sigma \epsilon_1\sqrt{\psi^{-1}(1/\epsilon_1)}. 
\end{align}

\item \textbf{Generalization bound:} 
We first show that $\bE_{p'}[\psi(X^2/\sigma^2)]\leq 1$ implies $\bE_{p'}[X^2]\leq 2\sigma^2$. Note that $\bE_{p'}[\psi(X^2/\sigma^2)]\leq 1$ is equivalent to
\begin{align}
    1 & \geq \bP_{p'}(|X|\leq \sigma)\bE_{p'}[\psi(X^2/\sigma^2) \mid  |X| \leq \sigma] + \bP_{p'}(|X|> \sigma)\bE_{p'}[\psi(X^2/\sigma^2) \mid  |X| > \sigma] \nonumber \\ 
    & \geq \bP_{p'}(|X|> \sigma)\bE_{p'}[X^2/\sigma^2 \mid  |X| > \sigma],\nonumber
\end{align}
since $\psi(x)\geq x$ for $x\geq 1$. Thus we have
\begin{align}
    \bE_{p'}[X^2/\sigma^2] & = \bP_{p'}(|X|\leq \sigma)\bE_{p'}[X^2/\sigma^2 \mid  |X| \leq \sigma] + \bP_{p'}(|X|> \sigma)\bE_{p'}[X^2/\sigma^2 \mid  |X| > \sigma]\nonumber \\ 
    & \leq 2. \nonumber
\end{align}

Thus we know that $p'$ has its variance bounded. 
Note that $p'$ is a $\epsilon_1$-deletion of $p^*$. By triangle inequality and the resilient condition for $p^*$, we have
\begin{align}
    |\bE_{p^*}[X]- \bE_q[X] | & \leq |\bE_{p^*}[X]- \bE_{p'}[X] |+  |\bE_{p'}[X]- \bE_{\hat p'}[X] | +   |\bE_{\hat p'}[X] - \bE_{q}[X]|  \nonumber \\
    & \leq \sigma \epsilon_1\sqrt{\psi^{-1}(1/\epsilon_1)} + |\bE_{p'}[X]- \bE_{\hat p'}[X] | +   |\bE_{\hat p'}[X] - \bE_{q}[X]|.
\end{align}
From Lemma~\ref{lem.bounded_support_sub_gaussian_tail} and the assumption that $\epsilon + \log(1/\delta)/n \leq \epsilon_1$, we know that with probability at least $1-\delta$, there exists some constant $C_3$ such that 
\begin{align}
    |\bE_{p'}[X]- \bE_{\hat p'}[X] | & \leq C_3\left(\sigma \sqrt{\frac{\log(1/\delta)}{n}}+ \frac{\sigma\log(1/\delta)}{n\sqrt{\epsilon + \log(1/\delta)/n}}\right)\nonumber \\
    & \leq C_3\left( \sigma \sqrt{\frac{\log(1/\delta)}{n}}+ \frac{\sigma\log(1/\delta)}{n\sqrt{  \log(1/\delta)/n}}\right) \nonumber \\
    & = 2 C_3 \sigma \sqrt{\frac{\log(1/\delta)}{n}}. \nonumber
\end{align}
Thus with probability at least $1-\delta$, there exists some constant $C_4$ such that
\begin{align}
        |\bE_{p^*}[X]- \bE_q[X] | \leq |\bE_{\hat p'}[X] - \bE_{q}[X]|  + \sigma \epsilon_1\sqrt{\psi^{-1}(1/\epsilon_1)}+ C_4 \sigma \sqrt{\frac{\log(1/\delta)}{n}}.
\end{align}
\end{enumerate}

Take $\epsilon_1 = \epsilon + \log(1/\delta)/n$. From the same argument as~\eqref{eqn.concave_upper_bound}, we have
\begin{align}
    \epsilon_1\sqrt{\psi^{-1}(1/\epsilon_1)}\leq  \epsilon\sqrt{\psi^{-1}(1/\epsilon)} + \frac{\log(1/\delta)}{n}\sqrt{\psi^{-1}\left (\frac{n}{\log(1/\delta)}\right )}.
\end{align}
From assumption $\tilde \epsilon < 1/2$ we know that $\frac{n}{\log(1/\delta)}>1$. From $\psi^{-1}(x)\leq x$ for $x\geq 1$, we know that $\frac{\log(1/\delta)}{n}\sqrt{\psi^{-1}(\frac{n}{\log(1/\delta)})} \leq \sqrt{\frac{\log(1/\delta)}{n}}$. Thus overall, we have
\begin{align}
        |\bE_{p^*}[X]- \bE_q[X] | \leq |\bE_{\hat p'}[X] - \bE_{q}[X]|  + \sigma \epsilon\sqrt{\psi^{-1}(1/\epsilon)}+ C_4 \sigma \sqrt{\frac{\log(1/\delta)}{n}}.
\end{align}

Combining the five conditions, from Theorem~\ref{thm.admissible-2} and taking , for projection algorithm $q = \Pi(\hat p_n; \TV, \cM)$ or $q = \Pi(\hat p_n; \TV, \cM, \tilde \epsilon/2)$, there exists some constant $C$ such that with probability at least $1-3\delta$,
\begin{align}
 |\bE_{p^*}[X]- \bE_q[X] | & \leq  C\sigma\cdot \left(\epsilon\sqrt{\psi^{-1}( 1/\epsilon)} +\sqrt{\frac{\log(1/\delta)}{n}}\right).
\end{align}

\end{proof}

Now we show that if we are able to get good mean estimator for one-dimensional random variable, we are guaranteed to get good mean estimator for high-dimensional random variable.  Similar idea also appears in~\cite{catoni2017dimension, joly2017estimation, prasad2019unified}.
\begin{lemma}\label{lem.linear_programming_mean}
Assume $X\sim p^*$ is a $d$-dimensional random variable, and  there exists an estimator $\hat \mu_v$ such that for any fixed $v\in\bR^d, \|v\|_2=1$, with probability at least $1-\delta$,
\begin{align}\label{eqn.onedimension_guarantee}
|\hat \mu_v -\bE_{p^*}[v^\top X]| \leq f(\delta). 
\end{align}
Denote the minimum $1/2$-covering of the unit sphere $\mathcal{S}^{d-1}=\{v\mid v\in\bR^d, \|v\|_2=1\}$ as $\mathcal{N}(\mathcal{S}^{d-1}, 1/2)$, i.e. $\mathcal{N}(\mathcal{S}^{d-1}, 1/2)$ is the set with minimum elements that satisfies $\forall v\in\mathcal{S}^{d-1}$, $\exists y \in \mathcal{N}(\mathcal{S}^{d-1}, 1/2)$ such that $\|v-y\|_2 \leq 1/2$. 
Then the linear programming
\begin{align}
    \hat \mu = \argmin_{\mu} \sup_{v\in\mathcal{N}(\mathcal{S}^{d-1}, 1/2)} |v^\top \mu - \hat \mu_v| 
\end{align}
satisfies 
\begin{align}
    \|\hat \mu -\bE_{p^*}[ X]\|_2 \leq 4f(\delta/5^d)
\end{align}
with probability at least $1-\delta$.
\end{lemma}
\begin{proof}
We have
\begin{align}
\|\hat \mu - \bE_{p^*}[X]\|_2 & = \sup_{v\in \mathcal{S}^{d-1}}  |\mathbb{E}_{p^*}[v^{\top}( X - \hat \mu)]| \\
& \leq  \sup_{v\in\mathcal{N}(\mathcal{S}^{d-1}, 1/2)}  \mathbb{E}_{p^*}[v^{\top}( X - \hat \mu)]  +  \frac{1}{2} \|\hat \mu - \bE_{p^*}[X]\|_2. 
\end{align}
Thus we know that 
\begin{align}
\|\hat \mu - \bE_{p^*}[X]\|_2 & \leq  2\sup_{v\in\mathcal{N}(\mathcal{S}^{d-1}, 1/2)}  \mathbb{E}_{p^*}[v^{\top}( X - \hat \mu)]  \\
& \leq  2(\sup_{v\in\mathcal{N}(\mathcal{S}^{d-1}, 1/2)}  | \mathbb{E}_{p^*}[v^{\top} X]  - \hat \mu_v| +  \sup_{v\in\mathcal{N}(\mathcal{S}^{d-1}, 1/2)} | v^\top \hat \mu - \hat \mu_v | )\\
& \leq 4 \sup_{v\in\mathcal{N}(\mathcal{S}^{d-1}, 1/2)}  | \mathbb{E}_{p^*}[v^{\top} X]  - \hat \mu_v|. 
\end{align}

Since $|\mathcal{N}(\mathcal{S}^{d-1}, 1/2)| \leq 5^d$~\citep{wainwright2019high},
by taking the union bound over all the vectors in $\mathcal{N}(\mathcal{S}^{d-1}, 1/2)$ in~\eqref{eqn.onedimension_guarantee},  we know that with probability at least $1-5^d\delta$,
\begin{align}
    \|\hat \mu - \bE_{p^*}[X]\|_2 \leq 4 f(\delta).
\end{align}
Substituting $\delta$ with $\tilde \delta = 5^d\delta$ gives the final result. 

\end{proof}

\begin{remark}
The above result also applies to  general case of estimating under $L = W_\sF(p,q) = \sup_{f\in \sF} |\bE_p f(X) - \bE_q f(X)|$ by designing $W_\sH$ where $\sH \subset \sF$ and $W_\sH(p,q)\geq \frac{1}{2}W_\sF(p,q)$. Furthermore, it is \emph{necessary} to estimate each $f(X)$ for $f\in\sF$ very well. Recall that the modulus of continuity for $L=W_\sF$ is a nearly tight upper bound for the population limit~(Lemma~\ref{lemma.population_limit_opt}). The modulus can be written as
\begin{align}
    \sup_{p_1, p_2\in\GG, \TV(p_1, p_2)\leq \epsilon} W_\sF(p_1, p_2) =  \sup_{f\in\sF} \sup_{p_1, p_2\in\GG, \TV(p_1, p_2)\leq \epsilon} |\bE_{p_1}[f(X)] -  \bE_{p_2}[f(X)]|,
\end{align}
while $\sup_{p_1, p_2\in\GG, \TV(p_1, p_2)\leq \epsilon} |\bE_{p_1}[f(X)] -  \bE_{p_2}[f(X)]|$ is the modulus for estimating $\bE_p[f(X)]$. Hence, robust estimation under $W_\sF$ is equivalent to robust estimation of each $f\in \sF$. 
\end{remark}

Combining Theorem~\ref{thm.oned_projection} and Lemma~\ref{lem.linear_programming_mean}, we have the following corollary.
\begin{corollary}\label{cor.kthmomentreduceapproach}
Assume the corruption model is either oblivious corruption (Definition~\ref{def.obliviouscorruption}) or adaptive corruption model (Definition~\ref{def.adaptivecont}) of level $\epsilon$. For some Orlicz function $\psi$ that satisfies $\psi(x)\geq x$ when $x\geq 1$, 
we take $\GG$ to be the set of $d$-dimensional bounded Orlicz norm distributions, and $\cM$ to be the set of $1$-dimensional bounded Orlicz norm distributions:
\begin{align}
    \GG &= \left\{p\mid \sup_{v\in\bR^d, \|v\|_2=1} \mathbb{E}_p \left[  \psi\left(\frac{(v^\top(X-\mu_p))^2}{\sigma^2}\right) \right]\leq 1 \right\}, \\
     \cM & = \left\{p \mid  \mathbb{E}_p \left[ \psi\left(\frac{(X-\mu_p)^2}{4\sigma^2}\right) \right] \leq 4 \right\}.
\end{align}
Denote
\begin{align}
    \tilde \epsilon = 4\left(\sqrt{\epsilon + \frac{\log(1/\delta)}{n}} + \sqrt{\frac{\log(1/\delta)}{2n}}\right)^2.
\end{align}
For some fixed $v\in\mathcal{S}^{d-1}$, denote the projection of distribution $\hat p_n$ along direction $v$ as $\hat p_n^v$.   
Denote $q_v = \Pi(\hat p_n^v; \TV, \cM, \tilde \epsilon / 2)$, and the algorithm outputs
\begin{align}
    \hat \mu = \argmin_{\mu} \sup_{v\in\mathcal{N}(\mathcal{S}^{d-1}, 1/2)} |v^\top \mu - \mu_{q_v}|
\end{align}
If $p^*\in\GG$ and $\tilde \epsilon<1/2$, then $\hat \mu$ satisfies the following with probability at least $1-3\delta$:
\begin{align}
 \|\mu_{p^*}- \hat \mu \|_2 & \leq C\sigma\cdot \left(\epsilon\sqrt{\psi^{-1}(1/\epsilon)} + \sqrt{\frac{d+\log(1/\delta)}{n}}\right),
\end{align}
where $C$ is some universal constant.
\end{corollary}

By taking $\psi(x) = x$, we are assuming the true distribution has bounded covariance, and guarantee estimation error $O(\sqrt{\epsilon} +\sqrt{\frac{d+\log(1/\delta)}{n}})$.

\subsection{Interpretation and Comparison for Mean Estimation}\label{sec.mean_comparison_discussion}

To contrast the approaches in Sections~\ref{sec.finite_sample_algorithm_weaken} and \ref{sec.expand}, here we compare their implied 
bounds for robust mean estimation (Table~\ref{tab.projection_summary}) and also discuss related literature.

If $p^*$ is sub-Gaussian, analysis in Theorem~\ref{thm.tTV_mean} 
implies an error $O(\epsilon\sqrt{\log(1/\epsilon)})$
when $n \gtrsim \frac{d + \log(1/\delta)}{\epsilon^2}$, while  analysis 
in Theorem~\ref{thm.general_subgaussian_tv_projection} shows that the same algorithm actually only  requires $n \gtrsim \frac{d + \log(1/\delta)}{\epsilon^2\log(1/\epsilon)}$, which is optimal. Here 
Theorem~\ref{thm.general_subgaussian_tv_projection} is not explicitly stated in the literature but should be well known among experts. We are also able to generalize it to all distributions with moment generating functions by Lemma~\ref{lem.subgaussian_empirical_resilience}, which provides better bound than~\citet[Lemma  4.4]{diakonikolas2019robust} for sub-Gaussian case.

If $p^*$ has bounded covariance, then Theorem~\ref{thm.tTV_mean} implies an error of order $O(\sqrt{\epsilon})$ when $n \gtrsim \frac{d + \log(1/\delta)}{\epsilon^2}$.
This dependence of $n$ on $\epsilon$ is sub-optimal; inspired by~\citet{prasad2019unified}, our Theorem~\ref{thm.projection_bounded_covariance} shows that the same $\tTV_\mathcal{H}$ projection algorithm reaches $O(\sqrt{\epsilon})$ error when $n \gtrsim \frac{d\log(d/\delta)}{\epsilon}$, 
    which has better dependence of $n$ on $\epsilon$ but worse dependence on $d$.
    
    When $\epsilon = 0$, the current analysis in Theorem~\ref{thm.tTV_mean} for projecting under resilient set
    yields an error of $O((\frac{d+\log(1/\delta)}{n})^{1/4})$ error, and our Theorem~\ref{thm.projection_bounded_covariance} for projecting under bounded covariance set
    gives $O((\frac{d\log(d/\delta)}{n})^{1/2})$,  while median-of-means tournament method~\cite{lugosi2019sub} can achieve $O((\frac{d+\log(1/\delta)}{n})^{1/2})$. In the follow-up work~\cite{diakonikolas2020outlier}, it is shown that the exact sub-Gaussian rate can be achieved by combining the bucket means in median-of-means and $\TV$ projection algorithm, while without the combination one might lose a logarithmic factor. %

\section{Related discussions and remaining proofs in Section~\ref{sec.population_w1}}
\subsection{A general design for $W_{c, k}$ perturbations}
We can extend the definition of friendly perturbation and generalized resilience set from $W_1$ perturbation to any Wasserstein perturbation $W_{c, k}$. It's mostly similar to that of $W_1$ perturbation, under some topological assumptions between $c$ and $f$.

\begin{definition}[Friendly perturbation]\label{def.friendly_perturbation_general}
For a distribution $p$ over $\mathcal{X}$, fix a function $f : \mathcal{X}\to \bR$. A distribution $r$ 
is an {$\eta$-friendly perturbation} of $p$ for $f$ under $W_{c,k}$, denoted as $r \in \mathbb{F}(p, \eta, W_{c, k}, f) $, if there is a coupling $\pi_{X,Y}$ between $X\sim p$ 
and $Y \sim r$ such that:
\begin{itemize}
\item The cost $(\bE_{\pi}[c^k(X,Y)])^{1/k}$ is at most $\eta$.
\item All points move towards the mean of $r$: $f(Y)$ is between $f(X)$ and $\bE_{r}[f(Y)]$ almost surely.
\end{itemize}
\end{definition}

We make the following 
topological assumptions regarding $c$ and $f$.
\begin{assumption}[Intermediate value property]
\label{ass:topo}
Given $W_{c, k}$, we assume that for all $x$ and $y$ and all 
$u$ with $f(x) < u < f(y)$, 
there is some $z$ satisfying 
$f(z) = u$ and $\max(c(x,z), c(z,y)) \leq c(x,y)$.
\end{assumption}

This holds for all of our examples and for many other $f$ and $c$, e.g.~when $c$ is a path metric~\citep[Definition 1.7]{gromov2007metric} and $f$ is continuous under the topology induced by the metric. Under this assumption, we generalize the midpoint lemma in Lemma~\ref{lem.wc_friend_pert} to arbitrary $W_{c, k}$.

\begin{lemma}[Midpoint lemma for $W_{c, k}$ perturbation]\label{lem.wc_friend_pert_general}
Suppose Assumption~\ref{ass:topo} holds. Then for
 any $p_1$ and $p_2$ such that $W_{c,k}(p_1,p_2)<\eta$ and any $f$, there exists a distribution $r$ such that
\begin{align}
 r \in \mathbb{F}(p_1, \eta, W_{c,k},f) \cap   \mathbb{F}(p_2, \eta, W_{c,k},f).
\end{align}
In other words, $r$ is an $\eta$-friendly perturbation of both $p_1$ and $p_2$ for $f$ under $W_{c,k}$.  
\end{lemma}

With the midpoint and friendly perturbation, we can define the generalized resilience set under $W_{c, k}$ perturbation:

\begin{definition}[$\GG^{W_{c, k}}$]\label{def.G^Wc_general}
We define \begin{align}\label{eqn.def_wc_G}
& \mathcal{G}^{W_{c, k}}(\rho_1, \rho_2, \eta) = \mathcal{G}^{W_{c, k}}_{\downarrow}(\rho_1, \eta) \cap \mathcal{G}^{W_{c, k}}_{\uparrow}(\rho_1, \rho_2, \eta),
\end{align}
where
\begin{align}
     \GG_{\downarrow}^{W_{c, k}}(\rho_1, \eta) = \{p \mid  & \sup_{f\in\sF_{\theta^*(p)},  r\in \mathbb{F}(p, \eta, W_{c, k}, f)} \bE_r[f(X)] - B^*(f, \theta^*(p)) \leq \rho_1  \}, \\
     \GG_{\uparrow}^{W_{c, k}}(\rho_1, \rho_2, \eta) =  \bigg\{p \mid  & \text{for all } \theta\in\Theta, \bigg( \Big(  \sup_{f\in \sF_\theta} \inf_{r\in \mathbb{F}(p, \eta, W_{c, k}, f)} \bE_r[f(X)] - B^*(f, \theta) \leq \rho_1 \Big) \Rightarrow  L(p, \theta) \leq \rho_2  \bigg) \bigg\}.
\end{align}
\end{definition}

And the modulus of the set is upper bounded:
\begin{theorem}\label{thm.G_Wc_fundamental_limit_general}

The modulus of continuity $\modu$ in (\ref{eqn.modulus}) for $\GG^{W_{c, k}}(\rho_1, \rho_2, \eta)$ is bounded above by 
$\modu(\GG^{W_{c, k}}(\rho_1, \rho_2, \eta), 2\epsilon) \leq \rho_2$ for any $ 2\epsilon \leq \eta$.
\end{theorem}

\subsection{Key Lemmas}

The following lemma produces an upper bound on $|\bE[g(X) - g(Y)]|$. 
\begin{lemma}\label{lem.coupling_g_bounded}
Assume we are given two distributions $p, q$, some function $g: \bR \mapsto \bR$, some non-negative cost function $c(x, y)$ and some Orlicz function $\psi$. For any coupling $\pi_{p, q}$ between $p, q$ such that $X\sim p, Y\sim q$, and any $\sigma > 0$, we have
\begin{align}
    |\bE[g(X) - g(Y)]| \leq \sigma \bE_{\pi_{p, q}}[c(X, Y)] \psi^{-1}\left(\frac{\bE_{\pi_{p, q}}\left[c(X, Y) \psi\left(\frac{|g(X) - g(Y)|}{\sigma c(X, Y)}\right)\right]}{\bE_{\pi_{p, q}}[c(X, Y)]}\right),
\end{align}
where $\psi^{-1}$ is the (generalized) inverse function of $\psi$.
\end{lemma}
We remark here that the left-hand side of the conclusion does not depend on the coupling $\pi_{p, q}$, so one can take the infimum over all couplings on the right-hand side.
\begin{proof}
We omit the $\pi$ in the subscript in $\bE$. Applying Jensen's inequality to the new measure $\frac{c(X,Y)}{\mathbb{E}[c(X,Y)]} d\pi_{p,q}$,  
\begin{align}
    \psi\left(\left|\frac{\bE[g(X)- g(Y)]}{\sigma \bE[c(X, Y) ]}\right|\right) & = \psi\left(\left|\bE\left[\frac{c(X, Y)}{\bE[c(X, Y) ]} \cdot \frac{g(X) - g(Y)}{\sigma c(X, Y)}  \right] \right|\right) \nonumber \\
    & \leq \bE\left[\frac{c(X, Y)}{\bE[c(X, Y) ]} \psi\left(\frac{|g(X)-g(Y)|}{\sigma c(X, Y)}\right) \right] \nonumber \\
    &  = \frac{\bE\left[{c(X, Y)} \psi\left(\frac{|g(X)-g(Y)|}{\sigma c(X, Y)}\right) \right]}{\bE[c(X, Y)]}.
\end{align}
It implies that
\begin{align}
     |\bE[g(X)-g(Y) ]| & \leq \sigma\bE[c(X, Y)] \psi^{-1}\left(\frac{\bE\Big[c(X, Y) \psi\left(\frac{|g(X) - g(Y)|}{\sigma c(X, Y)}\right)\Big]}{\bE[c(X, Y)]}\right). 
\end{align}

\end{proof}

With this lemma, we show that bounded Orlicz norm implies resilience for $k$-th moment estimation under $W_1$ perturbation.

\subsection{Proof of Lemma~\ref{lem.wc_friend_pert}}\label{proof.coupling_unimodal}

Given any two points $x$ and $y$, without loss of generality we assume $f(x)\leq f(y)$, define
\begin{align}
    s_{xy}(u) = \begin{cases} 
    \min(f(x), f(y)), & u \leq \min(f(x), f(y))\\
    u, & u\in [f(x), f(y)]\\
    \max(f(x), f(y)),& u\geq \max(f(x), f(y)).
    \end{cases}
\end{align}

If we imagine $u$ increasing from $-\infty$ to $+\infty$, 
we can think of $s_{xy}$ as a ``slider'' that tries to be 
as close to $u$ as possible while remaining between $f(x)$ and $f(y)$.

By Assumption~\ref{ass:topo}, there must exist some point $z$ such that 
$\max(c^k(x,z), c^k(z,y)) \leq c^k(x,y)$ and $f(z) = s_{xy}(u)$. Call this point 
$z_{xy}(u)$.

Given a coupling $\pi(x,y)$ from $p_1$ to $p_2$, 
if we map $y$ to $z_{xy}(u)$, we obtain a coupling 
$\pi_1(x,z)$ to some distribution $r(u)$, which by construction 
satisfies the monotonicity property, except that it is relative to 
$u$ rather than the mean $\mu(u) = \bE_{X\sim r(u)}[f(X)]$. 
However, note that $u - \mu(u)$ is a continuous, monotonically non-decreasing function (since 
$u - s_{xy}(u)$ is non-decreasing) that ranges from $-\infty$ to $+\infty$. It follows that 
there is a point with $\mu(u) = u$, in which case $r(u)$ satisfies 
the monotonicity property with respect to $\mu(u)$. 

Moreover, $\bE_{(X, Z) \sim \pi_1}[c^k(X, Z)] \leq \bE_{(X,Y) \sim \pi}[c^k(X, Y)] = W_{c, k}^k(p_1,p_2)$. 
The coupling $\pi_1$ therefore also has small enough cost, and so satisfies all 
of the properties required in Lemma~\ref{lem.wc_friend_pert}.

To finish, we need to construct $\pi_2$; but this can be done  
by taking the reverse coupling from $y$ to $z_{xy}(u)$, which satisfies the required 
properties by an identical argument as above.

\subsection{Proof of Theorem~\ref{thm.G_Wc_fundamental_limit} }\label{proof.thm_G_WcG_Wc_fundamental_limit}

This is a proof of more general result (Theorem~\ref{thm.G_Wc_fundamental_limit_general}) than Theorem~\ref{thm.G_Wc_fundamental_limit}.
\begin{proof}
It suffices to upper bound the modulus 
\begin{align}
\sup_{\begin{subarray}{c} (p_1, p_2):  W_{c,k}(p_1, p_2) \leq 2 \epsilon,\\  p_1 \in  \GG^{W_{c,k}}(\rho_1, \rho_2), p_2 \in  \GG^{W_{c,k}}(\rho_1, \rho_2)\end{subarray}} L(p_2, \theta^*(p_1)).
\end{align}

Note that $2 \epsilon \leq \eta$. It follows from the condition that $p_1\in\GG^{W_{c,k}}(\rho_1, \rho_2) \subset \mathcal{G}^{W_{c,k}}_{\downarrow}(\rho_1)$ that
\begin{align}\label{eqn.proof_design_Wc_G_limit}
\sup_{f\in\sF_{\theta^*(p_1)},  r\in \mathbb{F}(p_1, 2\epsilon, W_{c, k}, f)} \bE_r[f(X)] - B^*(f, \theta^*(p_1)) \leq \rho_1.
\end{align}
By Lemma~\ref{lem.wc_friend_pert}, for any $f\in\sF_{\theta^*(p_1)}$, we are able to pick $r$ such that $W_{c,k}(p_1, r)\leq 2\epsilon$,  $W_{c,k}(p_2, r)\leq 2\epsilon$, and $r \in  \mathbb{F}(p_1, 2\epsilon, W_{c,k}, f) \bigcap \mathbb{F}(p_2, 2\epsilon, W_{c,k}, f)$ is a friendly perturbation for both  $p_1$ and $p_2$, which implies that
\begin{align}
 \sup_{f\in \sF_{\theta^*(p_1)}} \inf_{r\in \mathbb{F}(p_2, 2\epsilon, W_{c, k}, f)} \bE_r[f(X)] - B^*(f, \theta^*(p_1))  \leq \rho_1. 
\end{align}
It then follows from $p_2\in\GG^{W_{c,k}}(\rho_1, \rho_2) \subset \mathcal{G}^{W_{c,k}}_{\uparrow}(\rho_1, \rho_2)$ that
\begin{align}
    L(p_2, \theta^*(p_1)) \leq \rho_2. 
\end{align}
\end{proof}

\subsection{$W_{c,k}$-Resilient Set Design for $L = W_{\sF}$}\label{sec.G_Wc_WF}

Recall that in Section~\ref{sec.robust_TV}, we design $\GG^{\TV}_{\mathsf{mean}}$ as that the means of any friendly perturbation distribution and the original distribution are close, and extend the idea to arbitrary $W_\sF$ pseudonorm in~\eqref{eqn.def_G_TV_WF}. Since we have already defined friendly perturbation for $W_{c, k}$, we similarly define $\GG^{W_{c, k}}_{W_\sF}$ as that the $W_\sF$ pseudonorm between any friendly perturbation distribution and the original distribution is close. Concretely, 
given $L = W_\sF$ defined in~\eqref{eqn.def_WF}
, the set $\GG^{W_{c, k}}_{W_\sF}$ can be defined as 
\begin{definition}[$\GG_{W_\sF}^{W_{c, k}}(\rho, \eta)$]
Assume Assumption~\ref{ass:topo} holds. We define 
\begin{align}\label{eqn.def_Wc_meanG}
\GG_{W_\sF}^{W_{c, k}}(\rho, \eta) = \{p \mid  \sup_{f\in\sF, r\in \mathbb{F}(p, \eta, W_{c, k}, f)} \bE_r[f(X)] -  \bE_p[f(X)] \leq \rho  \}. 
\end{align}
\end{definition}

If $c(x,y) = \mathbbm{1}(x\neq y)$ and $k=1$, then $W_{c, k}$ reduces to $\TV$, and one can show that the resilient set $\GG_{W_\sF}^{W_{c, k}}(\rho, \eta)$ defined in~(\ref{eqn.def_Wc_meanG}) reduces to that in~(\ref{eqn.def_G_TV_WF}) in the $\TV$ case. The reason is that these two definitions of resilient sets share the same \emph{worst case} perturbed $\bE_r[f(X)]$: it is always to either delete the largest (or smallest) $\epsilon$ probability mass of $f(X)$, or to move the largest (or smallest) $\epsilon$ probability mass to the $\bE_r[f(X)]$. Similar to $\TV$ case, 
we can show that the design of $\GG_{W_\sF}^{W_{c, k}}$ is not too big such that its population limit can be controlled, and not too small such that some usual non-parametric assumptions such as bounded $k$-th moment implies being in the set. 

\subsubsection{Not too big}
Similar to $\TV$ perturbation case, we show that $\mathcal{G}^{W_{c, k}}_{W_\sF}$ has controllable population limit by upper bounding its modulus of continuity.
\begin{theorem}\label{thm.G_WF_Wc_fundamental_limit}
Assume Assumption~\ref{ass:topo} holds.
The modulus of continuity $\modu$ in (\ref{eqn.modulus}) for $\GG_{W_\sF}^{W_{c, k}}(\rho, \eta)$ is bounded above as 
$\modu(\GG_{W_\sF}^{W_{c, k}}(\rho, \eta), 2\epsilon) \leq 2\rho$ for any $ 2\epsilon \leq \eta$.
\end{theorem}

\begin{proof}
The modulus is defined as 
\begin{align}
\sup_{\begin{subarray}{c} (p_1, p_2):  W_{c,k}(p_1, p_2) \leq 2 \epsilon,\\  p_1 \in  \GG^{W_{c,k}}_{W_\sF}(\rho, \eta), p_2 \in  \GG^{W_{c,k}}_{W_\sF}(\rho, \eta)\end{subarray}} W_\sF(p_1, p_2).
\end{align}

By Lemma~\ref{lem.wc_friend_pert}, for any $f\in\sF$, we are able to pick $r$ such that $W_{c,k}(p_1, r)\leq 2\epsilon$,  $W_{c,k}(p_2, r)\leq 2\epsilon$, and $r \in  \mathbb{F}(p_1, 2\epsilon, W_{c,k}, f) \bigcap \mathbb{F}(p_2, 2\epsilon, W_{c,k}, f)$ is a friendly perturbation for both  $p_1$ and $p_2$. Take\footnote{If the argmax is not achievable, we can take a sequence of $f^*_i$ such that $\bE_{p_1}[f^*_i(X)] - \bE_{p_2}[f^*_i(X)]$  goes to the maximum value as $i\rightarrow +\infty$.} 
\begin{align}
    f^* \in \argmax_{f\in\sF} \bE_{p_1}[f(X)] - \bE_{p_2}[f(X)]. 
\end{align}
From $p_1, p_2\in \GG^{W_{c,k}}_{W_\sF}(\rho, \eta)$ and the symmetricity of $\sF$, we know that for any $2 \epsilon \leq \eta$, 
\begin{align}
\bE_{p_1}[f^*(X)] - \bE_{r}[f^*(X)] \leq \rho,\\
\bE_{r}[f^*(X)] - \bE_{p_2}[f^*(X)] \leq \rho.
\end{align}
Combining the two equations together gives us 
\begin{align}
    \bE_{p_1}[f^*(X)] - \bE_{p_2}[f^*(X)]\leq 2\rho.
\end{align}
This shows that $W_\sF(p_1, p_2)\leq 2\rho$.
\end{proof}

\subsection{Reduction from $\GG^{W_{c, k}}$ to $\GG^{W_{c, k}}_{W_\sF}$}\label{appendix.Wck_reduction}
 We prove here that 
by taking $B = L = W_\sF$, $\rho_1 = \rho, \rho_2 = 2\rho$ in Equation (\ref{eqn.def_wc_G}), we can recover the definition of $\GG_{W_\sF}^{W_{c, k}}$.
 
\begin{proof}
Here we identify $\theta$ as $q$, and and take $\mathcal{F}_\theta = \mathcal{F}$ for any $\theta$. Under the choices of $B, \rho_1$,
we have $\GG_{W_\sF}^{W_{c, k}} = \GG^{W_{c, k}}_{\downarrow}$. 
We only need to show that under the choices of $\rho_1, \rho_2, B, L$, we have $\GG^{W_{c, k}}_{\downarrow}(\rho, \eta) \subset  \GG_{\uparrow}^{W_{c, k}}(\rho, 2\rho, \eta)$. For any $p\in \GG_{\downarrow}^{W_{c, k}}$, we have
\begin{align}
    \sup_{f\in\sF, r\in \mathbb{F}(p, \eta, W_{c, k}, f)} \bE_r[f(X)] -  \bE_p[f(X)] \leq \rho. 
\end{align}

Note that the condition $p \in \GG_{\uparrow}^{W_{c, k}}$ is equivalent to that for any $q$, we have
\begin{align}\label{eqn.proof_reduction_Gup}
    \Big(  \sup_{f\in \sF} \inf_{r\in \mathbb{F}(p, \eta, W_{c, k}, f)} \bE_r[f(X)] - \bE_q[f(X)] \leq \rho \Big) \Rightarrow \sup_{f\in\sF} \bE_p[f(X)] - \bE_q[f(X)] \leq 2\rho.
\end{align}
Now it suffices to show that for any $p\in \GG_{\downarrow}^{W_{c, k}}(\rho, \eta)$, we have $p\in \GG_{\uparrow}^{W_{c, k}}(\rho,2\rho,\eta)$. 

Assume $p\in \GG_{\downarrow}^{W_{c, k}}(\rho, \eta)$.
For any $q$  that satisfies the LHS condition in Equation~\eqref{eqn.proof_reduction_Gup}. Then we know for any $f\in \mathcal{F}$, there exists some $r \in \mathbb{F}(p, \epsilon, W_{c,k}, f)$ with $\bE_r[f(X)] - \bE_q[f(X)] \leq \rho$, denote by $f^*$ the function that satisfies \footnote{For any $a>0$, there always exists some $f^*$ such that $\bE_p[f^*(X)] - \bE_q[f^*(X)]\geq  \sup_{f\in\sF} \bE_p[f(X)] - \bE_q[f(X)] -a$. The following steps can go through using this limiting argument if the supremum is not attained. }
\begin{align}
    \sup_{f\in\sF} \bE_p[f(X)] - \bE_q[f(X)] & = \bE_p[f^*(X)] - \bE_q[f^*(X)].
\end{align}
Denote by $r^*$ the friendly perturbation in $\mathbb{F}(p, \epsilon, W_{c,k}, f^*)$. We have $\bE_r[f^*(X)] - \bE_q[f^*(X)] \leq \rho$. Then,
\begin{align}
    \sup_{f\in\sF} \bE_p[f(X)] - \bE_q[f(X)] & = \bE_p[f^*(X)] - \bE_q[f^*(X)] \\
    & \leq \bE_p[f^*(X)] - \bE_{r^*}[f^*(X)] + \bE_{r^*}[f^*(X)] - \bE_q[f^*(X)] \\
    & \leq 2\rho. 
\end{align}
Hence, we have shown that $\GG^{W_{c, k}}_{\downarrow}(\rho, \eta) \subset  \GG_{\uparrow}^{W_{c, k}}(\rho, 2\rho, \eta)$. Proof is completed. 

\end{proof}

\section{Related discussions and remaining proofs in Section~\ref{sec.finite_sample_w1}}

\subsection{Proof of Lemma~\ref{lem:tw1_cross_mean}}\label{proof.tw1_cross_mean}
\begin{proof}
We first show the result when $g(x) = x$. From~\citep{rothschild1978increasing}\citep[Proposition B.19.c, remark 2]{marshall1979inequalities}, we know that
it suffices to check the following two conditions to guarantee convex order of $X$:
\begin{enumerate}
    \item $\bE_{r_p}[X] = \bE_{r_q}[X]$,
    \item $\forall z \in \bR, \int_{-\infty}^{z}\bP_{r_p}[X\leq t ]dt \geq \int_{-\infty}^{z}\bP_{r_q}[X\leq t ]dt.$
\end{enumerate}
Denote function $(x)_+ = \max(x, 0)$, $(x)_- = -\min(-x, 0)$. Note that we have
\begin{align}
    \bE_p[(X-a)_+] = \int_{a}^{+\infty} \bP_p(X\geq t) dt,
\end{align}
which is also equivalent to the $W_1$ cost that moves all the mass left to $a$ to the point $a$. Similarly, the $W_1$ cost that moves all the mass right to $a$ to the point $a$ can be represented as
\begin{align}
    \bE_p[(X-a)_-] = \int_{-\infty}^{a} \bP_p(X\leq t) dt.
\end{align}

Given $p, q$, we construct $r_p$ and $r_q$ as follows. 
We consider different cases of $\mu_p, \mu_q$:

\begin{enumerate}
    \item Assume $\mu_p = \mu_q = \mu$ for some $\mu$. We construct some coupling for $p, r_p$ pair and 
    $q, r_q$ pair such that under the coupling, all the mass move towards $\mu$ and the mean is unchanged. 
    Since $\bE_p[(X-\mu)_+] - \bE_p[(X-\mu)_-] = \bE_p[X-\mu] = 0$, we know that $\bE_p[(X- \mu)_+] = \bE_p[(X-\mu)_-] = 0$. Similarly we have $\bE_q[(X-\mu)_+] = \bE_q[(X-\mu)_-]$.

    If $\bE_p[(X-\mu)_+] = \bE_p[(X-\mu)_-] \leq 2\epsilon$, from $\tW_1(p, q)\leq \epsilon$, we know that
    \begin{align}
        \bE_q[(X-\mu)_+] = \bE_q[(X-\mu)_-] \leq 3\epsilon.
    \end{align}
    Then we move all the mass of $p, q$ to a single point $\mu$ to get the new distribution $r_p, r_q$, from the above condition we know that $W_1(p, r_p)\leq 4\epsilon$, $W_1(q, r_q)\leq 6\epsilon$, and $r_p, r_q$ satisfies the two conditions required since they are identically distributed.
    
    Otherwise, we have 
    \begin{align}
        \bE_p[(X-\mu)_+] = \bE_p[(X-\mu)_-] > 2\epsilon,
    \end{align}
    and consequently, 
    \begin{align}
        \bE_q[(X-\mu)_+] = \bE_q[(X-\mu)_-] > \epsilon.
    \end{align}
    Then we are able to find some $\tau_2 < \mu < \tau_1$ such that
    \begin{align}
         \int_{-\infty}^{\tau_2} \bP_{q}(v^{\top}X \leq t) dt = \epsilon, \\
     \int_{\tau_1}^{+\infty} \bP_{q}(v^{\top}X \geq t) dt = \epsilon. 
    \end{align}
    We move all the mass of $q$ that is left to $\tau_2$ to $\tau_2$, and all the mass of $q$ that is right to $\tau_1$ to $\tau_1$ to get $r_q$, and keep $r_p = p$. Then we have $W_1(r_q, q) =2\epsilon$, $\mu_{r_q} = \mu_q = \mu_p = \mu_{r_p}$, and
    \begin{align}
    \forall z <\tau_2,  \int_{-\infty}^{z}\bP_{r_p}[X\leq t ]dt \geq 0 & = \int_{-\infty}^{z}\bP_{r_q}[X\leq t ]dt,\\
        \forall z \in [\tau_2, \tau_1] \int_{-\infty}^{z}\bP_{r_p}[X\leq t ]dt
        &  = \int_{-\infty}^{z}\bP_{p}[X\leq t ]dt  \nonumber \\
        & = \int_{-\infty}^{z}\bP_{p}[X\leq t ]dt - \int_{-\infty}^{z}\bP_{q}[X\leq t ]dt + \int_{-\infty}^{z}\bP_{q}[X\leq t ]dt \nonumber \\
        &\quad  - \int_{-\infty}^{z}\bP_{r_q}[X\leq t ]dt + \int_{-\infty}^{z}\bP_{r_q}[X\leq t ]dt \nonumber \\
        & \geq -\epsilon + \epsilon + \int_{-\infty}^{z}\bP_{r_q}[X\leq t ]dt  \nonumber \\
        & = \int_{-\infty}^{z}\bP_{r_q}[X\leq t ]dt, \\
        \forall z > \tau_2, \int_{-\infty}^{z}\bP_{r_p}[X\leq t ]dt
        &  = \int_{-\infty}^{z}\bP_{p}[X\leq t ]dt \nonumber \\
        & = \bE_p[(X-z)_-] \nonumber \\
        & = \bE_p[(X-z)_+] - \bE_p[X-z]  \nonumber \\
        & \geq  \bE_{r_q}[(X-z)_+] - \bE_{r_q}[X-z]  \nonumber \\
        & = \bE_{r_q}[(X-z)_-] \nonumber \\
        & = \int_{-\infty}^{z}\bP_{r_q}[X\leq t ]dt.
    \end{align}
    Thus the two conditions to guarantee convex order are satisfied.
    \item Assume $\mu_p > \mu_q$. From $W_1(p, q)\leq \epsilon$ we know that $|\mu_p-\mu_q |\leq \epsilon$. Take $\mu = \mu_p$.
    
    If $\bE_q[(X-\mu)_-] \leq \epsilon$, then from $\bE_q[(X-\mu)_+] - \bE_q[(X-\mu)_-] = \bE_q[X-\mu] < 0$, we know that $\bE_q[(X-\mu)_+] \leq \epsilon$. Thus from $W_1(p, q)\leq \epsilon$, we know that $\bE_p[(X-\mu)_-] \leq 2\epsilon$, $\bE_p[(X-\mu)_+] \leq 2\epsilon $. Thus we can move all the mass of $p, q$ to a single point $\mu$ to get the new distribution $r_p, r_q$, from the above condition we know that $W_1(p, r_p)\leq 4\epsilon$, $W_1(q, r_q)\leq 2\epsilon$, and $r_p, r_q$ satisfies the two conditions required since they are identically distributed.
    
    Otherwise, we know $\bE_q[(X-\mu)_-] > \epsilon$. Then we first move the left most part of $X$ under $q$ to make $\bE_{r_q}[X] = \mu$. Thus we have $\bE_{r_q}[(X-z)_-] \leq \bE_q[(X-z)_-] $ for any $z\in\bR$.  Starting from $p, r_q$, we know that their means are equal. Thus we repeat the first step to construct $r_p, r_q$ that satisfies the two conditions. Overall we have $W_1(q, r_q)\leq 7\epsilon$.
    
    \item Assume $\mu_p < \mu_q$. Denote $\mu =\mu_p$. Then if $\bE_q[(X-\mu)_-] \leq \epsilon$, we follow the same procedure as the case of $\mu_p>\mu_q$. Otherwise we first move the right most part of $X$ under $q$ to make $\bE_{r_q}[X] = \mu_p$.  Then we repeat the first step to construct $r_p, r_q$ that satisfies the two conditions.  Overall we have $W_1(q, r_q)\leq 7\epsilon$.
    
\end{enumerate}

When $g(x) = |x|$, note that the movement in above construction from $p, q$ to $r_p, r_q$ only includes deleting the left most or right most of $X$. By replacing $X$ with $|X|$, all above arguments go through without increasing the cost $W_1(p, r_q)$ and $W_1(q, r_q)$. This is because for any movement from $a$ to $b$ for $|X|$, where  both $a$ and $b$ are non negative, one can map it back to movement in $x$ space from $a$ to $b$ or from $-a$ to $-b$ without increasing the cost.  
Thus the result also holds for $g(x) = |x|$.
\end{proof}

\subsection{Proof of Lemma~\ref{lem.tW_1_empirical_to_population}}\label{proof.tW_1_empirical_to_population}
\begin{proof}

To show the result, we first prove the following Lemma.

\begin{lemma}

Consider any distribution $p$ and denote its empirical distribution of $n$ \iid samples as $\hat p_n$. For any $M>0$, define 
\begin{align}
    \xi_1 & = \bE_{p}  \left \| \frac{1}{n} \sum_{i = 1}^n  X_i - \mathbb{E}_p[X] \right \|_2,  \\
      \xi_2(M) & = \sup_{v \in \bR^d, \| v\|_2 = 1} \bE_p[ \max\left(0,  v^{\top}(X-\bE_p[X])-M\right)].
\end{align}
Then for any $M >0$, 
\begin{align}
    \bE_p[\tW_1(p, \hat p_n)] \leq 8\xi_1 +  \xi_2(M)+  2\frac{M}{\sqrt{n}},
\end{align}
where $C$ is some universal constant. 
\end{lemma}
Denote the contaminated population distribution as $p$ satisfying $W_1(p,p^*)\leq \epsilon$, and the empirical distribution of observed data as $\hat p_n$. Under the oblivious corruption model $\hat p_n$ represents $n$ \iid samples from $p$. Define
\begin{align}
 \mathcal{U}_1' & = \{v^{\top}x : v \in \bR^d, \|v\|_2\leq 1\} \\
    \mathcal{U}_2' & = \{\max(0, v^{\top}(x-a)): a, v\in \bR^d, \|v\|_2\leq 1 \} \\
      \mathcal{U}_3' & = \{-\max(0, v^{\top}(x-a)): a, v\in \bR^d, \|v\|_2\leq 1 \}.
\end{align}
Note that $\mathcal{U}_1' \bigcup \mathcal{U}_2'\bigcup \mathcal{U}_3' $ is symmetric. We have
\begin{align}
    \tW_1(p, \hat p_n) & = \sup_{u \in \mathcal{U}'} \left| \bE_{\hat p_n}[u(X)] - \bE_p[u(X)] \right| \nonumber \\
    & = \sup_{u \in \mathcal{U}_1' \bigcup \mathcal{U}_2'\bigcup \mathcal{U}_3' } \bE_p[u(X)]  - \frac{1}{n}\sum_{i=1}^n u(X_i)  \nonumber \\
    & = \max\left\{ \sup_{u \in \mathcal{U}_1'} \bE_p[u(X)]  - \frac{1}{n}\sum_{i=1}^n u(X_i)  , \sup_{u \in \mathcal{U}_2'}  \bE_p[u(X)]  - \frac{1}{n}\sum_{i=1}^n u(X_i)   , \sup_{u \in \mathcal{U}_3'}\bE_p[u(X)]  - \frac{1}{n}\sum_{i=1}^n u(X_i)  \right\}.
\end{align}

We have
\begin{align}
    \sup_{u \in \mathcal{U}_1'} \bE_p[u(X)]  - \frac{1}{n}\sum_{i=1}^n u(X_i) & = \sup_{v\in \bR^d, \|v\|_2 \leq 1} v^{\top}\mathbb{E}_p[X]   - \frac{1}{n} \sum_{i = 1}^n v^{\top} X_i  \\
    & = \left \| \frac{1}{n} \sum_{i = 1}^n  X_i - \mathbb{E}_p[X] \right \|_2 
\end{align}

Now we bound the uniform law of large number for $\mathcal{U}_2'$. Here we first shift $X$ to $\widetilde X = X - \bE_p[X]$. Since $a$ is taken in $\bR$, $\tW_1(p, \hat p_n)$ wouldn't change.  With a bit abuse of notation, we still use $X$ to represent the mean-0 shifted distribution. We have
\begin{align}
    \bE \left[ \sup_{u \in \mathcal{U}_2'} \bE_p[u(X)]  - \frac{1}{n}\sum_{i=1}^n u(X_i)  \right] & =  \bE \left[ \sup_{v \in \bR^d, \| v\|_2 = 1, a\in\bR} \bE_p\left[\max(0, v^{\top}X-a) \right] - \frac{1}{n}\sum_{i=1}^n \max(0, v^{\top}X_i-a)   \right]   
\end{align}

We bound three cases separately, i.e.  $-M \leq a \leq M$, $a> M$ and $a < -M$. 

For  $-M \leq a \leq M$, from  symmetrization inequality~\citep[Proposition 4.11]{wainwright2019high}, we have
\begin{align}
    & \bE_p \left[ \sup_{v \in \bR^d, \| v\|_2 = 1, a\in[-M, M]} \bE_p\left[\max(0, v^{\top}X-a) \right] - \frac{1}{n}\sum_{i=1}^n \max(0, v^{\top}X_i-a)   \right]   \\
     \leq & 2\bE_{p, \epsilon \sim \{\pm 1\}^d} \left[ \sup_{v \in \bR^d, \| v\|_2 = 1, a\in[-M, M]} \left[ \frac{1}{n}\sum_{i=1}^n  \epsilon_i \max(0, v^{\top}X_i-a) \right]  \right] 
\end{align}
From Talagrand contraction inequality~\citep[Exercise 6.7.7]{vershynin2018high} as below, we have  
\begin{align}
    & 2\bE_{p, \epsilon \sim \{\pm 1\}^d} \left[ \sup_{v \in \bR^d, \| v\|_2 = 1, a\in[-M, M]} \left[ \frac{1}{n}\sum_{i=1}^n  \epsilon_i \max(0, v^{\top}X_i-a) \right]  \right] \label{eqn.tW1_contraction1}\\ 
     \leq & 2\bE_{p, \epsilon \sim \{\pm 1\}^d} \left[ \sup_{v \in \bR^d, \| v\|_2 = 1, a\in[-M, M]} \left[ \frac{1}{n}\sum_{i=1}^n  \epsilon_i  (v^{\top}X_i-a) \right]  \right]\label{eqn.tW1_contraction2} \\
     \leq & 2 \bE_{p, \epsilon \sim \{\pm 1\}^d} \left[ \sup_{v \in \bR^d, \| v\|_2 = 1} \left[ \frac{1}{n}\sum_{i=1}^n  \epsilon_i v^{\top}X_i \right]  \right] + 2 \bE_{\epsilon \sim \{\pm 1\}^d} \left[ \sup_{a\in[-M, M]} \left[ \frac{1}{n}\sum_{i=1}^n  \epsilon_i a \right]  \right] \label{eqn.tW1_contraction3}\\
     \leq & 2 \bE_{p}  \left\| \frac{1}{n}\sum_{i=1}^n   ( X_i -\bE_p[X])\right\|_2  + 2M \bE_{\epsilon \sim \{\pm 1\}^d}   \left| \frac{1}{n}\sum_{i=1}^n  \epsilon_i \right|  \label{eqn.tW1_contraction4}\\
   \leq    &  4  \bE_{p}  \left\| \frac{1}{n}\sum_{i=1}^n  ( X_i -\bE_p[X]) \right\|_2  + 2M \sqrt{\bE_{\epsilon \sim \{\pm 1\}^d}   (\frac{1}{n}\sum_{i=1}^n  \epsilon_i)^2 } \label{eqn.tW1_contraction5}\\
   \leq & 4\bE_{p}  \left \| \frac{1}{n} \sum_{i = 1}^n  X_i - \mathbb{E}_p[X] \right \|_2  + 2\frac{M}{\sqrt{n}}.
\end{align}

Here Equation (\ref{eqn.tW1_contraction5}) comes from  symmetrization inequality~\citep[Proposition 4.11]{wainwright2019high}.

For $a > M$, we have
\begin{align}
     & \bE_p \left[ \sup_{v \in \bR^d, \| v\|_2 = 1, a>M} \bE_p\left[\max(0, v^{\top}X-a) \right] - \frac{1}{n}\sum_{i=1}^n \max(0, v^{\top}X_i-a)   \right]   \\
     \leq & \bE_p \left[ \sup_{v \in \bR^d, \| v\|_2 = 1, a>M} \bE_p\left[\max(0, v^{\top}X-a) \right]  \right]  \\
     = &  \sup_{v \in \bR^d, \| v\|_2 = 1, a>M} \bE_p\left[\max(0, v^{\top}X-a) \right]  \\
     < &   \sup_{v \in \bR^d, \| v\|_2 = 1 } \bE_p\left[\max(0, v^{\top}X-M) \right].
\end{align}

For $a < -M$, we have
\begin{align}
      & \bE_p \left[ \sup_{v \in \bR^d, \| v\|_2 = 1, a<-M} \bE_p\left[\max(0, v^{\top}X-a) \right] - \frac{1}{n}\sum_{i=1}^n \max(0, v^{\top}X_i-a)   \right]   \\
     = & \bE_p \Bigg[ \sup_{v \in \bR^d, \| v\|_2 = 1, a<-M} \bE_p\left[  v^{\top}X-a + \max\left(0, -v^{\top}X+a\right)\right] \nonumber \\ 
     & -  \frac{1}{n}\sum_{i=1}^n \left(  v^{\top}X_i-a + \max\left(0, -v^{\top}X+a\right) \right) \Bigg]  \\
     = & \bE_p \Bigg[ \sup_{v \in \bR^d, \| v\|_2 = 1, a<-M} \bE_p\left[  v^{\top}X-a \right] - \frac{1}{n}\sum_{i=1}^n \left(  v^{\top}X_i-a\right)  \nonumber \\
      & + \left(  \bE_p \left[ \max\left(0, -v^{\top}X+a\right) \right] - \frac{1}{n}\sum_{i=1}^n \max\left(0, -v^{\top}X_i+a\right)   \right) \Bigg]  \\
     \leq & \bE_{p} \left \| \frac{1}{n} \sum_{i = 1}^n  X_i - \mathbb{E}_p[X] \right \|_2  + \sup_{v \in \bR^d, \| v\|_2 = 1, a<-M} \bE_p \left[ \max\left(0, -v^{\top}X+a\right) \right] \\
     \leq & \bE_{p} \left \| \frac{1}{n} \sum_{i = 1}^n  X_i - \mathbb{E}_p[X] \right \|_2 +  \sup_{v \in \bR^d, \| v\|_2 = 1} \bE_p \left[ \max\left(0, -v^{\top}X-M\right) \right] \\
     = & \bE_{p} \left \| \frac{1}{n} \sum_{i = 1}^n  X_i - \mathbb{E}_p[X] \right \|_2 +  \sup_{v \in \bR^d, \| v\|_2 = 1} \bE_p \left[ \max\left(0,  v^{\top}X-M\right) \right]
\end{align}
Thus  we have
\begin{align}
     \sup_{u \in \mathcal{U}_2'} \bE_p[u(X)]  - \frac{1}{n}\sum_{i=1}^n u(X_i) & \leq 4\bE_{p}  \left \| \frac{1}{n} \sum_{i = 1}^n  X_i - \mathbb{E}_p[X] \right \|_2  + 2\frac{M}{\sqrt{n}} \nonumber \\ 
     & \quad +  \sup_{v \in \bR^d, \| v\|_2 = 1} \bE_p \left[ \max\left(0,  v^{\top}X-M\right) \right] .
\end{align}
Following a similar argument, we have
\begin{align}
     \sup_{u \in \mathcal{U}_3'} \bE_p[u(X)]  - \frac{1}{n}\sum_{i=1}^n u(X_i)  &\leq4\bE_{p}  \left \| \frac{1}{n} \sum_{i = 1}^n  X_i - \mathbb{E}_p[X] \right \|_2  + 2\frac{M}{\sqrt{n}} \nonumber \\ 
     & \quad +  \bE_p \left[ \sup_{v \in \bR^d, \| v\|_2 = 1} \frac{1}{n}\sum_{i=1}^n \max\left(0,  v^{\top}X_i-M\right) \right].
\end{align}

To see the final results, we first  show that
\begin{align}
    \sup_{v \in \bR^d, \| v\|_2 = 1} \bE_p \left[ \max\left(0,  v^{\top}X-M\right) \right]  \leq \bE_p \left[ \sup_{v \in \bR^d, \| v\|_2 = 1} \frac{1}{n}\sum_{i=1}^n \max\left(0,  v^{\top}X_i-M\right) \right].
\end{align}
This can be seen from that for any $v\in\bR^d, \| v\|_2 = 1$,
\begin{align}
    \frac{1}{n}\sum_{i=1}^n\max\left(0,  v^{\top}X_i-M\right) \leq \sup_{v \in \bR^d, \| v\|_2 = 1} \frac{1}{n}\sum_{i=1}^n \max\left(0,  v^{\top}X_i-M\right).
\end{align}
Taking expectation on both sides, we can see that for any $v\in\bR^d, \| v\|_2=1$,
\begin{align}
    \bE_p[\max\left(0,  v^{\top}X_i-M\right)] \leq \bE_p[\sup_{v \in \bR^d, \| v\|_2 = 1} \frac{1}{n}\sum_{i=1}^n \max\left(0,  v^{\top}X_i-M\right)].
\end{align}
Thus we only need to bound the RHS of the above equation.
Following the same approach of symmetrization and contraction inequality, we have
\begin{align}
    & \bE_p\left[\sup_{v \in \bR^d, \| v\|_2 = 1} \frac{1}{n}\sum_{i=1}^n \max\left(0,  v^{\top}X_i-M\right)\right] \nonumber \\ 
    = & \bE_p\left[\sup_{v \in \bR^d, \| v\|_2 = 1} \frac{1}{n}\sum_{i=1}^n \max\left(0,  v^{\top}X_i-M\right)\right] - \sup_{v \in \bR^d, \| v\|_2 = 1} \bE_p[ \max\left(0,  v^{\top}X-M\right)] \nonumber \\
    & + \sup_{v \in \bR^d, \| v\|_2 = 1} \bE_p[ \max\left(0,  v^{\top}X-M\right)]\nonumber \\
    \leq&  2\bE_{p, \epsilon \sim \{\pm 1\}^d}\left[\sup_{v \in \bR^d, \| v\|_2 = 1} \frac{1}{n}\sum_{i=1}^n \epsilon_i\max\left(0,  v^{\top}X_i-M\right)\right] + \sup_{v \in \bR^d, \| v\|_2 = 1} \bE_p[ \max\left(0,  v^{\top}X-M\right)]\nonumber \\
    \leq & 2\bE_{p, \epsilon \sim \{\pm 1\}^d}\left[\sup_{v \in \bR^d, \| v\|_2 = 1} \frac{1}{n}\sum_{i=1}^n \epsilon_i  v^{\top}X_i\right] + \sup_{v \in \bR^d, \| v\|_2 = 1} \bE_p[ \max\left(0,  v^{\top}X-M\right)]\nonumber \\
    \leq & 4\bE_{p}\left[\sup_{v \in \bR^d, \| v\|_2 = 1} \frac{1}{n}\sum_{i=1}^n   v^{\top}(X_i-\bE_p[X])\right]+ \sup_{v \in \bR^d, \| v\|_2 = 1} \bE_p[ \max\left(0,  v^{\top}X-M\right)]\nonumber \\
    = & 4 \bE_p\left[\|\frac{1}{n}\sum_{i=1}^n   v^{\top}X_i- \bE_p[X]\|_2\right] + \sup_{v \in \bR^d, \| v\|_2 = 1} \bE_p[ \max\left(0,  v^{\top}X-M\right)].
\end{align}

Overall, we have
\begin{align}
    \bE_p[ \tW_1(p, \hat p_n)] \leq & 8 \bE_{p}  \left \| \frac{1}{n} \sum_{i = 1}^n  X_i - \mathbb{E}_p[X] \right \|_2 +  \sup_{v \in \bR^d, \| v\|_2 = 1} \bE_p[ \max(0,  v^{\top}X-M)]  +  2\frac{M}{\sqrt{n}}.
\end{align}

Combining all the results give the conclusion. 

Now we are ready for the proof of main Lemma. Note that $\tW_1$  is a pseudometric. 
By triangle inequality, we have
\begin{align}
     \tW_1(p, \hat p_n) \leq \tW_1(p, p^*) + \tW_1(p^*, \hat p_n^*) + \tW_1(\hat p_n^*, \hat p_n).
\end{align}
Taking the expectation over the optimal coupling $\pi$ between $p, p^*$, by Lemma~\ref{lem.coupling}, we know that 
\begin{align}
    \bE_p[\tW_1(p, \hat p_n)] & \leq \tW_1(p, p^*) + \bE_{p^*}[\tW_1(p^*, \hat p_n^*)] + \bE_{\pi}[\tW_1(\hat p_n^*, \hat p_n)] \nonumber \\
    & \leq \epsilon + \bE_{p^*}[\tW_1(p^*, \hat p_n^*)] + \bE_{\pi}[W_1(\hat p_n^*, \hat p_n)] \nonumber \\
    & \leq 2\epsilon + \bE_{p^*}[\tW_1(p^*, \hat p_n^*)]. 
\end{align}
Thus it suffices to bound the term $\bE_{p^*}[\tW_1(p^*, \hat p_n^*)]  $.
By Lemma~\ref{lem.tW_1_empirical_to_population}, we have  for any $M >0$, 
\begin{align}
    \bE_{p^*}[\tW_1(p^*, \hat p_n^*)] \leq 8\xi_1 +  \xi_2(M)+  2\frac{M}{\sqrt{n}}.
\end{align}
where $  \xi_1  = \bE_{p^*}  \left \| \frac{1}{n} \sum_{i = 1}^n  X_i - \mathbb{E}_{p^*}[X] \right \|_2$, $\xi_2(M) = \sup_{v \in \bR^d, \| v\|_2 = 1 } \bE_{p^*}\left[ \max(0, v^{\top}(X-\bE_{p^*}[X])-M) \right]. $ Now we bound the two terms $\xi_1, \xi_2$ separately. 
From Lemma~\ref{lem.kth_convergence}, we know that 
\begin{align}
   \xi_1=  \bE_{p^*}  \left \| \frac{1}{n} \sum_{i = 1}^n  X_i - \mathbb{E}_{p^*}[X] \right \|_2 \leq \sigma \sqrt{\frac{d}{n}}.
\end{align}

Now we bound the term $\xi_2(M)$. From Lemma~\ref{lem.cvx_mean_resilience},
we know for some fixed $v\in\bR^d, \|v\|_2=1$,
\begin{align}
     \|\bE_{p^*}[ X \mid v^\top (X-\bE_{p^*}[X]) \geq M] - \bE_{p^*}[X]\|_2 & \leq \kappa\psi^{-1}(1/(1-\bP_{p^*}(v^\top (X-\bE_{p^*}[X]) \leq M))) \nonumber \\
    & = \kappa\psi^{-1}(1/\bP_{p^*}(v^\top (X-\bE_{p^*}[X]) \geq M)).
\end{align}
By Markov's inequality, we have
\begin{align}
    \bP_{p^*}(v^{\top}(X-\bE_{p^*}[X])\geq M) & \leq \bP_{p^*}(\psi(|v^{\top}(X-\bE_{p^*}[X])|/\kappa)\geq \psi(M/\kappa)) \nonumber \\
    & \leq \frac{\bE_{p^*}[\psi(|v^{\top}(X-\bE_{p^*}[X])/\kappa|)]}{\psi(M/\kappa)} \nonumber \\
    & \leq \frac{1}{\psi(M/\kappa)}.
\end{align}
Thus we have
\begin{align}
    \xi_2(M) & = \sup_{v \in \bR^d, \| v\|_2 = 1 } \bE_{p^*}\left[ \max(0, v^{\top}(X-\bE_{p^*}[X])-M) \right] \nonumber \\
    & = \sup_{v \in \bR^d, \| v\|_2 = 1 } \bP_{p^*}(v^{\top}(X-\bE_{p^*}[X])\geq M) \|\bE_{p^*}[X-\bE_{p^*}[X] \mid v^{\top}(X-\bE_{p^*}[X])\geq M] \|_2 \nonumber \\
    & \leq \sup_{v \in \bR^d, \| v\|_2 = 1 } \kappa \bP_{p^*}(v^{\top}(X-\bE_{p^*}[X])\geq M) \psi^{-1}(1/\bP_{p^*}(v^\top (X-\bE_{p^*}[X]) \geq M)) \nonumber \\
    & \leq \frac{M}{\psi(M/\kappa)}.
\end{align}
The last inequality uses the fact that $\epsilon \psi^{-1}(1/\epsilon)$ is nondecreasing from Lemma~\ref{lem.psi_increasing}.
Now we balance the term $\xi_2(M) + \frac{M}{\sqrt{n}}$. By taking $M = \kappa \psi^{-1}(\sqrt{n})$, we have
\begin{align}
    \xi_2(M) + 2\frac{M}{\sqrt{n}} \leq \frac{3\kappa \psi^{-1}(\sqrt{n})}{\sqrt{n}}.
\end{align}

\end{proof}

\subsection{Proof of Theorem~\ref{thm.sec_tw_1_proj}}\label{proof.psi_w1_resilience}
The theorem can be decomposed into the following two lemmas, on the population and finite sample results separately.

\begin{lemma}
 
Let $\psi$ be an Orlicz function  that further satisfies $\psi(x)\geq x$ for all $x \geq 1$, 
and define $\tilde \psi(x) = x\psi(2x)$. Suppose that 
\begin{align}
   \sup_{v\in\bR^d, \|v\|_2 =1} \bE_p\left[\tilde \psi\left(\frac{|v^{\top}X|}{\sigma}\right)\right] \leq 1.
\end{align}

 Then,
$p \in \GG^{W_1}_{\mathsf{sec}}(\rho(\eta), \eta)$ for 
$\rho(\eta) = \max(4\eta^2 + 2\sigma\eta, \sigma\eta\psi^{-1}(\frac{2\sigma}{\eta}))$, where $C$ is some universal constant.  %
The population limit when the perturbation level is $\epsilon$ is  $\Theta(\rho(2\epsilon))$.
\end{lemma}

In the above lemma, 
taking $\psi(x) = x^m$ for $m>1$, we know that when the $(m+1)$-th moment of $X$ is bounded by $2(\frac{\sigma}{2})^{m+1}$, 
the population limit is $\Theta(\min(\sigma^{1+1/m}\epsilon^{1-1/m}, \sigma^2))$.
\begin{lemma}

Assume $p^*$ has bounded $k$-th moment for $k>2$, i.e. 
    $\sup_{v\in\bR^d, \|v\|_2=1}\bE_{p^*}[|v^\top X|^k] \leq \sigma^k$
for some $\sigma>0$.
Denote $ \tilde \epsilon = \frac{C_1}{\delta} \left(\epsilon +  \sigma \sqrt{\frac{d}{n}}+\frac{\sigma }{\sqrt{n^{1-1/k}}}\right)$, where $C_1$ is some universal constant. 
Then the projection algorithm $q = \Pi(\hat p_n; \tW_1, \GG(k))$ or $q = \Pi(\hat p_n; \tW_1, \GG(k), \tilde \epsilon/2)$    satisfies
\begin{align}
    \|\bE_{q}[XX^{\top}] - \bE_{p^*}[XX^{\top}]\|_2 \leq  C_2\min(\sigma^2, \sigma^{1+1/(k-1)}\tilde \epsilon^{1-1/(k-1)})
\end{align}
with probability at least $1-\delta$, 
where $C_2$ is some universal constant.
\end{lemma}
We start with the proof of the first lemma.
\begin{proof}
We verify that the sufficient condition implies $p^* \in \GG_{\mathsf{sec}}^{W_1}$. 
From the fact that $r \in \mathbb{F}(p, \eta, W_1, |v^{\top}X|^2)$, we know that for any coupling $\pi_{p, r}$ that makes $r$ friendly perturbation, we have 
\begin{align}
  \sup_{v\in\bR^d, \|v\|_2=1} \bE_{\pi_{p, r}} |v^\top(X - Y)|\leq  \bE_{\pi_{p, r}} \sup_{v\in\bR^d, \|v\|_2=1} |v^\top(X - Y)| \leq  \eta.
\end{align}
For any fixed $v\in\bR^d, \|v\|_2 = 1$, we claim that the worst perturbation only happens when for any $(x, y) \in \mathsf{supp}(\pi_{p, r})$, $|v^\top x|^2 \geq |v^\top y|^2$ or for any $(x, y) \in \mathsf{supp}(\pi_{p, r})$, $|v^\top x|^2 \leq |v^\top y|^2$. If it is not one of the two cases, we can always remove the movement from $x$ to $y$ that decreases or increases $g$ to make $|\bE_\pi[|v^\top X|^2 - |v^\top Y|^2]|$ larger without increasing $\bE_\pi\|X-Y\|$.

Thus we can assume for any $(x, y) \in \mathsf{supp}(\pi_{p, r})$, $|v^\top x|^2 \geq |v^\top y|^2$ or for any $(x, y) \in \mathsf{supp}(\pi_{p, r})$, $|v^\top x|^2 \leq |v^\top y|^2$. For the first case,  by Lemma~\ref{lem.coupling_g_bounded}, we bound the worst case perturbation as follows. For any $v\in\bR^d, \|v\|_2=1$,
\begin{align}
    |\bE_{(X, Y) \sim \pi_{p, r}}[|v^\top X|^2 - |v^\top Y|^2]| & \leq \sigma \bE_{\pi_{p, r}}[|v^\top (X-Y)|] \psi^{-1}\left(\frac{\bE_{\pi_{p, r}}\Big[|v^\top (X-Y)| \psi\left(\left|\frac{|v^\top X|^2 - |v^\top Y|^2}{\sigma v^\top (X-Y)}\right|\right)\Big]}{ \bE_{\pi_{p, r}}[|v^\top (X-Y)|]}\right) \nonumber \\
    & \leq \sigma \eta \psi^{-1}\left(\frac{\bE_{\pi_{p, r}}\Big[|v^\top (X-Y)| \psi\left(\left|\frac{|v^\top X|^2 - |v^\top Y|^2}{\sigma v^\top (X-Y)}\right|\right)\Big]}{\eta}\right)\label{eqn.proof_w1_sec_orlicz_1}  \\
    & \leq \sigma \eta \psi^{-1}\left(\frac{\bE_{\pi_{p, r}}\Big[|v^\top (X-Y)| \psi\left(\left|\frac{2|v^\top X|}{\sigma}\right|\right)\Big]}{\eta}\right) \label{eqn.proof_w1_sec_orlicz_2}  \\
    & \leq \sigma \eta \psi^{-1}\left(\frac{\bE_{\pi_{p, r}}\Big[2|v^\top X| \psi\left(\left|\frac{2|v^\top X|}{\sigma}\right|\right)\Big]}{\eta}\right)  \\
    & = \sigma \eta \psi^{-1}\left(\frac{\bE_p\Big[2|v^\top X| \psi\left(\left|\frac{2|v^\top X|}{\sigma}\right|\right)\Big]}{\eta}\right)  \\
    & \leq \sigma \eta \psi^{-1}(\frac{2\sigma}{\eta})\label{eqn.proof_w1_sec_orlicz_3} . 
\end{align}
Here Equation (\ref{eqn.proof_w1_sec_orlicz_1}) comes from the fact that $x\psi^{-1}(C / x) $ is a non-decreasing function of $x$ for the region $[0,+\infty)$ for any $\sigma >0$ (Lemma~\ref{lem.psi_increasing}). Equation (\ref{eqn.proof_w1_sec_orlicz_2}) uses the fact that  for any $(x, y) \in \mathsf{supp}(\pi_{p, r})$, $|v^\top x| \geq |v^\top y|$. Equation (\ref{eqn.proof_w1_sec_orlicz_3}) is from the assumption given.

On the other hand, for any $(x, y) \in \mathsf{supp}(\pi_{p, r})$, $|v^\top x|^2 \leq |v^\top y|^2$, we bound the worst case perturbation as follows:
\begin{align}
     \bE_{(X, Y) \sim \pi_{p, r}}[|v^\top Y|^2 - |v^\top X|^2] & =  |\bE_{(X, Y) \sim \pi_{p, r}}[(|v^\top X| - |v^\top Y|)(|v^\top X| + |v^\top Y|)]| \\
     & \leq 2|\bE_{(X, Y) \sim \pi_{p, r}}[(|v^\top X| - |v^\top Y|)]\sqrt{\bE[|v^\top Y|^{2}]}|. 
     \\
     & \leq 2\eta \sqrt{\bE[|v^\top Y|^{2}]}. \label{eqn.proof_W1_sec_resilience1}
\end{align}
Here we use the fact that 
For any $(x, y) \in \mathsf{supp}(\pi_{p, r}), x \neq y$,  we have $|v^\top x|^2\leq |v^\top y|^2 \leq \bE[|v^\top Y|^2]$ from the definition of friendly perturbation.

Solving the inequality, we can get that 
\begin{align}
    \bE[|v^\top Y|^{2}] \leq (\eta + \sqrt{\eta^2 + \bE[|v^\top X|^{2}]})^2.
\end{align}
Thus
\begin{align}
        \bE[|v^\top Y|^{2}] - \bE[|v^\top X|^{2}] & \leq  (\eta + \sqrt{\eta^2 + \bE[|v^\top X|^{2}]})^2 - \bE[|v^\top X|^{2}] \nonumber \\ 
        & =2\eta^2 + 2\eta\sqrt{\eta^2 + \bE[|v^\top X|^{2}]} \nonumber \\ 
        & \leq 4\eta^2 + 2\eta\sqrt{\bE[|v^\top X|^{2}}.
\end{align}

Now we  show that  for any 1-d random variable $X$, $\bE_{p^*}[|\frac{X}{\sigma}| \psi(|\frac{2X}{\sigma}|)]\leq 1$ implies $\bE_{p^*}[X^2]\leq 2\sigma^2$. Note that $\bE_{p^*}[|\frac{X}{\sigma}| \psi(|\frac{2X}{\sigma}|)]\leq 1$ is equivalent to
\begin{align}
    1 & \geq \bP_{p^*}(|X|\leq \sigma)\bE_{p^*}[|\frac{X}{\sigma}| \psi(|\frac{X}{\sigma}|) \mid  |X| \leq \sigma] + \bP_{p^*}(|X|> \sigma)\bE_{p^*}[|\frac{X}{\sigma}| \psi(|\frac{X}{\sigma}|) \mid  |X| > \sigma] \nonumber \\ 
    & \geq \bP_{p^*}(|X|> \sigma)\bE_{p^*}[X^2/\sigma^2 \mid  |X| > \sigma],
\end{align}
since $\psi(x)\geq x$ for $x\geq 1$. Thus we have
\begin{align}
    \bE_{p^*}[X^2/\sigma^2] & = \bP_{p^*}(|X|\leq \sigma)\bE_{p^*}[X^2/\sigma^2 \mid  |X| \leq \sigma] + \bP_{p^*}(|X|> \sigma)\bE_{p^*}[X^2/\sigma^2 \mid  |X| > \sigma]\nonumber \\ 
    & \leq 2. 
\end{align}

Thus we can conclude that 
\begin{align}\label{eqn.moving_up_bound}
      \bE[|v^\top Y|^{2}] - \bE[|v^\top X|^{2}] & \leq 4\eta^2 + 2\sigma\eta.
\end{align}

Combining the two cases, we know that the movement is upper bounded by $\max(4\eta^2 + 2\sigma\eta, \sigma\eta\psi^{-1}(\frac{2\sigma}{\eta}))$. 
Thus 
we have  $p \in \GG_{\mathsf{sec}}^{W_1}(\max(4\eta^2 + 2\sigma\eta, \sigma\eta\psi^{-1}(\frac{2\sigma}{\eta})), \eta)$.

We remark here that the above proof also applies to the case when we requires $r\in\mathbb{F}(p, \eta, W_1, |v^\top X|)$ instead of $r\in\mathbb{F}(p, \eta, W_1, |v^\top X|^2)$. The only difference to note is in above~\eqref{eqn.proof_W1_sec_resilience1} where we need to apply Jensen's inequality to derive $|v^\top y|^2 \leq (\bE[|v^\top Y|])^2 \leq \bE[|v^\top Y|^2]$.
This proves to be crucial in finite sample algorithm design in Section~\ref{sec.finite_sample_w1}.

\end{proof}

Now we prove the second lemma on the finite-sample results.
\begin{proof}
First, we show that the projected distribution $q$ is close to $p^*$ in $\tW_1$.  We know from Lemma~\ref{lem.tW_1_empirical_to_population} that under appropriate choice of $C_1$, with probability at least $1-\delta$, we have
\begin{align}
    \tW_1(p^*, \hat p_n) \leq \frac{\tilde \epsilon} {2}=  \frac{C_1}{\delta} \left(\epsilon +  \sigma \sqrt{\frac{d}{n}}+\frac{\sigma }{\sqrt{n^{1-1/k}}}\right).
\end{align}

Note that $\tW_1$ satisfies triangle inequality and $\mathcal{U}' \subset \mathcal{U}$. Therefore 
\begin{align}
    \tW_1(q, p^*) & \leq \tW_1(q, \hat p_n) + \tW_1(\hat p_n, p^*) \leq \tilde \epsilon.
\end{align}

From Lemma~\ref{lemma.population_limit} we know that the final result can be upper bounded by the modulus of continuity, thus it suffices to bound the term $\sup_{p_1, p_2\in \GG, \tW_1(p_1, p_2)\leq \tilde \epsilon} \|M_{p_1} - M_{p_2}\|_2 $.
We apply Lemma~\ref{lem:tw1_cross_mean} to show that the modulus of continuity can be bounded.
By symmetry, without loss of generality we can take some $v^*\in\bR^d, \|v^*\|_2=1$, such that
\begin{align}
    v^{*\top}(M_{p_1} -M_{p_2} )v^{*} = \|M_{p_1} -M_{p_2} \|_2.
\end{align}

From $p_1, p_2\in\GG^{W_1}(\rho_1(7\tilde \epsilon), \rho_2(7\tilde \epsilon), 7\tilde \epsilon)$ and $\tW_1(p_1, p_2)\leq  \tilde \epsilon$ and Lemma~\ref{lem:tw1_cross_mean}, we know that  there exist an $r_{p_1}\in \mathbb{F}(p_1, 7\tilde \epsilon, W_1, |v^{*\top} X|)$ and an $r_{p_2} \in \mathbb{F}(p_2,  7\tilde \epsilon, W_1, |v^{*\top} X|)$ such that
    \begin{align}
        \bE_{r_{p_1}}[(v^{*\top} X)^2 ] \leq \bE_{r_{p_2}} [(v^{*\top} X)^2].
    \end{align}
    
 From $p_1, p_2 \in \GG^{W_1}_\downarrow$, we know that 
    \begin{align}
        \bE_{{p_1}}[(v^{*\top} X)^2 ] - \bE_{r_{p_1}}[(v^{*\top} X)^2 ] \leq \max(4(7\tilde \epsilon)^2 + 14\sigma \tilde \epsilon, 14\sigma^{1+1/(k-1)}\tilde \epsilon^{1-1/(k-1)}), \\
        \bE_{r_{p_2}}[(v^{*\top} X)^2 ] - \bE_{{p_2}}[(v^{*\top} X)^2 ] \leq \max(4(7\tilde \epsilon)^2 + 14\sigma \tilde \epsilon, 14\sigma^{1+1/(k-1)}\tilde \epsilon^{1-1/(k-1)}).
    \end{align}
for all friendly perturbations $r_{p_1}$, $r_{p_2}$ of $p_1$, $p_2$.
Thus we know that
\begin{align}
      \|M_{p_1} -M_{p_2} \|_2 = &  v^{*\top}(M_{p_1} -M_{p_2} )v^{*}  \nonumber \\
       = &\bE_{p_1}[(v^{*\top}X)^2] -  \bE_{p_2}[(v^{*\top}X)^2] \nonumber \\
       =& \bE_{p_1}[(v^{*\top}X)^2] -  \bE_{r_{p_1}}[(v^{*\top}X)^2] + \bE_{r_{p_1}}[(v^{*\top}X)^2] -  \bE_{r_{p_2}}[(v^{*\top}X)^2] \nonumber \\
      & + \bE_{r_{p_2}}[(v^{*\top}X)^2] -  \bE_{{p_2}}[(v^{*\top}X)^2] \nonumber \\
      \leq & 2\max(4(7\tilde \epsilon)^2 + 14\sigma \tilde \epsilon, 14\sigma^{1+1/(k-1)}\tilde \epsilon^{1-1/(k-1)}).
\end{align}
Furthermore, from the projection set we know that both $q$ and $p^*$ have their second moment bounded by $C\sigma^2$. Thus the final finite sample rate is $C\min(\sigma^2, \sigma^{1+1/(k-1)}\tilde \epsilon^{1-1/(k-1)}))$.
\end{proof}

\textbf{Lower bound on population limit.}

So far we have already shown that the population limit is upper bounded by $C\cdot\min(\epsilon^2, \sigma \epsilon \psi^{-1}(\sigma/\epsilon))$ when the perturbation level is $\epsilon$. Now we show that it is tight for the set of one-dimensional distribution with bounded Orlicz norm:
\begin{align}
    \GG = \left\{p \mid \bE_{p}\left[\frac{|X|}{\sigma}\psi\left(\frac{|X|}{\sigma}\right)\right] \leq 1\right\}.
\end{align}
We show the lower bound separately: when $\epsilon \geq \sigma$, the limit is lower bounded by $\epsilon^2$, when $\epsilon <\sigma$, the limit is lower bounded by $\sigma\epsilon\psi^{-1}(\sigma/\epsilon)$.

To see the first half, consider we observe the distribution $p$ with $X \equiv \sigma$. Since $\epsilon \geq \sigma$, the true distribution can perturb it to either $X \equiv 0$ or $X \equiv \sigma + \epsilon$. By standard Le Cam's two point argument, we can see that it is lower bounded by $\sigma^2$.

To see the second half, assume we observe the distribution $p$ with $X \equiv \frac{\sigma \psi^{-1}(\sigma/\epsilon)}{2}$. 
Since $\epsilon<\sigma$, we have $\sigma\psi^{-1}(\sigma/\epsilon)>\epsilon$. 
Now we construct two specific distributions  $p_1: X \equiv \frac{\sigma \psi^{-1}(\sigma/\epsilon)+\epsilon}{2}$ and $p_2: X \equiv \frac{\sigma \psi^{-1}(\sigma/\epsilon)-\epsilon}{2}$. 
One can see that $W_1(p_1, p)\leq \epsilon/2$, $W_1(p_2, p)\leq \epsilon/2$, $ p_1, p_2\in \GG$.   By standard Le Cam's two point argument, the error is lower bounded by $\frac{\sigma\epsilon\psi^{-1}(\sigma/\epsilon)}{2}$. From Lemma~\ref{lem.minimax_random} we know that this lower bound also holds for random decision rule with probability at least $1/2$.

\subsection{Proof of Theorem~\ref{thm.linreg_tw_1_proj}}\label{proof.linreg_tw_1_proj}
Similar to the second moment estimation, we decompose the theorem into two lemmas for population and finite-sample:
\begin{lemma}\label{lem.proof_linreg_w1_population}
Denote by $X' = [X, Z]$ the $d+1$ dimensional vector that concatenates $X$ with the noise $Z = Y-X^{\top}\theta^*(p^*)$. Given an Orlicz function
     $\psi$
that further satisfies $\psi(x)\geq x$ for all $x \geq 1$, denote $\tilde \psi(x) = x\psi(2x)$. 
Assume $p^*\in\GG(k)$, i.e. it satisfies:
\begin{align}\label{eqn.bounded_linreg_w1_condition}
   \sup_{v\in\bR^{d+1}, \|v\|_2=1}  \bE_{p^*}\bigg[\tilde \psi\bigg(\frac{|v^{\top}X'|}{\sigma_1}\bigg)\bigg] \leq 1,
   \bE_{p^*}[Z^2] \leq \sigma_2^2, \|\theta^*(p^*)\|_2\leq R.
\end{align} 
Denote $\bar R = \max(R, 1)$. 
Then we have
$p^* \in\GG^{W_1}(\sigma_2^2 + C_1 \Delta(\bar R, \eta), \sigma_2^2 + C_2\bar R^2\Delta(\bar R, \eta), \eta)$ for any $\eta>0$, where $\Delta(\bar R, \eta) = \sigma_1 \bar R \eta  \psi^{-1}(2\sigma_1/\bar R \eta) + \bar R^2\eta^2$  and $C_1, C_2$ are universal constants. Here we choose $B(p, \theta) = L(p, \theta) = \bE_p[(Y-X^\top \theta)^2]$.
The population limit for this set when the perturbation level is $\epsilon$  is upper bounded by $\sigma_2^2 + C_2 \bar R^2\Delta(\bar R, 2\epsilon) $.
\end{lemma}
\begin{lemma}
 Assume  $p^*$ satisfies:
\begin{align}
   \sup_{v\in\bR^{d+1}, \|v\|_2=1}  \bE_{p^*}\bigg[{|v^{\top}X'|^{k}}\bigg] \leq \sigma_1^{k}, 
   \bE_{p^*}[Z^2] \leq \sigma_2^2, \|\theta^*(p^*)\|_2\leq R
\end{align} 
Here $k> 2$. 
Denote   $ \tilde \epsilon = \frac{C_1}{\delta} \left(\epsilon + \sigma_1  \bar R\sqrt{ d/n}+ \sigma_1 \bar R/\sqrt{n^{1-1/k}}\right)$,  where $C_1$ is some universal constant. For $\GG^{W_1}$  with appropriate parameters, the projection algorithm $q = \Pi(\hat p_n; \tW_1, \GG^{W_1})$ or $q = \Pi(\hat p_n; \tW_1, \GG^{W_1},  \tilde \epsilon/2)$ satisfies
\begin{align}
    \bE_{p^*}[(Y-X^\top \theta^*(q))^2 ] \leq \sigma_2^2 + C_2 \bar R^2 (\sigma_1^{1+1/(k-1)} (\bar R \tilde \epsilon)^{1-1/(k-1)} + (\bar R \tilde \epsilon)^2)
\end{align} 
with probability at least $1-\delta$, where $C_2$ is some universal constant.
\end{lemma}
We begin with the proof of the first lemma.
\begin{proof}
We know that $\GG^{W_1} = \GG_{\downarrow}^{W_1}(\rho_1, \eta) \bigcap \GG_{\uparrow}^{W_1}(\rho_1, \rho_2, \eta)  $, where 
\begin{align}
    \GG_{\downarrow}^{W_1}(\rho_1, \eta) =  \{p \mid & \sup_{r \in \mathbb{F}(p, \eta, W_{c, k}, |X^{\top}\theta^*(p) - Y|^2) } \bE_r[(X^{\top}\theta^*(p) - Y)^2 ] \leq \rho_1 \}, \\
    \GG_{\uparrow}^{W_1}(\rho_2) =   \Bigg\{p \mid  & \forall \tau\geq 0, \forall \theta \in \Theta,  \forall r\in \mathbb{F}(p, \eta, W_{c, k}, |X^{\top}\theta - Y| ),\nonumber \\
    & \Big( \bE_r[(X^{\top}\theta - Y)^2  ]  \leq \rho_1  \Rightarrow  \bE_p[(X^{\top}\theta - Y)^2 ] \leq \rho_2  \Big) \Bigg\}. 
\end{align}

We first show that $\bE_{p^*}[|\frac{X}{\sigma}| \psi(|\frac{2X}{\sigma}|)]\leq 1$ implies $\bE_{p^*}[X^2]\leq 2\sigma^2$. Note that $\bE_{p^*}[|\frac{X}{\sigma}| \psi(|\frac{2X}{\sigma}|)]\leq 1$ is equivalent to
\begin{align}
    1 & \geq \bP_{p^*}(|X|\leq \sigma)\bE_{p^*}[|\frac{X}{\sigma}| \psi(|\frac{X}{\sigma}|) \mid  |X| \leq \sigma] + \bP_{p^*}(|X|> \sigma)\bE_{p^*}[|\frac{X}{\sigma}| \psi(|\frac{X}{\sigma}|) \mid  |X| > \sigma] \nonumber \\ 
    & \geq \bP_{p^*}(|X|> \sigma)\bE_{p^*}[X^2/\sigma^2 \mid  |X| > \sigma],
\end{align}
since $\psi(x)\geq x$ for $x\geq 1$. Thus we have
\begin{align}
    \bE_{p^*}[X^2/\sigma^2] & = \bP_{p^*}(|X|\leq \sigma)\bE_{p^*}[X^2/\sigma^2 \mid  |X| \leq \sigma] + \bP_{p^*}(|X|> \sigma)\bE_{p^*}[X^2/\sigma^2 \mid  |X| > \sigma]\nonumber \\ 
    & \leq 2. 
\end{align}

Denote $Z = Y - X^{\top}\theta^*({p^*})$. Then $X' = [X, Z]$, and $\bE_{p^*}[Z^2] \leq 2\sigma^2$. Furthermore, since $X' = [X, Y - X^\top \theta^*(p)]$, for any two distributions $p_1$, $p_2$ defined on $(X, Y)$ space with $W_1(p_1, p_2)\leq \eta $, converting them to $X'$ space to derive $\tilde p_1, \tilde p_2$ would give $W_1(\tilde p_1, \tilde p_2)\leq \eta\sqrt{R^2+2}$.  

From the same proof as in second moment estimation (Theorem~\ref{thm.sec_tw_1_proj})
we know that the condition  in~\eqref{eqn.bounded_linreg_w1_condition} implies that 
\begin{align}
& \sup_{v\in\bR^{d+1}, \|v \|_2 =1, r \in  \mathbb{F}({p^*}, \epsilon, W_{c, k}, |v^{\top}X'|^2 )}
    |\bE_{p^*}[(v^{\top}X')^2] - \bE_r[(v^{\top}X')^2]| \nonumber \\ 
    \leq &  \max(4(R^2+2)\eta^2 + 2\sigma\sqrt{R^2+2}\eta, \sigma \eta \sqrt{2R^2+4} \psi^{-1}(2\sigma/(\eta\sqrt{R^2/2+1})))
\end{align}
By setting the last element of $v$ as 1 and all others as 0, we have $p^* \in \GG_{\downarrow}^{W_{c,k}}(\rho_1, \eta)$ where $\rho_1 = \sigma^2 +  \max(4(R^2+2)\eta^2 + 2\sigma\sqrt{R^2+2}\eta, \sigma \eta \sqrt{2R^2+4} \psi^{-1}(2\sigma/(\eta\sqrt{R^2/2+1})))$. %

Now we show that  $p^* \in \GG_{\downarrow}^{W_{c,k}}$. Note that for any $\theta$, we know that 
\begin{align}
    X^{\top}\theta - Y = X^{\top}(\theta - \theta^*({p^*})) + (X^{\top}\theta^*({p^*}) - Y) = X^{\top}(\theta - \theta^*({p^*})) + Z & = |v^{\top}X'| \cdot \sqrt{\|\theta - \theta^*({p^*}) \|_2^2+1} \nonumber \\ 
    & \leq |v^{\top}X'|\cdot \sqrt{4R^2+1}.
\end{align}
where $v$ is the unit vector in the direction of  $(\theta-\theta^*({p^*}), 1)$. Since friendly perturbation is invariant to scaling,  this gives us 
\begin{align}
& \sup_{\theta \in \Theta,   r \in  \mathbb{F}({p^*}, \eta, W_{c, k}, |X^{\top}\theta - Y|^2 )}
    |\bE_{p^*}[(X^{\top}\theta-Y)^2] - \bE_r[(X^{\top}\theta-Y)^2]| \nonumber \\ 
    \leq &  (4R^2+1) \max(4(R^2+2)\eta^2 + 2\sigma\sqrt{R^2+2}\eta, \sigma \eta \sqrt{2R^2+4} \psi^{-1}(2\sigma/(\eta\sqrt{R^2/2+1}))).
\end{align}
Thus we know that if $\bE_r[(X^{\top}\theta - Y)^2  ]  \leq \rho_1  $, we have
\begin{align}
    \bE_{p^*}[(X^{\top}\theta - Y)^2] 
    & \leq \max(4(R^2+2)\eta^2 + 2\sigma\sqrt{R^2+2}\eta, \sigma \eta \sqrt{2R^2+4} \psi^{-1}(2\sigma/(\eta\sqrt{R^2/2+1}))).
\end{align}
Thus we have $p^* \in\GG^{W_1}(\sigma^2 + \Delta, \sigma^2 + (4R^2+2) \Delta, \eta)$, where 
\begin{align}
\Delta = \max(4(R^2+2)\eta^2 + 2\sigma\sqrt{R^2+2}\eta, \sigma \eta \sqrt{2R^2+4} \psi^{-1}(2\sigma/(\eta\sqrt{R^2/2+1}))).
\end{align}

\end{proof}

Now we prove the second lemma on the finite-sample results.
\begin{proof}
First, we show that the projected distribution $q$ is close to $p^*$ in $\tW_1$.  We know from Lemma~\ref{lem.tW_1_empirical_to_population} that under appropriate choice of $C_1$, with probability at least $1-\delta$, we have
\begin{align}
    \tW_1(p^*, \hat p_n) \leq \frac{\tilde \epsilon} {2}=  \frac{C_1}{\delta} \left(\epsilon +  \sigma_1\bar R \sqrt{\frac{d}{n}}+\frac{\sigma_1\bar R }{\sqrt{n^{1-1/k}}}\right).
\end{align}

Note that $\tW_1$ satisfies triangle inequality and $\mathcal{U}' \subset \mathcal{U}$. Therefore 
\begin{align}
    \tW_1(q, p^*) & \leq \tW_1(q, \hat p_n) + \tW_1(\hat p_n, p^*) \leq \tilde \epsilon.
\end{align}

From Lemma~\ref{lem.proof_linreg_w1_population}, we know that $p^*\in \GG^{W_1}(\rho_1(\eta), \rho_2(\eta), \eta)$ with parameters $\rho_1=\sigma_2^2 + C_1 \Delta(\bar R, \eta), \rho_2 = \sigma_2^2 + C_2\bar R^2\Delta(\bar R, \eta)$ for any $\eta>0$, where $\Delta(\bar R, \eta) = \sigma_1 \bar R \eta  \psi^{-1}(2\sigma_1/\bar R \eta) + \bar R^2\eta^2$.  
From Lemma~\ref{lemma.population_limit} we know that the final result can be upper bounded by the modulus of continuity, thus it suffices to bound the term $\sup_{p_1, p_2\in \GG, \tW_1(p_1, p_2)\leq \tilde \epsilon} L(p_1, \theta^*(p_2))$.
We apply Lemma~\ref{lem:tw1_cross_mean} to show that the modulus of continuity can be bounded.
From $p_1, p_2\in\GG^{W_1}$ and $\tW_1(p_1, p_2)\leq  \tilde \epsilon$ and Lemma~\ref{lem:tw1_cross_mean}, we know that  there exist an $r_{p_1}\in \mathbb{F}(p_1, 7\tilde \epsilon, W_1, |Y-X^\top \theta^*(p_2)|)$ and an $r_{p_2} \in \mathbb{F}(p_2,  7\tilde \epsilon, W_1, |Y-X^\top \theta^*(p_2)|)$ such that
    \begin{align}
        \bE_{r_{p_1}}[(Y-X^\top \theta^*(p_2))^2 ] \leq \bE_{r_{p_2}} [(Y-X^\top \theta^*(p_2))^2].
    \end{align}
    From $ p_2 \in \GG^{W_1}_\downarrow$, we know that 
    \begin{align}
        \bE_{r_{p_2}}[(Y-X^\top \theta^*(p_2))^2] \leq \rho_1(7\tilde \epsilon).
    \end{align} 
Thus from $p_1\in\GG^{W_1}_\uparrow$ we know that
\begin{align}\bE_{p_1}[(Y-X^\top \theta^*(p_2))^2] \leq \rho_2(7\tilde \epsilon).
\end{align}
\end{proof}

\subsubsection{Necessity of bounded $\theta$ assumption in $W_1$ linear regression} \label{appendix.necessity_hyper_w1}

To show  the necessity of the upper bound on $\|\theta\|_2$, we provide a lower bound for $W_1$ linear regression question showing that it is not sufficient to have bounded noise $Z$ and Gaussian $X$. The statement is illustrated in Figure~\ref{fig.w1-hypercontractive}.

\begin{theorem}

Taking  $B(p, \theta) = L(p, \theta)  = \bE_p[(Y - X \theta)^2] $  in (\ref{eqn.def_wc_G}).  
Denote $\GG$ as a set of two-dimensional distributions:
\begin{align}
    \GG = \{p \mid (X, Y) \sim p, Y = \theta X, X\sim\mathcal{N}(\mu, \sigma), \mu, \sigma, \theta \in \bR \}
\end{align}
Then the population information theoretic limit is infinity:
\begin{align}
    \inf_{\theta(p)} \sup_{(p^*, p): p^*\in\GG, W_1(p^*, p)\leq \epsilon} \bE_{p^*}[(Y- \theta(p) X)^2] = +\infty.
\end{align}

\end{theorem}

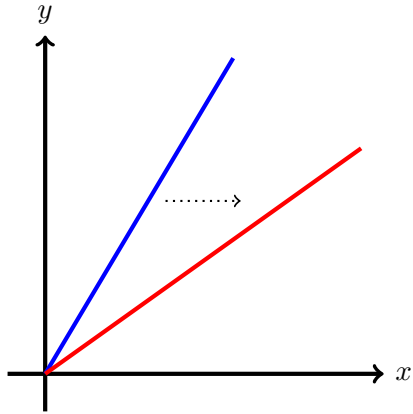
\begin{figure}[!htbp]
\centering
\begin{tikzpicture}[
  implies/.style={double,double equal sign distance,-implies},
]
\draw[->,ultra thick] (-2.5,-2)--(2.5,-2) node[right]{$x$};
\draw[->,ultra thick] (-2,-2.5)--(-2,2.5) node[above]{$y$};
\draw (-2, -2) edge[ultra thick, blue] (0.5,2.2);
\draw[->,dotted, thick] (-0.4,0.3)--(0.6, 0.3);
\draw (-2,-2) edge[ultra thick, red]  (2.2, 1.0);
\end{tikzpicture}

\caption{Illustration of necessity of boundedness assumption on $\|\theta\|_2$. In the  figure we have $Z \equiv 0$. The true distribution $Y = X^\top \theta^*$ lies on the blue line and the adversary can perturb it to the red line by moving $X$ to the right.  When the slope grows steeper, the optimal $\theta$ for the perturbed data would be significantly far away from the true $\theta^*$ when we only slightly perturb all $X$ to the right. }
\label{fig.w1-hypercontractive}
\end{figure}
\begin{proof}
Consider distribution $p_1$ that satisfies $Y = X \theta$, where $X \sim \mathcal{N}(1, 1), \theta > 0$. Now we construct a coupling $\pi$ by moving all points of $p_1$ along $X$ (without moving $Y$) such that under the coupling $\pi_{p_1, p_2}(X, X')$, we have $X' = aX$ for some $a>1$. Here $a$ is selected such that $\bE_\pi [|X-X'|]=\epsilon$. Thus we know that the marginal distribution for $p_2$ is also a Gaussian distribution, and under $p_2$, we have $Y = \frac{\theta X}{a}$. Thus we have $p_2 \in \GG$, and $W_1(p_1, p_2)\leq \epsilon$. 

Denote $\theta'  = \frac{\theta}{a}$. Thus we have $\bE[|\frac{Y}{\theta} - \frac{Y}{\theta'}|] = \epsilon $. Since $\theta>\theta' > 0 $, we have
\begin{align}
    \epsilon = \bE\left[|X|( \frac{\theta}{\theta'}-1)\right] \leq 2(\frac{\theta}{\theta'}-1).
\end{align}
Thus we have 
\begin{align}
    \theta' \leq \frac{\theta}{1+{\epsilon}/{2}}.
\end{align}
Then we have
\begin{align}
   \inf_{\theta(p)} \sup_{(p^*, p): p^*\in\GG, W_1(p^*, p)\leq \epsilon} \bE_{p^*}[(Y-X \theta(p))^2]  & \geq \inf_{\theta(p_1)} \sup_{p^*\in\GG, W_1(p^*, p_1)\leq \epsilon} \bE_{p^*}[(Y-X \theta(p_1))^2] \nonumber \\ 
    & \geq \inf_{\theta(p_1)} \frac{1}{2}(\bE_{p_1}[(Y-X \theta(p_1))^2] + \bE_{p_2}[(Y-X\theta(p_1))^2]) \nonumber \\
    & = \inf_{\theta(p_1)} ((\theta - \theta(p_1))^2 + (\theta' - \theta(p_1))^2) \nonumber \\
    & \geq \frac{1}{2}(\theta - \theta')^2 \nonumber \\ 
    & \geq \frac{\theta^2\epsilon^2}{18}.
\end{align}

As $\theta\rightarrow +\infty$, the limit goes to $+\infty$.
\end{proof}

\begin{remark}
Note that in the proof we only used the fact that $\bE[|X|]$ is upper bounded. Thus the proof applies to any $\GG$ with unbounded $\theta$ and upper bounded $\bE[|X|]$. 
\end{remark}

The proof also shows that the term $\bar R^2 \eta^2$ is necessary in the upper bound in Theorem~\ref{thm.linreg_tw_1_proj}.

\subsection{Hypercontractive-type sufficient conditions for linear regression under $W_1$ perturbation}\label{appendix.proof_W1_linreg_sufficient}

Here we show another set of sufficient conditions for $W_1$ linear regression to have finite population limit. The condition mimics the requirement of hyper-contractivity in $\TV$ linear regression case (Theorem~\ref{thm.linearregressiontvtildeproof}).

\begin{example}[Linear regression under $W_1$ perturbation]\label{example.linreg_w1}
Take $B(p, \theta) = L(p, \theta)  = \bE_p[(Y - X^\top \theta)^2] $.
Denote by $X' = [X, Y]$ the $d+1$ dimensional vector that concatenates $X$ and $Y$. Denote $Z = Y-X^{\top}\theta^*(p^*)$, where $\theta^*(p^*) \triangleq \argmin_{\theta} B(p^*,\theta)$. Given an Orlicz function
     $\psi$
that further satisfies $\psi(x)\geq x$ for all $x \geq 1$.
Assume $p^*$ satisfies:
\begin{align}\label{eqn.w1_linreg_condition1}
   \sup_{v\in\bR^{d+1}, \|v\|_2=1}  \bE_{p^*}\bigg[\frac{|v^{\top}X'|}{\kappa} \psi\bigg(\frac{\kappa|v^{\top}X'|}{\mathbb{E}[(v^\top X')^2]}\bigg)\bigg] \leq 1, \\
   \bE_{p^*}[Z^2] \leq \sigma^2.
\end{align} 
Assuming $\frac{\eta\psi^{-1}(2\kappa/\eta)}{\kappa}<1$, we have
$p^* \in\GG^{W_1}(\sigma^2 +4\eta^2 + 2\sigma\eta,  (\sigma^2 +4\eta^2 + 2\sigma\eta)/(1- \eta\psi^{-1}(2\kappa/\eta)/{\kappa}), \eta)$.
The population limit for this set when the perturbation level is $\epsilon$  is upper bounded by $(\sigma^2 +16\epsilon^2 + 4\sigma\epsilon)/(1- 2\epsilon\psi^{-1}(\kappa/\epsilon)/{\kappa})$.
\end{example}

\begin{proof}
We know that $\GG^{W_1} = \GG_{\downarrow}^{W_1}(\rho_1, \eta) \bigcap \GG_{\uparrow}^{W_1}(\rho_1, \rho_2, \eta)  $, where 
\begin{align}
    \GG_{\downarrow}^{W_1}(\rho_1, \eta) =  \{p \mid & \sup_{r \in \mathbb{F}(p, \eta, W_{c, k}, |X^{\top}\theta^*(p) - Y|^2) } \bE_r[(X^{\top}\theta^*(p) - Y)^2 ] \leq \rho_1 \}, \\
    \GG_{\uparrow}^{W_1}(\rho_2) =   \Bigg\{p \mid  & \forall \tau\geq 0, \forall \theta \in \Theta,  \forall r\in \mathbb{F}(p, \eta, W_{c, k}, |X^{\top}\theta - Y| ),\nonumber \\
    & \Big( \bE_r[(X^{\top}\theta - Y)^2  ]  \leq \rho_1  \Rightarrow  \bE_p[(X^{\top}\theta - Y)^2 ] \leq \rho_2  \Big) \Bigg\}. 
\end{align}

Since we already know that $\bE_{p^*}[(X^\top\theta - Y)^2] \leq \sigma_2^2$, it suffices to show that its friendly perturbation $r$ cannot drive it much larger to show that $p^* \in \GG_\downarrow$. 

Now we verify the condition in $\GG_{\downarrow}$. Firstly, it suffices to consider all coupling that is making $\bE[(X^\top \theta^*(p)- Y)^2]$ larger to guarantee $p^*\in\GG_{\downarrow}$. From the argument in second moment estimation (Equation~\eqref{eqn.moving_up_bound}, Appendix~\ref{proof.psi_w1_kth}), we know that for $r$ that only moves everything larger, we have
\begin{align}
    \bE_{r}[(X^{\top}\theta^*(p) - Y)^2] \leq 4\eta^2 + 2\sigma\eta + \sigma^2. 
\end{align}
Thus we have $p^* \in \GG_\downarrow(\rho_1, \eta)$ with $\rho_1 = 4\eta^2 + 2\sigma\eta + \sigma^2$. 

Now we verify that $p^* \in \GG_{\uparrow}$. It suffices to consider all $r$ that moves $p^*$ downwards. 
From Equation~\eqref{eqn.proof_w1_orlicz_3} and the assumption, we  also know that 
\begin{align}
\sup_{v\in\bR^{d+1}, \|v \|_2 =1, r \in  \mathbb{F}({p^*}, \epsilon, W_{c, k}, |v^{\top}X'|^2 )}
    \bE_{p^*}[(v^{\top}X')^2] - \bE_r[(v^{\top}X')^2] \leq \frac{\eta\psi^{-1}(2\kappa/\eta)}{\kappa} \bE_{p^*}[(v^{\top}X')^2].
\end{align}
Assume $\frac{\eta\psi^{-1}(2\kappa/\eta)}{\kappa}<1$, then we have
\begin{align}
    \bE_{p^*}[(v^{\top}X')^2] \leq \frac{ \bE_r[(v^{\top}X')^2]}{1- \frac{\eta\psi^{-1}(2\kappa/\eta)}{\kappa}}.
\end{align}

Thus we have $p^* \in\GG^{W_1}(\sigma^2 +4\eta^2 + 2\sigma\eta,  (\sigma^2 +4\eta^2 + 2\sigma\eta)/(1- \frac{\eta\psi^{-1}(2\kappa/\eta)}{\kappa}), \eta)$.

\end{proof}

\begin{remark}
Taking $\psi(x) = x^2$, we see that if the vector $X'=[X, Y]$ satisfies
\begin{align}
   \kappa \bE_{p^*}[|v^\top X'|^3] \leq     \bE_{p^*}[|v^\top X'|^2]^2
\end{align}
for any $\|v\|_2=1$, 
the population limit is upper bounded by $(\sigma^2 +4\eta^2 + 2\sigma\eta)/(1- \sqrt{2\eta/\kappa})$ assuming $\eta<\kappa/2$. 
Thus if we know the original distribution's optimal prediction error is bounded by $\sigma^2$, then our estimator's risk  approaches $\sigma^2$ as $\epsilon$ goes to $0$. 
\end{remark}

\textbf{Discussion on the first condition.}

Taking $\psi(x) = x^k$, the first condition becomes  $
{\kappa^{k-1} \bE_{p^*}[|v^\top X'|^{k+1}]} \leq {\bE_{p^*}[(v^\top X')^2]^k}$. We show that it is satisfied when $X, Z$ are independent Gaussians with variance lower bounded. 
Assume $X \sim \mathcal{N}(\mu, \Sigma)$ and $Z \sim \mathcal{N}(0, \sigma^2)$. Then we know $Y \sim \mathcal{N}(\mu^\top \theta, \theta^\top \Sigma\theta + \sigma^2)$. Thus $X'$ is also a Gaussian distribution with mean $[\mu, \mu^\top \theta]$ and covariance $\Sigma' = \left[\begin{matrix}
\Sigma & \Sigma \theta \\
(\Sigma \theta)^\top & \theta^\top \Sigma\theta + \sigma^2
\end{matrix}\right]$. From~\citep[Corollary 5.21]{boucheron2013concentration} we know that $\bE_{p^*}[|v^\top X'|^{k+1}] \leq C\bE_{p^*}[|v^\top X'|^{2}]^{(k+1)/2}$ for some constant $C>1$ that may depend on $k$. Thus it suffices to take $\kappa \leq \lambda_{n+1}(\Sigma')$ as the smallest eigenvalue of $\Sigma'$. Note here by assuming $\lambda_{n+1}(\Sigma')$ is bounded away from $0$, we implicitly impose some assumption between $\sigma$ and $\theta$ which gives upper bound on $\theta$.

The first condition is an analogy to the hyper-contractivity condition in $\TV$ case, which prevents the deletion of dimension, i.e. it guarantees for some $f(\eta, \kappa)>0$, when $\frac{\eta\psi^{-1}(2\kappa/\eta)}{\kappa}<1$, the following holds:
\begin{align}
   \forall v\in\bR^d,     \bE_{q}[(v^\top X')^2] \geq f(\eta, \kappa)  \bE_{p^*}[(v^\top X')^2],
\end{align}
where $q$ is any distribution such that $W_1(p^*, q)\leq \eta$. It is also worth noting that for linear regression under $\TV$ perturbation, we only require hypercontractivity for $X$; while for $W_1$ perturbation we need similar condition for $X' = [X, Y]$, and when $Z \equiv 0$ the condition fails to hold. In fact, under $W_1$ perturbation, appropriate noise is \emph{necessary} to guarantee controlled population limit:  we show that the population limit is infinity even when $X$ is a Gaussian (thus is hyper-contractive) and the noise $Z$ is $0$ in Appendix~\ref{appendix.necessity_hyper_w1}. This also  shows the necessity of having the first condition on $X'$ instead of $X$.  All the above discussion holds when we switch from  predictive loss to excess predictive loss.

\subsection{Proof of Theorem~\ref{thm.efficient_sec_w1}}\label{proof.efficient_sec_w1}

\begin{proof}
From Lemma~\ref{lem.kth_population_deletion_R}, we know that
 for $X\sim p^*$, with probability at least $1-\delta$, 
$\|X\|_2 \leq \frac{\sigma \sqrt{d}}{\delta^{1/k}}$.
Combining this with Lemma~\ref{lem.matrix_bernstein}, we know that with probability at least $1-2\delta$,
\begin{align}
    \|\bE_{\hat p_n^*}[XX^\top] - \bE_{p^*}[XX^\top] \|_2\leq C_1 \sigma^2 \max(\sqrt{\frac{ d\log(d/\delta)}{n\delta^{2/k}}},  \frac{ d\log(d/\delta)}{n\delta^{2/k}}).
\end{align}
From the projection algorithm, we know that either $W_1(\hat p_n^*, q)\leq 2\epsilon$ or $\tilde W_1(\hat p_n^*, q)\leq 2\epsilon$.  When $n\gtrsim (d\log(d/\delta))^{k/2}$, we have $\hat p_n^*\in\GG$ with probability at least $1-\delta$~\citep[Lemma 5.5]{kothari2017outlier}. Combining it with the bounded modulus of the set $\GG$ under either $W_1$ or $\tilde W_1$, we have
\begin{align}
     \|\bE_{\hat p_n^*}[XX^\top] - \bE_{q}[XX^\top] \|_2 \leq \min(\sigma^2, C_2 \sigma^{1+1/(k-1)}  \epsilon^{1-1/(k-1)}).
\end{align}
Then triangle inequality gives the final result. 
\end{proof}

\subsection{General proposition of $\tW_1$ projection algorithm}\label{proof.tw1_projection}

We first slightly generalize 
$\GG^{W_{c, k}}$ in Definition~\ref{def.G^Wc} to  $ \mathcal{G}^{W_{c, k}}(\rho_1, \rho_2, \eta) = \mathcal{G}^{W_{c, k}}_{\downarrow}(\rho_1, \eta) \cap \mathcal{G}^{W_{c, k}}_{\uparrow}(\rho_1, \rho_2, \eta)$, where
\begin{align}
    & \GG_{\downarrow}^{W_{c, k}}(\rho_1, \eta) = \{p \mid   \sup_{f\in\sF_{\theta^*(p)}, g\in \mathcal{P}(f), r\in \mathbb{F}(p, \eta, W_{c, k}, g)} \bE_r[f(X)] - B^*(f, \theta^*(p)) \leq \rho_1  \}, \\
    & \GG_{\uparrow}^{W_{c, k}}(\rho_1, \rho_2, \eta) =  \Bigg\{p \mid   \forall \theta,  \Bigg( \Big(  \sup_{f\in \sF_\theta, g\in \mathcal{P}(f)} \inf_{r\in \mathbb{F}(p, \eta, W_{c, k}, g)} \bE_r[f(X)] - B^*(f, \theta) \leq \rho_1 \Big) \Rightarrow  L(p, \theta) \leq \rho_2  \Bigg) \Bigg\}. 
\end{align}

Here the only difference between Definition~\ref{def.G^Wc} is that we allow the friendly perturbation has different projection function $g$ rather than $f$, which is needed in Theorem~\ref{thm.sec_tw_1_proj}. In second moment cases, we have $f(x)=(v^\top x)^2$ while $g(x)=|v^\top x|$. It follows the same proof as  Theorem~\ref{thm.G_Wc_fundamental_limit} that the modulus of continuity for the above set is bounded. 

In this section, we make the assumption that $\sF_\theta$ contains only the function of the form $f(|v^\top x|)$, where all the $f$ are convex functions.
We show the robustness guarantee for general projection algorithm 
$q = \Pi(\hat p_n; \tW_1, \GG^{W_1})$.  
The following proposition is a corollary of Theorem~\ref{thm.admissible-1}. 
\begin{proposition}\label{prop.tW1_general}
Under the oblivious corruption model of level $\epsilon$ with $\tW_1$ perturbation, where $\tW_1$ is defined in~(\ref{eqn.tw_def}).
Assume $p^*$ satisfies that 
\begin{align}
    \sup_{v\in\bR^d, \|v\|_2=1} \bE_{p^*}[\psi(|v^\top X|/\kappa)] & \leq 1, \\
    \|\Sigma_{p^*} \|_2& \leq \sigma^2.
\end{align}
Denote the empirical distribution of observed data as $\hat p_n$, and
\begin{align}
    \tilde \epsilon = 2(1+1/\delta)\epsilon +  \frac{8\sigma}{\delta} \sqrt{\frac{d}{n}}+\frac{3\kappa \psi^{-1}(\sqrt{n})}{\delta\sqrt{n}}.
\end{align}
Suppose $g(x) = $ either $v^\top x$ or $|v^\top x|$. 
When we take the projection set $\GG = \GG^{W_1}(\rho_1(7\tilde \epsilon), \rho_2(7\tilde \epsilon), 7\tilde \epsilon)$ in Definition~\ref{def.G^Wc}, where all functions in $\sF_\theta$ are of the form of $f(g(x))$ for some convex function $f$, and $\mathcal{P}(f(g(x))) = \{g(x)\} $,
we have with probability at least $1-\delta$, the projection algorithm $\Pi(\hat{p}_n; \tW_1, \GG^{W_1})$ or $\Pi(\hat{p}_n; \tW_1, \GG^{W_1}, \tilde{\epsilon}/2)$ satisfies
\begin{align}
     L(p^*, \theta^*(q)) \leq \rho_2(7\tilde \epsilon).
\end{align}
\end{proposition}

\begin{proof}
The two conclusions are all  corollaries of Theorem~\ref{thm.admissible-1}. 
We only need to verify the two conditions in Theorem~\ref{thm.admissible-1}. 
\begin{enumerate}
    \item \textbf{Robust to perturbation}: Note that $\tW_1$ satisfies triangle inequality and $\mathcal{U}' \subset \mathcal{U}$. For any $p_1, p_2, p_3$, we have
    \begin{align}
            |\tW_1(p_1, p_2) - \tW(p_1, p_3)| & = \sup_{u\in\mathcal{U}'}  |\bE_{p_1}[u(X)] -  \bE_{p_2}[u(X)]| -  \sup_{u\in\mathcal{U}'} |\bE_{p_1}[u(X)] -  \bE_{p_3}[u(X)]|   \nonumber \\
            & \leq \sup_{u\in\mathcal{U}'} | \bE_{p_2}[u(X)] - \bE_{p_3}[u(X)]| \nonumber \\
            & = \tW_1(p_2, p_3) \nonumber \\
            & \leq W_1(p_2, p_3).
    \end{align}
    \item \textbf{Generalized Modulus of Continuity}: Assume $p_1, p_2\in\GG^{W_1}(\rho_1(7\tilde \epsilon), \rho_2(7\tilde \epsilon), 7\tilde \epsilon)$ and $\tW_1(p_1, p_2)\leq  \tilde \epsilon$. For any  fixed $v$, from Lemma~\ref{lem:tw1_cross_mean}, we know that when $g(x)$ takes either $v^\top x$ or $|v^\top x|$,  there exists an $r_{p_1}\in \mathbb{F}(p_1, 7\tilde \epsilon, W_1, g(x))$ and an $r_{p_2} \in \mathbb{F}(p_2,  7\tilde \epsilon, W_1, g(x))$ such that for convex $f(x)$, we have
    \begin{align}
        \bE_{r_{p_1}}[f(g(X)) ] \leq \bE_{r_{p_2}} [f(g(X))].
    \end{align}
    
 From $p_2 \in \GG^{W_1}_\downarrow$, we know that
    \begin{align}
        \sup_{f\in\sF_{\theta^*(p_2)}, g\in \mathcal{P}(f), r\in \mathbb{F}(p_2, 7\tilde \epsilon, W_{c, k}, g(x))} \bE_{r}[f(X)] - B^*(f, \theta^*(p_2)) \leq \rho_1(7\tilde \epsilon).
    \end{align}
Thus we know for any given $f, g$, and any friendly perturbation $r_{p_2}$, the above inequality holds. Thus we also have
\begin{align}
    \bE_{r_{p_1}}[f(X)] - B^*(f, \theta^*(p_2)) \leq \bE_{r_{p_2}}[f(X)] - B^*(f, \theta^*(p_2)) \leq \rho_1(7\tilde \epsilon).
\end{align}
Rewriting the above statement, we know that
 \begin{align}
 \sup_{f\in\sF_{\theta^*(p_2)}, g\in \mathcal{P}(f)}\inf_{ r\in \mathbb{F}(p, 7\tilde \epsilon, W_{c, k}, g(x))} \bE_{r}[f(X)] - B^*(f, \theta^*(p_2)) 
\leq \leq \rho_1(7\tilde \epsilon).
    \end{align}
Since we also know that $p_1 \in \GG^{W_c}_\uparrow$, we have
\begin{align}
    L(p_1, \theta^*(p_2))\leq \rho_2(7\tilde \epsilon)
\end{align}
\end{enumerate}
Thus from Theorem~\ref{thm.admissible-1}, we have with probability at least $1-\delta$,
\begin{align}
      L(p^*, \theta^*(q)) \leq \rho_2(7\tilde \epsilon).
\end{align}

We provide ways to bound the statistical error term $\tW_1(\hat p_n, p)$ in Lemma~\ref{lem.tW_1_empirical_to_population}. Combining these two lemmas gives the results.
\end{proof}

\subsection{Estimating $k$-th moment projection}\label{proof.psi_w1_kth}

We extend the result above to estimating 1d projection of $k$-th moment.

\begin{example}[Bounded Orlicz norm implies resilience for $k$-th moment estimation under $W_1$ perturbation]\label{example.psi_w1_kth_resilience}
For $k>1$,
taking $\mathcal{F} = \{f\mid X \mapsto \xi |v^{\top}X|^k, v\in\bR^d, \|v\| =1, \xi \in\{\pm 1\} \}$, $W_{c, k} = W_1$ in Equation (\ref{eqn.def_Wc_meanG}), one can check that Assumption~\ref{ass:topo} holds. The  set $\GG^{W_1}_{W_\sF}$ becomes 
\begin{align}
    \GG^{W_1}_{\mathsf{kth}}(\rho, \eta) = \{p \mid \sup_{ v\in\bR^d, \|v\| = 1, r\in \mathbb{F}(p, \eta, W_1, |v^{\top}X|^k)} |\bE_{p}[|v^{\top}X|^k] - \bE_{r}[|v^{\top}X|^k]| \leq \rho  \}.
\end{align}

Assume that there exists some Orlicz function $\psi$ which satisfies $\psi(x)\geq x, \forall x \geq 1$. Denote $\tilde \psi(x) = x\psi(kx^{k-1})$. We assume that 
\begin{align}\label{eqn.proof_assumption_orlicz_w1}
   \sup_{v\in\bR^d, \|v\| =1} \bE_p\left[\tilde \psi\left(\frac{|v^{\top}X|}{\sigma}\right)\right] \leq 1.
\end{align}
 Then,
\begin{align}
        p \in \GG_{\mathsf{kth}}^{W_1}(\sigma^{k-1} \eta \psi^{-1}(\frac{2\sigma}{\eta}), \eta), \forall \eta<\min({\sigma}/{2k}, 2\sigma / \psi(\max(k, 8))),
\end{align}
where $\psi^{-1}$ is the (generalized) inverse of $\psi$.
\end{example}
\begin{proof}
From the fact that $r \in \mathbb{F}(p, \eta, W_1, |v^{\top}X|^k)$, we know that for any coupling $\pi_{p, r}$ that makes $r$ friendly perturbation, we have 
\begin{align}
  \sup_{v\in\bR^d, \|v\|_2=1} \bE_{\pi_{p, r}} |v^\top(X - Y)|\leq  \bE_{\pi_{p, r}} \sup_{v\in\bR^d, \|v\|_2=1} |v^\top(X - Y)| \leq  \eta.
\end{align}
For any fixed $v\in\bR^d, \|v\|_2 = 1$, we claim that the worst perturbation only happens when for any $(x, y) \in \mathsf{supp}(\pi_{p, r})$, $|v^\top x|^k \geq |v^\top y|^k$ or for any $(x, y) \in \mathsf{supp}(\pi_{p, r})$, $|v^\top x|^k \leq |v^\top y|^k$. If it is not one of the two cases, we can always remove the movement from $x$ to $y$ that decreases or increases $g$ to make $|\bE_\pi[|v^\top X|^k - |v^\top Y|^k]|$ larger without increasing $\bE_\pi\|X-Y\|$.

Thus we can assume for any $(x, y) \in \mathsf{supp}(\pi_{p, r})$, $|v^\top x|^k \geq |v^\top y|^k$ or for any $(x, y) \in \mathsf{supp}(\pi_{p, r})$, $|v^\top x|^k \leq |v^\top y|^k$. For the first case,  by Lemma~\ref{lem.coupling_g_bounded}, we bound the worst case perturbation as follows. For any $v\in\bR^d, \|v\|_2=1$,
\begin{align}
    \quad &|\bE_{(X, Y) \sim \pi_{p, r}}[|v^\top X|^k - |v^\top Y|^k]| \nonumber \\
    &\leq \sigma^{k-1} \bE_{\pi_{p, r}}[|v^\top (X-Y)|] \psi^{-1}\left(\frac{\bE_{\pi_{p, r}}\Big[|v^\top (X-Y)| \psi\left(\left|\frac{|v^\top X|^k - |v^\top Y|^k}{\sigma^{k-1} v^\top (X-Y)}\right|\right)\Big]}{ \bE_{\pi_{p, r}}[|v^\top (X-Y)|]}\right) \nonumber \\
    & \leq \sigma^{k-1} \eta \psi^{-1}\left(\frac{\bE_{\pi_{p, r}}\Big[|v^\top (X-Y)| \psi\left(\left|\frac{|v^\top X|^k - |v^\top Y|^k}{\sigma^{k-1} v^\top (X-Y)}\right|\right)\Big]}{\eta}\right)\label{eqn.proof_w1_orlicz_1}  \\
    & \leq \sigma^{k-1} \eta \psi^{-1}\left(\frac{\bE_{\pi_{p, r}}\Big[|v^\top (X-Y)| \psi\left(\left|\frac{k|v^\top X|^{k-1}}{\sigma^{k-1}}\right|\right)\Big]}{\eta}\right) \label{eqn.proof_w1_orlicz_2}  \\
    & \leq \sigma^{k-1} \eta \psi^{-1}\left(\frac{\bE_{\pi_{p, r}}\Big[2|v^\top X| \psi\left(\left|\frac{k|v^\top X|^{k-1}}{\sigma^{k-1}}\right|\right)\Big]}{\eta}\right)  \\
    & = \sigma^{k-1} \eta \psi^{-1}\left(\frac{\bE_p\Big[2|v^\top X| \psi\left(\left|\frac{k|v^\top X|^{k-1}}{\sigma^{k-1}}\right|\right)\Big]}{\eta}\right)  \\
    & \leq \sigma^{k-1} \eta \psi^{-1}(\frac{2\sigma}{\eta})\label{eqn.proof_w1_orlicz_3} . 
\end{align}
Here Equation (\ref{eqn.proof_w1_orlicz_1}) comes from the fact that $x\psi^{-1}(C / x) $ is a non-decreasing function of $x$ for the region $[0,+\infty)$ for any $\sigma >0$ (Lemma~\ref{lem.psi_increasing}). Equation (\ref{eqn.proof_w1_orlicz_2}) uses the fact that  for any $(x, y) \in \mathsf{supp}(\pi_{p, r})$, $|v^\top x| \geq |v^\top y|$. Equation (\ref{eqn.proof_w1_orlicz_3}) is from the assumption in (\ref{eqn.proof_assumption_orlicz_w1}).

On the other hand, for any $(x, y) \in \mathsf{supp}(\pi_{p, r})$, $|v^\top x|^k \leq |v^\top y|^k$, we bound the worst case perturbation as follows:
\begin{align}
     |\bE_{(X, Y) \sim \pi_{p, r}}[|v^\top X|^k - |v^\top Y|^k]| & =  |\bE_{(X, Y) \sim \pi_{p, r}}[(|v^\top X| - |v^\top Y|)(\sum_{i=0}^{k-1}|v^\top X|^i |v^\top Y|^{k-1-i})]| \\
     & \leq k|\bE_{(X, Y) \sim \pi_{p, r}}[(|v^\top X| - |v^\top Y|)]\bE[|v^\top Y|^{k}]^{1-1/k}|. 
     \\
     & \leq k\eta \bE[|v^\top Y|^{k}]^{1-1/k}. \label{eqn.proof_W1_resilience1}
\end{align}
Here we use the fact that 
For any $(x, y) \in \mathsf{supp}(\pi_{p, r}), x \neq y$,  we have $|v^\top x|^k\leq |v^\top y|^k \leq \bE[|v^\top Y|^k]$ from the definition of friendly perturbation.

If $\bE[|v^\top Y|^k] \leq \sigma^k$, we have $ \bE_{(X, Y) \sim \pi_{p, r}}[|v^\top Y|^k - |v^\top X|^k] \leq k \sigma^{k-1} \eta$. Otherwise we further upper bound it to get
\begin{align}
     \bE_{(X, Y) \sim \pi_{p, r}}[|v^\top Y|^k - |v^\top X|^k] \leq  k\sigma^{k-1}\eta \bE\left[\left|\frac{v^\top Y}{\sigma}\right|^k\right].
\end{align}
Solving the inequality, we know that when $\eta < \sigma/k$,
\begin{align}
    \bE[|v^\top Y|^k] \leq \frac{\bE[|v^\top X|^k] }{1-k \eta/\sigma}.\label{eqn.proof_W1_resilience2}
\end{align}

Now we bound the term $\bE[|v^\top X|^k]$. 
When $k|v^\top x|^{k-1} / \sigma^{k-1}\leq 1$, we have $  |v^\top x|^k \leq \sigma^k / k^{k/(k-1)}$. When $k|v^\top x|^{k-1} / \sigma^{k-1} > 1 $, by assumption $\psi(x)\geq x, \forall x \geq 1$, we have $ |v^\top x|^k \leq \frac{\sigma^{k-1}|v^\top x|}{k} \psi(\frac{k|v^\top X|^{k-1}}{\sigma^{k-1}}) $.
\begin{align}
    \bE[|v^\top X|^k] & \leq \frac{\sigma^{k}}{k^{k/(k-1)}} + \frac{\sigma^{k-1} \bE[|v^\top X|\psi(\frac{k|v^\top X|^{k-1}}{\sigma^{k-1}})]}{k} \nonumber \\
    & \leq \frac{\sigma^{k}}{k^{k/(k-1)}} + \frac{\sigma^{k}}{k} \nonumber \\
    & \leq \frac{ 2\sigma^{k}}{k}.
\end{align}
Thus combining Equation~\eqref{eqn.proof_W1_resilience1} and~\eqref{eqn.proof_W1_resilience2}, we have if $\eta < \sigma/2k$,
\begin{align*}
    \bE[|v^\top Y|^k] - \bE[|v^\top X|^k]\leq & {k\eta \bE[|v^\top Y|^k]^{1-1/k}} \leq  k\eta \left(\frac{\bE[|v^\top X|^k] }{1-k \eta/\sigma}\right)^{1-1/k} \nonumber \\ 
    \leq & k\eta \left(\frac{2\sigma^k }{k(1-k \eta/\sigma)}\right)^{1-1/k} \leq 8\sigma^{k-1}\eta.
\end{align*}

Combining the two cases, we know that when $\eta < \sigma/2k$, the movement is upper bounded by $\max(\sigma^{k-1}\eta\psi^{-1}(\frac{2\sigma}{\eta}), \max(k, 8)\sigma^{k-1}\eta )$ 
Thus 
for $\eta<\min({\sigma}/{2k}, 2\sigma / \psi(\max(k, 8)))$, we have  $p \in \GG_{\mathsf{kth}}^{W_1}(\sigma^{k-1} \eta \psi^{-1}(\frac{2\sigma}{\eta}), \eta)$.

We remark here that the above proof also applies to the case when we requires $r\in\mathbb{F}(p, \eta, W_1, |v^\top X|)$ instead of $r\in\mathbb{F}(p, \eta, W_1, |v^\top X|^k)$. The only difference to note is in above~\eqref{eqn.proof_W1_resilience1} where we need to apply Jensen's inequality to derive $|v^\top y|^k \leq (\bE[|v^\top Y|])^k \leq \bE[|v^\top Y|^k]$.
This proves to be crucial in finite sample algorithm design in Section~\ref{sec.finite_sample_w1}.
\end{proof}

\end{appendix}

\end{document}